\documentclass{amsart}
\usepackage{amssymb}
\usepackage{mathrsfs}
\usepackage{stmaryrd}
\usepackage{imakeidx}
\usepackage[colorlinks = true]{hyperref}
\usepackage{enumitem}
\usepackage{comment}
\input xy \xyoption {all}
\usepackage{tikz}
\usetikzlibrary{cd}

\usepackage{dsfont}

\usepackage{rotating}
\usepackage{lscape}

\mathchardef\mhyphen="2D

\DeclareMathOperator{\sEnd}{\mathscr{E}\hspace{-1pt}\mathit{nd}}

\newcommand{\Z}{\mathbb{Z}}

\newcommand{\F}{\mathbb{F}}

\newcommand{\bk}{\Bbbk}

\newcommand{\Gm}{\mathbb{G}_{\mathrm{m}}}

\newcommand{\Bimod}{\mathsf{Bimod}}
\newcommand{\Mod}{\mathsf{Mod}}

\newcommand{\Modqc}{\mathsf{Mod}_{\mathrm{qc}}}
\newcommand{\Modc}{\mathsf{Mod}_{\mathrm{c}}}
\newcommand{\Modfg}{\mathsf{Mod}_{\mathrm{fg}}}

\newcommand{\Kb}{K^{\mathrm{b}}}

\newcommand{\Db}{D^{\mathrm{b}}}

\newcommand{\fh}{\mathfrak{h}}

\newcommand{\uw}{{\underline{w}}}

\newcommand{\Tilt}{\mathsf{Tilt}}

\newcommand{\hatstar}{\mathbin{\widehat{\star}}}

\newcommand{\cL}{\mathcal{L}}

\newcommand{\cN}{\mathcal{N}}

\newcommand{\cZ}{\mathcal{Z}}

\newcommand{\Av}{\mathsf{Av}}
\newcommand{\For}{\mathsf{For}}
\newcommand{\IC}{\mathscr{I}\hspace{-1pt}\mathscr{C}}
\newcommand{\id}{\mathrm{id}}

\newcommand{\simto}{\xrightarrow{\sim}}

\DeclareMathOperator{\Hom}{Hom}
\DeclareMathOperator{\Ext}{Ext}
\DeclareMathOperator{\End}{End}

\newcommand{\Iw}{\mathrm{I}}
\newcommand{\Iwu}{\mathrm{I}_{\mathrm{u}}}
\newcommand{\Fl}{\mathrm{Fl}}
\newcommand{\tFl}{\widetilde{\mathrm{Fl}}}
\newcommand{\Gr}{\mathrm{Gr}}
\newcommand{\IW}{\mathcal{IW}}

\newcommand{\bG}{\mathbf{G}}
\newcommand{\bH}{\mathbf{H}}
\newcommand{\bK}{\mathbf{K}}

\newcommand{\bB}{\mathbf{B}}
\newcommand{\bP}{\mathbf{P}}
\newcommand{\bL}{\mathbf{L}}

\newcommand{\bU}{\mathbf{U}}
\newcommand{\bT}{\mathbf{T}}
\newcommand{\bZ}{\mathbf{Z}}

\newcommand{\bW}{\mathbf{W}}
\newcommand{\bWf}{\mathbf{W}_{\mathrm{f}}}
\newcommand{\Wf}{W_{\mathrm{f}}}
\newcommand{\bS}{\mathbf{S}}

\newcommand{\bSf}{\mathbf{S}_{\mathrm{f}}}
\newcommand{\Sf}{S_{\mathrm{f}}}

\newcommand{\aff}{\mathrm{aff}}
\newcommand{\bWaff}{\bW_{\aff}}
\newcommand{\bSaff}{\mathbf{S}_{\mathrm{aff}}}
\newcommand{\Saff}{S_{\mathrm{aff}}}
\newcommand{\bu}{\mathbf{u}}
\newcommand{\bg}{\mathbf{g}}
\newcommand{\bb}{\mathbf{b}}
\newcommand{\bt}{\mathbf{t}}
\newcommand{\bp}{\mathbf{p}}
\newcommand{\bl}{\mathbf{l}}

\newcommand{\hatotimes}{\mathbin{\widehat{\otimes}}}

\makeatletter
\def\lotimes{\@ifnextchar_{\@lotimessub}{\@lotimesnosub}}
\def\@lotimessub_#1{\mathchoice{\mathbin{\mathop{\otimes}^L}_{#1}}%
  {\otimes^L_{#1}}{\otimes^L_{#1}}{\otimes^L_{#1}}}
\def\@lotimesnosub{\mathbin{\mathop{\otimes}^L}}
\makeatother

\makeatletter
\def\rtimes{\@ifnextchar_{\@rtimessub}{\@rtimesnosub}}
\def\@rtimessub_#1{\mathchoice{\mathbin{\mathop{\times}^R}_{#1}}%
  {\times^R_{#1}}{\times^R_{#1}}{\times^R_{#1}}}
\def\@rtimesnosub{\mathbin{\mathop{\times}^R}}
\makeatother

\makeatletter
\def\lboxtimes{\@ifnextchar_{\@lboxtimessub}{\@lboxtimesnosub}}
\def\@lboxtimessub_#1{\mathchoice{\mathbin{\mathop{\boxtimes}^L}_{#1}}%
  {\boxtimes^L_{#1}}{\boxtimes^L_{#1}}{\boxtimes^L_{#1}}}
\def\@lboxtimesnosub{\mathbin{\mathop{\boxtimes}^L}}
\makeatother

 \makeatletter
 \def\hatotimes{\@ifnextchar_{\@hatotimessub}{\@hatotimesnosub}}
 \def\@hatotimessub_#1{\mathchoice{\mathbin{\mathop{\widehat{\otimes}}}_{#1}}{\widehat{\otimes}_{#1}}{\widehat{\otimes}_{#1}}{\widehat{\otimes}_{#1}}}
 \def\@lotimesnosub{\mathbin{\mathop{\widehat{\otimes}}}}
 \makeatother

\newcommand{\scO}{\mathscr{O}}

\newcommand{\pH}{{}^{\mathrm{p}} \hspace{-1.5pt} \mathscr{H}}
\newcommand{\pcH}{{}^{\mathrm{pc}} \hspace{-1.5pt} \mathscr{H}}
\newcommand{\Rep}{\mathsf{Rep}}

\newcommand{\Coh}{\mathsf{Coh}}
\newcommand{\QCoh}{\mathsf{QCoh}}
\newcommand{\BSCoh}{\mathsf{BSCoh}}
\newcommand{\SCoh}{\mathsf{SCoh}}

\newcommand{\st}{\mathsf{t}}

\newcommand{\Wak}{\mathscr{W}}

\newcommand{\fm}{\mathfrak{m}}
\newcommand{\fn}{\mathfrak{n}}

\newcommand{\sfC}{\mathsf{C}}
\newcommand{\sfD}{\mathsf{D}}

\newcommand{\sfN}{\mathsf{N}}
\newcommand{\sfP}{\mathsf{P}}
\newcommand{\sfR}{\mathsf{R}}
\newcommand{\sfL}{\mathsf{L}}
\newcommand{\sfT}{\mathsf{T}}

\newcommand{\scA}{\mathscr{A}}
\newcommand{\scB}{\mathscr{B}}
\newcommand{\scC}{\mathscr{C}}

\newcommand{\scF}{\mathscr{F}}
\newcommand{\scG}{\mathscr{G}}
\newcommand{\scH}{\mathscr{H}}
\newcommand{\scI}{\mathscr{I}}
\newcommand{\scJ}{\mathscr{J}}

\newcommand{\scL}{\mathscr{L}}
\newcommand{\scM}{\mathscr{M}}

\newcommand{\scR}{\mathscr{R}}
\newcommand{\scS}{\mathscr{S}}

\newcommand{\scX}{\mathscr{X}}

\newcommand{\scZ}{\mathscr{Z}}

\newcommand{\St}{\mathbf{St}}
\newcommand{\Stm}{\mathbf{St}_{\mathrm{m}}}

\newcommand{\cU}{\mathcal{U}}

\newcommand{\SRep}{\mathsf{SRep}}

\newcommand{\Spec}{\mathrm{Spec}}
\newcommand{\bbJ}{\mathbb{J}}
\newcommand{\bbI}{\mathbb{I}}

\newcommand{\bbM}{\mathbb{M}}

\newcommand{\op}{\mathrm{op}}

\newcommand{\ZFr}{Z_{\mathrm{Fr}}}
\newcommand{\ZHC}{Z_{\mathrm{HC}}}

\newcommand{\scD}{\mathscr{D}}
\newcommand{\cD}{\mathcal{D}}
\newcommand{\tD}{\widetilde{\mathcal{D}}}
\newcommand{\oD}{\overline{\mathcal{D}}}
\newcommand{\tsD}{\widetilde{\mathscr{D}}}
\newcommand{\osD}{\overline{\mathscr{D}}}

\newcommand{\fR}{\mathfrak{R}}
\newcommand{\ola}{\overline{\lambda}}
\newcommand{\tla}{\widetilde{\lambda}}
\newcommand{\hla}{\widehat{\lambda}}
\newcommand{\omu}{\overline{\mu}}
\newcommand{\tmu}{\widetilde{\mu}}
\newcommand{\hmu}{\widehat{\mu}}
\newcommand{\onu}{\overline{\nu}}
\newcommand{\tnu}{\widetilde{\nu}}
\newcommand{\hnu}{\widehat{\nu}}

\newcommand{\tbg}{\widetilde{\bg}}
\newcommand{\tbG}{\widetilde{\bG}}
\newcommand{\fRs}{\fR_{\mathrm{s}}}
\newcommand{\tcN}{\widetilde{\mathcal{N}}}

\newcommand{\FN}{\mathrm{FN}}
\newcommand{\Fr}{\mathrm{Fr}}
\newcommand{\BSHC}{\mathsf{BSHC}}
\newcommand{\SHC}{\mathsf{SHC}}
\newcommand{\HC}{\mathsf{HC}}
\newcommand{\Dist}{\mathrm{Dist}}

\newcommand{\sfU}{\mathsf{U}}
\newcommand{\nil}{\mathrm{nil}}
\newcommand{\Sat}{\mathsf{Sat}}
\newcommand{\Br}{\mathrm{Br}}
\newcommand{\Inv}{\mathrm{Inv}}

\newcommand{\Loop}{\mathrm{L}}

\numberwithin{equation}{section}
\numberwithin{figure}{section}
\newtheorem{thm}{Theorem}[section]
\newtheorem{lem}[thm]{Lemma}
\newtheorem{prop}[thm]{Proposition}
\newtheorem{cor}[thm]{Corollary}

\theoremstyle{definition}

\theoremstyle{remark}
\newtheorem{rmk}[thm]{Remark}

\title[On two modular geometric realizations of a Hecke algebra]{On two modular geometric realizations of an affine Hecke algebra}

\author[R. Bezrukavnikov]{Roman Bezrukavnikov}
\address{Department of Mathematics \\ Massachusetts Institute of Technology \\ Cambridge, MA \\ 02139 \\ USA.}
\email{bezrukav@math.mit.edu}

 \author[S.~Riche]{Simon Riche}
 \address{Universit\'e Clermont Auvergne, CNRS, LMBP, F-63000 Clermont-Ferrand, France.}
 \email{simon.riche@uca.fr}

\begin{document}

\begin{abstract}
In this paper we construct equivalences of monoidal categories relating three geometric or representation-theoretic categorical incarnations of the affine Hecke algebra of a connected reductive algebraic group $\bG$ over a field of positive characteristic:
\begin{itemize}
\item a category of Harish-Chandra bimodules for the Lie algebra of $\bG$;
\item the derived category of equivariant coherent sheaves on (a completed version of) the Steinberg variety of the Frobenius twist $\bG^{(1)}$ of $\bG$;
\item a derived category of constructible sheaves on the affine flag variety of the reductive group which is Langlands dual to $\bG^{(1)}$.
\end{itemize} 
These constructions build on the localization theory developed in~\cite{bmr,bmr2} and previous work (partly joint with L. Rider) in~\cite{brr-pt1,br-pt2}, and provide an analogue for positive-characteristic coefficients of the main result of~\cite{be}. As an application, we prove a conjecture by Finkelberg--Mirkovi{\'c} giving a geometric realization of the principal block of algebraic representations of $\bG$.
\end{abstract}

\maketitle

\setcounter{tocdepth}{1}
\tableofcontents

%%%%%%%%%%%%%%%%%%%%%%%%%%%%%%%%%%%%%%%%%%%%%%%
\section{Introduction}
\label{sec:intro}
%%%%%%%%%%%%%%%%%%%%%%%%%%%%%%%%%%%%%%%%%%%%%%%

%----------------------------------------------------
\subsection{Two geometric realizations of an affine Hecke algebra}
%----------------------------------------------------

This paper is the third (and last) part of a project started in~\cite{brr-pt1} (joint with L. Rider) and pursued in~\cite{br-pt2}, whose goal was to build an equivalence of monoidal categories between two geometric realizations (or ``categorifications'') of the affine Hecke algebra\footnote{Technically, the categories we consider are rather incarnations of the group algebra of the affine Weyl group.}
%Considering incarnations of the Hecke algebra would require working with mixed sheaves on the constructible side and $\Gm$-equivariant sheaves on the coherent side. This presents important technical difficulties which, as of now, are not even solved in the characteristic-$0$ setting, and which will not be considered here.
attached to a connected reductive algebraic group, in the case of positive-characteristic coefficients, thereby providing a ``modular'' counterpart to the main result of~\cite{be}.

Namely, consider a connected reductive algebraic group $\bG$ with simply-connected derived subgroup over an algebraically closed field $\bk$ of very good characteristic $\ell>0$.
% which is ``not too small.'' (The precise technical conditions will be discussed later.) 
 Fix also a Borel subgroup $\bB \subset \bG$ and a maximal torus $\bT \subset \bB$. Then the \emph{extended affine Weyl group} $\bW$ is the semi-direct product
\[
\bW := \bWf \ltimes X^*(\bT)
\]
where $\bWf$ is the (finite) Weyl group of $(\bG,\bT)$ and $X^*(\bT)$ is the character lattice of $\bT$ (a free $\Z$-module of finite type). The choice of $\bB$ determines on this group a quasi-Coxeter group structure (i.e.~$\bW$ is a semi-direct product of a Coxeter group by an abelian group acting by Coxeter group automorphisms), and in particular a length function $\ell : \bW \to \Z_{\geq 0}$ and an associated Hecke algebra $\mathcal{H}_{\bW}$. Here $\mathcal{H}_{\bW}$ is a $\Z[v,v^{-1}]$-algebra endowed with a basis $(H_w : w \in \bW)$, where multiplication satisfies
\[
H_y \cdot H_w = H_{yw} \quad \text{if $y,w \in \bW$ satisfy $\ell(yw)=\ell(y)+\ell(w)$}
\]
and the elements associated with simple reflections satisfy the usual quadratic relation.

This algebra appears naturally in two different contexts. First, consider a finite field $\F_0$ of cardinality $q$ which is prime with $\ell$, the split connected reductive algebraic group $G_0$ over $\F_0$ which is Langlangs dual to 
the Frobenius twist $\bG^{(1)}$, 
and the reductive group $G_{0, \F_0( \hspace{-1pt} (z) \hspace{-1pt} )}$ over the local field $\F_0( \hspace{-1pt} (z) \hspace{-1pt} )$ obtained by base change.\footnote{Instead of working over a local field of equal characteristic one can work over any local field with residue field $\mathbb{F}_0$; the same Hecke algebra plays the same role in this version. The geometric picture we will consider below is modeled on the case of equal characteristic, which explains our restriction here. A geometric picture closer to the mixed characteristic setting has been developed more recently by Zhu and Scholze (among others), and some specialists expect that an equivalence similar to ours can be developed in this setting; see~\cite{alwy} for work in this direction. We will not consider this here.} An important result of Iwahori--Matsumoto~\cite{im} states that the Hecke algebra associated with this group and an Iwahori subgroup identifies with the specialization of $\mathcal{H}_{\bW}$ to $v=q^{-\frac{1}{2}}$. (The importance of this algebra stems from later results of Borel~\cite{borel-Haff}, showing that it ``controls'' the category of smooth complex representations of $G_{0}(\F_0( \hspace{-1pt} (z) \hspace{-1pt} ))$ generated by their Iwahori-invariant vectors.) In view of Grothendieck's ``faisceau-fonctions'' dictionary, this suggests that, denoting by $\F$ an algebraic closure of $\F_0$ and by $G$ the reductive group obtained from $G_0$ by base change to $\F$, the category of Iwahori-equivariant constructible sheaves on the affine flag variety of $G$ (over $\F$) should be considered a categorical incarnation of $\mathcal{H}_{\bW}$.

On the other hand, consider the (Grothendieck--)Steinberg variety $\St$ of $\bG$. More explicitly we have
\[
\St = \tbg \times_{\bg^*} \tbg
\]
where $\tbg$ is the Grothendieck resolution associated with $\bG$, parametrizing pairs $(\xi,\bB')$ where $\bB'$ is a Borel subgroup of $\bG$ and $\xi$ is a linear form on the Lie algebra $\bg$ of $\bG$ which vanishes on the Lie algebra of the unipotent radical of $\bB'$, and the morphism $\tbg \to \bg^*$ is the obvious one.
 This variety admits a canonical action of $\bG$, and an action of the multiplicative group $\Gm$ by dilation along the fibers of the projection to $\bG/\bB \times \bG/\bB$. Taking Frobenius twists we obtain an action of $\bG^{(1)} \times (\Gm)^{(1)}$ on $\St^{(1)}$, so that we can consider the equivariant $K$-theory $\mathsf{K}^{\bG^{(1)} \times (\Gm)^{(1)}}(\St^{(1)})$. This abelian group has a natural structure of $\Z[v,v^{-1}]$-algebra, and a result of Kazhdan--Lusztig~\cite{kl} and Ginzburg~\cite{cg} (see also~\cite{lusztig-bases}) states that there exists a $\Z[v,v^{-1}]$-algebra isomorphism
\[
\mathsf{K}^{\bG^{(1)} \times (\Gm)^{(1)}}(\St^{(1)}) \cong \mathcal{H}_{\bW}.
\]
The bounded derived category of $\bG^{(1)} \times (\Gm)^{(1)}$-equivariant coherent sheaves on $\St^{(1)}$ can therefore also be considered a categorical incarnation of $\mathcal{H}_{\bW}$.

The main result of this paper is an equivalence of monoidal categories relating (appropriate versions of) these two categorical incarnations, which adapts a similar equivalence for characteristic-$0$ coefficients due to the first author~\cite{be}.

%--------------------------------------------------
\subsection{The equivalence}
\label{ss:intro-equiv}
%--------------------------------------------------

We continue with our algebraically closed field $\bk$, the connected reductive algebraic group $\bG$ over $\bk$ (which we no longer assume to have simply-connected derived subgroup), its subgroups $\bB$ and $\bT$, and the Weyl group $\bWf$. Let also $\bt$ be the Lie algebra of $\bT$. Then there exists a canonical morphism $\St \to \bt^* \times_{\bt^*/\bWf} \bt^*$, and we denote by $\St^\wedge$ the fiber product of $\St$ with the spectrum of the completion of $\scO(\bt^* \times_{\bt^*/\bWf} \bt^*)$ with respect to the maximal ideal corresponding to the closed point $(0,0)$. This scheme is noetherian, and admits a canonical (algebraic) action of $\bG$, so that we can consider the bounded derived category
\[
\Db \Coh^{\bG^{(1)}}(\St^{\wedge(1)})
\]
of $\bG^{(1)}$-equivariant coherent sheaves on $\St^{\wedge(1)}$. Standard considerations allow to construct a monoidal product on this category given by convolution.
%, see~\cite{dr}.

\begin{rmk}
We insist that $\St^\wedge$ is a plain scheme with an algebraic action of $\bG$, and not a formal scheme.
\end{rmk}

On the other hand, consider an algebraically closed field $\F$ of characteristic $p \neq \ell$, and the connected reductive algebraic group $G$ whose Langlands dual over $\bk$ is $\bG^{(1)}$. Let also $T$ be a maximal torus in $G$ with a fixed identification $X_*(T)=X^*(\bT^{(1)})$, and denote by $B$ the Borel subgroup containing $T$ whose roots are the coroots of $\bB^{(1)}$. Then we consider the loop group $\Loop G$ of $G$, the arc group $\Loop^+ G \subset \Loop G$, the Iwahori subgroup $\Iw \subset \Loop^+ G$ determined by $B$, and its pro-unipotent radical $\Iwu \subset \Iw$. We have the ``extended affine flag variety'' $\tFl_G := \Loop G/\Iwu$, a $T$-torsor over the usual affine flag variety $\Fl_G := \Loop G/\Iw$, and we consider the ``completed derived category'' $\sfD^\wedge_{\Iwu,\Iwu}$ with coefficients in $\bk$ in the sense of Yun~\cite[Appendix~A]{by} associated with this ind-scheme, the action of $\Iwu$, 
%(seen as a $T$-torsor over the affine flag variety $LG/\Iw$) 
and the stratification given by $\Iw$-orbits. This category also has a natural structure of monoidal category.

With this notation introduced one can state the main result of the paper. For this we assume that, for any indecomposable constituent in the root system of $(\bG,\bT)$, $\ell$ is strictly larger than the corresponding bound in Figure~\ref{fig:bounds}.
Under this assumption, we construct an equivalence of monoidal triangulated categories
\begin{equation}
\label{eqn:main-equiv-intro}
\Db \Coh^{\bG^{(1)}}(\St^{\wedge(1)}) \cong \sfD^\wedge_{\Iwu,\Iwu};
\end{equation}
see Theorem~\ref{thm:equivalences-comp}.
This equivalence satisfies various favorable properties, and in particular a compatibility with the geometric Satake equivalence for the group $G$ (see~\S\ref{ss:intro-FM} below), where perverse sheaves on the affine Grassmannian are ``lifted'' to objects of $\sfD^\wedge_{\Iwu,\Iwu}$ via (a variant of) Gaitsgory's central functor~\cite{gaitsgory}.

%DISCUSS ASSUMPTION

\begin{figure}
 \begin{tabular}{|c|c|c|c|c|c|c|c|c|}
  \hline
  $\mathbf{A}_n$ ($n \geq 1$) & $\mathbf{B}_n$ ($n \geq 2$) & $\mathbf{C}_n$ ($n \geq 3$) & $\mathbf{D}_n$ ($n \geq 4$) & $\mathbf{E}_6$ & $\mathbf{E}_7$ & $\mathbf{E}_8$ & $\mathbf{F}_4$ & $\mathbf{G}_2$ \\
  \hline
  $n+1$ & $2n$ & $2n$ & $2n-2$ & $12$ & $19$ & $31$ & $12$ & $6$\\
  \hline
 \end{tabular}
\caption{Bounds on $\ell$}
\label{fig:bounds}
\end{figure}

\begin{rmk}
\phantomsection
\label{rmk:main-equiv-intro}
\begin{enumerate}
\item
Both of the categories considered above are in fact attached to the Frobenius twist $\bG^{(1)}$ of $\bG$. The reason why we introduce the group $\bG$, and state the result in this form, is that we will also consider a third incarnation of this category, which really involves $\bG$ itself. See~\S\ref{ss:intro-HC-bim} below for details.
%will be explained in~\S\ref{ss:intro-HC-bim} below.
\item
As the reader will notice, for all types except $\mathbf{E}_7$ and $\mathbf{E}_8$ the bound in Figure~\ref{fig:bounds} is the Coxeter number $h$ of the root system; in the latter cases the bound is $h+1$. In the body of the paper we work under slightly different assumptions, that allow e.g.~the group $\mathrm{GL}_\ell(\bk)$ in characteristic $\ell$. 
See~\S\ref{ss:intro-nilp-unip} below for comments on our assumptions, and~\S\ref{ss:assumptions} for a discussion of the difference between the assumptions in the body of the paper and those considered here.
%We expect that the restriction on $\ell$ is an artefact of our approach, where we use localization theory (see~\S\ref{ss:intro-HC-bim}), and that an appropriate version of the equivalence should hold in good (or maybe very good) characteristic.
\item
\label{it:main-equiv-intro-class}
The equivalence~\eqref{eqn:main-equiv-intro} involves on the ``coherent'' side the scheme $\St^{\wedge(1)}$, which is not of finite type over $\bk$, and on the constructible side the ``completed'' category $\sfD^\wedge_{\Iwu,\Iwu}$. Once this equivalence is established, one can deduce a version which involves only a scheme of finite type and a ``usual'' category of constructible sheaves. Namely, consider the bounded derived category $\Db\Coh^{\bG^{(1)}}_{\cN}(\St^{(1)})$ of coherent sheaves on $\St^{(1)}$ which are supported set-theoretically on the preimage of $(0,0) \in \bt^{*(1)} \times_{\bt^{*(1)}/\bWf} \bt^{*(1)}$. On the other hand, consider the full triangulated subcategory $\sfD_{\Iwu,\Iwu}$ of the $\Iwu$-equivariant bounded derived category of sheaves on $\tFl_G$ generated by objects obtained by pullback from $\Iwu$-equivariant complexes on $\Fl_G$. Then there exist fully faithful functors
\[
\Db\Coh^{\bG^{(1)}}_{\cN}(\St^{(1)}) \to \Db\Coh^{\bG^{(1)}}(\St^{\wedge(1)}), \quad \sfD_{\Iwu,\Iwu} \to \sfD_{\Iwu,\Iwu}^{\wedge}
\]
whose essential images are ideals for the monoidal products, and the equivalence~\eqref{eqn:main-equiv-intro} restricts to an equivalence of monoidal categories
\begin{equation}
\label{eqn:main-equiv-intro-class}
\Db\Coh^{\bG^{(1)}}_{\cN}(\St^{(1)}) \cong \sfD_{\Iwu,\Iwu}.
\end{equation}
See Theorem~\ref{thm:equivalences} for details.
%Similar comments apply to the equivalences~\eqref{eqn:main-equiv-2-intro} and~\eqref{eqn:main-equiv-3-intro} considered above; see Theorem~\ref{thm:equivalences} for details.
\item
\label{it:intro-fixed-character}
Let $\St'$ be the preimage of $\bt^* \times_{\bt^*/\bWf} \{0\} \subset \bt^* \times_{\bt^*/\bWf} \bt^*$ in $\St$, and consider the $\Iwu$-equivariant bounded derived category $\sfD_{\Iwu,\Iw}$ of $\bk$-sheaves on the affine flag variety $\Fl_G$. We will also establish a variant of~\eqref{eqn:main-equiv-intro}, in the form of an equivalence of triangulated categories
\[
\Db \Coh^{\bG^{(1)}}(\St^{\prime(1)}) \cong \sfD_{\Iwu,\Iw}
\]
which intertwines the left actions of the categories in~\eqref{eqn:main-equiv-intro} on both sides, see Theorem~\ref{thm:equivalences-fixed}. This equivalence also has its counterpart in~\cite{be}. The latter paper contains a second variant of~\eqref{eqn:main-equiv-intro}, describing the $\Iw$-equivariant bounded derived category of sheaves on $\Fl_G$ in terms of equivariant coherent sheaves on a certain dg-scheme. We will not consider any version of this equivalence here. (It is likely that the two variants discussed here, and many more variants associated with choices of parabolic subgroups, can be deduced formally from~\eqref{eqn:main-equiv-intro} once an appropriate $\infty$-categorical framework has been developed, following the work done in~\cite{chen-dhillon} for characteristic-$0$ coefficients. This will be the subject of future work.)
\item
The construction of the characteristic-$0$ version of~\eqref{eqn:main-equiv-intro} involves as a preliminary step the construction (in~\cite{ab}) of an equivalence between the derived category of equivariant coherent sheaves on the Springer resolution of $\bG^{(1)}$ and an appropriate derived category of ``Iwahori--Whittaker sheaves'' on $\Fl_G$. The strategy followed here is completely different, but we establish in~\S\ref{ss:antispherical} a modular version of the equivalence of~\cite{ab}, which we obtain as a \emph{consequence} of the equivalence~\eqref{eqn:main-equiv-intro}.
\end{enumerate}
\end{rmk}

%--------------------------------------------------
\subsection{Harish-Chandra bimodule}
\label{ss:intro-HC-bim}
%--------------------------------------------------

As explained above, the equivalence~\eqref{eqn:main-equiv-intro} is a ``modular variant'' of an equivalence for characteristic-$0$ coefficients proved in~\cite{be}. Technical obstacles prevent us from using the same strategy as in~\cite{be} (based on the earlier work~\cite{ab} with Arkhipov) in the modular world, so our approach will be different. It is based on the use of a third monoidal category, which serves as a kind of bridge between $\Db \Coh^{\bG^{(1)}}(\St^{\wedge(1)})$ and $\sfD^\wedge_{\Iwu,\Iwu}$.

Namely, consider the algebra
\[
\sfU := \cU\bg \otimes_{\ZFr} \cU\bg^{\op}
\]
where (as above) $\bg$ is the Lie algebra of $\bG$, $\cU\bg$ is its universal enveloping algebra, and $\ZFr$ is the Frobenius center of $\cU\bg$. Consider also the Harish-Chandra center $\ZHC$ of $\cU\bg$; then $\sfU$ admits a central subalgebra isomorphic to
\[
\ZHC \otimes_{\ZFr \cap \ZHC} \ZHC \cong \scO(\bt^*/(\bWf, \bullet) \times_{\bt^{*(1)}/\bWf} \bt^*/(\bWf, \bullet)).
\]
Here the isomorphism is induced by the Harish-Chandra isomorphism
\[
\ZHC \cong \scO(\bt^*/(\bWf, \bullet))
\]
(where $\bullet$ is the usual ``dot-action'' of $\bWf$ on $\bt^*$)
and the morphism $\bt^*/(\bWf, \bullet) \to \bt^{*(1)}/\bWf$ is induced by the Artin--Schreier morphism $\bt^* \to \bt^{*(1)}$.
We will denote by
$\sfU^{\widehat{0}, \widehat{0}}$
the tensor product of $\sfU$ with the completion of $\ZHC \otimes_{\ZFr \cap \ZHC} \ZHC$ with respect to the ideal corresponding to the closed point
\[
(0,0) \in \bt^*/(\bWf, \bullet) \times_{\bt^{*(1)}/\bWf} \bt^*/(\bWf, \bullet).
\]
Then $\sfU^{\widehat{0}, \widehat{0}}$ is a noetherian ring endowed with a structure of algebraic $\bG$-module compatible with multiplication, and we can consider the category
\[
\HC^{\widehat{0}, \widehat{0}}
\]
of Harish-Chandra modules for this algebra, i.e.~finitely generated $\bG$-equivariant $\sfU^{\widehat{0}, \widehat{0}}$-modules such that the differential of the action of $\bG$ coincides with the restriction of the action of $\sfU^{\widehat{0}, \widehat{0}}$ along the morphism $\cU\bg \to \sfU^{\widehat{0}, \widehat{0}}$ induced by the assignment $x \mapsto x \otimes 1 - 1 \otimes x$ ($x \in \bg$). Here again, the bounded derived category $\Db \HC^{\widehat{0}, \widehat{0}}$ admits a natural monoidal structure.

\begin{rmk}
We insist that $\sfU^{\widehat{0}, \widehat{0}}$ is not defined as a completion of $\sfU$, but rather as a tensor product with a completion of a central subalgebra. In this way, and since $\bG$ acts trivially on this subalgebra, $\sfU^{\widehat{0}, \widehat{0}}$ inherits from $\sfU$ an \emph{algebraic} action of $\bG$.
\end{rmk}

On the way to establishing the equivalence~\eqref{eqn:main-equiv-intro}, we will also establish an equivalence of monoidal categories
\begin{equation}
\label{eqn:main-equiv-2-intro}
\Db \Coh^{\bG^{(1)}}(\St^{\wedge(1)}) \cong \Db \HC^{\widehat{0}, \widehat{0}}.
\end{equation}
This equivalence will be obtained using a variant of the localization theory developed by the first author with Mirkovi{\'c} and Rumynin in~\cite{bmr,bmr2,bm-loc}, which can be considered a ``modular'' counterpart to the Be{\u\i}linson--Bernstein localization theory for complex semisimple Lie algebras. Notice, however, that the statement of~\eqref{eqn:main-equiv-2-intro}
is simpler than the main result of~\cite{bmr} in that the latter involves modules over a nonsplit
Azumaya algebra, while the former does not due to cancellation of gerbe classes coming
from the two factors in the fiber product $\St^{(1)}$.

We believe that the equivalence~\eqref{eqn:main-equiv-2-intro} is interesting in its own right, in that it provides a third incarnation of the ``affine Hecke category.'' But technically, its role in our construction of~\eqref{eqn:main-equiv-intro} is that it allows us to check properties of a certain family of objects in $\Db \Coh^{\bG^{(1)}}(\St^{\wedge(1)})$ by translating the question in $\Db \HC^{\widehat{0}, \widehat{0}}$, where it becomes easy; see~\S\ref{ss:intro-Soergel-bim} below for more details. 
%A direct (and more geometric) proof of these properties would probably be needed in order to establish a version of the equivalence under weaker assumptions on $\ell$.

\begin{rmk}
\phantomsection
\label{rmk:intro-HC}
\begin{enumerate}
\item
The category $\HC^{\widehat{0}, \widehat{0}}$ has already appeared in our earlier paper~\cite{br-Hecke}, where we established a version of the equivalence~\eqref{eqn:main-equiv-2-intro} ``over the Kostant section.'' Some of our arguments are based on the results of~\cite{br-Hecke}, but the construction of the full equivalence requires further work.
\item
During the final stages of preparation of this manuscript a version of equivalence~\eqref{eqn:main-equiv-2-intro} was also
obtained independently by Losev~\cite{losev}. Our construction is partly similar although some of the
proofs are different.
%An equivalence~\eqref{eqn:main-equiv-2-intro} also appears in recent work of Losev~\cite{losev}. Our construction is similar, but some of our proofs are slightly different.
\item
As in Remark~\ref{rmk:main-equiv-intro}\eqref{it:main-equiv-intro-class}, once the equivalence~\eqref{eqn:main-equiv-2-intro} is proved one can deduce a version not involving any completion: if we denote by $\HC_{\mathrm{nil}}^{\widehat{0}, \widehat{0}}$ the category of finitely generated Harish-Chandra $\sfU$-modules such that the ideal of $\ZHC \otimes_{\ZFr \cap \ZHC} \ZHC$ corresponding to the point $(0,0)$ acts nilpotently, then the equivalence restricts to an equivalence of monoidal categories
\begin{equation}
\label{eqn:main-equiv-2-intro-class}
\Db \Coh_{\cN}^{\bG^{(1)}}(\St^{(1)}) \cong \Db \HC_{\mathrm{nil}}^{\widehat{0}, \widehat{0}}.
\end{equation}
\item
\label{it:intro-HC-character}
As in Remark~\ref{rmk:main-equiv-intro}\eqref{it:intro-fixed-character}, the results we have presented above have a variant where one imposes that the action of the Harish-Chandra center on the right factors through the trivial character; this involves a category of Harish-Chandra bimodules denoted $\HC^{\widehat{0},0}$.
%COMPLETE!
\end{enumerate}
\end{rmk}

%--------------------------------------------------
\subsection{Soergel bimodules}
\label{ss:intro-Soergel-bim}
%--------------------------------------------------

Our construction of the equivalence~\eqref{eqn:main-equiv-intro} follows a familiar pattern originating from the work of Soergel~\cite{soergel}: we identify on both sides a monoidal additive subcategory $\mathsf{A}$ such that the whole category can be reconstructed as the bounded homotopy category of $\mathsf{A}$, and then we identify these additive categories by comparing both of them to a category of ``Soergel bimodules'' for the group $\bW$. On the constructible side, this construction is the main result of~\cite{br-pt2}. In this case the additive category we consider is the subcategory of tilting perverse sheaves in $\sfD^\wedge_{\Iwu,\Iwu}$, and the category of ``Soergel bimodules'' we use is in fact a category of representations of a smooth affine group scheme over an affine scheme constructed out of the restriction of the universal centralizer to a Steinberg section. (A variant of this category is considered in~\cite{br-Hecke}; this variant can be related to a ``true'' category of Soergel bimodules via the work of Abe~\cite{abe-Hecke}.)

On the side of the category $\Db \Coh^{\bG^{(1)}}(\St^{\wedge(1)})$, we do not a priori have a ``nice'' t-structure such that we can consider tilting objects in the heart. We use the ``localization'' equivalence~\eqref{eqn:main-equiv-2-intro} to construct a subcategory ``$\mathsf{A}$'' which ``behaves like a category of tilting objects,'' i.e.~such that 
\[
\Hom_{\Db \Coh^{\bG^{(1)}}(\St^{\wedge(1)})}(M,N[n])=0
\]
for any $M,N \in \mathsf{A}$ and $n \in \Z_{\neq 0}$.
%there exists no nonzero morphisms from an object of $\mathsf{A}$ to a nonzero cohomological shift of an object of $\mathsf{A}$.
%; as explained above this is the main reason why we have to use a ``localization'' equivalence~\eqref{eqn:main-equiv-2-intro}. 
%This equivalence allows to transfer the problem (constructing a suitable additive monoidal subcategory and relating it to ``Soergel bimodules'') to the category $\Db \HC^{\widehat{0}, \widehat{0}}$, where it turns out to be much simpler. Here 
%The image of the category ``$\mathsf{A}$'' in $\Db \HC^{\widehat{0}, \widehat{0}}$ is 
More specifically, the corresponding subcategory of $\Db \HC^{\widehat{0}, \widehat{0}}$ is generated (under the monoidal product, directs sums and direct summands) by the ``wall crossing bimodules'' studied in~\cite{br-Hecke}, i.e.~the objects which realize wall-crossing functors on the principal block of the category $\Rep(\bG)$ of finite-dimensional algebraic $\bG$-modules. The relation with Soergel bimodules is obtained via a variant of the methods developed in~\cite{br-Hecke}.
%was essentially already obtained in~\cite{br-Hecke}.

%--------------------------------------------------
\subsection{Nilpotent cone versus unipotent cone}
\label{ss:intro-nilp-unip}
%--------------------------------------------------

There exists a variant of the Steinberg variety which is obtained by replacing $\tbg$ by the ``multiplicative'' Grothendieck resolution $\tbG$ parametrizing pairs $(g,\bB')$ where $\bB'$ is a Borel subgroup of $\bG$ and $g \in \bB'$, and $\bg^*$ by $\bG$. This scheme, denoted $\Stm$, admits a natural morphism to $\bT \times_{\bT/\bWf} \bT$, and we denote by $\Stm^\wedge$ the fiber product of $\Stm$ with the spectrum of the completion of $\scO(\bT \times_{\bT/\bWf} \bT)$ with respect to the ideal corresponding to the point $(e,e) \in \bT \times_{\bT/\bWf} \bT$ (where $e \in \bT$ is the unit element). As part of our constructions we prove that under our assumptions there exists a $\bG$-equivariant isomorphism of schemes $\Stm^\wedge \cong \St^\wedge$, which induces an equivalence of monoidal categories
\begin{equation}
\label{eqn:main-equiv-3-intro}
\Db \Coh^{\bG^{(1)}}(\Stm^{\wedge(1)}) \cong \Db \Coh^{\bG^{(1)}}(\St^{\wedge(1)}).
\end{equation}
Combining this equivalence with~\eqref{eqn:main-equiv-intro} we deduce an equivalence of monoidal categories
\[
\Db \Coh^{\bG^{(1)}}(\Stm^{\wedge(1)}) \cong \sfD^\wedge_{\Iwu,\Iwu}.
\]

We expect this version of our equivalence to be more ``canonical,'' and in fact to hold under much weaker assumptions on $\ell$ (e.g.~for $\ell$ good). In fact, the scheme that appears naturally from the study of $\sfD^\wedge_{\Iwu,\Iwu}$ is $\Stm^\wedge$. Some aspects of the proofs in~\cite{brr-pt1,br-pt2}, which we expect to be artificial, had already forced us to impose stronger assumptions. As explained above the constructions of the present paper rely on the localization theory of~\cite{bmr,bmr2,bm-loc}, which involves the scheme $\tbg$ and not $\tbG$, and moreover requires the condition that $\ell$ is larger than the Coxeter number $h$ (so that the weight $0 \in \bt^*$ is regular). This explains the appearance of $\St^\wedge$ here, and the further restrictions on $\ell$ that we impose. A proof in larger generality would most likely require a direct and better understanding of the structure of the category $\Db \Coh^{\bG^{(1)}}(\Stm^{\wedge(1)})$.

%DISCUSS PASSAGE FROM MULTIPLICATIVE VERSION TO ADDITIVE VERSION?

%--------------------------------------------------
\subsection{Application to the Finkelberg--Mirkovi{\'c} conjecture}
\label{ss:intro-FM}
%--------------------------------------------------

Our desire to construct the equivalence~\eqref{eqn:main-equiv-intro} comes from expected applications to representation theory of reductive algebraic groups in positive characteristic. In fact the characteristic-$0$ version of this equivalence from~\cite{be}, and its preliminary step obtained in~\cite{ab}, have already found important applications to the representation theory of quantum groups at roots of unity~\cite{bez-tilting} and of representations of Lie algebras of reductive groups in large characteristic~\cite{bm-loc,blo1,blo2}. We expect its modular version to allow new applications where representations of quantum groups are replaced by algebraic representations of reductive groups, and to extend the results on representations of Lie algebras to general characteristics.

\begin{rmk}
The equivalences of~\cite{ab,be} have also found important applications in another direction, namely the Geometric Langlands Program, see e.g.~\cite{bchn,hz} for two examples. It is likely that its modular counterpart will lead to refinements of these results for positive-characteristic coefficients, but no concrete result in this direction will be discussed here.
\end{rmk}

As an illustration of this idea, here we establish an equivalence of categories conjectured by Finkelberg--Mirkovi{\'c}~\cite{fm}. Namely, recall that the \emph{geometric Satake equivalence} provides an equivalence of monoidal categories between the category $\Rep(\bG^{(1)})$ of finite-dimensional algebraic representations of $\bG^{(1)}$ (endowed with the tensor product of representations) and the category $\sfP_{\Loop^+ G,\Loop^+ G}$ of $\Loop^+ G$-equivariant $\bk$-perverse sheaves on the affine Grassmannian $\Gr_G := \Loop G / \Loop^+ G$. 

One can ``enlarge'' the categories on both sides of this equivalence in the following way. 
Let $\Rep_{\langle 0 \rangle}(\bG)$ be the ``extended principal block'' in $\Rep(\bG)$, i.e.~the Serre subcategory generated by the simple objects whose highest weights belong to the $\bW$-orbit of $0$ (for the ``dot'' action of $\bW$ on $X^*(\bT)$). Then there exists a canonical action of the category $\Rep(\bG^{(1)})$ on $\Rep_{\langle 0 \rangle}(\bG)$ given by tensor product with pullbacks under the Frobenius morphism.
On the other hand, consider the ``opposite'' affine Grassmannian $\Gr'_G=\Loop^+ G \backslash \Loop G$ with the natural action of $\Iwu$ (induced by multiplication on the right on $\Loop G$), and denote by $\sfP_{\Loop^+ G, \Iwu}$ the category of $\Iwu$-equivariant $\bk$-perverse sheaves on $\Gr'_G$. Then there exists a canonical action of $\sfP_{\Loop^+ G,\Loop^+ G}$ on $\sfP_{\Loop^+ G,\Iwu}$ defined by convolution on the left.

The conjecture of Finkelberg--Mirkovi{\'c} that we establish here (under the assumptions considered in~\S\ref{ss:intro-equiv}) is that there exists an equivalence of abelian categories
\[
\Rep_{\langle 0 \rangle}(\bG) \cong \sfP_{\Loop^+ G, \Iwu}
\]
intertwining the actions of $\Rep(\bG^{(1)})$ and $\sfP_{\Loop^+ G,\Loop^+ G}$ (identified via the geometric Satake equivalence); see Theorem~\ref{thm:FM}. This equivalence induces the obvious bijection between the sets of isomorphism classes of simple objects on both sides. One of the motivations for this conjecture is that it makes Lusztig's character formula~\cite{lusztig-pbs} for simple representations in $\Rep_{\langle 0 \rangle}(\bG)$ in large characteristics completely transparent, since it translates it to the standard observation that the cohomology of fibers of any given intersection cohomology sheaf in $\sfP_{\Loop^+ G, \Iwu}$ has the same dimension in large characteristic and in characteristic $0$.

A crucial step in the proof of this theorem is the identification of the image of the perverse t-structure under the equivalence
\[
\sfD_{\Iwu,\Iwu} \cong \Db \HC_{\mathrm{nil}}^{\widehat{0},\widehat{0}}
\]
obtained by combining~\eqref{eqn:main-equiv-intro-class} and~\eqref{eqn:main-equiv-2-intro-class}; the answer is given in terms of the \emph{perverse coherent t-structure} constructed and studied in~\cite{arinkin-bezrukavnikov}. Here again, a closely related result appears in~\cite{losev} (with other applications in mind), but our proof is different; in fact it follows ideas that already played a crucial role in~\cite{bm-loc}.

In a different direction,
notice that by~\cite[Proposition~6.3]{losev} there is a derived equivalence between the category $\HC^{\widehat{0},0}$ of Remark~\ref{rmk:intro-HC}\eqref{it:intro-HC-character}
%Harish-Chandra bimodules in positive characteristic 
and an appropriately defined modular category $\mathscr{O}$. Thus our main result is related to the modular
version of Gaitsgory's conjectural extension of the Kazhdan--Lusztig equivalence~\cite{gaitsgory-2}
recently proved in~\cite{chen-fu} (see also~\cite{losev2,situ} for alternative proofs of closely related results). In fact, a study of this modular category $\mathscr{O}$ using constructible sheaves based (among other things) on the results of the present paper is the subject of a joint project of the second author with P.~Achar and G.~Dhillon.

%--------------------------------------------------
\subsection{Contents}
%--------------------------------------------------

In Section~\ref{sec:HCBim} we prove preliminary results on Harish-Chandra bimodules. In Section~\ref{sec:compl-loc} we establish a ``completed'' version of the results of~\cite{bmr,bmr2} adapted to the schemes and rings we want to consider. In Section~\ref{sec:loc} we construct equivalences relating our categories of Harish-Chandra bimodules to some categories of $\mathscr{D}$-modules on (partial) flag varieties. In Section~\ref{sec:splitting} we deduce a construction of the equivalence~\eqref{eqn:main-equiv-2-intro}. In Section~\ref{sec:monoidality} we discuss the compatibility of our constructions with the natural convolution products on the categories involved.

In Section~\ref{sec:braid-gp} we study a family of objects in $\Db \Coh^{\bG^{(1)}}(\St^{\wedge(1)})$ and $\Db \HC^{\widehat{0}, \widehat{0}}$ associated with elements in the braid group of $\bW$. (These elements form a categorical incarnation of the ``standard basis'' of the Hecke algebra $\mathcal{H}_\bW$.) In Section~\ref{sec:Steinberg-section} we introduce and study a functor of ``restriction to a Steinberg section'' that will allow us to apply the strategy outlined in~\S\ref{ss:intro-Soergel-bim} in the category $\Db \Coh^{\bG^{(1)}}(\St^{\wedge(1)})$. (This functor will play the role of ``functor $\mathbb{V}$'' from Soergel's point of view.) In Section~\ref{sec:equiv} we prove the equivalence~\eqref{eqn:main-equiv-intro}, together with the various variants discussed above.

In Section~\ref{sec:t-structures} we define the ``perverse coherent t-structure'' on the category $\Db \HC_{\mathrm{nil}}^{\widehat{0}, \widehat{0}}$, and show that it is the image of the perverse t-structure on $\sfD_{\Iwu,\Iwu}$. In Section~\ref{sec:FM} we apply this to the proof of the Finkelberg--Mirkovi{\'c} conjecture.

The paper finishes with two appendices treating some technical constructions used in the body of the paper.

%--------------------------------------------------
\subsection{Some notations}
%--------------------------------------------------

If $\bk$ is a perfect field of characteristic $\ell>0$, given a $\bk$-scheme $X$, we will denote by $X^{(1)}$ the associated Frobenius twist defined as the fiber product
\[
X^{(1)}:=X \times_{\Spec(\bk)} \Spec(\bk),
\]
where the morphism $\Spec(\bk) \to \Spec(\bk)$ is associated with the map $x \mapsto x^\ell$. (The projection $X^{(1)} \to X$ is an isomorphism of abstract schemes, but not of $\bk$-schemes.) The Frobenius morphism then defines a morphism of $\bk$-schemes $\Fr_X : X \to X^{(1)}$.

If $X=\Spec(A)$ is an affine scheme and $Y \subset X$ is a closed subscheme, defined by an ideal $I \subset A$, we will denote by $\FN_X(Y)$ the spectrum of the completion of $A$ with respect to $I$. (Here, ``$\FN$'' stands for ``formal neighborhood.'') In this setting we have a canonical morphism of schemes $\FN_X(Y) \to X$.

If $A$ is a ring, resp.~a left noetherian ring, we will denote by $\Mod(A)$, resp.~by $\Modfg(A)$, the category of left $A$-modules, resp.~of finitely generated left $A$-modules. If $k$ is a commutative ring and $H$ is a flat affine group scheme over $k$, we will denote by $\Rep^\infty(H)$ the category of algebraic $H$-modules, i.e.~of $\scO(H)$-comodules, and, in case $k$ is noetherian, by $\Rep(H)$ the subcategory of representations whose underlying $k$-module is finitely generated. If $A$ is a $k$-algebra endowed with an action of $H$ we will denote by $\Mod^H(A)$ the category of $H$-equivariant left $A$-modules and, in case $A$ is left noetherian, by $\Modfg^H(A)$ the subcategory of modules which are finitely generated over $A$.

%--------------------------------------------------
\subsection{Acknowledgements}
%--------------------------------------------------

This paper is the third step of a project we initially started as a joint work with L. Rider. We thank her for her early contributions to this program. The second author also thanks P. Achar and G. Dhillon for stimulating discussions on subjects related to this project (in particular in relation with Section~\ref{sec:monoidality}), and I. Losev for his help to complete some proofs in~\S\ref{ss:comp-loc-end-proof} and Section~\ref{sec:monoidality}.

R.B.~was supported by NSF Grant No.~DMS-2101507.
This project has received funding from the European Research Council (ERC) under the European Union's Horizon 2020 research and innovation programme (S.R., grant agreement No.~101002592).

%%%%%%%%%%%%%%%%%%%%%%%%%%%%%%%%%%%%%%%%%%%%%%%
\section{Harish-Chandra bimodules}
\label{sec:HCBim}
%%%%%%%%%%%%%%%%%%%%%%%%%%%%%%%%%%%%%%%%%%%%%%%

%------------------------------------------------------
\subsection{Notation}
\label{ss:HC-notation}
%------------------------------------------------------

We let $\bk$ be an algebraically closed field of characteristic $\ell>0$, and $\bG$ be a connected reductive algebraic group over $\bk$. We also fix a Borel subgroup $\bB \subset \bG$ and a maximal torus $\bT \subset \bB$, and denote by $\bU$ the unipotent radical of $\bB$. 
%We will denote by $\Rep^\infty(\bG)$ the category of algebraic representations of $\bG$, and by $\Rep(\bG)$ the full subcategory of finite-dimensional representations. 
We will also denote by $\bg$, $\bb$, $\bt$, $\bu$ the Lie algebras of $\bG$, $\bB$, $\bT$, $\bU$ respectively.

We will denote by $\fR \subset X^*(\bT)$ the root system of $(\bG,\bT)$. The choice of $\bB$ determines a system of positive roots $\fR_+ \subset \fR$ (chosen as the $\bT$-weights in $\bg/\bb$), and we denote by $\fRs$ the associated system of simple roots. For any $\alpha \in \fR$, we will denote by $\alpha^\vee \in X_*(\bT)$ the corresponding coroot. We will also denote by $\bWf$ the Weyl group of $(\bG,\bT)$; the reflection associated with a root $\alpha$ will be denoted $s_\alpha$. The set $\bSf = \{s_\alpha : \alpha \in \fRs\}$ is a subset of Coxeter generators of $\bWf$; the longest element in this group will (as usual) be denoted $w_\circ$.
Given a weight $\lambda \in X^*(\bT)$, we will denote by $\ola \in \bt^*$ its differential.

We will make the following (standard) assumptions:
%\footnote{Maybe we need very good characteristic; cf.~\cite{br-Hecke}.}
\begin{enumerate}
\item 
\label{it:ass-1}
the derived subgroup $\mathscr{D}\bG$ of $\bG$ is simply connected;
\item $\ell$ is odd and good for $\bG$;
\item 
\label{it:ass-3}
$\bg$ admits a nondegenerate $\bG$-invariant bilinear form;
\item 
\label{it:ass-4}
$X^*(\bT)/\Z\fR$ has no $\ell$-torsion.
\end{enumerate}
In particular, assumption~\eqref{it:ass-1} ensures that there exists $\varsigma \in X^*(\bT)$ which satisfies $\langle \varsigma, \alpha^\vee \rangle = 1$ for any $\alpha \in \fRs$. We fix such a weight once and for all. (None of our constructions below will depend on this choice) In view of~\eqref{it:ass-3} we can (and will) fix a $\bG$-equivariant isomorphism $\bg \simto \bg^*$. It is clear that these assumptions are stable under passage to a Levi subgroup.

\begin{rmk}
\phantomsection
\label{rmk:assumptions-group}
\begin{enumerate}
\item
Assumptions~\eqref{it:ass-3} and~\eqref{it:ass-4} are automatic in case $\ell$ is very good for $\bG$; see~\cite[Proposition~2.5.12]{letellier} for~\eqref{it:ass-3}.
\item
In~\cite{br-Hecke} we imposed the further assumption that $\bG$ is semisimple. However, all the results of that paper hold under the present assumptions.
\end{enumerate}
\end{rmk}

Chevalley's theorem implies that the isomorphism classes of simple objects in $\Rep(\bG)$ are in a canonical bijection (via the theory of highest weights) with the subset $X^*(\bT)^+ \subset X^*(\bT)$ of dominant weights; the simple module of highest weight $\lambda \in X^*(\bT)^+$ will be denoted $\sfL(\lambda)$. For $\lambda \in X^*(\bT)$ we will denote by $\Rep_{\langle \lambda \rangle}(\bG)$ the Serre subcategory of $\Rep(\bG)$ generated by the simple objects whose highest weights belong to
\[
\{w(\lambda+\varsigma)-\varsigma + \ell \mu : w \in \bWf, \, \mu \in X^*(\bT)\} \cap X^*(\bT)^+.
\] 
The linkage principle implies that this subcategory is a direct summand in $\Rep(\bG)$. The category $\Rep(\bG)$ admits a canonical highest weight structure with weight poset $X^*(\bT)^+$, and with standard, resp.~costandard, objects the Weyl, resp.~induced, modules. This structure induces a highest weight structure on each subcategory $\Rep_{\langle \lambda \rangle}(\bG)$.

We will denote by $\Tilt(\bG)$ the full subcategory of $\Rep(\bG)$ consisting of tilting modules.
It is a standard fact that the set of isomorphism classes of indecomposable objects in $\Tilt(\bG)$ is in a canonical bijection with $X^*(\bT)^+$. For any $\lambda \in X^*(\bT)^+$ we fix an object $\sfT(\lambda)$ in the class corresponding to $\lambda$;
% Denoting by $w_\circ$ the longest element in $\bWf$, 
 it is well known that for any $\lambda \in X^*(\bT)^+$ we have $\sfT(\lambda)^* \cong \sfT(-w_\circ(\lambda))$.

%------------------------------------------------------
\subsection{Center of the universal enveloping algebra}
\label{ss:center-Ug}
%------------------------------------------------------

We consider the universal enveloping algebra $\cU\bg$ of $\bg$. This algebra admits a central subalgebra $\ZFr$, called the ``Frobenius center,'' and defined as the subalgebra generated by the elements of the form $x^\ell - x^{[\ell]}$ for $x \in \bg$. (Here, $(-)^{[\ell]} : \bg \to \bg$ is the restricted $\ell$-th power operation.) We have a canonical algebra isomorphism
\begin{equation}
\label{eqn:ZFr}
 \ZFr \simto \scO(\bg^{*(1)})
\end{equation}
induced by the map $x \mapsto x^\ell-x^{[\ell]}$.

Consider now the adjoint action of $\bG$ on $\bg$, and the induced action on $\cU\bg$.
It is easily seen that the subalgebra $\ZHC:=(\cU\bg)^{\bG} \subset \cU\bg$ (called the ``Harish-Chandra center'') is also central.
Under our present assumptions, by~\cite[\S 9]{jantzen-prime} we have a ``Harish-Chandra isomorphism''
\begin{equation}
\label{eqn:HC-isom}
(\cU\bg)^{\bG} \simto \mathrm{S}(\bt)^{(\bWf,\bullet)}=\scO(\bt^*/(\bWf,\bullet)),
\end{equation}
where $\bWf$ acts on $\bt^*$ via the ``dot action'' defined by
\[
w \bullet \xi = w(\xi + \overline{\varsigma})-\overline{\varsigma}.
\]
Defining a $\bk$-algebra morphism $\ZHC \to \bk$ is therefore equivalent to fixing a $\bk$-point in $\bt^*/(\bWf,\bullet)$, i.e.~a $(\bWf, \bullet)$-orbit in $\bt^*$. The characters we will mostly be interested in are those coming from characters of $\bT$; for $\lambda \in X^*(\bT)$, we will denote by $\tla$ the image of $\ola$ in $\bt^*/(\bWf,\bullet)$, and by $\fm^\lambda \subset \ZHC$ the maximal ideal defined by $\tla$.

Note that, since the isomorphism~\eqref{eqn:ZFr} is $\bG$-equivariant (for the action on $\bg^{*(1)}$ obtained from the adjoint action of $\bG^{(1)}$ via the Frobenius morphism $\Fr_\bG$), we have
\[
\ZHC \cap \ZFr = (\ZFr)^\bG \cong \scO(\bg^{*(1)}/\bG^{(1)}).
\]
Under our assumptions we have the Chevalley isomorphism
\[
\bg^{*(1)}/\bG^{(1)} \simto \bt^{*(1)}/\bWf
\]
(induced by restriction of linear forms), see~\cite[\S 4.1]{bc}, and under this isomorphism the embedding
\[
\ZHC \cap \ZFr \hookrightarrow \ZHC
\]
corresponds to the morphism $\bt^*/(\bWf,\bullet) \to \bt^{*(1)}/\bWf$ induced by the Artin--Schreier morphism
\begin{equation}
\label{eqn:Artin-Schreier}
\bt^* \to \bt^{*(1)},
\end{equation}
which itself corresponds to the map $\scO(\bt^{*(1)}) \to \scO(\bt^*)$ given by $x \mapsto x^\ell - x^{[\ell]}$ for $x \in \bt$. We therefore obtain an isomorphism
\[
\ZFr \otimes_{\ZFr \cap \ZHC} \ZHC \cong \scO \bigl( \bg^{*(1)} \times_{\bt^{*(1)}/\bWf} \bt^*/(\bWf,\bullet) \bigr),
\]
hence a canonical algebra morphism
\begin{equation}
\label{eqn:morph-center}
\scO \bigl( \bg^{*(1)} \times_{\bt^{*(1)}/\bWf} \bt^*/(\bWf,\bullet) \bigr) \to \cU \bg
\end{equation}
with central image. In fact it is well known that under our assumptions this morphism identifies the left-hand side with the center $\cZ$ of $\cU\bg$; see e.g.~\cite[Theorem~3.5(5)]{brown-gordon}. Using this morphism, one can view $\cU\bg$ as a coherent sheaf of $\scO_{\bg^{*(1)} \times_{\bt^{*(1)}/\bWf} \bt^*/(\bWf,\bullet)}$-algebras on the affine scheme $\bg^{*(1)} \times_{\bt^{*(1)}/\bWf} \bt^*/(\bWf,\bullet)$. (We will use the same notation for this sheaf of rings.)

Below we will be interested in some ``completions''\footnote{Here we use this term since it is suggestive of the basic idea of this construction; but we insist that these algebras are not completions of $\cU\bg$ in the formal sense (of e.g.~\cite[\href{https://stacks.math.columbia.edu/tag/00M9}{Tag 00M9}]{stacks-project}).} of $\cU\bg$ defined as follows. Given $\lambda \in X^*(\bT)$, we will denote by $(\cU \bg)^{\widehat{\lambda}}$ the pullback of $\cU\bg$ under the natural morphism
\[
\bg^{*(1)} \times_{\bt^{*(1)}/\bWf} \FN_{\bt^*/(\bWf,\bullet)}(\{\tla\}) \to \bg^{*(1)} \times_{\bt^{*(1)}/\bWf} \bt^*/(\bWf,\bullet).
\]
Since the left-hand side is an affine scheme, we can (and will) consider $(\cU \bg)^{\widehat{\lambda}}$ as a plain $\bk$-algebra.
%, and
%(We will often use the term ``completed'' to qualify these algebras; note however that $(\cU \bg)^{\widehat{\lambda}}$ is \emph{not} a completion of $\cU\bg$ in the usual sense, since $\cU\bg$ is not finite over $\ZHC$.) 
%we will denote by $\Modfg((\cU \bg)^{\widehat{\lambda}})$ the category of finitely generated modules over this algebra. 
Note that $(\cU \bg)^{\widehat{\lambda}}$ is left and right noetherian, since it is finite over the commutative ring $\scO(\bg^{*(1)} \times_{\bt^{*(1)}/\bWf} \FN_{\bt^*/(\bWf,\bullet)}(\{\tla\}))$, which itself is noetherian because it is of finite type over the noetherian ring $\scO(\FN_{\bt^*/(\bWf,\bullet)}(\{\tla\}))$.

Below we will use the following properties.

\begin{prop}
\label{prop:flatness-Ug}
The algebra $\cU\bg$ is free over $\ZFr$, and flat over $\ZHC$.
\end{prop}

\begin{proof}
The freeness claim is classical, and follows from the Poincar\'e--Birkhoff--Witt theorem. To prove the flatness statement, we consider the Poincar\'e--Birkhoff--Witt filtration of $\cU\bg$, and the filtration on $\ZHC \cong \scO(\bt^* / (\bWf, \bullet))$ induced by the degree filtration on $\scO(\bt^*)$. Then $\cU\bg$ is a filtered module over the filtered ring $\ZHC$, and the associated graded is $\scO(\bg^*)$ seen as a module over $\scO(\bt^* / \bWf) \cong \scO(\bg^*/\bG)$. Now the quotient morphism $\bg \to \bg/\bG$ is flat (see~\cite[Proposition~4.2.6]{bc}; in our present setup all torsion primes are bad, so that this result applies), hence the same claim holds for $\bg^*$, and then the desired claim follows from~\cite[Chap.~2, Proposition~3.12]{bjork}.
\end{proof}

%------------------------------------------------------
\subsection{Harish-Chandra bimodules}
\label{ss:def-HC}
%------------------------------------------------------

%As in~\cite[\S 3.4]{br-Hecke} w
We will denote by $\widetilde{\HC}$ the category of ``Harish-Chandra bimodules'' for $\bG$, i.e.~$\cU\bg$-bimodules $M$ endowed with an (algebraic) action of $\bG$ such that
\begin{enumerate}
 \item the action morphisms $\cU\bg \otimes M \to M$ and $M \otimes \cU\bg \to M$ are $\bG$-equivariant (with respect to the diagonal $\bG$-actions on $\cU\bg \otimes M$ and $M \otimes \cU\bg$);
 \item the $\bg$-action obtained by differentiating the $\bG$-action is given by $(x,v) \mapsto x \cdot v - v \cdot x$.
% \item 
% \label{it:HC-cond-3}
% $M$ is finitely generated both as a left and as a right $\cU\bg$-module.
\end{enumerate}
The morphisms in this category are morphisms of bimodules which commute with the $\bG$-actions. 
%The tensor product of bimodules (with diagonal $\bG$-action) endows $\HC$ with a structure of monoidal category.
As explained in~\cite[\S 3.4]{br-Hecke}, the action of $\cU\bg \otimes_\bk \cU\bg^\op$ on any Harish-Chandra bimodule factors through an action of the algebra
\[
\sfU := \cU \bg \otimes_{\ZFr} \cU\bg^\op,
\]
which is finite as an $\scO(\ZFr)$-module; hence a Harish-Chandra bimodule is finitely generated as a $\cU\bg \otimes_\bk \cU\bg^\op$-module if and only if it is finitely generated as a left $\cU\bg$-module, if and only if it is finitely generated as a right $\cU\bg$-module, if and only if it is finitely generated as a $\ZFr$-module (either for the left or for the right action). The full abelian subcategory of $\widetilde{\HC}$ whose objects are those satisfying these conditions will be denoted $\HC$.

We will consider the category $\Mod^{\bG}(\sfU)$ of $\bG$-equivariant $\sfU$-modules, and the full subcategory $\Modfg^{\bG}(\sfU)$ of finitely generated modules. Then
%and the condition~\eqref{it:HC-cond-3} above is equivalent to requiring that $M$ is finitely generated as a $\cU$-module, or as a left $\cU\bg$-module, or as a right $\cU\bg$-module. In particular
we have canonical fully faithful exact functors
\[
 \HC \to \Modfg^{\bG}(\sfU), \qquad \widetilde{\HC} \to \Mod^{\bG}(\sfU).
\]
%where the right-hand side is the category of $\bG$-equivariant finitely generated $\sfU$-modules (where $\bG$ acts diagonally on $\sfU$). Since $\cU\bg$ is finite over $\ZFr$, 
The tensor product of $\cU\bg$-bimodules provides a monoidal product
\[
 (-) \otimes_{\cU\bg} (-) : \Mod^{\bG}(\sfU) \times \Mod^{\bG}(\sfU) \to \Mod^{\bG}(\sfU)
\]
which stabilizes the subcategories $\widetilde{\HC}$, $\Modfg^{\bG}(\sfU)$ and $\HC$.

It is easily seen that forgetting the right action of $\cU\bg$ defines an equivalence of categories
\begin{equation}
 \label{eqn:HC-modules}
 \HC \simto \Modfg^{\bG}(\cU\bg)
\end{equation}
where the right-hand side is the category of $\bG$-equivariant finitely generated $\cU\bg$-modules; see~\cite[\S 3.4]{br-Hecke} for details. (Of course we also have an equivalence $\HC \simto \Modfg^{\bG}(\cU\bg^\op)$ given by forgetting the \emph{left} action.)

%\begin{rmk}
%\label{rmk:HC-diag-ind}

%There is another important way of producing 
One can produce objects of $\widetilde{\HC}$ out of objects of $\Rep^\infty(\bG)$ by ``diagonal induction," see~\cite[\S 3.4]{br-Hecke}. Namely, given $V \in \Rep^\infty(\bG)$ one can consider the object $V \otimes \cU\bg$ where $\bG$ acts diagonally, the left copy of $\cU\bg$ acts diagonally (via the differential of the $\bG$-action on $V$ and left multiplication on $\cU\bg$) and the right copy acts by right multiplication on $\cU\bg$. In fact, this object is isomorphic to the object $\cU\bg \otimes V$ where $\bG$ acts diagonally,
the left copy of $\cU\bg$ acts by left multiplication on $\cU\bg$ and the right copy acts diagonally, via (the opposite of) the differential of the $\bG$-action on $V$ and right multiplication on $\cU\bg$.

%\end{rmk}

%------------------------------------------------------
\subsection{Central characters and completions}
\label{ss:central-characters}
%------------------------------------------------------

We set
\[
 \bZ := \bt^*/(\bWf,\bullet) \times_{\bt^{*(1)}/\bWf} \bt^*/(\bWf,\bullet) \cong \Spec(\ZHC \otimes_{\ZHC \cap \ZFr} \ZHC)
\]
(see~\S\ref{ss:center-Ug}); then $\scO(\bZ)$ is a central subalgebra in $\sfU$.
For $\lambda,\mu \in X^*(\bT)$, we also set
\[
 \bZ^{\hla,\hmu} = \FN_{\bZ}( \{ (\tla,\tmu) \} ), \quad
 \sfU^{\widehat{\lambda},\widehat{\mu}} := 
 %(\cU\bg \otimes_{\ZFr} \cU\bg^\op) 
 \sfU \otimes_{\scO(\bZ)} \scO (\bZ^{\hla,\hmu}).
\]
The (diagonal) $\bG$-action on $\sfU$
%$\cU \bg \otimes_{\ZFr} \cU\bg^\op$ 
induces an (algebraic) action on $\sfU^{\hla,\hmu}$. Moreover $\sfU^{\hla,\hmu}$ is left and right noetherian, see~\cite[\S 3.5]{br-Hecke}. We can therefore consider the abelian category
\[
 \Mod^{\bG}(\sfU^{\hla,\hmu})
\]
of $\bG$-equivariant modules over this algebra, its full subcategory $\Modfg^{\bG}(\sfU^{\hla,\hmu})$ of finitely generated modules, and the full abelian subcategories
\[
\HC^{\hla,\hmu} \subset \Modfg^{\bG}(\sfU^{\hla,\hmu}), \qquad \widetilde{\HC}^{\hla,\hmu} \subset \Mod^{\bG}(\sfU^{\hla,\hmu})
\]
of ``Harish-Chandra bimodules,'' i.e.~modules on which the differential of the $\bG$-action coincides with the action given by $(x,m) \mapsto x \cdot m - m \cdot x$. 

If $M$ is a $\bG$-equivariant $\sfU^{\hla,\hmu}$-module, the assignment $(x,m) \mapsto x \cdot m - m \cdot x$ defines an action of $\cU\bg$ which factors through an action of the distribution algebra $\Dist(\bG_1)$ of the Frobenius kernel $\bG_1$, i.e.~the (scheme-theoretic) kernel of $\Fr_{\bG}$. (Here we use the canonical identification $\Dist(\bG_1)=\cU\bg \otimes_{\ZFr} \bk$ where $\bk$ is the trivial $\ZFr$-module.) The module $M$ acquires in this way a structure of $\bG_1$-module which, combined with the action of $\bG$, provides an action of $\bG \ltimes \bG_1$. Playing with the definition one sees that the objects of $\widetilde{\HC}^{\hla,\hmu}$ are the objects of $\Mod^{\bG}(\sfU^{\hla,\hmu})$ such that this action of $\bG \ltimes \bG_1$ factors through the multiplication morphism $\bG \ltimes \bG_1 \to \bG$ or, in other words, on which the action of the antidiagonal copy of $\bG_1$ is trivial. 

% Note that the algebra
% \[
%  \bigl( Z(\cU\bg) \otimes_{\ZFr} Z(\cU\bg) \bigr) \otimes_{\scO(\fD)} \scO(\fD^{\hla,\hmu})
% \]
% is Noetherian, since it is finitely generated as an algebra of the Noetherian ring $\scO(\fD^{\hla,\hmu})$. Since $\cU^{\hla,\hmu}$ is finite as a module over this algebra, it is left and right Noetherian, so that the category $\Modfg^{\bG}(\cU^{\hla,\hmu})$ is abelian. Its subcategory $\HC^{\hla,\hmu}$ is an abelian subcategory.

%It will sometimes be convenient to relax the condition of finite generation: we will denote by $\widetilde{\HC}^{\hla,\hmu}$ the abelian category of $\bG$-equivariant $\cU^{\hla,\hmu}$-modules on which the differential of the $\bG$-action coincides with the action given by $(x,m) \mapsto x \cdot m - m \cdot x$. 

The following statement follows from standard arguments based on the fact that $\sfU^{\hla,\hmu}$ is a noetherian ring (see e.g.~\cite[Proof of Corollary~2.11]{arinkin-bezrukavnikov}).

\begin{lem}
\label{lem:HC-tHC}
The obvious functor
\[
\Db \HC^{\hla,\hmu} \to \Db \widetilde{\HC}^{\hla,\hmu}
\]
is fully faithful, and identifies the left-hand side with the full subcategory of the right-hand side whose objects are the complexes all of whose cohomology objects belong to $\HC^{\hla,\hmu}$.
\end{lem}

%This description shows that the embedding
%\[
% \HC^{\hla,\hmu} \to \Modfg^{\bG}(\cU^{\hla,\hmu})
%\]
%admits left and right adjoints, given respectively by
%\[
% M \mapsto M \otimes_{\Dist(\bG_1)} \bk \quad \text{and} \quad M \mapsto M^{\bG_1},
%\]
%where the action of $\bG_1$ we consider here is that given by the diagonal embedding $\bG_1 \hookrightarrow \bG \ltimes \bG_1$ and the construction above. 

%Standard results show that the category $\Modfg^{\bG}(\cU^{\hla,\hmu})$ has enough injectives; we deduce that $\HC^{\hla,\hmu}$ has enough injective objects, and that any injective object is a direct summand in an object of the form $M^{\bG_1}$ with $M \in 

%If we denote by $\Modfg^{\bG}(\sfU)$ the category of $\bG$-equivariant $\sfU$-modules, then w
We have a natural exact functor
\[
 \sfC^{\hla,\hmu} : \Mod^{\bG}(\sfU) \to \Mod^{\bG}(\sfU^{\hla,\hmu})
\]
given by $\scO (\bZ^{\hla,\hmu}) \otimes_{\scO(\bZ)} (-)$, which restricts to functors 
\[
\Modfg^{\bG}(\sfU) \to \Modfg^{\bG}(\sfU^{\hla,\hmu}), \quad \widetilde{\HC} \to \widetilde{\HC}^{\hla,\hmu} \quad \text{and} \quad
% with $\scO(\FN_{\bt^*/(\bW,\bullet) \times_{\bt^{*(1)}/\bW} \bt^*/(\bW,\bullet)}(\tla,\tmu))$ over
% \[
% \scO(\bt^*/(\bW,\bullet) \times_{\bt^{*(1)}/\bW} \bt^*/(\bW,\bullet)) \cong \ZHC \otimes_{\ZHC \cap \ZFr} \ZHC.
% \]
%This functor of course restricts to an exact functor 
\HC \to \HC^{\hla,\hmu}.
\] 
Recall from~\S\ref{ss:def-HC}
%Remark~\ref{rmk:HC-diag-ind} 
that for $V \in \Rep^\infty(\bG)$ we have an object $V \otimes \cU\bg \in \widetilde{\HC}$.
%For any $V \in \Rep(\bG)$ the tensor product $V \otimes \cU\bg$ defines an object of $\HC$ where the left copy of $\cU\bg$ acts diagonally (with respect to the differential of the $\bG$-action on $V$ and multiplication on the left on $\cU\bg$), the right copy of $\cU\bg$ acts via multiplication on the right on $\cU\bg$, and $\bG$ acts diagonally. 
Below we will be particularly interested in the modules of the form $\sfC^{\hla,\hmu}(V \otimes \cU\bg)$ where $V$ is a finite-dimensional tilting module. We will denote by
\[
 \HC^{\hla,\hmu}_{\mathrm{diag}}
\]
the full idempotent-complete additive subcategory of $\HC^{\hla,\hmu}$ whose objects are the direct sums of direct summands of objects of the form $\sfC^{\hla,\hmu}(V \otimes \cU\bg)$ with $V \in \Rep(\bG)$, and by
\[
 \HC^{\hla,\hmu}_{\mathrm{diag},\mathrm{tilt}}
\]
the full idempotent-complete additive subcategory defined in the same way with the further condition that $V$ is tilting.

For some constructions and proofs below it will be convenient to treat all values of $\lambda$ and $\mu$ at once, in the following way. 
Consider the extended affine Weyl group
\[
 \bW := \bWf \ltimes X^*(\bT).
\]
For any $\nu \in X^*(\bT)$, the image of $\nu$ in $\bW$ will be denoted $\st(\nu)$.
%Then $\bW$ acts naturally 
We have a ``dot action'' of $\bW$ on $X^*(\bT)$ given by 
\begin{equation}
\label{eqn:dot-action}
(\st(\nu) w) \bullet \eta = w(\eta+\varsigma)-\varsigma+\ell\nu
\end{equation}
for $w \in \bWf$ and $\nu,\eta \in X^*(\bT)$. Then the differentiation morphism $X^*(\bT) \to \bt^*$ is $\bW$-equivariant, where $\bW$ acts on $\bt^*$ via the dot-action of its quotient $\bWf$ considered in~\S\ref{ss:center-Ug}. As a consequence,
the point $\tla$, resp.~$\tmu$, only depends on the image of $\lambda$, resp.~$\mu$, in $X^*(\bT)/(\bW,\bullet)$. We can therefore consider the notation $\sfC^{\hla,\hmu}$, $\HC^{\hla,\hmu}$, $\bZ^{\hla,\hmu}$, etc. for $\lambda,\mu \in X^*(\bT)/(\bW,\bullet)$. By abuse we will also denote by $0$ the image of $0 \in \bt^{*(1)}$ in $\bt^{*(1)}/\bWf$, and set
\[
 \bZ^\wedge := \bZ \times_{\bt^{*(1)}/\bWf} \FN_{\bt^{*(1)}/\bWf}(\{0\}).
\]
We will denote by $\fn \subset \scO(\bt^{*(1)}/\bWf)$ the maximal ideal corresponding to $0$; then, since the morphism $\bZ \to \bt^{*(1)}/\bWf$ is finite, $\scO(\bZ^\wedge)$ identifies with the completion of $\scO(\bZ)$ with respect to the ideal generated by $\fn$, see~\cite[Equation~(3.8)]{br-Hecke}.
% If we denote by $\bZ^\wedge$ the spectrum of the completion of $\scO(\bZ)$ with respect to the ideal generated by the maximal ideal of $\scO(\bt^{*(1)}/\bWf)$ corresponding to the image of $0 \in \bt^{*(1)}$,
% %formal neighborhood in $\fD$ of the preimage of the image of $0 \in \bt^{*(1)}$ under the projection $\fD \to \bt^{*(1)}/\bWf$, 

If we set
\[
 \sfU^\wedge = \sfU \otimes_{\scO(\bZ)} \scO(\bZ^\wedge),
\]
then $\sfU^\wedge$ is a $\bG$-equivariant $\scO(\bZ^\wedge)$-algebra, so that
we can as above consider the category
\[
\Mod^{\bG}(\sfU^\wedge)
\]
of $\bG$-equivariant $\sfU^\wedge$-modules, the full subcategory $\Modfg^{\bG}(\sfU^\wedge)$ of finitely generated modules, and the full subcategories 
\[
\widetilde{\HC}^\wedge \subset \Mod^{\bG}(\sfU^\wedge), \qquad
\HC^\wedge \subset \Modfg^{\bG}(\sfU^\wedge)
\]
of Harish-Chandra bimodules. In fact, setting
 \[
  (\cU\bg)^\wedge := \cU\bg \otimes_{\scO(\bt^{*(1)}/\bWf)} \scO(\FN_{\bt^{*(1)}/\bWf}(\{0\})),
 \]
%since the morphism $\bZ \to \bt^{*(1)}/\bWf$ is finite 
we have
\begin{equation}
\label{eqn:Uwedge-tensor-product}
 \sfU^\wedge \cong (\cU\bg)^\wedge
 \otimes_{\ZFr \otimes_{\scO(\bt^{*(1)}/\bWf)} \scO(\FN_{\bt^{*(1)}/\bWf}(\{0\}))} (\cU\bg)^{\wedge, \op},
\end{equation}
so that a $\sfU^\wedge$-module is nothing but a $(\cU\bg)^\wedge$-bimodule on which the left and right actions of the central subalgebra
\[
\ZFr \otimes_{\scO(\bt^{*(1)}/\bWf)} \scO(\FN_{\bt^{*(1)}/\bWf}(\{0\}))
\]
coincide. From this point of view it is clear that this category has a natural monoidal structure, defined by a bifunctor
\begin{equation}
\label{eqn:def-hatotimes-2}
 (-) \hatotimes_{\cU\bg} (-) : \Mod^{\bG}(\sfU^\wedge) \times \Mod^{\bG}(\sfU^\wedge) \to \Mod^{\bG}(\sfU^\wedge),
\end{equation}
which stabilizes the subcategories of finitely generated modules and Harish-Chandra bimodules in the obvious sense.
We also have a canonical monoidal exact functor
\[
 \sfC^\wedge : \Mod^{\bG}(\sfU) \to \Mod^{\bG}(\sfU^\wedge)
\]
given by
\[
\scO(\bZ^\wedge) \otimes_{\scO(\bZ)} (-) \cong \scO(\FN_{\bt^{*(1)}/\bWf}(\{0\})) \otimes_{\scO(\bt^{*(1)}/\bWf)} (-),
\]
which sends Harish-Chandra bimodules to Harish-Chandra bimodules.

\begin{rmk}
\label{rmk:completion-diag-induced}
Let $V \in \Rep^\infty(\bG)$, and recall the Harish-Chandra bimodule $V \otimes \cU\bg$ considered in~\S\ref{ss:def-HC}. Since $\ZFr$ acts trivially on $V$, we have a canonical identification
\[
\sfC^\wedge(V \otimes \cU\bg) \cong V \otimes (\cU\bg)^\wedge,
\]
for the natural actions on the right-hand side.
\end{rmk}

%Using the isomorphism~\eqref{eqn:decomp-Dwedge} we see that 
By~\cite[Lemma~3.4]{br-Hecke} the natural morphism
\begin{equation}
\label{eqn:decomp-Zwedge}
 \scO(\bZ^\wedge) \to \prod_{\lambda,\mu \in X^*(\bT)/(\bW,\bullet)} \scO(\bZ^{\hla,\hmu})
\end{equation}
is an isomorphism. (Note that the indexing set on the right-hand side is finite, so that the product is a direct sum.) It follows that
there are canonical decompositions
\begin{align}
 \Mod^{\bG}(\sfU^\wedge) &= \bigoplus_{\lambda,\mu \in X^*(\bT)/(\bW,\bullet)} \Mod^{\bG}(\sfU^{\hla,\hmu}), \\
 \label{eqn:decomp-Zwedge-cat}
 \widetilde{\HC}^\wedge &= \bigoplus_{\lambda,\mu \in X^*(\bT)/(\bW,\bullet)} \widetilde{\HC}^{\hla,\hmu}
\end{align}
%and by construction the bifunctors~\eqref{eqn:def-hatotimes} are obtained by restriction from the bifunctor~\eqref{eqn:def-hatotimes-2}. Moreover, 
which restrict to subcategories of finitely generated modules,
and for any $M \in \Mod^{\bG}(\sfU)$ we have a canonical isomorphism
\begin{equation}
\label{eqn:decomp-Zwedge-C}
 \sfC^\wedge(M) \cong \bigoplus_{\lambda,\mu \in X^*(\bT)/(\bW,\bullet)} \sfC^{\hla,\hmu}(M).
\end{equation}
Moreover the ring 
 $(\cU\bg)^\wedge$
is left and right noetherian (by the same arguments as for $(\cU\bg)^{\hla}$ or $\sfU^{\hla,\hmu}$) and, as for~\eqref{eqn:HC-modules}, forgetting the right action of $(\cU\bg)^\wedge$ defines equivalences of categories
\begin{equation}
 \label{eqn:HC-modules-wedge}
 \widetilde{\HC}^\wedge \simto \Mod^{\bG} \bigl( (\cU\bg)^\wedge \bigr), \quad \HC^\wedge \simto \Modfg^{\bG} \bigl( (\cU\bg)^\wedge \bigr).
\end{equation}
Of course, similar comments apply to forgetting the left action instead.

As explained in~\cite[\S 3.7]{br-Hecke}, for any $\lambda,\mu,\nu \in X^*(\bT)$ the bifunctor~\eqref{eqn:def-hatotimes-2} restricts to a bifunctor
\begin{equation}
\label{eqn:def-hatotimes}
% (-) \hatotimes_{\cU\bg} (-) : 
 \Mod^{\bG}(\sfU^{\hla,\hmu}) \times \Mod^{\bG}(\sfU^{\hmu,\hnu}) \to \Mod^{\bG}(\sfU^{\hla,\hnu})
\end{equation}
which stabilizes the subcategories of finitely-generated modules and Harish-Chandra bimodules in the obvious sense. These bifunctors are associative and unital in the appropriate sense; in particular, in case $\lambda=\mu=\nu$, we obtain a monoidal structure on the category $\Mod^{\bG}(\sfU^{\hla,\hla})$.

Let us note the following other consequence of these considerations.

\begin{lem}
\label{lem:surjection-diag-ind-comp}
Let $\lambda,\mu \in X^*(\bT)$. Then for any $M$ in $\widetilde{\HC}^{\widehat{\lambda},\widehat{\mu}}$, resp.~in $\HC^{\widehat{\lambda},\widehat{\mu}}$, there exists $V$ in $\Rep^\infty(\bG)$, resp.~in $\Rep(\bG)$, and a surjection $\sfC^{\widehat{\lambda},\widehat{\mu}}(V \otimes \cU\bg) \twoheadrightarrow M$, where the structure on $V \otimes \cU\bg$ is as in~\S\ref{ss:def-HC}.
%Remark~\ref{rmk:HC-diag-ind}.
%If $M$ belongs to $\HC^{\widehat{\lambda},\widehat{\mu}}$, then $V$ can be chosen finite-dimensional.
%\footnote{Add a version for $\widetilde{\HC}^{\widehat{\lambda},\widehat{\mu}}$?}
\end{lem}

\begin{proof}
We treat the case of $\widetilde{\HC}^{\widehat{\lambda},\widehat{\mu}}$; the other one is similar.
Consider $M$ as an object of $\widetilde{\HC}^\wedge$ via the decomposition in~\eqref{eqn:decomp-Zwedge-cat}. 
%Then $M$ is finitely generated as a left $(\cU\bg)^\wedge$-module. 
If $V \subset M$ is a 
%finite-dimensional 
$\bG$-stable subspace which generates $M$ as a right $(\cU\bg)^\wedge$-module, then we have a surjection of $\bG$-equivariant right $(\cU\bg)^\wedge$-modules
\begin{equation}
\label{eqn:surjection-diag-ind-comp}
 V \otimes (\cU\bg)^\wedge \twoheadrightarrow M
\end{equation}
where $\bG$ acts diagonally on the left-hand side and $(\cU\bg)^\wedge$ acts via right multiplication on the second factor. By the analogue for right modules of~\eqref{eqn:HC-modules-wedge}, the left-hand side can be canonically ``lifted'' to $\widetilde{\HC}^\wedge$ so that this morphism becomes a morphism in $\widetilde{\HC}^\wedge$; moreover, with this structure we have $V \otimes (\cU\bg)^\wedge = \sfC^\wedge(V \otimes \cU\bg)$, see Remark~\ref{rmk:completion-diag-induced}. Now by~\eqref{eqn:decomp-Zwedge-C} we have a canonical decomposition
\[
\sfC^\wedge(V \otimes \cU\bg) = \bigoplus_{\nu,\eta \in X^*(\bT)/(\bW,\bullet)} \sfC^{\widehat{\nu},\widehat{\eta}}(V \otimes \cU\bg).
\]
Since the action of $\scO(\bZ^\wedge)$ on $M$ factors through an action of $\scO(\bZ^{\widehat{\lambda},\widehat{\mu}})$, our surjection~\eqref{eqn:surjection-diag-ind-comp} must factor through a surjection $\sfC^{\widehat{\lambda},\widehat{\mu}}(V \otimes \cU\bg) \twoheadrightarrow M$.
\end{proof}

%-----------------------------------
\subsection{\texorpdfstring{$\Ext$}{Ext}-vanishing}
%-----------------------------------

This subsection is devoted to the proof of the following claim.

\begin{prop}
\label{prop:Ext-vanishing}
For any $\lambda,\mu \in X^*(\bT)$ and any $V,V' \in \Tilt(\bG)$, we have
 \[
  \Ext^n_{\HC^{\hla,\hmu}} \bigl( \sfC^{\hla,\hmu}(V \otimes \cU\bg), \sfC^{\hla,\hmu}(V' \otimes \cU\bg) \bigr)=0
 \]
 for any $n \in \Z_{>0}$.
\end{prop}

Before explaining the proof of this proposition we note its most important consequence (in fact, an equivalent formulation).

\begin{cor}
\label{cor:Kb-Db-HC}
 For any $\lambda,\mu \in X^*(\bT)$, the natural functor
 \[
  \Kb \HC^{\hla,\hmu}_{\mathrm{diag},\mathrm{tilt}} \to \Db \HC^{\hla,\hmu}
 \]
is fully faithful.
\end{cor}

\begin{proof}
 The corollary follows from Proposition~\ref{prop:Ext-vanishing} and Be{\u\i}linson's lemma.
\end{proof}

%The considerations in~\S\ref{ss:central-characters}
The decompositions in~\eqref{eqn:decomp-Zwedge-cat} and~\eqref{eqn:decomp-Zwedge-C}
reduce Proposition~\ref{prop:Ext-vanishing} to the following claim.

\begin{prop}
\label{prop:Ext-vanishing-2}
For any $V,V' \in \Tilt(\bG)$ and any $n \in \Z_{>0}$ we have
 \[
  \Ext^n_{\HC^\wedge} \bigl( \sfC^\wedge(V \otimes \cU\bg), \sfC^\wedge(V' \otimes \cU\bg) \bigr)=0.
 \]
\end{prop}

The proof of Proposition~\ref{prop:Ext-vanishing-2} will use the following property, where we denote by $\mathrm{S}(\bg)$ the symmetric algebra of $\bG$, which we endow with the obvious $\bG$-module structure induced by the adjoint action on $\bg$.

\begin{lem}
\label{lem:Sg-good-filtration}
The $\bG$-module $\mathrm{S}(\bg)$ admits a good filtration.
\end{lem}

\begin{proof}
This property follows from the classical claim in~\cite[\S II.4.22]{jantzen}. Namely,
 by assumption we have an isomorphism of $\bG$-modules $\bg \simto \bg^*$, so that our claim is equivalent to the claim that $\mathrm{S}(\bg^*)$ admits a good filtration. 
% By~\cite[Proposition~3.2.7]{donkin}, to prove this it suffices to prove that $\mathrm{S}(\bg^*)$ admits a good filtration as a module for the derived subgroup $\mathscr{D}(\bG)$. 
 If we denote by $\bg'$ the Lie algebra of $\mathscr{D}(\bG)$ then we have an exact sequence of $\bG$-modules $\bg' \hookrightarrow \bg \twoheadrightarrow \bg/\bg'$, and $\bG$ acts trivially on $\bg/\bg'$. We deduce an exact sequence of $\bG$-modules $(\bg/\bg')^* \hookrightarrow \bg^* \twoheadrightarrow (\bg')^*$, which implies that for any $m \geq 0$ the $\bG$-module $\mathrm{S}^m(\bg^*)$ admits a (finite) filtration with associated graded
 \[
  \bigoplus_{i+j=m} \mathrm{S}^i((\bg')^*) \otimes_\bk \mathrm{S}^j((\bg/\bg')^*),
 \]
 where the action on $\mathrm{S}^j((\bg/\bg')^*)$ is trivial. In view of our assumptions and~\cite[Equation~(1) in~\S II.4.22]{jantzen}, each $\mathrm{S}^i((\bg')^*)$ admits a good filtration as a $\mathscr{D}(\bG)$-module, hence as a $\bG$-module by~\cite[Proposition~3.2.7]{donkin}, which concludes the proof.
\end{proof}

\begin{proof}[Proof of Proposition~\ref{prop:Ext-vanishing-2}]
 As in the proof of~\cite[Lemma~3.6]{br-Hecke} one easily sees that the functor
 \[
  \sfC^\wedge(V \otimes \cU\bg) \hatotimes_{\cU\bg} (-) : \HC^\wedge \to \HC^{\wedge}
 \]
is left adjoint to
 \[
  \sfC^\wedge(V^* \otimes \cU\bg) \hatotimes_{\cU\bg} (-) : \HC^\wedge \to \HC^{\wedge}.
 \]
 Since tensor products of tilting $\bG$-modules are tilting (see~\cite[\S II.E.7]{jantzen}), this reduces the proof of the proposition to the case $V=\bk$ is the trivial module.
 We will therefore replace the notation $V'$ by $V$. 
% Next, as explained in~\cite[\S 3.4]{br-Hecke} (and recalled in~\S\ref{ss:def-HC}), for any $V \in \Rep(\bG)$ there exists an isomorphism of Harish-Chandra bimodules
% \[
%  V \otimes \cU\bg \simto \cU\bg \otimes V
% \]
%where in the right-hand side the left action of $\cU\bg$ is by left mutliplication on $\cU\bg$, and the $\bG$-action and the right action of $\cU\bg$ are diagonal. To conclude, it therefore suffices to prove that 
% \[
%  \Ext^n_{\HC^\wedge} \bigl( \sfC^\wedge(\cU\bg), \sfC^\wedge(\cU\bg \otimes V) \bigr)=0 \quad \text{for any $V \in \Tilt(\bG)$ and $n>0$.}
% \]
 Next, by Remark~\ref{rmk:completion-diag-induced}, for any $V \in \Rep(\bG)$ we have
 \begin{equation}
 \label{eqn:completion-diag-induced}
  \sfC^\wedge(V \otimes \cU\bg) \cong V \otimes (\cU\bg)^\wedge
 \end{equation}
 with the obvious actions on the right-hand side.
%where $(\cU\bg)^\wedge$ acts on the right-hand side in the natural, diagonal, way.
Using the equivalence~\eqref{eqn:HC-modules-wedge}, these considerations show that to prove the proposition it suffices to prove that
 \[
  \Ext^n_{\Modfg^{\bG} ((\cU\bg)^\wedge)} \bigl( (\cU\bg)^\wedge, V \otimes (\cU\bg)^\wedge \bigr)=0 \quad \text{for any $V \in \Tilt(\bG)$ and $n>0$.}
 \]
% for any $V \in \Tilt(\bG)$ tilting and any $n>0$. 
 
 Since $(\cU\bg)^\wedge$ is left noetherian, as in Lemma~\ref{lem:HC-tHC},
% if we denote by $\Mod^{\bG} ((\cU\bg)^\wedge)$ the category of all (i.e.~not necessarily finitely generated) $\bG$-equivariant $(\cU\bg)^\wedge$-modules, then 
 the canonical functor
 \[
  \Db \Modfg^{\bG} ((\cU\bg)^\wedge) \to \Db \Mod^{\bG} ((\cU\bg)^\wedge)
 \]
is fully faithful, so that what we have to prove is that
 \[
  \Ext^n_{\Mod^{\bG} ((\cU\bg)^\wedge)} \bigl( (\cU\bg)^\wedge, V \otimes (\cU\bg)^\wedge \bigr)=0 \quad \text{for any $V \in \Tilt(\bG)$ and $n>0$.}
 \]
% for any $V \in \Tilt(\bG)$ and any $n>0$. 
%Now if we denote by $\Rep^{\infty}(\bG)$ the category of all algebraic $\bG$-modules, 
The functor
 \[
 (\cU\bg)^\wedge \otimes (-) : \Rep^{\infty}(\bG) \to \Mod^{\bG} ((\cU\bg)^\wedge)
 \]
is left adjoint to the forgetful functor
\[
 \Mod^{\bG} ((\cU\bg)^\wedge) \to \Rep^{\infty}(\bG),
\]
which reduces the desired claim to the fact that
%Hence, what we have to prove is that
\[
 \Ext^n_{\Rep^{\infty}(\bG)} \bigl( \bk, V \otimes (\cU\bg)^\wedge \bigr)=0,
\]
i.e.
\[
 \mathsf{H}^n \bigl( \bG, V \otimes (\cU\bg)^\wedge \bigr)=0,
\]
 for any $V \in \Tilt(\bG)$ and any $n>0$. Now, using e.g.~the fact that the cohomology can be computed using the Hochschild complex (see~\cite[\S I.4.14]{jantzen}) and the flatness of $\scO(\FN_{\bt^{*(1)}/\bWf}(\{0\}))$ over $\scO(\bt^{*(1)}/\bWf)$, one sees that for any $V \in \Rep(\bG)$ and $n \in \Z$ we have a canonical isomorphism
 \[
 \mathsf{H}^n \bigl( \bG, V \otimes (\cU\bg)^\wedge \bigr)=\mathsf{H}^n ( \bG, V \otimes \cU\bg) \otimes_{\scO(\bt^{*(1)}/\bWf)} \scO(\FN_{\bt^{*(1)}/\bWf}(\{0\})).
\]
To conclude the proof it therefore suffices to prove that
\begin{equation}
\label{eqn:HC-vanishing-cohomology}
 \mathsf{H}^n ( \bG, V \otimes \cU\bg) = 0 \quad \text{for any $V \in \Tilt(\bG)$ and $n>0$,}
\end{equation}
%for any $V \in \Tilt(\bG)$ and $n>0$.
where $\bG$ acts diagonally on $V \otimes \cU\bg$.

Consider the PBW filtration $((\cU_{\leq m} \bg) : m \geq 0)$ of $\cU\bg$. Then since cohomology commutes with filtrant direct limits (see~\cite[Lemma~I.4.17]{jantzen}), to prove~\eqref{eqn:HC-vanishing-cohomology} it suffices to prove that
\[
 \mathsf{H}^n ( \bG, V \otimes \cU_{\leq m}\bg ) = 0 \quad \text{for any $V \in \Tilt(\bG)$, $m \geq 0$ and $n>0$.}
\]
Arguing by induction on $m$, to prove this it suffices to prove that
\[
 \mathsf{H}^n ( \bG, V \otimes \mathrm{S}(\bg)) = 0 \quad \text{for any $V \in \Tilt(\bG)$ and $n>0$.}
\]
%where $\mathrm{S}(\bg)$ is the symmetric algebra of $\bg$. 
This fact follows from Lemma~\ref{lem:Sg-good-filtration} since
modules admitting a good filtration have no cohomology in positive degree (see~\cite[Proposition~II.4.16]{jantzen}), $V$ admits a good filtration (by assumption), and
%, in view of~\cite[Proposition~II.4.16]{jantzen} and the fact 
tensor products of $\bG$-modules admitting a good filtration admit a good filtration (see~\cite[Proposition~II.4.21]{jantzen}).
%, to conclude the proof it therefore suffices to show that the $\bG$-module $\mathrm{S}(\bg)$ admits a good filtration.
%
 %Then using~\cite[Corollary~2.3.9]{letellier} we see that one can assume $\bG$ is semisimple (and simply connected). In 
%
%Since tensor products of modules with good filtrations have good filtrations (see~\cite[Proposition~II.4.21]{jantzen}), this claim implies that $V \otimes \cU\bg$ has a good filtration, so that $\mathsf{H}^n ( \bG, V \otimes \cU\bg )$=0, which finishes the proof.
\end{proof}

%------------------------------------
\subsection{Monoidal structure for derived categories}
\label{ss:monoidal-structure-DHC}
%------------------------------------

%As explained in~\cite[\S 3.7]{br-Hecke}, there are monoidal structures on categories of completed Harish-Chandra bimodules. More specifically, 
Recall that in view of~\eqref{eqn:Uwedge-tensor-product} a ($\bG$-equivariant) $\sfU^\wedge$-module is a ($\bG$-equivariant) $(\cU\bg)^\wedge$-bimodule on which the two actions of the central subalgebra $\ZFr \otimes_{\scO(\bt^{*(1)}/\bWf)} \scO(\FN_{\bt^{*(1)}/\bWf}(\{0\}))$ coincide. The category $\Bimod^\bG((\cU\bg)^\wedge)$ of $\bG$-equivariant $(\cU\bg)^\wedge$-bimodules admits a monoidal structure given by the tensor product $(-) \otimes_{(\cU\bg)^\wedge} (-)$, which stabilizes the subcategories $\Mod^\bG(\sfU^\wedge)$, $\Modfg^\bG(\sfU^\wedge)$, $\widetilde{\HC}^\wedge$ and $\HC^\wedge$; this is the monoidal structure considered in~\S\ref{ss:central-characters}. 
Passing to derived categories, for $M,N$ in $D^- \Bimod^\bG((\cU\bg)^\wedge)$ we set
\begin{equation}
\label{eqn:star-HC}
M \star N := M \lotimes_{(\cU\bg)^\wedge} N;
\end{equation}
here again this bifunctor defines a monoidal structure on $D^- \Bimod^\bG((\cU\bg)^\wedge)$, with unit object the ``diagonal'' bimodule $(\cU\bg)^\wedge$.
%admits a monoidal structure given by the derived tensor product $(-) \lotimes_{(\cU\bg)^\wedge} (-)$

%\begin{lem}
%If $M,N \in \Db \Bimod((\cU\bg)^\wedge)$, then $M \star N$ belongs to $\Db \Bimod((\cU\bg)^\wedge)$.
%\end{lem}

\begin{lem}
\label{lem:tensor-prod-HC-derived}
There exist canonical bifunctors
\[
(-) \star (-) : D^- \widetilde{\HC}^\wedge \times D^- \widetilde{\HC}^\wedge \to D^- \widetilde{\HC}^\wedge
\]
and
\[
(-) \star (-) : D^- \HC^\wedge \times D^- \HC^\wedge \to D^- \HC^\wedge
\]
which define monoidal structures on $D^- \widetilde{\HC}^\wedge$ and $D^- \HC^\wedge$ and such that the forgetful functors
\[
D^- \HC^\wedge \to D^- \widetilde{\HC}^\wedge \to D^- \Bimod^\bG((\cU\bg)^\wedge)
\]
are monoidal.
\end{lem}

\begin{proof}
The lemma follows from the fact that any object of $\HC^\wedge$, resp.~$\widetilde{\HC}^\wedge$, admits a resolution by objects of $\HC^\wedge$, resp.~$\widetilde{\HC}^\wedge$, which are flat both as left and as right $(\cU\bg)^\wedge$-modules; see Lemma~\ref{lem:surjection-diag-ind-comp} and Remark~\ref{rmk:completion-diag-induced}.
%~\eqref{eqn:completion-diag-induced}.
\end{proof}

%Given $\lambda,\mu \in X^*(\bT)$, as explained in~\cite[Remark~3.5]{br-Hecke} we have a 
Now, recall the isomorphism~\eqref{eqn:decomp-Zwedge}. We similarly have a canonical isomorphism
\[
\bt^* / (\bWf,\bullet) \times_{\bt^{*(1)}/\bWf} \FN_{\bt^{*(1)}/\bWf}(\{0\}) \simto \prod_{\lambda \in X^*(\bT)/(\bW,\bullet)} \FN_{\bt^* / (\bWf,\bullet)}(\{\tla\}).
\]
%If we set
%\[
%(\cU\bg)^{\hla} := \cU\bg \otimes_{\ZHC} \scO(\FN_{\bt^* / (\bWf,\bullet)}(\{\tla\})),
%\]
We deduce a canonical decomposition
\begin{equation}
\label{eqn:Ug-wedge-product}
(\cU\bg)^\wedge = \prod_{\lambda \in X^*(\bT)/(\bW,\bullet)} (\cU\bg)^{\hla},
\end{equation}
where we use the notation of~\S\ref{ss:center-Ug}.
For $\lambda,\mu \in X^*(\bT)$ we have an identification
\begin{equation}
\label{eqn:U-hla-hmu-tensor-prod}
\sfU^{\hla,\hmu} = (\cU\bg)^{\hla} \otimes_{\ZFr \otimes_{\scO(\bt^{*(1)}/\bWf)} \scO(\FN_{\bt^{*(1)}/\bWf}(\{0\}))} (\cU\bg)^{\hmu,\op},
\end{equation}
and we deduce that a $\sfU^{\hla,\hmu}$-module is a $((\cU\bg)^{\hla},(\cU\bg)^{\hmu})$-bimodule on which the two actions of $\ZFr \otimes_{\scO(\bt^{*(1)}/\bWf)} \scO(\FN_{\bt^{*(1)}/\bWf}(\{0\}))$ coincide. 
For $\lambda,\mu \in X^*(\bT)$ we can consider the category $\Bimod^{\bG}((\cU\bg)^{\hla},(\cU\bg)^{\hmu})$ of $\bG$-equivariant $((\cU\bg)^{\hla},(\cU\bg)^{\hmu})$-bimodules, and for $\lambda,\mu,\nu \in X^*(\bT)$ the derived tensor product functor
\begin{multline*}
(-) \lotimes_{(\cU\bg)^{\hmu}} (-) : 
D^- \Bimod^{\bG}((\cU\bg)^{\hla},(\cU\bg)^{\hmu}) \times D^- \Bimod^{\bG}((\cU\bg)^{\hmu},(\cU\bg)^{\hnu}) \\
\to D^- \Bimod^{\bG}((\cU\bg)^{\hla},(\cU\bg)^{\hnu}).
\end{multline*}
Then we have
\[
D^- \Bimod^\bG((\cU\bg)^\wedge) = \prod_{\lambda,\mu \in X^*(\bT) / (\bW,\bullet)} D^- \Bimod^{\bG}((\cU\bg)^{\hla},(\cU\bg)^{\hmu}),
\]
and the bifunctor~\eqref{eqn:star-HC} is the product of the bifunctors $(-) \lotimes_{(\cU\bg)^{\hmu}} (-)$ (which will therefore also be denoted $\star$).
We also have
\[
\HC^\wedge = \prod_{\lambda,\mu \in X^*(\bT) / (\bW,\bullet)} \HC^{\hla,\hmu},
\]
and by Lemma~\ref{lem:tensor-prod-HC-derived}
%From this point of view, it is clear that 
the bifunctor $\star$ considered above restricts, for any $\lambda,\mu,\nu \in X^*(\bT)$, to a bifunctor
\begin{equation}
\label{eqn:convolution-HC-la-mu-nu}
D^- \HC^{\hla,\hmu} \times D^- \HC^{\hmu,\hnu} \to D^- \HC^{\hla,\hnu}.
\end{equation}
%given by the derived tensor product over $(\cU\bg)^{\hmu}$.
In particular, in case $\lambda=\mu=\nu$ this bifunctor defines a monoidal structure on $D^- \HC^{\hla,\hla}$, with unit object $(\cU\bg)^{\hla}$.

Recall the full subcategory $\HC^{\hla,\hmu}_{\mathrm{diag}} \subset \HC^{\hla,\hmu}$ defined in~\S\ref{ss:central-characters}. As explained in~\cite[\S 3.7]{br-Hecke}, for $\lambda,\mu,\nu \in X^*(\bT)$ and $V \in \Rep(\bG)$ we have a canonical isomorphism
\[
\sfC^{\hla,\hmu}(V \otimes V' \otimes \cU\bg) \cong
\bigoplus_{\nu \in X^*(\bT)/(\bW,\bullet)} \sfC^{\hla,\hnu}(V \otimes \cU\bg) \otimes_{(\cU\bg)^\wedge} \sfC^{\hnu,\hmu}(V' \otimes \cU\bg).
\]
Combined with the considerations in
the proof of Lemma~\ref{lem:tensor-prod-HC-derived}, this implies that the subcategories $\HC^{\hla,\hmu}_{\mathrm{diag}}$ are ``stable under convolution'' in the sense that for $\lambda,\mu,\nu \in X^*(\bT)$ and any objects $M \in \HC^{\hla,\hmu}_{\mathrm{diag}}$ and $N \in \HC^{\hmu,\hnu}_{\mathrm{diag}}$ we have
\[
M \star N \in \HC^{\hla,\hnu}_{\mathrm{diag}}.
\]
%, for any $\lambda,\mu \in X^*(\bT)$, $\HC^{\hla,\hmu}_{\mathrm{diag}}$ is a monoidal subcategory in $D^- \HC^{\hla,\hmu}$.
Similar comments apply to the subcategories
$\HC^{\hla,\hmu}_{\mathrm{diag},\mathrm{tilt}}$. In particular, when $\lambda=\mu$ we obtain a monoidal structure on the category $\HC^{\hla,\hla}_{\mathrm{diag},\mathrm{tilt}}$, hence an induced monoidal structure on its bounded homotopy category, and it is clear that the composition of the functor of
Corollary~\ref{cor:Kb-Db-HC} with the obvious functor $\Db \HC^{\hla,\hmu} \to D^- \HC^{\hla,\hmu}$ is monoidal.

%------------------------------------
\subsection{Nilpotent Harish-Chandra bimodules}
\label{ss:nilpotent-HC}
%------------------------------------

For $\lambda,\mu$ in $X^*(\bT)$ (or in the quotient $X^*(\bT)/(\bW,\bullet)$),
we will denote by 
\[
\HC_\nil^{\widehat{\lambda},\widehat{\mu}}
\]
the full subcategory of $\HC$ whose objects are the bimodules such that the left, resp.~right, action of $\ZHC$ vanishes on a power of $\fm^\lambda$, resp.~$\fm^\mu$. (See~\S\ref{ss:center-Ug} for the notation.)
%the maximal ideal corresponding to $\tla$, resp.~$\tmu$. 
It is clear from classical properties of completions of rings (see~\cite[\href{https://stacks.math.columbia.edu/tag/05GG}{Tag 05GG}]{stacks-project}) that $\HC_\nil^{\widehat{\lambda},\widehat{\mu}}$ identifies naturally with the full subcategory of $\HC^{\hla,\hmu}$ consisting of modules such that the left, resp.~right, action of $\ZHC$ vanishes on a power of $\fm^\lambda$, resp.~$\fm^\mu$. In fact, 
%as a full subcategory of $\HC^{\hla,\hmu}$, 
each of these conditions is sufficient to ensure that an object belongs to $\HC_\nil^{\widehat{\lambda},\widehat{\mu}}$, as explained in the following lemma.

\begin{lem}
\label{lem:nilp-bimodules}
Let $\lambda,\mu \in X^*(\bT)$. For an object $M \in \HC^{\hla,\hmu}$ the following conditions are equivalent:
\begin{enumerate}
\item
$M$ belongs to $\HC_\nil^{\widehat{\lambda},\widehat{\mu}}$;
\item
the left action of $\ZHC$ on $M$ vanishes on a power of $\fm^\lambda$;
\item
the right action of $\ZHC$ on $M$ vanishes on a power of $\fm^\mu$.
\end{enumerate}
\end{lem}

\begin{proof}
We prove the equivalence between the first two properties; the equivalence between the first and third properties can be established similarly. By definition an object which satisfies the first property satisfies the second one, so what we have to prove is that if $M \in \HC^{\hla,\hmu}$ and if the left action of $(\fm^\lambda)^N$ vanishes for some $N$, then the right action of a power of $\fm^\mu$ vanishes. For this it suffices to prove that the image of $\fm^\mu$ in
\[
\scO(\bZ^{\hla,\hmu}) / (\fm^\lambda)^N \cdot \scO(\bZ^{\hla,\hmu})
\]
vanishes. Now we have
\[
\bZ^{\hla,\hmu} \cong \FN_{\bt^*/(\bWf,\bullet)}(\{\tla\}) \times_{\FN_{\bt^{*(1)}/\bWf}(\{0\})} \FN_{\bt^*/(\bWf,\bullet)}(\{\tmu\}),
\]
see~\cite[Equation~(3.11)]{br-Hecke}, hence to conclude it suffices to prove that the image of $\fm^\mu$ in
\[
\scO(\FN_{\bt^*/(\bWf,\bullet)}(\{\tmu\})) / \fn^N \cdot \scO(\FN_{\bt^*/(\bWf,\bullet)}(\{\tmu\}))
\]
is nilpotent. (See~\S\ref{ss:central-characters} for the definition of $\fn$.) By exactness of completion, this quotient identifies with the completion of the quotient
\[
\scO(\bt^*/(\bWf,\bullet)) / \fn^N \cdot \scO(\bt^*/(\bWf,\bullet))
\]
with respect to $\fm^\mu$. By the general theory of artinian rings, and as in the proof of~\cite[Lemma~3.4]{br-Hecke}, the quotient
\[
\scO(\bt^*/(\bWf,\bullet)) / (\fn^N \cdot \scO(\bt^*/(\bWf,\bullet)) + (\fm^\mu)^{N'})
\]
does not depend on $N'$ for $N' \gg 0$, hence identifies with this completion, which finishes the proof.
\end{proof}

%one check that its essential image consists of bimodules on which the left action of $\ZHC$ vanishes on a power of the maximal ideal corresponding to $\tla$, and also of bimodules on which the right action of $\ZHC$ vanishes on a power of the maximal ideal corresponding to $\tmu$.

\begin{lem}
\label{lem:HC-central-char}
Let $\lambda,\mu \in X^*(\bT)$.
The functor
\[
\Db \HC_\nil^{\widehat{\lambda},\widehat{\mu}} \to \Db \HC^{\hla,\hmu}
\]
induced by the embedding $\HC_\nil^{\widehat{\lambda},\widehat{\mu}} \to \HC^{\hla,\hmu}$ is fully faithful. Its essential image is the full subcategory whose objects are the complexes $M$ which satisfy one of the following equivalent conditions:
\begin{enumerate}
\item
\label{it:HC-central-char-cond-1}
for any $n \in \Z$ the object $\mathsf{H}^n(M)$ belongs to $\HC_\nil^{\widehat{\lambda},\widehat{\mu}}$;
\item
\label{it:HC-central-char-cond-2}
the action morphism $\scO(\bZ^{\hla,\hmu}) \to \End(M)$ vanishes on a power of the unique maximal ideal;
\item
\label{it:HC-central-char-cond-3}
there exists $n \geq 0$ such that the morphism $(\fm^\lambda)^n \to \End(M)$ defined by the left action of $\ZHC$ vanishes;
\item
\label{it:HC-central-char-cond-4}
there exists $n \geq 0$ such that the morphism $(\fm^\mu)^n \to \End(M)$ defined by the right action of $\ZHC$ vanishes.
\end{enumerate}
\end{lem}

In fact, it will be more convenient to prove a version of Lemma~\ref{lem:HC-central-char} which treats all choices of $\lambda,\mu$ simultaneously. Namely, recall the category $\HC^\wedge$ and the decomposition in~\eqref{eqn:decomp-Zwedge-cat}. If we denote by $\HC^\wedge_{\mathrm{nil}}$ the full subcategory of $\HC$ whose objects are the bimodules annihilated by a power of $\fn \subset \scO(\bt^{*(1)}/\bWf)$, then $\HC^\wedge_{\mathrm{nil}}$ identifies with a full subcategory in $\HC^{\wedge}$, and we similarly have a canonical decomposition
\begin{equation}
\label{eqn:decomp-HC-nil}
\HC^\wedge_{\mathrm{nil}} \cong \bigoplus_{\lambda,\mu \in X^*(\bT)/(\bW,\bullet)} \HC_\nil^{\widehat{\lambda},\widehat{\mu}}.
\end{equation}
Lemma~\ref{lem:HC-central-char} will be deduced below from the following claim.
%(together with the considerations in the proof of Lemma~\ref{lem:nilp-bimodules}).

\begin{lem}
\label{lem:HC-wedge-nil}
The functor
\[
\Db \HC^\wedge_{\mathrm{nil}} \to \Db \HC^\wedge
\]
induced by the embedding $\HC^\wedge_{\mathrm{nil}} \to \HC^\wedge$ is fully faithful. Its essential image is the full subcategory whose objects are the complexes $M$ which satisfy one of the following equivalent conditions:
\begin{enumerate}
\item
for any $n \in \Z$ the object $\mathsf{H}^n(M)$ belongs to $\HC^\wedge_{\mathrm{nil}}$;
\item
there exists $n \geq 0$ such that the morphism $\fn^n \to \End(M)$ vanishes.
\end{enumerate}
\end{lem}

\begin{proof}
%Note that once the first claim is known, one obtains that the essential image of the functor is a triangulated subcategory, hence that it coincides with the triangulated subcategory generated by $\HC^{\lambda,\mu}$, which clearly consists of the complexes all of whose cohomology objects belong to $\HC^{\lambda,\mu}$. It therefore suffices to prove that our functor is fully faithful.
%
By~\eqref{eqn:HC-modules-wedge} we have an equivalence
$\HC^\wedge \simto \Modfg^{\bG} \bigl( (\cU\bg)^\wedge \bigr)$, under which the full subcategory $\HC^\wedge_{\mathrm{nil}}$ identifies with the full subcategory of 
%$\bG$-equivariant $\cU\bg$-
modules annihilated by a power of $\fn$. The lemma is therefore an application of the considerations
of~\S\ref{ss:app-special-case}
%Proposition~\ref{prop:nilp-modules-ff} 
for the noetherian commutative ring $\scO(\bt^{*(1)}/\bWf)$, its ideal $\fn$, the (affine) scheme $\Spec(\ZFr)$ over $\bt^{*(1)}/\bWf$, the sheaf of algebras corresponding to $\cU\bg$ and the group scheme obtained from $\bG$ by base change.
%
%$\ZFr \otimes_{\scO(\bt^{*(1)}/\bWf)} \scO(\FN_{\bt^{*(1)}/\bWf}(0))$, the ideal generated by the maximal ideal of $\scO(\bt^{*(1)}/\bWf)$ corresponding to $0 \in \bt^{*(1)}/\bWf$, the algebra $(\cU\bg)^\wedge$ and the group scheme obtained by base change from $\bG$. Then the category of finitely generated torsion modules of the proposition identifies with the category of finitely generated $\bG$-equivariant $\cU\bg$-modules annihilated by a power of the ideal of $0 \in \bt^{*(1)}/\bWf$, which corresponds to the subcategory $\HC^{\mathrm{nil}} \subset \HC^\wedge$ under the equivalence~\eqref{eqn:HC-modules-wedge}. Hence Proposition~\ref{prop:nilp-modules-ff} exactly provides the desired full faithfulness claim.
\end{proof}

\begin{proof}[Proof of Lemma~\ref{lem:HC-central-char}]
Fully faithfulness of our functor directly follows from Lem\-ma~\ref{lem:HC-wedge-nil}, since it is a direct summand of the functor considered in the latter lemma. Once this is established, we know that the essential image of this functor coincides with the full triangulated subcategory of $\Db \HC^{\hla,\hmu}$ generated by $\HC_\nil^{\widehat{\lambda},\widehat{\mu}}$, i.e.~that its objects are the complexes which satisfy~\eqref{it:HC-central-char-cond-1}. Among these conditions, it is clear that we have the implications 
$\eqref{it:HC-central-char-cond-2} \Rightarrow \eqref{it:HC-central-char-cond-3}$ and $\eqref{it:HC-central-char-cond-2} \Rightarrow \eqref{it:HC-central-char-cond-4}$,
and it follows from Lemma~\ref{lem:nilp-bimodules} that both~\eqref{it:HC-central-char-cond-3} and~\eqref{it:HC-central-char-cond-4} imply~\eqref{it:HC-central-char-cond-1}. On the other hand, any bounded complex of objects in $\HC_\nil^{\widehat{\lambda},\widehat{\mu}}$ satisfies~\eqref{it:HC-central-char-cond-2}, so that this condition is implied by~\eqref{it:HC-central-char-cond-1}, which finishes the proof.
\end{proof}

\begin{lem}
For any $\lambda,\mu,\nu \in X^*(\bT)$, the bifunctor~\eqref{eqn:convolution-HC-la-mu-nu} restricts to bifunctors
\[
\Db \HC^{\hla,\hmu}_{\mathrm{nil}} \times \Db \HC^{\hmu,\hnu} \to \Db \HC^{\hla,\hnu}_{\mathrm{nil}}, \quad
\Db \HC^{\hla,\hmu} \times \Db \HC^{\hmu,\hnu}_{\mathrm{nil}} \to \Db \HC^{\hla,\hnu}_{\mathrm{nil}}.
\]
\end{lem}

\begin{proof}
We treat the first case; the second one is similar. What we have to prove is that if $M \in \HC^{\hla,\hmu}_{\mathrm{nil}}$ and $N \in \Db \HC^{\hmu,\hnu}$ then $M \star N$ is bounded, and has all of its cohomology objects in $\Db \HC^{\hla,\hnu}_{\mathrm{nil}}$. The second property is clear from Lemma~\ref{lem:nilp-bimodules}, computing the tensor product using a flat resolution of $N$. For the first property, we observe that we have
\[
M \cong M \otimes_{\ZHC} \scO(\FN_{\bt^*/(\bWf,\bullet)}(\{\tmu\}))
\]
(where we consider the action of $\ZHC$ via the \emph{right} action of $\cU\bg$). Hence if $M^\bullet \to M$ is a flat resolution of $M$ as a right $\cU\bg$-module then we obtain a flat resolution $M^\bullet \otimes_{\ZHC} \scO(\FN_{\bt^*/(\bWf,\bullet)}(\{\tmu\})) \to M$ as a right $(\cU\bg)^{\hmu}$-module, which can be used to compute $M \star N$. Since $\cU\bg$ has finite cohomological dimension, the complex $M^\bullet$ can be chosen bounded, so that $M \star N$ is indeed bounded.
%shows that we have
%\[
%M \star N \cong M \lotimes_{\cU\bg} N.
%\]
%Now the right-hand side is bounded because $\cU\bg$ has finite cohomological dimension, which finishes the proof.
\end{proof}

%------------------------------------------------
\subsection{Bimodules with a fixed (right) central character}
\label{ss:HC-fixed-character}
%------------------------------------------------

For $\mu \in X^*(\bT)$ we now set
\[
(\cU\bg)^\mu := \cU\bg / \fm^\mu \cdot \cU\bg = \cU\bg \otimes_{\ZHC} \bk_\mu,
\]
where $\bk_\mu$ is the $1$-dimensional module over $\ZHC \cong \scO(\bt^*/(\bWf,\bullet))$ corresponding to the closed point $\tmu$. 
The algebra
\[
\cU\bg \otimes_{\ZFr} (\cU\bg)^{\mu,\op}
\]
%has a central algebra which identifies with 
receives a morphism with central image from the
structure algebra of the finite affine scheme
\[
\bZ \times_{\bt^*/(\bWf,\bullet)} \{\tmu\} \cong \bt^*/ (\bWf,\bullet) \times_{\bt^{*(1)}/\bWf} \{0\},
\]
where the morphism $\bZ \to \bt^*/(\bWf,\bullet)$ is induced by projection on the second factor. For $\lambda \in X^*(\bT)$ we then set
\[
\bZ^{\hla,\mu} := \FN_{\bZ \times_{\bt^*/(\bWf,\bullet)} \{\tmu\}}( \{ (\tla,\tmu) \} )
\]
(a finite scheme)
and
%\begin{multline*}
\[
\sfU^{\widehat{\lambda},\mu} :=
\bigl( \cU\bg \otimes_{\ZFr} (\cU\bg)^{\mu,\op} \bigr) \otimes_{\scO(\bZ \times_{\bt^*/(\bWf,\bullet)} \{\tmu\})} \scO(\bZ^{\hla,\mu}).
\]
%\end{multline*}
Note that by the structure theory of artinian rings (see e.g.~\cite[\href{https://stacks.math.columbia.edu/tag/00JB}{Tag 00JB}]{stacks-project}) the natural morphism
\[
\bigsqcup_{\lambda \in X^*(\bT)/(\bW,\bullet)} \bZ^{\hla,\mu} \to
\bZ \times_{\bt^*/(\bWf,\bullet)} \{\tmu\}
\]
is an isomorphism, so that we have
\[
\cU\bg \otimes_{\ZFr} (\cU\bg)^{\mu,\op} \cong \prod_{\lambda \in X^*(\bT)/(\bW,\bullet)} \sfU^{\widehat{\lambda},\mu}.
\]

A $\sfU^{\widehat{\lambda},\mu}$-module is exactly a $\cU\bg$-bimodule which satisfies the following properties with respect to the actions of the central subalgebras in each factor:
\begin{itemize}
\item
the left and right actions of $\ZFr$ coincide;
\item
the right action of the ideal $\fm^\mu \subset \ZHC$ vanishes;
\item
the left action of a power of the ideal $\fm^\lambda \subset \ZHC$ vanishes.
\end{itemize}
This algebra admits a canonical action of $\bG$, and we can consider the category $\Mod^\bG(\sfU^{\hla,\mu})$ of $\bG$-equivariant $\sfU^{\hla,\mu}$-modules, its full subcategory $\Modfg^\bG(\sfU^{\hla,\mu})$ of finitely generated modules, and the full subcategories
\[
\widetilde{\HC}^{\hla,\mu} \subset \Mod^\bG(\sfU^{\hla,\mu}), \quad \HC^{\hla,\mu} \subset \Modfg^\bG(\sfU^{\hla,\mu})
\]
of Harish-Chandra bimodules. We have a surjective algebra morphism $\sfU^{\hla,\hmu} \to \sfU^{\hla,\mu}$ which induces a fully faithful exact functor $\Mod^\bG(\sfU^{\hla,\mu}) \to \Mod^\bG(\sfU^{\hla,\hmu})$ which restricts to a fully faithful exact functor
\begin{equation}
\label{eqn:functor-HC-hla-mu-hmu}
\HC^{\hla,\mu} \to \HC^{\hla,\hmu}.
\end{equation}

This functor has a left adjoint functor
\begin{equation}
\label{eqn:specialization-ZHC-char}
\Mod^\bG(\sfU^{\hla,\hmu}) \to \Mod^\bG(\sfU^{\hla,\mu})
\end{equation}
given by $M \mapsto \sfU^{\hla,\mu} \otimes_{\sfU^{\hla,\hmu}} M = M / M \cdot \fm^\mu$, whose composition with the embedding $\widetilde{\HC}^{\hla,\hmu} \subset \Mod^\bG(\sfU^{\hla,\hmu})$ factor through a functor
\[
\widetilde{\HC}^{\hla,\hmu} \to \widetilde{\HC}^{\hla,\mu}.
\]
For $V \in \Rep(\bG)$, we will denote by $\sfC^{\hla,\mu}(V \otimes \cU\bg)$ the image of $\sfC^{\hla,\hmu}(V \otimes \cU\bg)$ under this functor.

\begin{lem}
\phantomsection
\label{lem:flatness}
\begin{enumerate}
\item
\label{it:flatness-1}
The algebra $\sfU^{\hla,\hmu}$ is flat over $\scO(\FN_{\bt^*/(\bWf,\bullet)}(\{\tmu\}))$ (for the action induced by multiplication on the right).
\item
\label{it:flatness-2}
For any $V \in \Rep(\bG)$, the object $\sfC^{\hla,\hmu}(V \otimes \cU\bg)$ is flat as a module over $\scO(\FN_{\bt^*/(\bWf,\bullet)}(\{\tmu\}))$ (for the action induced by multiplication on the right). 
\end{enumerate}
\end{lem}

\begin{proof}
\eqref{it:flatness-1}
Using the identifications~\eqref{eqn:Ug-wedge-product} and~\eqref{eqn:U-hla-hmu-tensor-prod} we see that to prove the lemma it suffices to prove that
\[
(\cU\bg)^\wedge \otimes_{\ZFr \otimes_{\scO(\bt^{*(1)}/\bWf)} \scO(\FN_{\bt^{*(1)}/\bWf}(\{0\}))} (\cU\bg)^{\hmu,\op}
\]
is flat over $\scO(\FN_{\bt^*/(\bWf,\bullet)}(\{\tmu\}))$. Now $\cU\bg$ is free over $\ZFr$ (see Proposition~\ref{prop:flatness-Ug}), hence $(\cU\bg)^\wedge$ is free over $\ZFr \otimes_{\scO(\bt^{*(1)}/\bWf)} \scO(\FN_{\bt^{*(1)}/\bWf}(\{0\}))$. The claim then follows from the fact that $(\cU\bg)^{\hmu,\op}$ is flat over $\scO(\FN_{\bt^*/(\bWf,\bullet)}(\{\tmu\}))$, which itself follows from the fact that $\cU\bg$ is flat over $\ZHC$, see Proposition~\ref{prop:flatness-Ug}.

\eqref{it:flatness-2}
We have
\[
\prod_{\lambda \in X^*(\bT)/(\bW,\bullet)} \sfC^{\hla,\hmu}(V \otimes \cU\bg) \cong V \otimes (\cU\bg)^{\hmu},
\]
so once again the claim follows from the fact that $\cU\bg$ is flat over $\ZHC$.
\end{proof}

Lemma~\ref{lem:flatness}\eqref{it:flatness-1} implies that the left derived functor of~\eqref{eqn:specialization-ZHC-char}, namely the functor
\[
D^- \Mod^\bG(\sfU^{\hla,\hmu}) \to D^- \Mod^\bG(\sfU^{\hla,\mu})
\]
given by $M \mapsto \sfU^{\hla,\mu} \lotimes_{\sfU^{\hla,\hmu}} M$, can also be expressed as $M \mapsto M \lotimes_{\scO(\FN_{\bt^*/(\bWf,\bullet)}(\{\tmu\}))} \bk_\mu$. In particular, since $\bt^*/(\bWf,\bullet)$ is isomorphic to an affine space by~\cite{demazure}, the ring $\scO(\FN_{\bt^*/(\bWf,\bullet)}(\{\tmu\}))$ has finite global dimension, so that this functor restricts to a functor from $\Db \Mod^\bG(\sfU^{\hla,\hmu})$ to $\Db \Mod^\bG(\sfU^{\hla,\mu})$. Since the ring $\sfU^{\hla,\hmu}$ is noetherian, this functor in turn induces a functor from $\Db \Modfg^\bG(\sfU^{\hla,\hmu})$ to $\Db \Modfg^\bG(\sfU^{\hla,\mu})$.

We claim that the composition of the latter functor with the natural functor $\Db \HC^{\hla,\hmu} \to \Db \Modfg^\bG(\sfU^{\hla,\hmu})$ factors canonically through a ``specialization'' functor
\begin{equation*}
%\label{eqn:specialization-ZHC-char-der}
\mathsf{Sp}_{\lambda,\mu} : \Db \HC^{\hla,\hmu} \to \Db \HC^{\hla,\mu}.
\end{equation*}
In fact, in view of the comments above, to prove this it suffices to prove that any object in $\HC^{\hla,\hmu}$ is a quotient of an object which is flat over $\scO(\FN_{\bt^*/(\bWf,\bullet)}(\{\tmu\}))$. This follows from Lemma~\ref{lem:surjection-diag-ind-comp} and Lemma~\ref{lem:flatness}\eqref{it:flatness-2}.

The following statement is an analogue in this setting of Proposition~\ref{prop:Ext-vanishing}.

\begin{prop}
\label{prop:Ext-vanishing-fixed-char}
For any $\lambda,\mu \in X^*(\bT)$ and any $V,V' \in \Tilt(\bG)$, the functor $\mathsf{Sp}_{\lambda,\mu}$
%~\eqref{eqn:specialization-ZHC-char-der} 
induces an isomorphism
\begin{multline*}
\bk_\mu \otimes_{\scO(\FN_{\bt^*/(\bWf,\bullet)}(\{\tmu\}))} \Hom_{\HC^{\hla,\hmu}} \bigl( \sfC^{\hla,\hmu}(V \otimes \cU\bg), \sfC^{\hla,\hmu}(V' \otimes \cU\bg) \bigr) \simto \\
\Hom_{\HC^{\hla,\mu}} \bigl( \sfC^{\hla,\mu}(V \otimes \cU\bg), \sfC^{\hla,\mu}(V' \otimes \cU\bg) \bigr),
\end{multline*}
and moreover we have
 \[
  \Ext^n_{\HC^{\hla,\mu}} \bigl( \sfC^{\hla,\mu}(V \otimes \cU\bg), \sfC^{\hla,\mu}(V' \otimes \cU\bg) \bigr)=0
 \]
 for any $n \in \Z_{>0}$.
\end{prop}

\begin{proof}
%As in~\eqref{eqn:HC-modules}, restricting the action to the right copy of $\cU\bg$ induces a fully faithful functor
%\[
%\Db \HC^{\hla,\hmu} \to \Db \Modfg^{\bG}((\cU\bg)^{\hmu,\op})
%\]
%%of $\bG$-equivariant right modules for the algebra $(\cU\bg)^\hmu$, 
%and a fully faithful functor 
%\[
%\Db \HC^{\hla,\mu} \to \Db \Modfg^{\bG}((\cU\bg)^{\mu,\op}).
%\]
%%to the derived category of $\bG$-equivariant right modules for the algebra $(\cU\bg)^\mu$. 
%Under this identification the functor $\mathsf{Sp}_{\lambda,\mu}$
%%~\eqref{eqn:specialization-ZHC-char-der} 
%is induced by the functor
In view of Lemma~\ref{lem:flatness}\eqref{it:flatness-2},
using Proposition~\ref{prop:Ext-vanishing} and the same technique as in its proof (treating all $\lambda$'s at the same time), one reduces the proof to showing that the functor
\[
M \mapsto M \lotimes_{(\cU\bg)^{\hmu}} (\cU\bg)^\mu,
\]
%which can also be expressed as 
i.e.~the functor $M \mapsto M \lotimes_{\scO(\FN_{\bt^*/(\bWf,\bullet)}(\{\tmu\}))} \bk_\mu$, satisfies the following property.
%In view of Proposition~\ref{prop:Ext-vanishing}, the proposition will therefore follow if we prove that 
For any $M,N \in \Modfg^{\bG}((\cU\bg)^{\hmu,\op})$ which are flat over $\scO(\FN_{\bt^*/(\bWf,\bullet)}(\{\tmu\}))$ and satisfy
\[
\Ext^n_{\Modfg^{\bG}((\cU\bg)^{\hmu,\op})}(M,N)=0
\]
for $n >0$, the functor above induces an isomorphism
\begin{multline}
\label{eqn:Hom-fixed-char}
\bk_\mu \otimes_{\scO(\FN_{\bt^*/(\bWf,\bullet)}(\{\tmu\}))} \Hom_{\Modfg^{\bG}((\cU\bg)^{\hmu,\op})}(M,N) \\
\simto \Hom_{\Modfg^{\bG}((\cU\bg)^{\mu,\op})}(M / M \cdot \fm^\mu,N / N \cdot \fm^\mu),
\end{multline}
and moreover we have
\[
\Ext^n_{\Modfg^{\bG}((\cU\bg)^{\mu,\op})}(M / M \cdot \fm^\mu,N / N \cdot \fm^\mu)=0
\]
for $n>0$. Now by Lemma~\ref{lem:RHom-equiv-ext-scalars} we have an identification
\begin{multline*}
\bk_\mu \lotimes_{\scO(\FN_{\bt^*/(\bWf,\bullet)}(\{\tmu\}))} R\Hom_{\Modfg^{\bG}((\cU\bg)^{\hmu,\op})}(M,N) \\
\simto R\Hom_{\Modfg^{\bG}((\cU\bg)^{\mu,\op})}(M / M \cdot \fm^\mu,N / N \cdot \fm^\mu).
\end{multline*}
Our assumptions guarantee that the left-hand side is concentrated in nonpositive degrees, while the right-hand side is concentrated in nonnegative degrees. Hence they are both concentrated in degree $0$, which proves the desired vanishing result and provides an isomorphism as in~\eqref{eqn:Hom-fixed-char}. We leave it to the reader to check that this isomorphism is induced by the given functor.
%functor $\mathsf{Sp}_{\lambda,\mu}$.
%~\eqref{eqn:specialization-ZHC-char-der}.
\end{proof}

The same considerations as in~\S\ref{ss:monoidal-structure-DHC} lead to the construction, for any $\lambda,\mu,\nu \in X^*(\bT)$, of a convolution bifunctor
\begin{equation*}
D^- \HC^{\hla,\hmu} \times D^- \HC^{\hmu,\nu} \to D^- \HC^{\hla,\nu}
\end{equation*}
which is compatible with the bifunctors~\eqref{eqn:convolution-HC-la-mu-nu} in the natural sense. It particular, when $\lambda=\mu$, this bifunctor makes $D^- \HC^{\hmu,\nu}$ a module category for the monoidal category $D^- \HC^{\hmu,\hmu}$.

%------------------------------------------------
\subsection{Bimodules with trivial left action}
\label{ss:bimod-trivial-action}
%------------------------------------------------

There is a canonical exact fully faithful functor
\begin{equation}
\label{eqn:Rep-HC}
\Rep(\bG) \to \HC
\end{equation}
defined as follows. (This functor will usually be omitted from notation.) Given $V \in \Rep(\bG)$, we have a left action of $\cU\bg$ on $V$ obtained by taking the differential of the $\bG$-action. We associate to $V$ the object whose underlying vector space is $V$, endowed the given action of $\bG$, the \emph{trivial} left action of $\cU\bg$, and the right action of $\cU\bg$ deduced from the left action considered above by composition with the antiautomorphism of $\cU\bg$ sending each $x \in \bg$ to $-x$. It is easily seen that these structures define a Harish-Chandra bimodule, which provides the desired functor~\eqref{eqn:Rep-HC}. The essential image of this functor consists exactly of the Harish-Chandra bimodules on which the left action is trivial, i.e.~factors through the augmentation morphism $\cU\bg \to \bk$ sending each $x \in \bg$ to $0$.

Given $\lambda \in X^*(\bT)$, one can consider the full subcategory $\Rep_{\langle \lambda \rangle}(\bG)$ of $\Rep(\bG)$ defined in~\S\ref{ss:HC-notation}, i.e.~the full subcategory whose objects are the representations on which $\fm^\lambda$ acts nilpotently (for the action of $\cU\bg$ obtained by differentiation).
%, if we denote by $\Rep_{\langle \lambda \rangle}(\bG)$ the extended block of $\lambda$, then 
It is clear that the functor~\eqref{eqn:Rep-HC} restricts to an exact fully faithful functor
\begin{equation}
\label{eqn:Rep-HC-lambda}
\Rep_{\langle \lambda \rangle}(\bG) \to \HC_\nil^{\widehat{0},\widehat{-w_\circ \lambda}}.
\end{equation}
%whose essential image consists of bimodules on which the right action of $\cU\bg$ is trivial.

In particular, we can consider this construction in case $\lambda=0$, and with the object $\bk=\sfL(0)$. The resulting Harish-Chandra bimodule will be called the \emph{trivial Harish-Chandra bimodule}.

\begin{lem}
\label{lem:Rep-cohomology-convolution}
%For $\lambda=0$, 
The essential image of the functor~\eqref{eqn:Rep-HC-lambda} consists of the objects of the form $H^n(\bk \star M)$ with $M \in \Db \HC^{\widehat{0},\widehat{-w_\circ \lambda}}$.
\end{lem}

\begin{proof}
Computing $\bk \star M$ using a flat resolution of $M$ and using Lemma~\ref{lem:HC-central-char} one sees that $H^n(\bk \star M)$ belongs to the essential image of~\eqref{eqn:Rep-HC-lambda} for any $n \in \Z$. On the other hand, if $V \in \Rep_{\langle \lambda \rangle}(\bG)$ we observe that
$H^0(\bk \star V) \cong V$,
 %\bk \star \sfC^{\widehat{0},\widehat{0}}(\cU\bg \otimes V) \cong V,
so that any object in the essential image of our functor appears in this way.
\end{proof}

\subsection{Translation bimodules}
\label{ss:translation-bimod}
%------------------------------------

Recall the group $\bW$ introduced in~\S\ref{ss:central-characters}, and consider its subgroup
\[
\bWaff := \bWf \ltimes \Z \fR,
\]
called the affine Weyl group.
%, and $\bW$ will be called the \emph{extended affine Weyl group}. 
%The $\bullet$-action of $\bWf$ on $X^*(\bT)$ extends to an 
We have an action of $\bW$ on $\mathbb{R} \otimes_{\Z} X^*(\bT)$ extending the dot-action on $X^*(\bT)$ and defined by the same formula as in~\eqref{eqn:dot-action}, now for $\eta \in \mathbb{R} \otimes_{\Z} X^*(\bT)$.
%\[
%(t_\lambda w) \bullet v = w(v + \varsigma) -\varsigma + \ell \lambda
%\]
%for $\lambda \in X^*(\bT)$, $w \in \bWf$ and $v \in \mathbb{R} \otimes_{\Z} X^*(\bT)$
%where in the right-hand side we consider the natural action of $\bWf$ on $\mathbb{R} \otimes_{\Z} X^*(\bT)$. 
The closure $\overline{A_0}$ of the fundamental alcove
\[
A_0 = \{v \in \mathbb{R} \otimes_{\Z} X^*(\bT) \mid \forall \alpha \in \fR_+, \, 0 < \langle v+\varsigma, \alpha^\vee \rangle < \ell\}
\]
constitutes a fundamental domain for the restriction of this action to $\bWaff$, see~\cite[\S 6.2]{jantzen}. 

The affine space $\mathbb{R} \otimes_{\Z} X^*(\bT)$ is divided into facets, see also~\cite[\S 6.2]{jantzen}. The codimension-$1$ facets are called walls, and to these walls one can naturally attach reflections in $\bWaff$, see~\cite[\S 6.3]{jantzen}.
If we denote by $\bSaff \subset \bWaff$ the subset consisting of reflections attached to walls contained in $\overline{A_0}$,
%elements $s$ such that the closed subset
%\[
%(\overline{A_0})^s := \{v \in \overline{A_0} \mid w \bullet v = v\}
%\]
%has codimension $1$ in $\overline{A_0}$, 
then $\bSaff$ is a set of Coxeter generators for $\bWaff$ which contains $\bSf$. If we set
\[
\mathbf{\Omega} := \{w \in \bW \mid w \bullet A_0 = A_0\},
\]
then $\mathbf{\Omega}$ is an abelian subgroup of $\bW$ such that $\omega s \omega^{-1} \in \bSaff$ for any $\omega \in \mathbf{\Omega}$ and $s \in \bSaff$, and the multiplication morphism
\[
\mathbf{\Omega} \ltimes \bWaff \to \bW
\]
is an isomorphism. We use this isomorphism to extend the length function $\ell$ of $\bWaff$ by setting $\ell(\omega w)=\ell(w)$ for $w \in \bWaff$ and $\omega \in \mathbf{\Omega}$. (We then also have $\ell(w \omega)=\ell(w)$ for any $w \in \bWaff$ and $\omega \in \mathbf{\Omega}$.)

Following the terminology used in~\cite{br-Hecke}, we will call \emph{lower closure of the fundamental alcove} the set
\[
%A_0^{\downarrow} = 
\{v \in \mathbb{R} \otimes_{\Z} X^*(\bT) \mid \forall \alpha \in \fR_+, \, 0 \leq \langle \nu+\varsigma, \alpha^\vee \rangle < \ell\}.
\]
%say that a weight $\nu \in X^*(\bT)$ belongs to the \emph{lower closure of the fundamental alcove} if it satisfies
%\[
%0 \leq \langle \nu+\varsigma, \alpha^\vee \rangle < \ell
%\]
%for any $\alpha \in \fR_+$. 
As explained in~\cite[\S 3.5]{br-Hecke} (see also~\cite[Lemma~3.1]{br-Hecke}), if $\nu \in X^*(\bT)$ belongs to the lower closure of the fundamental alcove then the stabilizer of $\onu$ for the $\bullet$-action of $\bWf$ is the parabolic subgroup generated by the subset of $\bSf$ consisting of the elements $s$ such that $s \bullet \nu = \nu$, which corresponds to the subset $\{ \alpha \in \fRs \mid \langle \nu+\varsigma, \alpha^\vee \rangle=0 \} \subset \fRs$ under the canonical bijection between $\fRs$ and $\bSf$. 
%In particular, the regular weights (i.e.~those with trivial stabilizer) in the lower closure of the fundamental alcove are those in $X^*(\bT) \cap A_0$.

For any $\lambda,\mu \in X^*(\bT)$ we have a ``translation bimodule'' $\sfP^{\hla,\hmu}$ defined as follows. (For a justification of the terminology, see~\cite[Lemma~6.1]{br-Hecke}.) Let $\nu \in X^*(\bT)$ be the unique dominant weight in $\bWf(\lambda-\mu)$, and set
\[
\sfP^{\hla,\hmu} := \sfC^{\hla,\hmu}(\sfL(\nu) \otimes \cU\bg) \quad \in \HC^{\hla,\hmu}_{\mathrm{diag}}.
\]
The following lemma shows that this object belongs to $\HC^{\hla,\hmu}_{\mathrm{diag},\mathrm{tilt}}$ for appropriate $\lambda$ and $\mu$.
%coincides with the object denoted $\sfP^{\lambda,\mu}$ in~\cite[\S 3.5]{br-Hecke}.

\begin{lem}
\label{lem:translation-bimod-simples}
If $\lambda,\mu$ belong to $\overline{A_0}$, with one of them at least in $A_0$, we have an isomorphism
\[
\sfP^{\hla,\hmu} \cong \sfC^{\hla,\hmu}(\sfT(\nu) \otimes \cU\bg).
\]
\end{lem}

\begin{proof}
%By Proposition~\ref{prop:restriction-Kostant-HC-ff}, it suffices to prove the isomorphism in the lemma after application of the functor~\eqref{eqn:restriction-S-HC}. 
In~\cite[\S 3.8--3.9]{br-Hecke} we have constructed, for any $\mu,\nu \in X^*(\bT)$ a functor of ``restriction to a Kostant slice'' on the category $\HC^{\hmu,\hnu}$, which by~\cite[Proposition~3.7]{br-Hecke} is fully faithful on the subcategory $\HC^{\hmu,\hnu}_{\mathrm{diag}}$. To prove the lemma it suffices to prove that our two bimodules become isomorphic after application of this functor.
Now in the proof of~\cite[Lemma~5.5]{br-Hecke} one can replace the simple module $\sfL(\nu)$ by $\sfT(\nu)$ without altering any argument. This lemma therefore holds in the two cases, which implies the desired isomorphism.
\end{proof}

The same proof as in~\cite[Lemma~3.6]{br-Hecke} shows that for any $\lambda,\mu,\eta \in X^*(\bT)$ the functor
\[
\sfP^{\hla,\hmu} \star (-) : \Db \HC^{\hmu,\widehat{\eta}} \to \Db \HC^{\hla,\widehat{\eta}}
\]
is both left and right adjoint to the functor
\[
\sfP^{\hmu,\hla} \star (-) : \Db \HC^{\hla,\widehat{\eta}} \to \Db \HC^{\hmu,\widehat{\eta}},
\]
and similarly for convolution on the right.
(These adjunctions are not canonical; they depend on a choice of isomorphism $\sfL(\nu)^* \cong \sfL(-w_\circ(\nu))$ where $\nu$ is as above.)
%The bimodule $\sfP^{\hla,\hmu}$ realises the translation functor for $\bG$-modules from the block of $\mu$ to the block of $\lambda$, in a sense made precise in~\cite[Lemma~6.1]{br-Hecke}. From this point of view, Lemma~\ref{lem:translation-bimod-simples} is essentially a reformulation of~\cite[Remark~7.6(1)]{jantzen}.

\begin{rmk}
\label{rmk:P-action-W}
It is clear from definitions that for $\lambda, \mu \in X^*(\bT)$ and $y \in \bW$ we have $\sfP^{\hla,\hmu} = \sfP^{\widehat{y \bullet \lambda}, \widehat{y \bullet \mu}}$.
\end{rmk}

%------------------------------------
\subsection{Bott--Samelson and Soergel type Harish-Chandra bimodules}
\label{ss:BSHC}
%------------------------------------

In this subsection we
assume that $\ell \geq h$ where $h$ is the Coxeter number of $\bG$. In this case we have $A_0 \cap X^*(\bT) \neq \varnothing$, and 
each wall contained in $\overline{A_0}$ has nonempty intersection with $X^*(\bT)$, see~\cite[\S 6.3]{jantzen}.
%$(\overline{A_0})^s \cap X^*(\bT) \neq \varnothing$ for any $s \in \bSaff$.
We can therefore fix a weight $\lambda \in A_0 \cap X^*(\bT)$ and, for each $s \in \bSaff$, a weight $\mu_s \in X^*(\bT)$ on the wall attached to $s$ contained in $\overline{A_0}$.
%there exist weights which belong to the fundamental alcove; we fix such a weight $\lambda$.
%Recall also that to any $s \in \bSaff$ is attached a ``wall'' of the fundamental alcove, which contains elements of $X^*(\bT)$. For any $s \in \bSaff$ we also fix a weight $\mu_s$ belonging to the $s$-wall of the fundamental alcove, and 
For $s \in \bSaff$
we set
%we will denote by $\nu_s$ the unique dominant $\bWf$-translate of $\lambda-\mu_s$, and set
\[
\sfR_s := \sfP^{\hla,\widehat{\mu_s}} \star \sfP^{\widehat{\mu_s},\hla} \quad \in \HC^{\hla,\hla}_{\mathrm{diag},\mathrm{tilt}}.
\]
For $\omega \in \mathbf{\Omega}$ we also 
%denote by $\nu_\omega$ the unique dominant $\bWf$-translate of $\lambda - \omega \bullet \lambda$, and 
set
\[
\sfR_\omega := \sfP^{\hla,\widehat{\omega \bullet \lambda}} \quad \in \HC^{\hla,\hla}_{\mathrm{diag},\mathrm{tilt}}.
\]
(Note that here the images of $\lambda$ and $\omega \bullet \lambda$ in $\bt^*/(\bWf,\bullet)$ coincide.)

\begin{lem}
\phantomsection
\label{lem:sfP-properties}
\begin{enumerate}
%\item
%\label{it:sfP-properties-1}
%For any $s \in \bSaff$, resp.~$\omega \in \Omega$, there exists an isomorphism
%\[
%\sfR_s \cong \sfC^{\hla,\widehat{\mu_s}}(\sfL(\nu_s) \otimes \cU\bg) \star \sfC^{\widehat{\mu_s},\hla}(\sfL(\nu_s)^* \otimes \cU\bg), \quad \text{resp.} \quad \sfR_\omega := \sfC^{\hla,\hla}(\sfL(\nu_\omega) \otimes \cU\bg).
%\]
\item
\label{it:sfP-properties-2}
For any $s \in \bSaff$ and $\omega \in \mathbf{\Omega}$ there exists an isomorphism
\[
\sfR_\omega \star \sfR_s \star \sfR_{\omega^{-1}} \cong \sfR_{\omega s \omega^{-1}}.
\]
\item
\label{it:sfP-properties-3}
For any $\omega,\omega' \in \mathbf{\Omega}$ there exists an isomorphism
\[
\sfR_\omega \star \sfR_{\omega'} \cong \sfR_{\omega \omega'}.
\]
\end{enumerate}
\end{lem}

\begin{proof}
%By Proposition~\ref{prop:restriction-Kostant-HC-ff} again, it suffices to prove the isomorphisms in the lemma after application of the functor~\eqref{eqn:restriction-S-HC}. 
We explain the proof of~\eqref{it:sfP-properties-2}; that of~\eqref{it:sfP-properties-3} is similar. This proof will rely on the formalism of~\cite{br-Hecke}.

%\eqref{it:sfP-properties-1}
%In the proof of~\cite[Lemma~5.5]{br-Hecke} one can replace the simple module $\sfL(\nu)$ by $\sfT(\nu)$ without altering any argument. This lemma therefore holds in the two cases, showing that in this setting the bimodules defined using simple modules or indecomposable tilting modules are isomorphic, which implies the desired isomorphisms.

Recall, for any $\mu,\nu \in X^*(\bT)$, the fully faithful functor on $\HC^{\hmu,\hnu}$ considered in the proof of Lemma~\ref{lem:translation-bimod-simples}.
%In~\cite[\S 3.8--3.9]{br-Hecke} we have constructed, for any $\mu,\nu \in X^*(\bT)$ a functor of ``restriction to a Kostant slice'' on the category $\HC^{\hmu,\hnu}$, which by~\cite[Proposition~3.7]{br-Hecke} is fully faithful on the subcategory $\HC^{\hmu,\hnu}_{\mathrm{diag}}$. 
In~\cite[\S 5.2]{br-Hecke} we have defined an object $\mathsf{Q}_{\mu,\nu}$ in the target category of this functor.
% \in \HC_\bS^{\hmu,\hnu}$ considered in~\cite[\S 5.2]{br-Hecke}. 
 Then 
%Lemma~\ref{lem:translation-bimod-simples}
%~\eqref{it:sfP-properties-1} 
by~\cite[Lemma~5.5]{br-Hecke} the image of $\sfR_s$, resp.~$\sfR_\omega$, resp.~$\sfR_{\omega^{-1}}$, under this functor is
%under~\eqref{eqn:restriction-S-HC} 
%we have isomorphisms
\[
%\mathsf{Q}_{\omega \bullet \lambda, \omega \bullet \mu_s} \hatotimes_{\cU_\bS\bg} \mathsf{Q}_{\omega \bullet \mu_s, \omega \bullet \lambda}, 
\mathsf{Q}_{\lambda, \mu_s} \hatotimes_{\cU_\bS\bg} \mathsf{Q}_{\mu_s, \lambda},
\quad \text{resp.} \quad \mathsf{Q}_{\omega^{-1} \bullet \lambda, \lambda},
\quad \text{resp.} \quad \mathsf{Q}_{\lambda, \omega^{-1} \bullet \lambda}.
\]
(See~\cite[\S 3.9]{br-Hecke} for the definition of the bifunctor $\hatotimes_{\cU_\bS\bg}$.)
%restriction to $\HC^{\hla,\hla}_\bS$.
Since the functor is monoidal, the image of $\sfR_\omega \star \sfR_s \star \sfR_{\omega^{-1}}$ is
\[
\mathsf{Q}_{\omega^{-1} \bullet \lambda, \lambda} \hatotimes_{\cU_\bS\bg} \mathsf{Q}_{\lambda, \mu_s} \hatotimes_{\cU_\bS\bg} \mathsf{Q}_{\mu_s, \lambda} \hatotimes_{\cU_\bS\bg} \mathsf{Q}_{\lambda, \omega^{-1} \bullet \lambda}.
\]
Now, by~\cite[Lemma~5.3]{br-Hecke} we have isomorphisms
\[
\mathsf{Q}_{\omega^{-1} \bullet \lambda, \lambda} \hatotimes_{\cU_\bS\bg} \mathsf{Q}_{\lambda, \mu_s} \cong \mathsf{Q}_{\omega^{-1} \bullet \lambda, \mu_s}, \quad \mathsf{Q}_{\mu_s, \lambda} \hatotimes_{\cU_\bS\bg} \mathsf{Q}_{\lambda, \omega^{-1} \bullet \lambda} \cong \mathsf{Q}_{\mu_s, \omega^{-1} \bullet \lambda},
\]
and 
%by Lemma~\ref{lem:translation-bimod-simples}
%~\eqref{it:sfP-properties-1} 
by~\cite[Lemma~5.5]{br-Hecke} we have isomorphisms
\[
\mathsf{Q}_{\omega^{-1} \bullet \lambda, \mu_s} \cong \mathsf{Q}_{\lambda, \omega \bullet \mu_s}, \quad
\mathsf{Q}_{\mu_s, \omega^{-1} \bullet \lambda} \cong \mathsf{Q}_{\omega \bullet \mu_s, \lambda}.
\]
Finally, using the fact that $\omega \bullet \mu_s$ and $\mu_{\omega s \omega^{-1}}$ both belong to the wall of $\overline{A_0}$ attached to the simple reflection $\omega s \omega^{-1}$ and~\cite[Lemma~5.15]{br-Hecke} one checks that
\[
\mathsf{Q}_{\lambda, \omega \bullet \mu_s} \hatotimes_{\cU_\bS\bg} \mathsf{Q}_{\omega \bullet \mu_s, \lambda} \cong 
\mathsf{Q}_{\lambda, \mu_{\omega s \omega^{-1}}} \hatotimes_{\cU_\bS\bg} \mathsf{Q}_{\mu_{\omega s \omega^{-1}}, \lambda}.
\]
The desired claim follows, using 
%again Lemma~\ref{lem:translation-bimod-simples}
%~\eqref{it:sfP-properties-1} 
again~\cite[Lemma~5.5]{br-Hecke}.
%
%\eqref{it:sfP-properties-3}
%The proof is similar to that of~\eqref{it:sfP-properties-2}; details are left to the reader.
\end{proof}

%For any $s \in \bSaff$ we fix adjunctions
%\[
%(\sfP^{\hla,\widehat{\mu_s}} \star (-), \sfP^{\widehat{\mu_s},\hla} \star(-)) \quad \text{and} \quad 
%(\sfP^{\widehat{\mu_s},\hla} \star (-), \sfP^{\hla,\widehat{\mu_s}} \star (-))
%\]
%as in the comments after Lemma~\ref{lem:translation-bimod-simples}. 
As explained e.g.~in~\cite[\S 6.6]{br-Hecke}, the results of that paper imply that for any $s \in \bSaff$ the morphism spaces
\[
\Hom_{\HC^{\hla,\hla}}((\cU\bg)^{\hla}, \sfR_s) \quad \text{and} \quad \Hom_{\HC^{\hla,\hla}}(\sfR_s, (\cU\bg)^{\hla})
\]
are free of rank $1$ as left (or right) modules over $\scO(\FN_{\bt^*/(\bWf, \bullet)}(\{\tla\}))$. (The left and right actions on these spaces coincide, since they do on the module $(\cU\bg)^{\hla}$.) We choose generators
%This provides morphisms
$(\cU\bg)^{\hla} \to \sfR_s$ and $\sfR_s \to (\cU\bg)^{\hla}$,
%in $\HC^{\hla,\hla}$. We 
and define the objects $\sfN_s$ and $\sfD_s$ of these spaces as their cone and cocone respectively, so that we have
%in $\Db \HC^{\hla,\hla}$ so that they fit in 
distinguished triangles
\begin{equation}
\label{eqn:triangles-Ds-Ds'}
(\cU\bg)^{\hla} \to \sfR_s \to \sfN_s \xrightarrow{[1]} \quad \text{and} \quad \sfD_s \to \sfR_s \to (\cU\bg)^{\hla} \xrightarrow{[1]}.
\end{equation}
(These objects are therefore well defined up to isomorphism.)

We will denote by
\[
\BSHC^{\hla,\hla}
\]
the strictly full subcategory of $\HC^{\hla,\hla}_{\mathrm{diag},\mathrm{tilt}}$ generated under the monoidal product $\star$ by the unit object and the objects $\sfR_s$ ($s \in \bSaff$) and $\sfR_\omega$ ($\omega \in \mathbf{\Omega}$). In view of Lemma~\ref{lem:sfP-properties}, any object in $\BSHC^{\hla,\hla}$ is isomorphic to an object of the form
\[
\sfR_{s_1} \star \cdots \star \sfR_{s_r} \star \sfR_\omega
\]
where $s_1, \cdots, s_r \in \bSaff$ and $\omega \in \mathbf{\Omega}$. We will also denote by
\[
\SHC^{\hla,\hla}
\]
the karoubian envelope of the additive hull of $\BSHC^{\hla,\hla}$, i.e.~the strictly full subcategory of $\HC^{\hla,\hla}$ whose objects are the direct summands of the direct sums of objects in $\BSHC^{\hla,\hla}$.

\begin{rmk}
In this notation the letters ``$\mathsf{BS}$'' and ``$\mathsf{S}$'' refer to Bott--Samelson and Soergel respectively, since these categories will later be related to some category of Soergel bimodules, in such a way that $\BSHC^{\hla,\hla}$ corresponds to Bott--Samelson bimodules.
\end{rmk}

%----------------------------------------------------
\subsection{Further properties of the braid bimodules}
\label{ss:further-properties}
%----------------------------------------------------

We continue with the setting and notation of~\S\ref{ss:BSHC}.
In this subsection we prove for later use some properties of the objects $(\sfR_s : s \in \bSaff)$, $(\sfD_s : s \in \bSaff)$ and $(\sfN_s : s \in \bSaff)$.
%in $\Db \HC^{\hla,\hla}$. 
The proofs of these properties rely on the constructions of~\cite{br-Hecke}.

\begin{lem}
\label{lem:SHC-w0}
Let $w_\circ = s_1 \cdots s_r$ be a reduced decomposition. Then
the object $\sfP^{\hla,\widehat{-\varsigma}} \star \sfP^{\widehat{-\varsigma},\hla}$ is a direct summand in $\sfR_{s_1} \star \sfR_{s_2} \star \cdots \star \sfR_{s_r}$;
in particular, 
%$\sfP^{\hla,\widehat{-\varsigma}} \star \sfP^{\widehat{-\varsigma},\hla}$ 
it belongs to $\SHC^{\hla,\hla}$.
\end{lem}

\begin{proof}
Consider the ``Hecke category'' $\mathsf{H}'_\aff$ attached to $\bWaff$ as in~\cite[\S 2.1]{br-Hecke}, and
the monoidal functor 
\begin{equation}
\label{eqn:functor-br-Hecke}
\mathsf{H}_\aff' \to \BSHC^{\hla,\hla}
\end{equation}
of~\cite[Theorem~6.3]{br-Hecke}. (In~\cite{br-Hecke}, the category $\mathsf{H}'_\aff$ is denoted $\mathsf{D}_{\mathrm{BS}}$.) 
%Since $\HC^{\hla,\hla}$ is additive and karoubian, 
This functor extends uniquely to a monoidal functor 
\begin{equation}
\label{eqn:functor-br-Hecke-2}
\mathsf{H}_\aff \to \SHC^{\hla,\hla}
\end{equation}
where $\mathsf{H}_\aff$ is the karoubian closure of the additive envelope of $\mathsf{H}_\aff'$.
% to $\SHC^{\hla,\hla}$.

The indecomposable objects in $\mathsf{H}_\aff$
%the Hecke category (i.e.~the karoubian closure of the additive envelope of the category denoted $\mathsf{D}_{\mathrm{BS}}$ in~\cite{br-Hecke}) 
are classified (up to isomorphism and shift) by the elements in $\bWaff$. Moreover, the indecomposable object $B_{w_\circ}$ associated with 
%the longest element 
$w_\circ$ 
%in $\bWf$ 
is described (in terms of Abe's incarnation of the Hecke category, see~\cite[\S 2.2]{br-Hecke}) in~\cite[Proposition~2.10]{abe}
%\footnote{EXPLAIN WHY THE ASSUMPTIONS ARE SATISFIED!}. 
(whose assumptions are satisfied in our setting by~\cite[Remark~2.8]{abe}).
Using this description, the same arguments as for~\cite[Proposition~6.6]{br-Hecke} show that the functor~\eqref{eqn:functor-br-Hecke-2}
%of~\cite[Theorem~6.3]{br-Hecke} 
sends $B_{w_\circ}$ to $\sfP^{\hla,\widehat{-\varsigma}} \star \sfP^{\widehat{-\varsigma},\hla}$. 
%By definition, for any reduced decomposition $w_\circ = s_1 \cdots s_r$ (with all $s_i$'s in $\bSf$) 
With the notation in the statement, the object $B_{w_\circ}$ is a direct summand in $B_{s_1} \star \cdots \star B_{s_r}$, where $B_s \in \mathsf{H}_\aff$ is the indecomposable object associated with $s$. Applying the (monoidal) functor~\eqref{eqn:functor-br-Hecke-2},
%of~\cite[Theorem~6.3]{br-Hecke}, 
which sends $B_s$ to $\sfR_s$ for any $s \in \bSaff$,
%by~\cite[Proposition~6.6]{br-Hecke}, 
we deduce the desired claim.
\end{proof}

For $\uw = (s_1, \cdots, s_r)$ a word in $\bSaff$ we set
\[
\sfD_{\uw} := \sfD_{s_1} \star \cdots \star \sfD_{s_r}, \quad \sfN_{\uw} := \sfN_{s_1} \star \cdots \star \sfN_{s_r}.
\]

\begin{lem}
\phantomsection
\label{lem-properties-D-N-HC}
\begin{enumerate}
\item
\label{it:D-N-HC-1}
For any $\omega \in \mathbf{\Omega}$ and $s \in \bSaff$ we have
\[
\sfR_\omega \star \sfD_s \star \sfR_{\omega^{-1}} \cong \sfD_{\omega s \omega^{-1}}, \quad
\sfR_\omega \star \sfN_s \star \sfR_{\omega^{-1}} \cong \sfN_{\omega s \omega^{-1}}.
\]
\item
\label{it:D-N-HC-2}
For any $s \in \bSaff$ there exist isomorphisms
\[
\sfN_s \star \sfD_s \cong (\cU\bg)^{\hla} \cong \sfD_s \star \sfN_s.
\]
\item
\label{it:D-N-HC-3}
If $\uw$ is a reduced expression for an element in $\bWaff$, then the objects $\sfD_{\uw}$ and $\sfN_{\uw}$ only depend (up to isomorphism) on the image of $\uw$ in $\bWaff$.
\end{enumerate}
\end{lem}

\begin{proof}
Statement~\eqref{it:D-N-HC-1} follows from Lemma~\ref{lem:sfP-properties}.
%\eqref{it:sfP-properties-2}.

To prove~\eqref{it:D-N-HC-2} and~\eqref{it:D-N-HC-3},
we will again use the monoidal functor $\mathsf{H}_\aff \to \HC^{\hla,\hla}$ considered in the course of the proof of Lemma~\ref{lem:SHC-w0},
%~\eqref{eqn:functor-br-Hecke},
% of~\cite[Theorem~6.3]{br-Hecke}, 
 and the induced functor between bounded homotopy categories. In~\cite{arv} it is explained (in a general setting that admits the category $\mathsf{H}_\aff$ as a special case) how to associate with any element in $\bWaff$ a ``standard object'' and a ``costandard object'' in $\Kb \mathsf{H}_\aff$.
% the bounded homotopy category of the Hecke category. 
 Inspecting on definitions we see that the composition of the functor considered above with the canonical functor
\[
\Kb \SHC^{\hla,\hla} \to D^- \HC^{\hla,\hla}
\]
sends, for any $s \in \bSaff$, the standard, resp.~costandard, object associated with $s$ to $\sfD_s$, resp.~$\sfN_s$. Since the latter functor is monoidal, the desired properties then follow from the similar properties of standard and costandard objects in $\Kb \mathsf{H}_\aff$, see~\cite[Proposition~6.11]{arv}.
\end{proof}

In view of Lemma~\ref{lem-properties-D-N-HC}
%\eqref{it:D-N-HC-3} 
we can define complexes $(\sfD_w : w \in \bW)$ and $(\sfN_w : w \in \bW)$ as follows. For $w \in \bW$, if $\omega \in \mathbf{\Omega}$ is the unique element such that $w \omega^{-1} \in \bWaff$ and if 
%$\ell(w)=r$ we write 
$w \omega^{-1} = s_1 \cdots s_r$ is a reduced expression in $\bWaff$, we
%for some $s_1, \cdots, s_r \in \bSaff$ and $\omega \in \Omega$, and 
set
\[
\sfD_w := \sfD_{s_1} \star \cdots \star \sfD_{s_r} \star \sfR_\omega, \quad
\sfN_w := \sfN_{s_1} \star \cdots \star \sfN_{s_r} \star \sfR_\omega.
\]
%The properties above 
Lemma~\ref{lem-properties-D-N-HC}
guarantees that these objects do not depend on the choice of $s_1, \cdots, s_r$ and $\omega$, up to isomorphism, and that they satisfy the following properties:
\begin{enumerate}
\item
\label{it:properties-D-N-1}
for any $w \in \bW$ we have
\[
\sfD_w \star \sfN_{w^{-1}} \cong (\cU\bg)^{\hla} \cong \sfN_{w^{-1}} \star \sfD_w;
\]
\item
\label{it:properties-D-N-2}
if $y,w \in \bW$ satisfy $\ell(yw)=\ell(y)+\ell(w)$ then
\[
\sfD_y \star \sfD_w \cong \sfD_{yw}, \quad \sfN_y \star \sfN_w \cong \sfN_{yw}.
\]
\end{enumerate}

Recall the braid group $\Br_{\bW}$ associated with the group $\bW$, see~\cite[\S 1.1]{br-Baff}. (The subgroup $\mathbf{\Omega}$ considered in the present paper coincides with the subgroup denoted similarly in~\cite{br-Baff}.) This group admits a presentation with generators $(T_w : w \in \bW)$ and relations
\[
T_y T_w = T_{yw} \quad \text{if $\ell(yw)=\ell(y)+\ell(w)$.}
\]
The subgroup $\Br_{\aff}$ of $\Br_{\bW}$ generated by the elements $(T_w : w \in \bWaff)$ (or, equivalently, by the elements $(T_s : s \in \bSaff)$) is the usual Artin braid group associated with the Coxeter system $(\bWaff,\bSaff)$. The action of $\mathbf{\Omega}$ on $\bSaff$ induces an action on $\Br_\aff$, such that the assignment $(\omega, b) \mapsto T_\omega \cdot b$ induces a group isomorphism
\begin{equation}
\label{eqn:Br-semi-direct-prod}
\mathbf{\Omega} \ltimes \Br_{\aff} \simto \Br_{\bW}.
\end{equation}
The group $\Br_\bW$ also contains some ``Bernstein elements'' $(\theta_\lambda : \lambda \in X^*(\bT))$, and it is also generated by the elements
\begin{equation}
\label{eqn:generators-BrW}
\{T_s : s \in \bSf \} \cup \{\theta_\lambda : \lambda \in X^*(\bT)\}.
\end{equation}
(See~\cite[\S 1.1]{br-Baff} for a description of a presentation of $\Br_{\bW}$ in terms of these generators.)

The properties~\eqref{it:properties-D-N-1}--\eqref{it:properties-D-N-2} above show that the assignment $T_w \mapsto \sfN_w$ can be uniquely extended to a group morphism from $\Br_\bW$ to the group of isomorphism classes of invertible objects in the monoidal category $D^- \HC^{\hla,\hla}$. The image of $b$ (or, more specifically, a representative of this image) will be denoted $\sfN_b$. (It is clear from construction that this complex is bounded.) Note that for any $w \in \bW$ we have $\sfN_{(T_w)^{-1}} = \sfD_{w^{-1}}$.

\begin{lem}
\label{lem:conjugation-sfR}
Let $s,t \in \bSaff$ and $b \in \Br_\bW$, and assume that $bT_s b^{-1} = T_t$. Then there exists an isomorphism
\[
\sfN_b \star \sfR_s \star \sfN_{b^{-1}} \cong \sfR_t.
\]
\end{lem}

\begin{proof}
Using the isomorphism~\eqref{eqn:Br-semi-direct-prod} and Lemma~\ref{lem:sfP-properties}\eqref{it:sfP-properties-2} one can assume that $b \in \Br_\aff$. Then
the proof is similar to that of Lemma~\ref{lem-properties-D-N-HC}\eqref{it:D-N-HC-2}--\eqref{it:D-N-HC-3}. Namely, it suffices to prove a similar isomorphism in $\Kb \mathsf{H}_\aff$,
%the homotopy category of the Hecke category, 
which follows from the characterization of the simple object in the heart of the perverse t-structure attached to $t$ given in~\cite[\S 8.1]{arv}.
\end{proof}

%%%%%%%%%%%%%%%%%%%%%%%%%%%%%%%%%%%%%%%%%%%%%%%
\section{Completed localization for \texorpdfstring{$\cU\bg$}{Ug}-modules}
\label{sec:compl-loc}
%%%%%%%%%%%%%%%%%%%%%%%%%%%%%%%%%%%%%%%%%%%%%%%

%----------------------------------------------------------
\subsection{(Parabolic) Grothendieck resolutions}
\label{ss:parabolic-Groth}
%----------------------------------------------------------

Recall that the parabolic subgrou\-ps of $\bG$ containing $\bB$ are in a natural bijection with the subsets of $\fRs$.
Given a subset $I \subset \fRs$, we will denote by $\bP_I \subset \bG$ the associated parabolic subgroup, by $\bU_I$ its unipotent radical, by $\bL_I \subset \bP_I$ the Levi factor containing $\bT$, and by $\bp_I$, $\bu_I$ and $\bl_I$ their respective Lie algebras. The ``Grothendieck resolution'' associated with $I$ is the induced variety
\[
\tbg_I := \bG \times^{\bP_I} (\bg/\bu_I)^*.
\]
This variety is a vector bundle over $\bG/\bP_I$; in particular it is smooth. It is equipped with a $\bG$-action induced by left multiplication on $\bG$.
(When $I=\varnothing$ this variety is the usual Grothendieck resolution associated with $\bG$; in this case we will often omit the subscript $\varnothing$. At the other extreme, when $I=\fRs$ we have $\tbg_{\fRs} = \bg^*$.) If $I \subset J \subset \fRs$ we have a natural $\bG$-equivariant projective morphism $\tbg_I \to \tbg_J$; in particular, for any $I \subset \fRs$, taking $J=\fRs$ we obtain
%The coadjoint $\bG$-action on $\bg^*$ induces 
a $\bG$-equivariant projective morphism $\tbg_I \to \bg^*$.

Consider the action of $\bL_I$ on $\bG/\bU_I$ induced by right multiplication on $\bG$, and the induced action on $T^* (\bG/\bU_I)$. Then we have a canonical identification
\[
T^* (\bG/\bU_I) \cong \bG \times^{\bU_I} (\bg/\bu_I)^*,
\]
under which the action of $\bL_I$ is given by $h \cdot [g : \xi] = [gh^{-1} : h \cdot \xi]$ for $h \in \bL_I$, $g \in \bG$ and $\xi \in (\bg/\bu_I)^*$, hence there exists a canonical morphism
\[
T^* (\bG/\bU_I) \to \tbg_I
\]
which is an $\bL_I$-torsor.

To $I \subset \fRs$ we also associate the subgroup $\bW_I$ of $\bWf$ generated by the simple reflections $s_\alpha$ with $\alpha \in I$, which identifies with the Weyl group of $\bL_I$. We have a Chevalley isomorphism
\[
\bl_I^* / \bL_I \simto \bt^* / \bW_I, 
\]
from which we obtain $\bP_I$-equivariant morphisms
\[
(\bg/\bu_I)^* \to \bl_I^* \to \bt^* / \bW_I
\]
(where the first map is induced by the restriction morphism $(\bg/\bu_I)^* \to (\bp_I/\bu_I)^* = \bl_I^*$, and where $\bP_I$ acts via the quotient $\bP_I \to \bL_I$ on $\bl_I^*$ and trivially on the right-hand side). We deduce a morphism
\begin{equation}
\label{eqn:morph-tbg-t}
\tbg_I \to \bt^* / \bW_I
\end{equation}
which is $\bG$-equivariant for the trivial action on $\bt^* / \bW_I$, and such that the following diagram commutes, where the bottom horizontal map is the coadjoint quotient morphism and the right vertical map is induced by the quotient morphism $\bt^* \to \bt^*/\bW$:
\[
\xymatrix{
\tbg_I \ar[r] \ar[d] & \bt^* / \bW_I \ar[d] \\
\bg^* \ar[r] & \bt^*/\bW.
}
\]
If $I \subset J \subset \fRs$ these morphisms (for $I$ and $J$) and the canonical morphism $\tbg_I \to \tbg_J$ satisfy various obvious compatibility properties, whose precise formulation is left to the reader.

\begin{lem}
\label{lem:flatness-Groth}
For any subset $I \subset \fRs$, the morphism $\tbg_I \to \bt^* / \bW_I$ is flat.
\end{lem}

\begin{proof}
Since flatness is local on the source, and by $\bG$-equivariance, it suffices to check that the restriction of the morphism to the preimage of the open subscheme $\bU_I^+ \subset \bG/\bP_I$ is flat, where $\bU_I^+$ is the unipotent radical of the parabolic subgroup opposite to $\bP_I$ (which identifies with the open subscheme $\bU_I^+ \cdot \bP_I / \bP_I \subset \bG/\bP_I$). Over this open subscheme $\tbg_I$ identifies with
\[
\bU_I^+ \times (\bg/\bu_I)^* \cong \bU_I^+ \times \bl_I^* \times (\bu_I^+)^*
\]
(where $\bu_I^+$ is the Lie algebra of $\bU_I^+$), and our morphism identifies with the composition of projection to $\bl_I^*$ with the coadjoint quotient morphism for $\bl_I$. As seen in the proof of Proposition~\ref{prop:flatness-Ug} the latter morphism is flat by results of~\cite{bc}, hence so is our morphism.
\end{proof}

%In case $I=\varnothing$, we will denote by $\tbS \subset \tbg$ the preimage of $\bS^*$ in $\tbg=\tbg_\varnothing$. Then the morphism~\eqref{eqn:morph-tbg-t} restricts to an isomorphism
%\[
%\tbS \simto \bt^*.
%\]

\begin{rmk}
Below, the cases we will mainly consider are when $\#I \leq 1$. In case $I=\{s\}$ for some $s \in \bSf$, we will sometimes write $\tbg_s$ for $\tbg_{\{s\}}$.
\end{rmk}

%------------------------------------------------------
\subsection{Sheaves of algebras and modules}
\label{ss:sheaves-alg-mod}
%------------------------------------------------------

Let $X$ be a $\bk$-scheme, and $\scA$ be a sheaf of (not necessarily commutative) algebras on $X$ endowed with a morphism of sheaves of algebras $\scO_X \to \scA$ (which does not necessarily take values in the center of $\scA$) which makes $\scA$ a quasi-coherent $\scO_X$-module. In this setting we will denote by
\[
\Modqc(\scA)
\]
the abelian category of sheaves of left $\scA$-modules which are quasi-coherent as sheaves of $\scO_X$-modules. We will say that $\scA$ is left, resp.~right, noetherian if 
%there exists a finite affine open covering $X = \bigcup_{i \in I} U_i$ such that each 
for any affine open subscheme $U \subset X$ the ring
$\Gamma(U, \scA_{|U})$ is left, resp.~right, noetherian; if $\scA$ is left noetherian we will denote by
\[
\Modc(\scA)
\]
the abelian category of sheaves of left $\scA$-modules which are coherent, i.e.~quasi-coherent and locally finitely generated.

We will consider these constructions in particular in the following setting.
If $X$ is a smooth $\bk$-variety, we will denote by $\cD_X$ the sheaf of algebras of crystalline differential operators on $X$ (see~\cite[\S 1.2]{bmr}), i.e.~the enveloping algebra of the tangent bundle. There exists a canonical morphism $\scO_X \to \cD_X$ which makes $\cD_X$ a quasi-coherent $\scO_X$-module, and we can consider the abelian category $\Modqc(\cD_X)$ and its abelian subcategory $\Modc(\cD_X)$. (Here $\cD_X$ is left and right noetherian.)
%, and the covering of $X$ as above can be chosen to be any finite affine open covering.

Since the Frobenius morphism $\Fr_X : X \to X^{(1)}$ is affine, the functor $\Fr_{X*}$ induces equivalences of categories
\[
\Modqc(\cD_X) \simto \Modqc(\Fr_{X*} \cD_X), \quad \Modc(\cD_X) \simto \Modc(\Fr_{X*} \cD_X).
\]
Moreover, $\Fr_{X*} \cD_X$ is a sheaf of $\scO_{X^{(1)}}$-algebras (i.e.~the natural morphism $\scO_{X^{(1)}} \to \Fr_{X*} \cD_X$ takes values in the center of $\Fr_{X*} \cD_X$); in fact there exists a canonical identification between the center of $\Fr_{X*} \cD_X$ and the pushforward of the structure sheaf of $T^* X^{(1)}$ under the canonical morphism $\varpi_X : T^* X^{(1)} \to X^{(1)}$, see~\cite[Lemma~1.3.2]{bmr}, and $\Fr_{X*} \cD_X$ is locally finitely generated over its center. Since $\varpi_X$ is an affine morphism, this implies that there exists a coherent sheaf of $\scO_{T^* X^{(1)}}$-algebras $\scD_X$ such that $\varpi_{X*} \scD_X = \Fr_{X*} \cD_X$. Then we have equivalences of categories
\begin{equation}
\label{eqn:equivalences-Dmod}
\Modqc(\cD_X) \simto \Modqc(\scD_X), \quad \Modc(\cD_X) \simto \Modc(\scD_X),
\end{equation}
and a quasi-coherent sheaf of $\scD_X$-modules is coherent iff it is coherent as a sheaf of $\scO_{T^*X^{(1)}}$-modules.

%------------------------------------------------------
\subsection{Monodromic variant}
\label{ss:monodromic-variant}
%------------------------------------------------------

Let us recall a variant of these constructions taken from~\cite[\S 1.2.1]{bmr2}, specialized to the setting we will require below.
Let $I \subset \fRs$ be a subset. We consider the natural morphism
\[
p_I : \bG/\bU_I \to \bG/\bP_I,
\]
a Zariski locally trivial torsor for the action of $\bL_I$ on $\bG/\bU_I$ considered in~\S\ref{ss:parabolic-Groth},
%induced by right multiplication on $\bG$,
and the sheaf of algebras
\[
\tD_I := (p_I)_* (\cD_{\bG/\bU_I})^{\bL_I},
\]
where the superscript means $\bL_I$-invariants (for the action induced by the action of $\bL_I$ on $\bG/\bU_I$).

The action of $\bL_I$ on $\bG/\bU_I$ defines an algebra morphism $\cU\bl_I \to \Gamma(\bG/\bP_I, \tD_I)$. Restricting this morphism to the Harish-Chandra center $(\cU \bl_I)^{\bL_I} \cong \scO(\bt^*/(\bW_I,\bullet))$ we obtain a canonical morphism
\begin{equation}
\label{eqn:center-tDI}
\scO(\bt^*/(\bW_I,\bullet)) \to \Gamma(\bG/\bP_I, \tD_I);
\end{equation}
moreover this morphism takes values in the center of $\tD_I$. (Note that here the two possible meanings of the dot-action of $\bW_I$, obtained by seeing $\bW_I$ either as the Weyl group of $\bL_I$ or as a subgroup of $\bWf$, coincide.)

On the other hand, the sheaf of algebras $(\Fr_{\bG/\bU_I})_* \cD_{\bG/\bU_I}$ has a center isomorphic canonically to $(\varpi_{\bG/\bU_I})_* \scO_{T^*(\bG/\bU_I)^{(1)}}$, see~\S\ref{ss:sheaves-alg-mod}. We deduce a canonical morphism from the pushforward to $(\bG/\bP_I)^{(1)}$ of the structure sheaf of $\tbg_I^{(1)} = T^*(\bG/\bU_I)^{(1)} / \bL_I^{(1)}$ to the center of $(\Fr_{\bG/\bP_I})_* \tD_I$. Combining this construction with~\eqref{eqn:center-tDI}, we obtain a canonical morphism from the pushforward to $(\bG/\bP_I)^{(1)}$ of the structure sheaf of
\[
\tbg_I^{(1)} \times_{\bt^{*(1)} / \bW_I} \bt^*/(\bW_I,\bullet)
\]
to the center of $(\Fr_{\bG/\bP_I})_* \tD_I$. (Here the morphism $\bt^*/(\bW_I,\bullet) \to \bt^{*(1)} / \bW_I$ is induced by the Artin--Schreier morphism~\eqref{eqn:Artin-Schreier}.) Since all the morphisms under consideration are affine, there exists a coherent sheaf of $\scO_{\tbg_I^{(1)} \times_{\bt^{*(1)} / \bW_I} \bt^*/(\bW_I,\bullet)}$-algebras
$\tsD_I$
whose pushforward to $(\bG/\bP_I)^{(1)}$ is $(\Fr_{\bG/\bP_I})_* \tD_I$. Moreover, there exist canonical equivalences of categories
\[
\Modqc(\tD_I) \simto \Modqc(\tsD_I), \quad \Modc(\tD_I) \simto \Modc(\tsD_I),
\]
and a quasi-coherent $\tsD_I$-module is coherent iff it is coherent as a quasi-coherent sheaf on $\tbg_I^{(1)} \times_{\bt^{*(1)} / \bW_I} \bt^*/(\bW_I,\bullet)$.

%-------------------------------------------------------------
\subsection{Completed localization, I}
\label{ss:comp-loc-statement}
%-------------------------------------------------------------

We continue with our subset $I \subset \fRs$. We will consider below the quotient morphisms
\[
\kappa : \bt^* \to \bt^*/(\bWf,\bullet), \quad \kappa_I : \bt^* \to \bt^*/(\bW_I,\bullet),
\]
and denote by
\[
\kappa^I : \bt^*/(\bW_I,\bullet) \to \bt^*/(\bWf,\bullet)
\]
the morphism such that $\kappa=\kappa^I \circ \kappa_I$.

To lighten notation, in this subsection we set
\begin{gather*}
X_I := \tbg_I^{(1)} \times_{\bt^{*(1)} / \bW_I} \bt^*/(\bW_I,\bullet), \qquad Y := \bg^{*(1)} \times_{\bt^{*(1)} / \bWf} \bt^*/(\bWf,\bullet), \\
Y_I := Y \times_{\bt^*/(\bWf,\bullet)} \bt^*/(\bW_I,\bullet) = \bg^{*(1)} \times_{\bt^{*(1)} / \bWf} \bt^*/(\bW_I,\bullet).
\end{gather*}

For $\xi \in \bt^*$ we set
\[
(\cU \bg)^{\widehat{\xi}} := \cU \bg \otimes_{\scO(\bt^*/(\bWf,\bullet))} \scO(\FN_{\bt^*/(\bWf,\bullet)}(\{\kappa(\xi)\}))
\]
and
\[
\tD_I^{\widehat{\xi}} := \tD_I \otimes_{\scO(\bt^*/(\bW_I,\bullet))} \scO(\FN_{\bt^*/(\bW_I,\bullet)}(\{\kappa_I(\xi)\})).
\]
We will also denote by $\tsD_I^{\widehat{\xi}}$ the pullback of $\tsD_I$ under the natural morphism
from
\[
X_I^{\widehat{\xi}} := X_I \times_{\bt^*/(\bW_I,\bullet)} \FN_{\bt^*/(\bW_I,\bullet)}(\{\kappa_I(\xi)\}) = \tbg_I^{(1)} \times_{\bt^{*(1)} / \bW_I} \FN_{\bt^*/(\bW_I,\bullet)}(\{\kappa_I(\xi)\})
\]
to $X_I$.
%\[
%\tbg_I^{(1)} \times_{\bt^{*(1)} / \bW_I} \FN_{\bt^*/(\bW_I,\bullet)}(\{\pi_I(\xi)\}) \to \tbg_I^{(1)} \times_{\bt^{*(1)} / \bW_I} \bt^*/(\bW_I,\bullet).
%\]
Then, as in~\eqref{eqn:equivalences-Dmod}, there exist canonical equivalences of categories
\begin{equation}
\label{eqn:equivalences-Dmod-xi}
\Modqc(\tD_I^{\widehat{\xi}}) \cong \Modqc(\tsD_I^{\widehat{\xi}}), \quad
\Modc(\tD_I^{\widehat{\xi}}) \cong \Modc(\tsD_I^{\widehat{\xi}}).
\end{equation}

%By~\cite[Lemma~1.2.1]{bmr2} (see also~\cite[Lemma~2.3.1]{bmr}), there exists an equivalence of Azumaya algebras

\begin{rmk}
\label{rmk:translation-DI}
Recall that any $\lambda \in X^*(\bT)$ which satisfies $\langle \lambda, \alpha^\vee \rangle=0$ for any $\alpha \in I$ defines in a natural way a line bundle $\scO_{\bG/\bP_I}(\lambda)$ on $\bG/\bP_I$. On the other hand, for such $\lambda$, the automorphism of $\bt^*$ given by $\xi \mapsto \xi + \ola$ commutes with the dot-action of $\bW_I$, hence induces an automorphism $\tau_{I,\lambda}$ of $\bt^*/(\bW_I,\bullet)$, which itself induces an isomorphism
\begin{multline*}
X_I^{\widehat{\xi}} = \tbg_I^{(1)} \times_{\bt^{*(1)} / \bW_I} \FN_{\bt^*/(\bW_I,\bullet)}(\{\kappa_I(\xi)\}) \simto \\
\tbg_I^{(1)} \times_{\bt^{*(1)} / \bW_I} \FN_{\bt^*/(\bW_I,\bullet)}(\{\kappa_I(\xi+\ola)\}) = X_I^{\widehat{\xi+\ola}}.
\end{multline*}
By~\cite[Lemma~1.2.1]{bmr2} (see also~\cite[Lemma~2.3.1]{bmr}), there exists a canonical equivalence of Azumaya algebras between the pushforward of $\tsD_I^{\widehat{\xi}}$ under this isomorphism and $\tsD_I^{\widehat{\xi+\ola}}$, which induces equivalences of categories
\[
\Modqc(\tsD_I^{\widehat{\xi}}) \simto \Modqc(\tsD_I^{\widehat{\xi+\ola}}), \quad \Modc(\tsD_I^{\widehat{\xi}}) \simto \Modqc(\tsD_I^{\widehat{\xi+\ola}}).
\]
Under the identifications~\eqref{eqn:equivalences-Dmod-xi}, this equivalence is given by the assignment
\[
\scF \mapsto \scO_{\bG/\bP_I}(\lambda) \otimes_{\scO_{\bG/\bP_I}} \scF.
\]

Below we will also consider an analogue of this construction for \emph{right} modules. In fact, as in~\cite[Lemma~3.0.6]{bmr2} we have a canonical isomorphism of sheaves of rings $\tD_I \simto \tD_I^\op$ which, for any $\xi$, induces equivalences of categories
\[
\Modqc(\tsD_I^{\widehat{\xi},\op}) \simto \Modqc(\tsD_I^{\widehat{-\xi-2\overline{\rho}}}), \quad
\Modc(\tsD_I^{\widehat{\xi},\op}) \simto \Modc(\tsD_I^{\widehat{-\xi-2\overline{\rho}}}).
\]
(Here, as usual, $\rho$ is the halfsum of the positive roots.)
Hence we have equivalences of categories
\[
\Modqc(\tsD_I^{\widehat{\xi}, \op}) \simto \Modqc(\tsD_I^{\widehat{\xi+\ola}, \op}), \quad \Modc(\tsD_I^{\widehat{\xi}, \op}) \simto \Modqc(\tsD_I^{\widehat{\xi+\ola}, \op})
\]
given by the assignment
\[
\scF \mapsto \scO_{\bG/\bP_I}(-\lambda) \otimes_{\scO_{\bG/\bP_I}} \scF.
\]
\end{rmk}

%The main result of this section is the following theorem, which is a ``completed" variant of~\cite[Theorem~1.5.1(a)]{bmr2} (see also~\cite[Theorem~3.2]{bmr} for the case $I=\varnothing$). A slightly different variant (involving sheaves on a formal scheme, and completion with respect to a Frobenius character), appears in~\cite[Theorem~5.4.1]{bmr}.
%
%\begin{thm}
%\label{thm:comp-loc}
%Assume that the stabilizer of $\xi$ for the $\bullet$-action of $\bWf$ on $\bt^*$ is contained in $\bW_I$. Then there exist canonical equivalences of categories
%\[
%\Db \Modc(\tD_I^{\widehat{\xi}}) \simto \Db \Modfg((\cU \bg)^{\widehat{\xi}}).
%\]
%\end{thm}
%
%%------------------------------------------------------
%\subsection{Proof of Theorem~\ref{thm:comp-loc}}
%%------------------------------------------------------
%
%We consider $\xi$ and $I$ as in Theorem~\ref{thm:comp-loc}.

%Then we have the sheaf of rings $\tsD_I$ on $X_I$, and 
Setting
\[
\widetilde{\cU}_I \bg := \cU \bg \otimes_{\scO(\bt^*/(\bWf,\bullet))} \scO(\bt^*/(\bW_I,\bullet)),
\]
by~\cite[Proposition~1.2.3(b)]{bmr2} we have a canonical algebra isomorphism
\begin{equation}
\label{eqn:global-sections-tsDI}
\widetilde{\cU}_I \bg \simto \Gamma(X_I, \tsD_I).
\end{equation}
The $\scO(Y_I)$-algebra $\widetilde{\cU}_I \bg$ defines a coherent sheaf of algebras on the affine scheme $Y_I$, which we will denote similarly; then in view of~\eqref{eqn:global-sections-tsDI} we have a natural morphism of ringed spaces
\[
(X_I, \tsD_I) \to (Y_I, \widetilde{\cU}_I \bg).
\]
We can also consider the sheaf of rings on the affine scheme $Y$ defined by $\cU\bg$, which we also denote by $\cU\bg$, and the natural morphism of ringed spaces
\[
(Y_I, \widetilde{\cU}_I \bg) \to (Y,\cU\bg).
\]

If $\xi \in \bt^*$ we set
\begin{multline*}
(\widetilde{\cU}_I \bg)^{\widehat{\xi}} := \widetilde{\cU}_I \bg \otimes_{\scO(\bt^*/(\bW_I,\bullet))} \scO(\FN_{\bt^*/(\bW_I,\bullet)}(\{\kappa_I(\xi)\})) \\
= \cU \bg \otimes_{\scO(\bt^*/(\bWf,\bullet))} \scO(\FN_{\bt^*/(\bW_I,\bullet)}(\{\kappa_I(\xi)\})),
\end{multline*}
so that $(\widetilde{\cU}_I \bg)^{\widehat{\xi}}$ defines a sheaf of rings on the affine scheme
\[
Y_I^{\widehat{\xi}} := Y_I \times_{\bt^*/(\bW_I,\bullet)} \FN_{\bt^*/(\bW_I,\bullet)}(\{\kappa_I(\xi)\}) = Y \times_{\bt^*/(\bWf,\bullet)} \FN_{\bt^*/(\bW_I,\bullet)}(\{\kappa_I(\xi)\}).
\]
If $\mathrm{Stab}_{(\bWf,\bullet)}(\xi) \subset \bW_I$, then
%Consider the morphism $\pi^I : \bt^*/(\bW_I,\bullet) \to \bt^*/(\bWf,\bullet)$. 
by the standard criterion~\cite[Exp.~V, Proposition~2.2]{sga1} the morphism $\kappa^I$ is \'etale at $\kappa_I(\xi)$, 
%hence this morphism 
and induces an isomorphism of schemes
\begin{equation}
\label{eqn:isom-FN}
\FN_{\bt^*/(\bW_I,\bullet)}(\{\kappa_I(\xi)\}) \simto \FN_{\bt^*/(\bWf,\bullet)}(\{\kappa(\xi)\}).
\end{equation}
(Here we use the standard identification between $\scO(\FN_{\bt^*/(\bW_I,\bullet)}(\{\kappa_I(\xi)\}))$ and the completion of the local ring of $\bt^*/(\bW_I,\bullet)$ at $\kappa_I(\xi)$, and similarly for the right-hand side.)
%(More specifically, since the residue fields of $\bt^*/(\bW_I,\bullet)$ at $\pi_I(\xi)$ and $\bt^*/(\bWf,\bullet)$ at $\pi(\xi)$ are both $\bk$, our morphism induces an isomorphism between the completions of the local rings at these points by the comments following~\cite[\href{https://stacks.math.columbia.edu/tag/039M}{Tag 039M}]{stacks-project}, and these completions identify with the rings of functions on the affine schemes in~\eqref{eqn:isom-FN}.)
We deduce an isomorphism of schemes from $Y_I^{\widehat{\xi}}$
%\[
%Y_I^\xi := Y_I \times_{\bt^*/(\bW_I,\bullet)} \FN_{\bt^*/(\bW_I,\bullet)}(\pi_I(\xi)) = Y \times_{\bt^*/(\bWf,\bullet)} \FN_{\bt^*/(\bW_I,\bullet)}(\pi_I(\xi))
%\]
to
\[
Y^{\widehat{\xi}} := Y \times_{\bt^*/(\bWf,\bullet)} \FN_{\bt^*/(\bWf,\bullet)}(\{\kappa(\xi)\}) 
%= \bg^{*(1)} \times_{\bt^{*(1)}/\bWf} \FN_{\bt^*/(\bWf,\bullet)}(\pi(\xi))
\]
and an isomorphism of rings between $(\widetilde{\cU}_I \bg)^{\widehat{\xi}}$
%\begin{multline*}
%(\widetilde{\cU}_I \bg)^{\widehat{\xi}} := \widetilde{\cU}_I \bg \otimes_{\scO(\bt^*/(\bW_I,\bullet))} \scO(\FN_{\bt^*/(\bW_I,\bullet)}(\pi_I(\xi))) \\
%= \cU \bg \otimes_{\scO(\bt^*/(\bWf,\bullet))} \scO(\FN_{\bt^*/(\bW_I,\bullet)}(\pi_I(\xi))).
%\end{multline*}
and $(\cU \bg)^{\widehat{\xi}}$. In other words, under this assumption we have an isomorphism of ringed spaces
\begin{equation}
\label{eqn:identification-Ug-xi}
(Y_I^{\widehat{\xi}}, (\widetilde{\cU}_I \bg)^{\widehat{\xi}}) \simto (Y^{\widehat{\xi}}, (\cU \bg)^{\widehat{\xi}})
\end{equation}
where once again we denote by the same symbol a quasi-coherent sheaf of $\scO_Z$-algebras over an affine scheme $Z$ and its global sections.

%For $\xi \in \bt^*$
%we set
%\[
%X_I^\xi := X_I \times_{\bt^*/(\bW_I,\bullet)} \FN_{\bt^*/(\bW_I,\bullet)}(\pi_I(\xi)) = \tbg_I^{(1)} \times_{\bt^{*(1)} / \bW_I} \FN_{\bt^*/(\bW_I,\bullet)}(\pi_I(\xi)),
%\]
%so that $\tsD_I^{\widehat{\xi}}$ is a sheaf of algebras on $X_I^\xi$.
%Then
By~\eqref{eqn:global-sections-tsDI} and the flat base change theorem (see~\cite[\href{https://stacks.math.columbia.edu/tag/02KH}{Tag 02KH}]{stacks-project}), we have a canonical algebra isomorphism
\begin{equation}
\label{eqn:global-sections-DI-xi}
(\widetilde{\cU}_I \bg)^{\widehat{\xi}} \simto \Gamma(X_I^{\widehat{\xi}}, \tsD_I^{\widehat{\xi}});
\end{equation}
%Denoting again by $(\cU \bg)^{\widehat{\xi}}$ the sheaf of algebras on the affine scheme $Y^\xi$ associated with the $\scO(Y^\xi)$-algebra $(\cU \bg)^{\widehat{\xi}}$, 
we therefore have a canonical morphism of ringed spaces
\[
f^{\prime,\xi}_I : (X_I^{\widehat{\xi}}, \tsD_I^{\widehat{\xi}}) \to (Y_I^{\widehat{\xi}}, (\widetilde{\cU}_I \bg)^{\widehat{\xi}}).
\]
In case $\mathrm{Stab}_{(\bWf,\bullet)}(\xi) \subset \bW_I$, the composition of this morphism
with the isomorphism~\eqref{eqn:identification-Ug-xi} will be denoted $f_I^\xi$. We can then consider the
(derived) push/pull functors
\[
R(f_I^\xi)_* : D^+ \Modqc (\tsD_I^{\widehat{\xi}}) \to D^+ \Mod((\cU \bg)^{\widehat{\xi}})
\]
and
\[
L(f_I^\xi)^* : D^- \Mod((\cU \bg)^{\widehat{\xi}}) \to D^- \Modqc (\tsD_I^{\widehat{\xi}}).
\]

Note that by standard arguments (see e.g.~\cite[Proof of Corollary~2.11]{arinkin-bezrukavnikov}) the natural functors
\[
\Db \Modc (\tsD_I^{\widehat{\xi}}) \to \Db \Modqc (\tsD_I^{\widehat{\xi}}), \quad \Db \Modfg((\cU \bg)^{\widehat{\xi}}) \to \Db \Mod((\cU \bg)^{\widehat{\xi}})
\]
are fully faithful, and their essential images consist of complexes with coherent (resp.~finitely generated) cohomology objects.

The following statement is a ``completed'' version of the main results of~\cite{bmr,bmr2}.
% \textbf{CURRENTLY THE PROOF OF~\eqref{it:comp-loc-3} WORKS ONLY WHEN $I=\varnothing$!}

\begin{thm}
\label{thm:comp-loc}
Assume that $\mathrm{Stab}_{(\bWf,\bullet)}(\xi) \subset \bW_I$.
%the stabilizer of $\xi$ for the $\bullet$-action of $\bWf$ on $\bt^*$ is contained in $\bW_I$. 
\begin{enumerate}
\item
\label{it:comp-loc-1}
The functor $R(f_I^\xi)_*$ restricts to a functor
\begin{equation}
\label{eqn:completed-loc-functor-1}
\Db \Modc (\tsD_I^{\widehat{\xi}}) \to \Db \Modfg((\cU \bg)^{\widehat{\xi}}),
\end{equation}
still denoted $R(f_I^\xi)_*$.
\item
\label{it:comp-loc-2}
The functor $L(f_I^\xi)^*$ restricts to a functor
\begin{equation}
\label{eqn:completed-loc-functor-2}
\Db \Modfg((\cU \bg)^{\widehat{\xi}}) \to \Db \Modc (\tsD_I^{\widehat{\xi}}),
\end{equation}
still denoted $L(f_I^\xi)^*$.
\item
\label{it:comp-loc-3}
The functors~\eqref{eqn:completed-loc-functor-1} and~\eqref{eqn:completed-loc-functor-2} are quasi-inverse equivalences of categories.
\end{enumerate}
%The functor $L(f_I^\xi)^*$ of~\eqref{eqn:completed-loc-functor-2} is left adjoint to the functor $R(f_I^\xi)_*$ of~\eqref{eqn:completed-loc-functor-1}, and moreover the adjunction morphism $\id \to R(f_I^\xi)_* \circ L(f_I^\xi)^*$ is an isomorphism; in particular, $L(f_I^\xi)^*$ is fully faithful.
%\end{enumerate}
%
%Then there exist canonical equivalences of categories
%\[
%\Db \Modc(\tD_I^{\widehat{\xi}}) \simto \Db \Modfg((\cU \bg)^{\widehat{\xi}}).
%\]
\end{thm}

To prove Theorem~\ref{thm:comp-loc}  we will repeat the construction of the functors
%is essentially identical to that of its counterparts 
in~\cite{bmr,bmr2}, and then reduce the proof that they are equivalences to the results of these references. In this subsection we prove statements~\eqref{it:comp-loc-1} and~\eqref{it:comp-loc-2}, and the following weaker variant of~\eqref{it:comp-loc-3}:
\begin{enumerate}
\setcounter{enumi}{3}
\item
\label{it:comp-loc-4}
The functor $L(f_I^\xi)^*$ of~\eqref{eqn:completed-loc-functor-2} is left adjoint to the functor $R(f_I^\xi)_*$ of~\eqref{eqn:completed-loc-functor-1}, and moreover the adjunction morphism $\id \to R(f_I^\xi)_* \circ L(f_I^\xi)^*$ is an isomorphism; in particular, $L(f_I^\xi)^*$ is fully faithful.
\end{enumerate}
Statement~\eqref{it:comp-loc-3} will be deduced in~\S\ref{ss:comp-loc-end-proof} below.

\begin{proof}[Proof of~\eqref{it:comp-loc-1}, \eqref{it:comp-loc-2} and~\eqref{it:comp-loc-4}]
We assume that $\mathrm{Stab}_{(\bWf,\bullet)}(\xi) \subset \bW_I$.

\eqref{it:comp-loc-1}
We have forgetful functors
\[
\Modqc (\tsD_I^{\widehat{\xi}}) \to \QCoh(X_I^{\widehat{\xi}}), \quad \Mod((\cU \bg)^{\widehat{\xi}}) \to \QCoh(Y^{\widehat{\xi}})
\]
which send coherent (resp.~finitely generated) modules to coherent sheaves,
and denoting by $g_I^\xi : X_I^{\widehat{\xi}} \to Y^{\widehat{\xi}}$ the morphism of schemes underlying $f_I^\xi$, the diagram
\begin{equation}
\label{eqn:diagram-pushforward-DI}
\vcenter{
\xymatrix@C=1.5cm{
D^+ \Modqc (\tsD_I^{\widehat{\xi}}) \ar[r]^-{R(f_I^\xi)_*} \ar[d] & D^+ \Mod((\cU \bg)^{\widehat{\xi}}) \ar[d] \\
D^+ \QCoh(X_I^{\widehat{\xi}}) \ar[r]^-{R(g_I^\xi)_*} & D^+ \QCoh(Y^{\widehat{\xi}})
}
}
\end{equation}
commutes by the same arguments as in~\cite[\S 3.1.9]{bmr}. Note that $g_I^\xi$ is proper, since it can be written as the composition
\begin{multline*}
X_I^{\widehat{\xi}} \to \tbg_I^{(1)} \times_{\bt^{*(1)} / \bWf} \FN_{\bt^*/(\bW_I,\bullet)}(\{\kappa_I(\xi)\}) \cong \tbg_I^{(1)} \times_{\bt^{*(1)} / \bWf} \FN_{\bt^*/(\bWf,\bullet)}(\{\kappa(\xi)\}) \\
\to \bg^{*(1)} \times_{\bt^{*(1)} / \bWf} \FN_{\bt^*/(\bWf,\bullet)}(\{\kappa(\xi)\}) = Y^{\widehat{\xi}}
\end{multline*}
where the first morphism is a closed immersion and the last one is obtained from the proper morphism $\tbg_I^{(1)} \to \bg^{*(1)}$ by base change.

Now, note that $X_I^{\widehat{\xi}}$ 
%and $Y^\xi$ are 
is a scheme of finite type over $\FN_{\bt^*/(\bW_I,\bullet)}(\{\kappa_I(\xi)\})$. Here $\scO(\FN_{\bt^*/(\bW_I,\bullet)}(\{\kappa_I(\xi)\}))$ is the completion of a noetherian local ring of finite Krull dimension (namely, the local ring of $\bt^*/(\bW_I,\bullet)$ at $\kappa_I(\xi)$), hence it is itself noetherian (see~\cite[\href{https://stacks.math.columbia.edu/tag/0316}{Tag 0316}]{stacks-project}) and of finite Krull dimension (see~\cite[\href{https://stacks.math.columbia.edu/tag/07NV}{Tag 07NV}]{stacks-project}). We deduce that $X_I^{\widehat{\xi}}$ 
%and $Y^\xi$ are 
is noetherian of finite Krull dimension, so that the functor $(g_I^\xi)_*$ has finite homological dimension,
% on $\QCoh(X_I^\xi)$, 
see~\cite[p.~88]{hartshorne-RD}. Since $g_I^\xi$ is proper, we deduce using~\cite[\href{https://stacks.math.columbia.edu/tag/02O5}{Tag 02O5}]{stacks-project} that $R(g_I^\xi)_*$ restricts to a functor $\Db \Coh(X_I^{\widehat{\xi}}) \to \Db \Coh(Y^{\widehat{\xi}})$.
%Since the lower horizontal arrow restricts to a functor $\Db \Coh(X_I^\xi) \to \Db \Coh(Y^\xi)$ because $g_I^\xi$ is a proper morphism between noetherian schemes of finite Krull dimension, 
Using the diagram~\eqref{eqn:diagram-pushforward-DI} we deduce that $R(f_I^\xi)_*$ restricts to a functor
\[
\Db \Modc (\tsD_I^{\widehat{\xi}}) \to \Db \Modfg((\cU \bg)^{\widehat{\xi}}),
\]
as desired.

%Now we analyze the functor $L(f_I^\xi)^*$.
%First, since any finitely generated $(\cU \bg)^{\widehat{\xi}}$-module is a quotient of a \emph{finitely generated} flat module, this functor restricts to a functor
%\[
%L(f_I^\xi)^* : D^- \Modfg((\cU \bg)^{\widehat{\xi}}) \to D^- \Modc(\tsD_I^{\widehat{\xi}}).
%\]
\eqref{it:comp-loc-2}
To prove this claim we will use the identifications~\eqref{eqn:equivalences-Dmod-xi}. 
%Let us denote by $\pi^I : \bt^*/(\bW_I,\bullet) \to \bt^*/(\bWf,\bullet)$ the morphism such that $\pi=\pi^I \circ \pi_I$. Then t
The closed points in $(\kappa^I)^{-1}(\xi)$ consist of the $(\bW_I,\bullet)$-orbits $\theta \subset \bt^*$ contained in $\bWf \bullet \xi$, and as in~\cite[Lemma~3.4]{br-Hecke} we have a canonical identification
\begin{equation}
\label{eqn:decomp-formal-neighborhood}
\FN_{\bt^*/(\bW_I,\bullet)}(\{(\kappa^I)^{-1}(\xi)\}) = \bigsqcup_{\theta \subset \bWf \bullet \xi} \FN_{\bt^*/(\bW_I,\bullet)}(\{\theta\})
\end{equation}
where $\theta$ runs over the $(\bW_I,\bullet)$-orbits contained in $\bWf \bullet \xi$. The functor
\[
\tD_I \otimes_{\cU\bg} (-) : \Mod(\cU\bg) \to \Modqc(\tD_I),
\]
i.e.~the pullback functor associated with the natural morphism of ringed spaces
$(X_I, \tsD_I) \to (Y, \cU \bg)$, induces a functor
\begin{equation}
\label{eqn:loc-functor}
\Mod((\cU\bg)^{\widehat{\xi}}) \to \Modqc(\tD_I \otimes_{\scO(\bt^*/(\bW_I,\bullet))} \scO(\FN_{\bt^*/(\bW_I,\bullet)}(\{(\kappa^I)^{-1}(\xi)\}))),
\end{equation}
and the right-hand side identifies with the direct sum
\[
\bigoplus_{\theta \subset \bWf \bullet \xi} \Modqc(\tD_I \otimes_{\scO(\bt^*/(\bW_I,\bullet))} \scO(\FN_{\bt^*/(\bW_I,\bullet)}(\theta)))
\]
where $\theta$ is as above.
By definition, $(f_I^\xi)^*$ coincides with the composition of~\eqref{eqn:loc-functor} with projection to the direct summand $\Modqc(\tD_I^{\widehat{\xi}})$. Since any flat $(\cU\bg)^{\widehat{\xi}}$-module is also flat over $\cU\bg$, this identification also holds at the derived level, and we deduce that for any $M$ in $D^- \Mod((\cU \bg)^{\widehat{\xi}})$ the image in $D^- \Modqc(\tD_I)$ of $L(f_I^\xi)^*(M)$ is a direct summand in
\[
\tD_I \lotimes_{\cU\bg} M.
\]
Since $\cU\bg$ has finite homological dimension, it follows that $L(f_I^\xi)^*$ sends bounded complexes to bounded complexes. 
%Now, as for Lemma~\ref{lem:HC-tHC}, the obvious functors
%\[
%\Db \Modfg((\cU\bg)^{\widehat{\xi}}) \to \Db \Mod((\cU\bg)^{\widehat{\xi}}), \quad
%\Db \Modc(\tD_I^{\widehat{\xi}}) \to \Modqc(\tD_I^{\widehat{\xi}})
%\]
%are fully faithful, and 
%Since 
The desired claim follows, using the fact that any finitely generated $(\cU\bg)^{\widehat{\xi}}$-module is a quotient of a \emph{finitely generated} flat module.

% the functor $L(f_I^\xi)^*$ finally restricts to a functor
%\[
%L(f_I^\xi)^* : \Db \Modfg((\cU\bg)^{\widehat{\xi}}) \to \Db \Modc(\tD_I^{\widehat{\xi}}).
%\]
\eqref{it:comp-loc-4}
The fact that our functors are adjoint follows from the fact that $(f_I^\xi)^*$ is left adjoint to $(f_I^\xi)_*$ and general properties of derived functors (see~\cite[\href{https://stacks.math.columbia.edu/tag/0DVC}{Tag 0DVC}]{stacks-project}).

Then we consider the adjunction morphism
$\id \to R(f_I^\xi)_* \circ L(f_I^\xi)^*$.
Recall from~\cite[Proposition~1.2.3(b)]{bmr2} that we have $\mathsf{H}^i(X_I, \scD_I)=0$ for any $i>0$. By the flat base change theorem (see~\cite[\href{https://stacks.math.columbia.edu/tag/02KH}{Tag 02KH}]{stacks-project}) this implies that $\mathsf{H}^i(X_I^{\widehat{\xi}}, \tsD^{\widehat{\xi}}_I)=0$ for any $i>0$. Using also~\eqref{eqn:global-sections-DI-xi} and the identification~\eqref{eqn:identification-Ug-xi}, we deduce that adjunction induces an isomorphism
\[
(\cU \bg)^{\widehat{\xi}} \simto R(f_I^\xi)_* \circ L(f_I^\xi)^* (\cU \bg)^{\widehat{\xi}}.
\]
Since $(f_I^\xi)_*$ has finite homological dimension (see the proof of~\eqref{it:comp-loc-1}), it follows that this property holds for any bounded above complex of finitely generated free $(\cU \bg)^{\widehat{\xi}}$-modules, hence finally for any object in $\Db \Modfg((\cU \bg)^{\widehat{\xi}})$, proving that this adjunction morphism $\id \to R(f_I^\xi)_* \circ L(f_I^\xi)^*$ is an isomorphism on $\Db \Modfg((\cU \bg)^{\widehat{\xi}})$. Fully faithfulness of the functor $L(f_I^\xi)^*$ is a standard consequence.
\end{proof}

%-------------------------------------------------------------
\subsection{Completed localization, II}
\label{ss:comp-loc-end-proof}
%-------------------------------------------------------------

We will now explain how to complete the proof of Theorem~\ref{thm:comp-loc}. We continue to assume that $\mathrm{Stab}_{(\bWf,\bullet)}(\xi) \subset \bW_I$. In view of the work done so far, to conclude it suffices to prove that the adjunction morphism
\begin{equation}
\label{eqn:adjunction-morph-Loc}
L(f_I^\xi)^* \circ R(f_I^\xi)_*(\scM) \to \scM
\end{equation}
is an isomorphism for any $\scM$ in $\Db \Modc (\tsD_I^{\widehat{\xi}})$.

Here and below we will identify $X^*(\bT^{(1)})$ with $X^*(\bT)$ in such a way that pullback under the Frobenius morphism $\bT \to \bT^{(1)}$ corresponds to the map $\lambda \mapsto \ell \cdot \lambda$.
As noted in Remark~\ref{rmk:translation-DI},
any $\lambda \in X^*(\bT) = X^*(\bT^{(1)})$ orthogonal to any $\alpha^\vee$ with $\alpha \in I$ defines a line bundle on $(\bG/\bP_I)^{(1)}$, whose pullback to $X_I^{\widehat{\xi}}$ will be denoted $\scO_{X_I^{\widehat{\xi}}}(\lambda)$. We can then consider the left $\tsD_I^{\widehat{\xi}}$-module
\[
\tsD_I^{\widehat{\xi}}(\lambda) = \tsD_I^{\widehat{\xi}} \otimes_{\scO_{X_I^{\widehat{\xi}}}} \scO_{X_I^{\widehat{\xi}}}(\lambda).
\]

\begin{lem}
\phantomsection
\label{lem:comp-loc-end}
\begin{enumerate}
\item
\label{it:comp-loc-end-1}
For any nonzero $\scM$ in $\Db \Modc (\tsD_I^{\widehat{\xi}})$, there exist $\lambda \in X^*(\bT)$ orthogonal to any $\alpha^\vee$ with $\alpha \in I$ and $n \in \Z$ such that
\[
\Hom_{\Db \Modc (\tsD_I^{\widehat{\xi}})}(\tsD_I^{\widehat{\xi}}(\lambda), \scM[n]) \neq 0.
\]
\item
\label{it:comp-loc-end-2}
For any $\lambda \in X^*(\bT)$ orthogonal to any $\alpha^\vee$ with $\alpha \in I$, the morphism~\eqref{eqn:adjunction-morph-Loc} is an isomorphism when $\scM = \tsD_I^{\widehat{\xi}}(\lambda)$.
\end{enumerate}
\end{lem}

The proof of~\eqref{it:comp-loc-end-2} will use the following preliminary result.

\begin{lem}
\label{lem:closedness}
The image under the morphism $X_I^{\widehat{\xi}} \to \FN_{\bt^*/(\bW_I,\bullet)}(\{\kappa_I(\xi)\})$ of any closed $\bG$-invariant subset is closed.
\end{lem}

\begin{proof}
In view of the isomorphism~\eqref{eqn:isom-FN}, to prove the lemma it suffices to prove that the image in $\FN_{\bt^*/(\bWf,\bullet)}(\{\kappa(\xi)\})$ of any $\bG$-invariant closed subset of $X_I^{\widehat{\xi}}$ is closed. Then, considering the factorization
\[
X_I^{\widehat{\xi}} \to Y^{\widehat{\xi}} \to \FN_{\bt^*/(\bWf,\bullet)}(\{\kappa(\xi)\})
\]
in which the first map is $\bG$-equivariant and proper (hence closed), we see that it suffices to prove that the image under the second map of any $\bG$-equivariant closed subset of $Y^{\widehat{\xi}}$ is closed. We will prove below that the latter morphism is the quotient morphism with respect to the action of $\bG$ on $Y^{\widehat{\xi}}$; in view of~\cite[Corollary~2 on p.~262]{seshadri} (applied to the natural closed immersion $Y^{\widehat{\xi}} \hookrightarrow \bg^* \times_\bk \FN_{\bt^*/(\bWf,\bullet)}(\{\kappa(\xi)\})$ and the reductive group $\bG \times_\bk \FN_{\bt^*/(\bWf,\bullet)}(\{\kappa(\xi)\})$ over $\FN_{\bt^*/(\bWf,\bullet)}(\{\kappa(\xi)\})$) this will imply that the image of any $\bG$-invariant closed subset is closed, as desired.

To conclude we therefore have to prove that the natural morphism
\[
\scO(\FN_{\bt^*/(\bWf,\bullet)}(\{\kappa(\xi)\})) \to \scO(Y^{\widehat{\xi}})^{\bG}
\]
is an isomorphism. Since $\scO(\bt^*)$ is free over $\scO(\bt^*/(\bWf, \bullet))$ by~\cite{demazure}, it suffices to prove this claim after base change along the quotient morphism $\bt^* \to \bt^*/(\bWf, \bullet)$. Now, by finiteness we have
\[
\bt^* \times_{\bt^*/(\bWf, \bullet)} \FN_{\bt^*/(\bWf,\bullet)}(\{\kappa(\xi)\}) \cong \FN_{\bt^*}(\{\kappa^{-1}(\kappa(\xi))\}),
\]
and by commutation of fixed points with flat base change (see~\cite[Equation~(3) in~\S I.2.10]{jantzen}) the base change of our morphism identifies with the natural morphism
\[
\scO(\FN_{\bt^*}(\{\kappa^{-1}(\kappa(\xi))\})) \to \scO(\bg^{*(1)} \times_{\bt^{*(1)}/\bWf} \FN_{\bt^*}(\{\kappa^{-1}(\kappa(\xi))\}))^\bG.
\]
Using again commutation of fixed points with flat base change,
the latter morphism is obtained from the Chevalley isomorphism $\bg^{*(1)} / \bG \simto \bt^{*(1)}/\bWf$ by base change along the composition of flat morphisms
\[
\FN_{\bt^*}(\{\kappa^{-1}(\kappa(\xi))\}) \to \bt^* \to \bt^{*(1)} \to \bt^{*(1)}/\bWf,
\]
which finishes the proof. (Here the second morphism is the Artin--Schreier morphism, which is \'etale hence flat, and the third morphism is flat by~\cite{demazure}.)
%
%Denote by $\xi'$ the image of $\pi(\xi)$ under the finite morphism $\bt^*/(\bWf,\bullet) \to \bt^{*(1)}/\bWf$, and by $Z_\xi$ the inverse image of $\xi'$ in $\bt^*/(\bWf,\bullet)$; then we have a closed immersion $\{\pi(\xi)\} \hookrightarrow Z_\xi$, hence a closed immersion
%\[
%\FN_{\bt^*/(\bWf,\bullet)}(\{\pi(\xi)\}) \hookrightarrow \FN_{\bt^*/(\bWf,\bullet)}(Z_\xi),
%\]
%so that to conclude it suffices to prove that the image under the morphism
%\[
%\bg^{*(1)} \times_{\bt^{*(1)}/\bWf} \FN_{\bt^*/(\bWf,\bullet)}(Z_\xi) \to \FN_{\bt^*/(\bWf,\bullet)}(Z_\xi)
%\]
%of any closed $\bG$-stable subset is closed. Now, by finiteness we have
%\[
%\FN_{\bt^*/(\bWf,\bullet)}(Z_\xi) = \bt^*/(\bWf,\bullet) \times_{\bt^{*(1)}/\bWf} \FN_{\bt^{*(1)}/\bWf}(\xi')
%\]
\end{proof}

\begin{proof}[Proof of Lemma~\ref{lem:comp-loc-end}]
\eqref{it:comp-loc-end-1}
It is clear that for any $\lambda$ and $n$ we have
\begin{multline*}
\Hom_{\Db \Modc (\tsD_I^{\widehat{\xi}})}(\tsD_I^{\widehat{\xi}}(\lambda), \scM[n]) \cong \Hom_{\Db \Coh(X_I^{\widehat{\xi}})}(\scO_{X_I^{\widehat{\xi}}}(\lambda), \scM[n]) \\
\cong \mathsf{H}^n(X_I^{\widehat{\xi}}, \scM \otimes_{\scO_{X_I^{\widehat{\xi}}}} \scO_{X_I^{\widehat{\xi}}}(-\lambda)).
\end{multline*}
Now the line bundle on $(\bG/\bP_I)^{(1)}$ defined by $-\lambda$ is ample if $\langle \lambda, \beta^\vee \rangle <0$ for any $\beta \in \fRs \smallsetminus I$ (see e.g.~\cite[Remark after Proposition~II.4.4]{jantzen}), hence in this case $\scO_{X_I^{\widehat{\xi}}}(-\lambda)$ is ample by~\cite[\href{https://stacks.math.columbia.edu/tag/0892}{Tag 0892}]{stacks-project}. This implies the desired claim by~\cite[\href{https://stacks.math.columbia.edu/tag/01Q3}{Tag 01Q3}]{stacks-project}.

\eqref{it:comp-loc-end-2}
In fact, we will prove this claim for any module $\scM$ which is $\bG$-equivariant as a coherent sheaf (with respect to the obvious action of $\bG$ on $X_I^{\widehat{\xi}}$). Namely, consider the cone $\scC$ of our morphism, and assume for a contradiction that $\scC \neq 0$. This cone is easily seen to be a complex of $\bG$-equivariant coherent sheaves, hence so are its cohomology sheaves. Recall the notion of support of a coherent sheaf, and its characterization in~\cite[\href{https://stacks.math.columbia.edu/tag/056J}{Tag 056J}]{stacks-project}. The support of the nonzero cohomology object of $\scC$ of highest degree is a $\bG$-invariant closed subset in $X_I^{\widehat{\xi}}$. By Lemma~\ref{lem:closedness} its image in $\FN_{\bt^*/(\bW_I,\bullet)}(\{\kappa_I(\xi)\})$ is closed. This implies that this image contains the unique closed point, and then that the (derived) pullback of $\scC$ under the closed immersion
\[
\imath_{I,\xi} : \tbg_I^{(1)} \times_{\bt^{*(1)} / \bW_I} \{\kappa_I(\xi)\} \hookrightarrow
\tbg_I^{(1)} \times_{\bt^{*(1)} / \bW_I} \FN_{\bt^*/(\bW_I,\bullet)}(\{\kappa_I(\xi)\}) =
X_I^{\widehat{\xi}}
\]
is nonzero.

Now, using the general version of the base change theorem (see~\cite[Theorem~3.10.3]{lipman}) and Lemma~\ref{lem:flatness-Groth} one sees that the composition of the functor $L(f_I^\xi)^* \circ R(f_I^\xi)_*$ with the pullback functor $L(\imath_{I,\xi})^*$ identifies with the composition of the latter functor with the composition of the equivalence in~\cite[Equation (5) in Theorem~1.5.1]{bmr2} and its inverse (see also~\cite[\S 1.6]{bmr2}); it follows that $L(\imath_{I,\xi})^* \scC=0$, which provides a contradiction.
\end{proof}

\begin{proof}[Proof of Theorem~\ref{thm:comp-loc}\eqref{it:comp-loc-3}]
In view of Lemma~\ref{lem:comp-loc-end}\eqref{it:comp-loc-end-1}, to prove the desired claim it suffices to prove that for any $\scM$ in $\Db \Modc (\tsD_I^{\widehat{\xi}})$, for any $\lambda \in X^*(\bT)$ and $n \in \Z$ the morphism
\[
\Hom(\tsD_I^{\widehat{\xi}}(\lambda), L(f_I^\xi)^* \circ R(f_I^\xi)_*(\scM)[n]) \to 
\Hom(\tsD_I^{\widehat{\xi}}(\lambda), \scM[n])
\]
induced by~\eqref{eqn:adjunction-morph-Loc} is an isomorphism. Now by Lemma~\ref{lem:comp-loc-end}\eqref{it:comp-loc-end-2}, fully faithfulness of $L(f_I^\xi)^*$, and adjunction, we have
\begin{multline*}
\Hom(\tsD_I^{\widehat{\xi}}(\lambda), L(f_I^\xi)^* \circ R(f_I^\xi)_*(\scM) [n]) \cong \\
\Hom(L(f_I^\xi)^* \circ R(f_I^\xi)_*(\tsD_I^{\widehat{\xi}}(\lambda)), L(f_I^\xi)^* \circ R(f_I^\xi)_*(\scM) [n]) \cong \\
\Hom( R(f_I^\xi)_*(\tsD_I^{\widehat{\xi}}(\lambda)), R(f_I^\xi)_*(\scM) [n]) \cong \\
\Hom( L(f_I^\xi)^* \circ R(f_I^\xi)_*(\tsD_I^{\widehat{\xi}}(\lambda)), \scM [n]).
\end{multline*}
We conclude by Lemma~\ref{lem:comp-loc-end}\eqref{it:comp-loc-end-2} again.
\end{proof}

\section{Derived localization for Harish-Chandra bimodules}
\label{sec:loc}
%%%%%%%%%%%%%%%%%%%%%%%%%%%%%%%%%%%%%%%%%%%%%%%

%------------------------------------------------------
\subsection{Steinberg schemes}
\label{ss:Steinberg-schemes}
%------------------------------------------------------

Our goal in this section is to prove an analogue of Theorem~\ref{thm:comp-loc} for the categories of (completed) Harish-Chandra bimodules introduced in~\S\ref{ss:central-characters}. First we introduce the geometry that will be involved in this theorem.

Given two subsets $I,J \subset \fRs$, we will consider the associated ``Steinberg scheme''
\[
 \St_{I,J} = \tbg_I \times_{\bg^*} \tbg_J.
\]
This scheme is equipped with a diagonal $\bG$-action, and with a canonical $\bG$-equiva\-riant morphism
\[
 \St_{I,J} \to \bt^* / \bW_I \times \bt^* / \bW_J
\]
(for the trivial action on the codomain) given by~\eqref{eqn:morph-tbg-t} on each factor,
which factors through $\bt^* / \bW_I \times_{\bt^*/\bWf} \bt^* / \bW_J$;
we will set
\[
 \St^{\wedge}_{I,J} := \St_{I,J} \times_{\bt^{*}/\bW_I \times_{\bt^*/\bWf} \bt^{*}/\bW_J} \FN_{\bt^{*}/\bW_I \times_{\bt^*/\bWf} \bt^{*}/\bW_J}( \{ (0,0) \} ).
\]
In case $I=J=\varnothing$, these subsets will usually be omitted from notation.

\begin{rmk}
\label{rmk:St-fiber-product}
By finiteness of the map $\bt^{*}/\bW_I \times_{\bt^*/\bWf} \bt^{*}/\bW_J \to \bt^*/\bWf$, and since $(0,0)$ is the only closed point in the preimage of $0$ under this map,
the scheme $\St^\wedge_{I,J}$ also identifies with the fiber product
\[
\St_{I,J} \times_{\bt^*/\bWf} \FN_{\bt^*/\bWf}(\{0\}).
\]
\end{rmk}

We set
\[
 \bZ_{I,J} := \bt^*/(\bW_I,\bullet) \times_{\bt^{*(1)}/\bWf} \bt^*/(\bW_J,\bullet).
 \]
(In case $I=J=\fRs$, this scheme is the scheme $\bZ$ from~\S\ref{ss:central-characters}.)\footnote{For most notation used in the paper, omission of a subset $I \subset \fRs$ in the notation means that $I=\varnothing$. The notation $\bZ$ is one the only exceptions to this rule.}
Consider the fiber product
\begin{multline}
\label{eqn:base-osDIJ}
\left( \tbg_I^{(1)} \times_{\bt^{*(1)} / \bW_I} \bt^*/(\bW_I,\bullet) \right) \times_{\bg^{*(1)}} \left( \tbg_J^{(1)} \times_{\bt^{*(1)} / \bW_J} \bt^*/(\bW_J,\bullet) \right) \\
\cong \St_{I,J}^{(1)} \times_{\bt^{*(1)} / \bW_I \times_{\bt^{*(1)}/\bWf} \bt^{*(1)} / \bW_J} \bZ_{I,J},
\end{multline}
where the morphisms to $\bg^{*(1)}$ are induced by the natural morphisms $\tbg_I \to \bg^*$ and $\tbg_J \to \bg^*$ (see~\S\ref{ss:parabolic-Groth}).

If $\lambda \in X^*(\bT)$, we will denote by $\tla_I$, resp.~$\tla_J$, the image of $\ola$ in $\bt^*/(\bW_I,\bullet)$, resp.~$\bt^*/(\bW_J,\bullet)$; in the notation of~\S\ref{ss:comp-loc-statement} we therefore have $\tla_I=\kappa_I(\ola)$, $\tla_J = \kappa_J(\ola)$. For $\lambda,\mu \in X^*(\bT)$ we also set
\[
\bZ^{\widehat{\lambda}, \widehat{\mu}}_{I,J} = \FN_{\bZ_{I,J}}( \{ (\tla_I,\tmu_J) \} ).
\]

%As in~\cite{br-Hecke}, we will say that a weight $\lambda \in X^*(\bT)$ belongs to the \emph{lower closure of the fundamental alcove} if it satisfies
%\[
%0 \leq \langle \lambda+\varsigma, \alpha^\vee \rangle < p
%\]
%for any $\alpha \in \fR_+$. 
%For appropriate choices of $\lambda,\mu$, the following lemma will allow us to view $\osD_{I,J}$ as a sheaf of rings on $\St^{\wedge(1)}_{I,J}$.

\begin{lem}
\label{lem:steinberg-completion}
Let $\lambda,\mu \in X^*(\bT)$ and $I,J \subset \fRs$,
% which belong to the lower closure of the fundamental alcove, 
 and assume that
 \[
 \bW_I \subset \mathrm{Stab}_{(\bWf,\bullet)}(\ola), \quad \bW_J \subset \mathrm{Stab}_{(\bWf,\bullet)}(\omu).
 \]
%$\langle \lambda+\varsigma, \alpha^\vee \rangle=0$ for any $\alpha \in I$, and $\langle \mu+\varsigma, \beta^\vee \rangle=0$ for any $\beta \in J$.
%\[
%I=\{\alpha \in \fRs \mid \langle \lambda+\varsigma, \alpha^\vee \rangle=0\}, \quad J=\{\alpha \in \fRs \mid \langle \mu+\varsigma, \alpha^\vee \rangle=0\}.
%\]
Then the projection morphism
\[
\St_{I,J}^{(1)} \times_{\bt^{*(1)} / \bW_I \times_{\bt^{*(1)}/\bWf} \bt^{*(1)} / \bW_J} \bZ_{I,J} \to \St_{I,J}^{(1)}
\]
induces an
isomorphism of schemes
\[
\St_{I,J}^{(1)} \times_{\bt^{*(1)} / \bW_I \times_{\bt^{*(1)} / \bWf} \bt^{*(1)} / \bW_J} 
%\FN_{\bZ_{I,J}}((\tla_I,\tmu_J)) 
\bZ^{\widehat{\lambda}, \widehat{\mu}}_{I,J}
\simto \St^{\wedge(1)}_{I,J}.
\]
\end{lem}

\begin{proof}
%The morphism under consideration is induced by the natural projection
%\[
%\St_{I,J}^{(1)} \times_{\bt^{*(1)} / \bW_I \times \bt^{*(1)} / \bW_J} \bigl( \bt^*/(\bW_I,\bullet) \times \bt^*/(\bW_J,\bullet) \bigr) \to \St_{I,J}.
%\]
%Hence 
The lemma will follow from the definitions once we prove that the morphism
\[
\bt^*/(\bW_I,\bullet) \to \bt^{*(1)} / \bW_I, \quad \text{resp.} \quad
\bt^*/(\bW_J,\bullet) \to \bt^{*(1)} / \bW_J,
\]
is \'etale at $\tla_I$, resp.~$\tmu_J$. The two cases are similar, so we concentrate on the first one. We consider the composition
\[
\bt^* \to \bt^{*(1)} \to \bt^{*(1)} / \bW_I,
\]
where the first morphism is the Artin--Schreier morphism~\eqref{eqn:Artin-Schreier}. This morphism is the quotient morphism for the natural ``dot'' action of $\bW_I \ltimes (X^*(\bT)/\ell X^*(\bT))$ on $\bt^*$ (where $X^*(\bT)/\ell X^*(\bT)$ acts by translation on $\bt^* = \bk \otimes_\Z X^*(\bT)$). Since the point $\ola \in \bt^*$ is stabilized by the action of $\bW_I$, whereas its orbit under $X^*(\bT)/\ell X^*(\bT)$ is free, its stabilizer is exactly $\bW_I$, so the claim follows from the general criterion~\cite[Exp. V, Proposition 2.2]{sga1}.
\end{proof}

We will denote by
\[
\Coh^{\bG^{(1)}}_{\cN}(\St_{I,J}^{(1)})
\]
the category of $\bG^{(1)}$-equivariant coherent sheaves on $\St_{I,J}^{(1)}$ set-theoretically supported on the preimage of the point $(0,0) \in \bt^{*(1)}/\bW_I \times_{\bt^{*(1)}/\bWf} \bt^{*(1)}/\bW_J$, or equivalently set-theoretically supported on the preimage of $0 \in \bt^{*(1)}/\bWf$. The following lemma is an application of the considerations in~\S\ref{ss:app-special-case}.

\begin{lem}
\label{lem:DbCoh-nil}
The obvious functor
\[
\Db \Coh^{\bG^{(1)}}_{\cN}(\St_{I,J}^{(1)}) \to \Db \Coh^{\bG^{(1)}}(\St^{\wedge(1)}_{I,J})
\]
is fully faithful; its essential image is the full subcategory whose objects are the complexes $\scF$ such that the morphism
\[
\scO(\FN_{\bt^{*(1)}/\bW_I \times_{\bt^{*(1)}/\bWf} \bt^{*(1)}/\bW_J}(\{(0,0)\})) \to \End(\scF)
\]
vanishes on a power of the unique maximal ideal.
\end{lem}

%------------------------------------------------------
\subsection{Equivariant sheaves of algebras and modules}
\label{ss:equiv-sheaves-alg-mod}
%------------------------------------------------------

As in~\S\ref{ss:sheaves-alg-mod}, consider a $\bk$-scheme $X$, and a sheaf of algebras $\scA$ on $X$ endowed with a morphism of sheaves of algebras $\scO_X \to \scA$ which makes $\scA$ a quasi-coherent $\scO_X$-module. Assume furthermore that $X$ is endowed with an action of a smooth affine $\bk$-group scheme $H$, and that the $\scO_X$-module $\scA$ admits a structure of $H$-equivariant quasi-coherent sheaf such that the morphism $\scO_X \to \scA$ and the multiplication morphism $\scA \otimes_{\scO_X} \scA \to \scA$ are $H$-equivariant. In this setting, a \emph{weakly $H$-equivariant $\scA$-module} is an object $\scM$ of $\Modqc(\scA)$ endowed with a structure of $H$-equivariant quasi-coherent $\scO_X$-module (with respect to the restriction of the $\scA$-action to $\scO_X$) such that the action morphism $\scA \otimes_{\scO_X} \scM \to \scM$ is a morphism of $H$-equivariant quasi-coherent sheaves. These are naturally objects of an abelian category, which will be denoted
\[
\Modqc^H(\scA),
\]
and we have an obvious faithful exact functor
\[
\For^H_{\scA} : \Modqc^H(\scA) \to \Modqc(\scA).
\]
(We will omit the subscript ``$\scA$'' in this notation in case it is obvious from the context.)
If $\scA$ is noetherian, we will denote by
\[
\Modc^H(\scA)
\]
the full subcategory of $\Modqc^H(\scA)$ whose objects are those $\scM$ such that $\For_{\scA}^H(\scM)$ belongs to $\Modc(\scA)$.

There is a stronger notion of equivariance in the following setting. Consider again a smooth affine $\bk$-group scheme $H$ acting on a $\bk$-scheme $X$, and let $\fh$ be the Lie algebra of $H$. Recall that if $\scM$ is an $H$-equivariant quasi-coherent $\scO_X$-module, there exists a canonical Lie algebra morphism $L_{\scM} : \fh \to \End_{\bk}(\scM)$; see~\cite[\S 3.1]{kashiwara}.\footnote{In~\cite{kashiwara} it is assumed that the base field is $\mathbb{C}$; but the considerations used here and below hold over any base field.} Assume we are given a sheaf of algebras $\scA$ as above with an $H$-equivariant structure satisfying the properties considered at the beginning of the present subsection, and a morphism of Lie algebras $\sigma : \fh \to \Gamma(X, \scA)$ which is $H$-equivariant and satisfies
\begin{equation}
\label{eqn:cond-strong-equiv-modules}
L_{\scA}(x) = [\sigma(x),-] \quad \text{for any $x \in \fh$}
\end{equation}
(where the right-hand side is the commutator in the algebra $\Gamma(X, \scA)$). Then an object $\scM$ of $\Modqc^H(\scA)$, with structure morphism $\alpha_\scM : \scA \to \sEnd_\bk(\scM)$, is called a \emph{strongly $H$-equivariant $\scA$-module} if we have
\[
\alpha_\scM(\sigma(x)) = L_\scM(x) \quad \text{for any $x \in \fh$.}
\]
(For this notion, see also~\cite[\S 2.16]{bl2}.) We will denote by
\[
\Modqc(\scA,H)
\]
the full subcategory of $\Modqc^H(\scA)$ whose objects are the strongly equivariant modules, and by
\[
\Modc(\scA,H)
\]
the intersection of the full subcategories $\Modqc(\scA,H)$ and $\Modc^H(\scA)$ in the category $\Modqc^H(\scA)$.

The following statement 
%(which generalizes Lemma~\ref{lem:HC-tHC}) 
follows from the same arguments as for Lemma~\ref{lem:HC-tHC}.
%standard arguments (see e.g.~\cite[Proof of Corollary~2.11]{arinkin-bezrukavnikov}).

\begin{lem}
\label{lem:Modc-Modqc}
If $\scA$ is noetherian, the obvious functor
\[
\Db \Modc(\scA,H) \to \Db \Modqc(\scA,H)
\]
is fully faithful, and identifies the left-hand side with the full subcategory of the right-hand side whose objects are the complexes all of whose cohomology objects belong to $\Modc(\scA,H)$.
\end{lem}

In particular, if $X$ is a smooth $\bk$-variety endowed with an action of a smooth affine $\bk$-group scheme $H$, the sheaf of algebras $\cD_X$ admits a canonical structure of $H$-equivariant quasi-coherent sheaf which satisfies the conditions above, and there exists a canonical $H$-equivariant morphism of Lie algebras $\fh \to \Gamma(X, \cD_X)$ such that~\eqref{eqn:cond-strong-equiv-modules} is satisfied. We can therefore consider the abelian categories
\[
\Modqc^H(\cD_X), \quad \Modc^H(\cD_X), \quad \Modqc(\cD_X,H), \quad \Modc(\cD_X,H).
\]
As explained in~\cite[Lemma~3.1.4]{kashiwara}, the category $\Modqc(\cD_X,H)$ can also be decribed as the full subcategory of $\Modqc^H(\cD_X)$ whose objects are the weakly $H$-equivariant $\cD_X$-modules $\scM$ such that the isomorphism
\[
a_X^* \scM \simto p_X^* \scM
\]
giving the structure of $H$-equivariant quasi-coherent sheaf is an isomorphism of $\cD_{H \times X}$-modules.
(Here, $a_X,p_X : H \times X \to X$ are the action and projection morphisms, respectively.)

In terms of the equivalences~\eqref{eqn:equivalences-Dmod}, there exists a canonical action of $H^{(1)}$ on $T^* X^{(1)}$, which by composition with $\Fr_H$ provides an action of $H$. The sheaf of algebras $\scD_X$ admits a canonical structure of $H$-equivariant quasi-coherent sheaf, and there exist canonical equivalences
\begin{equation*}
%\label{eqn:equivalences-Dmod-equiv}
\Modqc^H(\cD_X) \simto \Modqc^H(\scD_X), \quad \Modc^H(\cD_X) \simto \Modc^H(\scD_X).
\end{equation*}
We can also consider the morphism $\fh \to \Gamma(X,\cD_X) = \Gamma(T^* X^{(1)}, \scD_X)$, hence the associated categories $\Modqc(\scD_X,H)$ and $\Modc(\scD_X,H)$, and we have canonical equivalences
\[
\Modqc(\cD_X,H) \simto \Modqc(\scD_X,H), \quad \Modc(\cD_X,H) \simto \Modc(\scD_X,H).
\]

Let us now consider a variant of these constructions in the setting of~\S\ref{ss:monodromic-variant}. Given a subset $I \subset \fRs$,
the $\scO_{\bG/\bP_I}$-module $\tD_I$ also admits a natural structure of $\bG$-equivariant quasi-coherent sheaf (with respect to the obvious action of $\bG$) which satisfies the conditions of~\S\ref{ss:sheaves-alg-mod}, and
the action of $\bG$ on $\bG/\bU_I$ induces a $\bG$-equivariant morphism of Lie algebras
\begin{equation}
\label{eqn:morphism-fg-tDI}
\bg \to \Gamma(\bG/\bP_I, \tD_I).
\end{equation}
We can therefore consider the categories
\begin{equation}
\label{eqn:categories-tDI-G}
\Modqc^{\bG}(\tD_I), \quad \Modc^{\bG}(\tD_I), \quad \Modqc(\tD_I,\bG), \quad \Modc(\tD_I,\bG).
\end{equation}
Similarly, $\tsD_I$ has a canonical $\bG$-equivariant structure with respect to the pullback (under $\Fr_\bG$) of the action of $\bG^{(1)}$ on $\tbg_I^{(1)} \times_{\bt^{*(1)} / \bW_I} \bt^*/(\bW_I,\bullet)$ induced by the action on $\tbg_I^{(1)}$, and the right-hand side in~\eqref{eqn:morphism-fg-tDI} identifies with $\Gamma \bigl( \tbg_I^{(1)} \times_{\bt^{*(1)} / \bW_I} \bt^*/(\bW_I,\bullet), \tsD_I \bigr)$. We can therefore consider the categories
\[
\Modqc^{\bG}(\tsD_I), \quad \Modc^{\bG}(\tsD_I), \quad \Modqc(\tsD_I,\bG), \quad \Modc(\tsD_I,\bG),
\]
and each of these categories identifies with the corresponding category in~\eqref{eqn:categories-tDI-G}.

\begin{rmk}
Consider the action of $\bL_I$ on $\bG/\bU_I$ as in~\S\ref{ss:monodromic-variant}. Then the sheaf of algebras $\cD_{\bG/\bU_I}$ has a canonical structure of $\bL_I$-equivariant quasi-coherent sheaf, and
we have a canonical equivalence of categories
$\Modqc^{\bL_I}(\cD_{\bG/\bU_I}) \simto \Modqc(\tD_I)$.
\end{rmk}

%------------------------------------------------------
\subsection{Harish-Chandra \texorpdfstring{$\cD$}{D}-modules}
\label{ss:HC-Dmod}
%------------------------------------------------------

From now on we fix two subsets $I,J \subset \fRs$, and set
\[
\tD_{I,J} := \tD_I \boxtimes \tD_J^\op, \qquad \tsD_{I,J} := \tsD_I \boxtimes \tsD_J^\op.
\]
Then $\tD_{I,J}$ is a sheaf of algebras on $\bG/\bP_I \times \bG/\bP_J$, $\tsD_{I,J}$ is a sheaf of algebras on
\begin{equation}
\label{eqn:base-tsDIJ}
\left( \tbg_I^{(1)} \times_{\bt^{*(1)} / \bW_I} \bt^*/(\bW_I,\bullet) \right) \times \left( \tbg_J^{(1)} \times_{\bt^{*(1)} / \bW_J} \bt^*/(\bW_J,\bullet) \right),
\end{equation}
and their pushforwards under the natural maps to $(\bG/\bP_I)^{(1)} \times (\bG/\bP_J)^{(1)}$ coincide canonically. We will consider the diagonal $\bG$-actions on $\bG/\bP_I \times \bG/\bP_J$ and on~\eqref{eqn:base-tsDIJ}, and the composition
\[
\bg \to \bg \times \bg^\op \to \Gamma(\bG/\bP_I \times \bG/\bP_J, \tD_{I,J}) = \Gamma(\tsD_{I,J})
\]
(where in the right-hand side we omit the scheme~\eqref{eqn:base-tsDIJ} to save space) where the first morphism is the antidiagonal embedding and the second one is induced by the morphisms~\eqref{eqn:morphism-fg-tDI}. With respect to these structures we can consider the categories
\[
\Modqc^\bG(\tD_{I,J}), \quad \Modqc(\tD_{I,J},\bG), \quad \Modqc^\bG(\tsD_{I,J}), \quad \Modqc(\tsD_{I,J},\bG)
\]
and their subcategories of coherent modules. We have canonical equivalences
\[
\Modqc^\bG(\tD_{I,J}) \cong \Modqc^\bG(\tsD_{I,J}), \quad \Modqc(\tD_{I,J},\bG) \cong \Modqc(\tsD_{I,J},\bG),
\]
and similarly for coherent modules.

Consider the algebra morphism
\[
\scO(\bg^{*(1)} \times \bg^{*(1)}) \to \Gamma(\tsD_{I,J}) = \Gamma(\bG/\bP_I \times \bG/\bP_J, \tD_{I,J})
\]
obtained from the natural morphism from~\eqref{eqn:base-tsDIJ} to $\bg^{*(1)} \times \bg^{*(1)}$. Using this morphism we can define the sheaf of algebras
\[
\oD_{I,J} := \tD_{I,J} \otimes_{\scO(\bg^{*(1)} \times \bg^{*(1)})} \scO(\Delta \bg^{*(1)})
\]
on $\bG/\bP_I \times \bG/\bP_J$,
where $\Delta \bg^{*(1)} \subset \bg^{*(1)} \times \bg^{*(1)}$ is the diagonal copy of $\bg^{*(1)}$.
If we denote by
\[
\osD_{I,J}
\]
the pullback of $\tsD_{I,J}$ to the closed subscheme~\eqref{eqn:base-osDIJ}, then we have equivalences
\[
\Modqc(\oD_{I,J}) \cong \Modqc(\osD_{I,J}), \quad \Modc(\oD_{I,J}) \cong \Modc(\osD_{I,J}).
\]
We can also consider the categories of weakly or strongly $\bG$-equivariant modules for these sheaves of algebras, and have similar equivalences for the categories of equivariant modules.

Since a $\cU\bg$-module obtained by differentiation from a $\bG$-module has trivial action of the Frobenius center, any strongly $\bG$-equivariant $\tsD_{I,J}$-module is scheme-theoretically supported on~\eqref{eqn:base-osDIJ}; in other words we have canonical equivalences of categories
\[
\Modqc(\tsD_{I,J}, \bG) \cong \Modqc(\osD_{I,J}, \bG), \quad
\Modc(\tsD_{I,J}, \bG) \cong \Modc(\osD_{I,J}, \bG).
\]
%(On the right-hand sides we consider the obvious structures on $\osD_{I,J}$ deduced from the corresponding structures on $\tsD_{I,J}$.)

For $\lambda,\mu \in X^*(\bT)$ we will denote by
\[
\osD_{I,J}^{\widehat{\lambda},\widehat{\mu}}
\]
the pullback of $\osD_{I,J}$ to
\begin{equation}
\label{eqn:scheme-oD-hla-hmu}
\St_{I,J}^{(1)} \times_{\bt^{*(1)} / \bW_I \times_{\bt^{*(1)} / \bWf} \bt^{*(1)} / \bW_J} \bZ_{I,J}^{\hla,\hmu}.
\end{equation}
We can also consider the natural morphism
\[
\scO(\bZ_{I,J}) \to \Gamma(\osD_{I,J}) = \Gamma(\bG/\bP_I \times \bG/\bP_J, \oD_{I,J}),
\]
and then the sheaf of algebras
\[
\oD_{I,J}^{\widehat{\lambda},\widehat{\mu}} := \oD_{I,J} \otimes_{\scO(\bZ_{I,J})} \scO(\bZ_{I,J}^{\hla,\hmu})
\]
on $\bG/\bP_I \times \bG/\bP_J$. As above we have canonical equivalences
\begin{equation}
\label{eqn:equivalences-oD-osD}
\Modqc(\oD_{I,J}^{\widehat{\lambda},\widehat{\mu}}) \cong \Modqc(\osD_{I,J}^{\widehat{\lambda},\widehat{\mu}}), \quad \Modc(\oD_{I,J}^{\widehat{\lambda},\widehat{\mu}}) \cong \Modc(\osD_{I,J}^{\widehat{\lambda},\widehat{\mu}}),
\end{equation}
and similarly for the categories of equivariant modules.

%------------------------------------------------------
\subsection{Localization theorem}
%------------------------------------------------------

%The following theorem will be proved in REF below.
%The main result of this section is the following theorem.

%\begin{thm}
%\label{thm:localization-HC}
%Let $\lambda, \mu \in X^*(\bT)$ which belong to the lower closure of the fundamental alcove, and assume that
%\[
%I=\{\alpha \in \fRs \mid \langle \lambda+\varsigma, \alpha^\vee \rangle=0\}, \quad J=\{\alpha \in \fRs \mid \langle \lambda+\varsigma, \alpha^\vee \rangle=0\}.
%\]
%Then there exists a canonical equivalence of triangulated categories
%\[
%\Db \Modc(\osD_{I,J}^{\widehat{\lambda}, \widehat{\mu}}, \bG) \simto \Db \HC^{\widehat{\lambda}, \widehat{\mu}}.
%\]
%\end{thm}

From now on we fix $\lambda, \mu \in X^*(\bT)$.
%which belong to the lower closure of the fundamental alcove, and assume that
%\[
%I=\{\alpha \in \fRs \mid \langle \lambda+\varsigma, \alpha^\vee \rangle=0\}, \quad J=\{\alpha \in \fRs \mid \langle \lambda+\varsigma, \alpha^\vee \rangle=0\}.
%\]
%The proof of Theorem~\ref{thm:localization-HC} will use 
Our localization theorem for completed Harish-Chandra bimodules will be based on (a slight variant of)
Theorem~\ref{thm:comp-loc}, applied to the group $\bG \times \bG$
and the subset $I \times J \subset \fRs \times \fRs$.
%, and the character $(\ola,\omu)$. We start by making this theorem more explicit in this special case.
% More specifically, consider two weights $\lambda,\mu \in X^*(\bT)$. 

%We will identify
%\[
%\bG^\op / \bP_J^\op = \bP_J \backslash \bG
%\]
%with $\bG/\bP_J$ via $\bP_J g \mapsto g^{-1} \bP_J$; under this identification $\tD_{\bG^\op / \bP_J^\op}$ corresponds to $\tD_{\bG/\bP_J}^{\op}$.
 To save space we set
\begin{multline*}
X_{I,J}^{\widehat{\lambda},\widehat{\mu}} := \left( \tbg_I^{(1)} \times_{\bt^{*(1)} / \bW_I} \bt^*/(\bW_I,\bullet) \right) \times \left( \tbg_J^{(1)} \times_{\bt^{*(1)} / \bW_J} \bt^*/(\bW_J,\bullet) \right) \\
\times_{\bt^*/(\bW_I,\bullet) \times \bt^*/(\bW_J,\bullet)} \FN_{\bt^*/(\bW_I,\bullet) \times \bt^*/(\bW_J,\bullet)}( \{ (\tla_I,\tmu_J) \} )
\end{multline*}
and
\begin{multline*}
Y^{\widehat{\lambda},\widehat{\mu}} := \left( \bg^{*(1)} \times_{\bt^{*(1)} / \bWf} \bt^*/(\bWf,\bullet) \right) \times \left( \bg^{*(1)} \times_{\bt^{*(1)} / \bWf} \bt^*/(\bWf,\bullet) \right) \\
\times_{\bt^*/(\bWf,\bullet) \times \bt^*/(\bWf,\bullet)} \FN_{\bt^*/(\bWf,\bullet) \times \bt^*/(\bWf,\bullet)}( \{ (\tla,\tmu) \} ).
\end{multline*}
We have a sheaf of algebra $\tsD_{I,J}^{\widehat{\lambda},\widehat{\mu}}$ on $X_{I,J}^{\widehat{\lambda},\widehat{\mu}}$ (obtained by pullback from the algebra $\tsD_{I,J}$ of~\S\ref{ss:HC-Dmod}), and a sheaf of algebras on $Y^{\widehat{\lambda},\widehat{\mu}}$ corresponding to the $\scO(Y^{\widehat{\lambda},\widehat{\mu}})$-algebra
\begin{multline*}
(\cU \bg \otimes \cU\bg^\op)^{\widehat{\lambda},\widehat{\mu}} := \\
(\cU \bg \otimes \cU\bg^\op) \otimes_{\scO(\bt^*/(\bWf,\bullet) \times \bt^*/(\bWf,\bullet))} \scO(\FN_{\bt^*/(\bWf,\bullet) \times \bt^*/(\bWf,\bullet)}( \{ (\tla,\tmu) \} ))
\end{multline*}
(which will be denoted by the same symbol).
%$(\cU \bg \otimes \cU\bg^\op)^{\widehat{\lambda},\widehat{\mu}}$. 
We have a canonical morphism of ringed spaces
\[
f_{I,J}^{\lambda,\mu} : \bigl( X_{I,J}^{\widehat{\lambda},\widehat{\mu}}, \tsD_{I,J}^{\widehat{\lambda},\widehat{\mu}} \bigr) \to
\bigl( Y^{\widehat{\lambda},\widehat{\mu}}, (\cU \bg \otimes \cU\bg^\op)^{\widehat{\lambda},\widehat{\mu}} \bigr),
\]
and Theorem~\ref{thm:comp-loc} says that, under the assumption that the stabilizer of $\ola$, resp.~$\omu$, for the $(\bWf,\bullet)$-action on $\bt^*$ is contained in $\bW_I$, resp.~$\bW_J$, the push/pull functors 
%$R(f_{I,J}^{\lambda,\mu})_*$ and $L(f_{I,J}^{\lambda,\mu})^*$ 
associated with this morphism restrict to quasi-inverse equivalences of categories
\begin{gather*}
L(f_{I,J}^{\lambda,\mu})^* : \Db \Modfg((\cU \bg \otimes \cU\bg^\op)^{\widehat{\lambda},\widehat{\mu}}) \simto \Db \Modc(\tsD_{I,J}^{\widehat{\lambda},\widehat{\mu}}), \\
R(f_{I,J}^{\lambda,\mu})_* : \Db \Modc(\tsD_{I,J}^{\widehat{\lambda},\widehat{\mu}}) \simto \Db \Modfg((\cU \bg \otimes \cU\bg^\op)^{\widehat{\lambda},\widehat{\mu}}).
\end{gather*}
%such that the adjunction morphism $\id \to R(f_{I,J}^{\lambda,\mu})_* \circ L(f_{I,J}^{\lambda,\mu})^*$ is an isomorphism.

\begin{rmk}
This statement is not exactly an application of Theorem~\ref{thm:comp-loc} for $\bG \times \bG$, since we work with the algebras $\tsD_I \boxtimes \tsD_J^\op$ and $\cU\bg \otimes \cU\bg^\op$ instead of the algebras $\tsD_I \boxtimes \tsD_J$ and $\cU\bg \otimes \cU\bg$. However, as in Remark~\ref{rmk:translation-DI}, a slight variant of~\cite[Lemma~3.0.6]{bmr2} provides an isomorphism between $\tsD_J^\op$ and the pullback of $\tsD_J$ under the automorphism of $\tbg_J^{(1)} \times_{\bt^{*(1)}/\bW_J} \bt^*/(\bW_J,\bullet)$ given by
\[
(g\bP_J^{(1)},\xi ; \kappa_I(\eta)) \mapsto (g\bP_J^{(1)},-\xi,\kappa_I(-\eta-2\overline{\rho}))
\]
for $g \in \bG^{(1)}$, $\xi \in \bg^{*(1)}$ and $\eta \in \bt^*$. In particular, there exists an equivalence of categories, compatible with the global sections functors, between $\Modc(\tsD_{I,J}^{\widehat{\lambda},\widehat{\mu}})$ and the category of coherent modules for the pullback of $\tsD_I \boxtimes \tsD_J$ to $X_{I,J}^{\widehat{\lambda}, \widehat{-\mu-2\rho}}$. Since $\mathrm{Stab}_{(\bWf,\bullet)}(\omu) = \mathrm{Stab}_{(\bWf,\bullet)}(\overline{-\mu-2\rho})$, this justifies our claim above regarding the functors $L(f_{I,J}^{\lambda,\mu})^*$ and $R(f_{I,J}^{\lambda,\mu})_*$.
\end{rmk}

Below we will assume that $\lambda$ and $\mu$ satisfy
\begin{equation}
\label{eqn:conditions-la-mu}
\mathrm{Stab}_{(\bWaff,\bullet)}(\lambda) = \bW_I, \quad
\mathrm{Stab}_{(\bWaff,\bullet)}(\mu) = \bW_J
\end{equation}
(where we consider the stabilizers for the restriction of the dot-action of $\bW$ on $X^*(\bT)$ to $\bWaff$).
In view of~\cite[Lemma~3.1(ii)]{br-Hecke}, these conditions imply that $\mathrm{Stab}_{(\bWf,\bullet)}(\ola) = \bW_I$ and $\mathrm{Stab}_{(\bWf,\bullet)}(\omu) = \bW_J$, so that the conditions considered above are satisfied.

\begin{rmk}
\phantomsection
\label{rmk:assumption-weights-loc}
\begin{enumerate}
\item
%Recall the notation of~\S\ref{ss:BSHC}. 
If $\lambda,\mu$ belong to the lower closure of the fundamental alcove (see~\S\ref{ss:BSHC}) and if $I=\{\alpha \in \fRs \mid \langle \lambda+\varsigma, \alpha^\vee \rangle=0\}$, $J=\{\alpha \in \fRs \mid \langle \mu+\varsigma, \alpha^\vee \rangle=0\}$, then the conditions~\eqref{eqn:conditions-la-mu} are satisfied, see~\cite[\S II.6.3]{jantzen}.
\item
\label{it:assumption-weights-loc-2}
Recall the extended affine Weyl group $\bW$ introduced in~\S\ref{ss:central-characters}.
From the proof of~\cite[Lemma~3.1(ii)]{br-Hecke} one sees that for any $\lambda \in X^*(\bT)$ we have $\mathrm{Stab}_{(\bW,\bullet)}(\lambda) \subset \bWaff$, hence $\mathrm{Stab}_{(\bWaff,\bullet)}(\lambda)=\mathrm{Stab}_{(\bW,\bullet)}(\lambda)$. In particular, if $\mathrm{Stab}_{(\bWaff,\bullet)}(\lambda)=\bW_I$ and $w \in \bW$ commutes with $\bW_I$ then $\mathrm{Stab}_{(\bWaff,\bullet)}(w \bullet \lambda) = \bW_I$.
\end{enumerate}
\end{rmk}

%belong to the lower closure of the fundamental alcove, and that
%\[
%I=\{\alpha \in \fRs \mid \langle \lambda+\varsigma, \alpha^\vee \rangle=0\}, \quad J=\{\alpha \in \fRs \mid \langle \mu+\varsigma, \alpha^\vee \rangle=0\}.
%\]
%These conditions imply that the assumption on the stabilizers of $\ola$ and $\omu$ considered above is satisfied, in view of~\cite[Lemma~3.1(ii)]{br-Hecke} and~\cite[\S II.6.3]{jantzen}.

%induce quasi-inverse equivalences of categories
%\[
%\Db \Modc(\tsD_{I,J}^{\widehat{\lambda},\widehat{\mu}}) \cong \Db \Modfg((\cU \bg \otimes \cU\bg^\op)^{\widehat{\lambda},\widehat{\mu}}),
%\]
%under the assumption that the stabilizer of $\ola$, resp.~$\omu$, for the $(\bWf,\bullet)$-action on $\bt^*$ is contained in $\bW_I$, resp.~$\bW_J$. (Note that, under the assumption of Theorem~\ref{thm:localization-HC}, the stabilizer of $\ola$, resp.~$\omu$, is \emph{equal} to $\bW_I$, resp.~$\bW_J$, by~\cite[Lemma~3.1(ii)]{br-Hecke} and~\cite[\S II.6.3]{jantzen}.)

We have canonical forgetful functors
\begin{equation}
\label{eqn:forget-HC-loc-thm}
\Modqc(\osD_{I,J}^{\widehat{\lambda}, \widehat{\mu}}, \bG) \to \Modqc(\tsD_{I,J}^{\widehat{\lambda},\widehat{\mu}}), \quad 
\widetilde{\HC}^{\widehat{\lambda}, \widehat{\mu}} \to \Mod((\cU \bg \otimes \cU\bg^\op)^{\widehat{\lambda},\widehat{\mu}}).
\end{equation}
The following claims will be proved in~\S\ref{ss:proof-acyclicity-diag-induced} 
%(for~\eqref{it:acyclicity-HC}) 
and~\S\ref{ss:proof-injectives-acyclic} 
%(for~\eqref{it:acyclicity-D}) 
below, respectively.

\begin{prop}
\label{prop:acyclicity}
%Let $\lambda,\mu,I,J$ be as in Theorem~\ref{thm:localization-HC}.
Let $\lambda, \mu \in X^*(\bT)$ and $I,J \subset \fRs$, and assume that the conditions~\eqref{eqn:conditions-la-mu} are satisfied.
%\[
%\mathrm{Stab}_{(\bWaff,\bullet)}(\lambda) = \bW_I, \quad
%\mathrm{Stab}_{(\bWaff,\bullet)}(\mu) = \bW_J.
%\]

%which belong to the lower closure of the fundamental alcove, and assume that
%\[
%I=\{\alpha \in \fRs \mid \langle \lambda+\varsigma, \alpha^\vee \rangle=0\}, \quad J=\{\alpha \in \fRs \mid \langle \mu+\varsigma, \alpha^\vee \rangle=0\}.
%\]
\begin{enumerate}
\item
\label{it:acyclicity-D}
The category $\Modqc(\osD_{I,J}^{\widehat{\lambda}, \widehat{\mu}}, \bG)$ has enough injective objects. Moreover, the image in $\Modqc(\tsD_{I,J}^{\widehat{\lambda},\widehat{\mu}})$ of any injective object is acyclic for the functor $(f_{I,J}^{\lambda,\mu})_*$.
\item
\label{it:acyclicity-HC}
For any $V \in \Rep^\infty(\bG)$
%\footnote{Allow infinite-dimensional reps?} 
the image in $\Mod((\cU \bg \otimes \cU\bg^\op)^{\widehat{\lambda},\widehat{\mu}})$ of $\sfC^{\widehat{\lambda},\widehat{\mu}}(V \otimes \cU\bg)$ is acyclic for the functor $(f_{I,J}^{\lambda,\mu})^*$.
\end{enumerate}
\end{prop}

%For now, we explain why Proposition~\ref{prop:acyclicity} allows to prove Theorem~\ref{thm:localization-HC}.

The morphism $f_{I,J}^{\lambda,\mu}$ induces a morphism of ringed spaces $\overline{f}_{I,J}^{\lambda,\mu}$ from
\[
\left( \St_{I,J}^{(1)} \times_{\bt^{*(1)} / \bW_I \times_{\bt^{*(1)} / \bWf} \bt^{*(1)} / \bW_J} 
%\FN_{\bZ_{I,J}}((\tla_I,\tmu_J))
\bZ^{\widehat{\lambda}, \widehat{\mu}}_{I,J}, \osD_{I,J}^{\widehat{\lambda}, \widehat{\mu}} \right)
\]
to
\[
\left( \bg^{*(1)} \times_{\bt^{*(1)}/\bWf} \bZ^{\widehat{\lambda}, \widehat{\mu}}, \sfU^{\widehat{\lambda}, \widehat{\mu}} \right),
\]
where once again we still denote by $\sfU^{\widehat{\lambda}, \widehat{\mu}}$ the sheaf of algebras on $\bg^{*(1)} \times_{\bt^{*(1)}/\bWf} \bZ^{\widehat{\lambda}, \widehat{\mu}}$ associated with this $\scO(\bg^{*(1)} \times_{\bt^{*(1)}/\bWf} \bZ^{\widehat{\lambda}, \widehat{\mu}})$-algebra. We have push/pull functors associated with this morphism, which induce adjoint functors
\[
(\overline{f}_{I,J}^{\lambda,\mu})^* : \widetilde{\HC}^{\widehat{\lambda}, \widehat{\mu}} \to \Modqc(\osD_{I,J}^{\widehat{\lambda}, \widehat{\mu}}, \bG), \quad
(\overline{f}_{I,J}^{\lambda,\mu})_* : \Modqc(\osD_{I,J}^{\widehat{\lambda}, \widehat{\mu}}, \bG) \to \widetilde{\HC}^{\widehat{\lambda}, \widehat{\mu}}.
\]
These functors are compatible with the functors $(f_{I,J}^{\lambda,\mu})^*$ and $(f_{I,J}^{\lambda,\mu})_*$ via the functors~\eqref{eqn:forget-HC-loc-thm}.

Since the category $\Modqc(\osD_{I,J}^{\widehat{\lambda}, \widehat{\mu}}, \bG)$ has enough injectives (see Proposi\-tion~\ref{prop:acyclicity}\eqref{it:acyclicity-D}) we can consider the right derived functor
\[
R(\overline{f}_{I,J}^{\lambda,\mu})_* : D^+ \Modqc(\osD_{I,J}^{\widehat{\lambda}, \widehat{\mu}}, \bG) \to D^+ \widetilde{\HC}^{\widehat{\lambda}, \widehat{\mu}}.
\]
On the other hand, Lemma~\ref{lem:surjection-diag-ind-comp} and Proposi\-tion~\ref{prop:acyclicity}\eqref{it:acyclicity-HC} imply that the category $\widetilde{\HC}^{\widehat{\lambda}, \widehat{\mu}}$ has enough objects whose images are acyclic for the functor $(f_{I,J}^{\lambda,\mu})^*$. Hence there is a left derived functor
\[
L(\overline{f}_{I,J}^{\lambda,\mu})^* : D^- \widetilde{\HC}^{\widehat{\lambda}, \widehat{\mu}} \to D^- \Modqc(\osD_{I,J}^{\widehat{\lambda}, \widehat{\mu}}, \bG)
\]
such that the diagram
\begin{equation}
\label{eqn:pullback-diagram}
\vcenter{
\xymatrix@C=1.5cm@R=0.5cm{
D^- \widetilde{\HC}^{\widehat{\lambda}, \widehat{\mu}} \ar[r]^-{L(\overline{f}_{I,J}^{\lambda,\mu})^*} \ar[d] & D^- \Modqc(\osD_{I,J}^{\widehat{\lambda}, \widehat{\mu}}, \bG) \ar[d] \\
D^- \Mod((\cU \bg \otimes \cU\bg^\op)^{\widehat{\lambda},\widehat{\mu}}) \ar[r]^-{L(f_{I,J}^{\lambda,\mu})^*} & D^- \Modqc(\tsD_{I,J}^{\widehat{\lambda},\widehat{\mu}})
}
}
\end{equation}
commutes, where the vertical arrows are induced by the functors in~\eqref{eqn:forget-HC-loc-thm}.

Note that we have fully faithful functors
\[
\Db \Modc(\osD_{I,J}^{\widehat{\lambda}, \widehat{\mu}}, \bG) \to \Db \Modqc(\osD_{I,J}^{\widehat{\lambda}, \widehat{\mu}}, \bG)
\]
(see Lemma~\ref{lem:Modc-Modqc}) and
\[
\Db \HC^{\widehat{\lambda}, \widehat{\mu}} \to \Db \widetilde{\HC}^{\widehat{\lambda}, \widehat{\mu}}
\]
(see Lemma~\ref{lem:HC-tHC}), whose essential images consist of complexes with coherent or finitely generated cohomology objects, respectively.
The following statement is the promised localization theorem for completed Harish-Chandra bimodules.

\begin{thm}
\label{thm:localization-HC}
Let $\lambda, \mu \in X^*(\bT)$ and $I,J \subset \fRs$ be such that
\[
\mathrm{Stab}_{(\bWaff,\bullet)}(\lambda) = \bW_I, \quad
\mathrm{Stab}_{(\bWaff,\bullet)}(\mu) = \bW_J.
\]

%which belong to the lower closure of the fundamental alcove, and assume that
%\[
%I=\{\alpha \in \fRs \mid \langle \lambda+\varsigma, \alpha^\vee \rangle=0\}, \quad J=\{\alpha \in \fRs \mid \langle \mu+\varsigma, \alpha^\vee \rangle=0\}.
%\]

\begin{enumerate}
\item
\label{it:loc-HC-1}
The functor $R(\overline{f}_{I,J}^{\lambda,\mu})_*$ restricts to a functor
\begin{equation*}
%\label{eqn:completed-loc-functor-HC-1}
\Gamma_{I,J}^{\lambda,\mu} : \Db \Modc(\osD_{I,J}^{\widehat{\lambda}, \widehat{\mu}}, \bG) \to \Db \HC^{\widehat{\lambda}, \widehat{\mu}}.
\end{equation*}
%still denoted $R(\overline{f}_{I,J}^{\lambda,\mu})_*$.
\item
\label{it:loc-HC-2}
The functor $L(\overline{f}_{I,J}^{\lambda,\mu})^*$ restricts to a functor
\begin{equation*}
%\label{eqn:completed-loc-functor-HC-2}
\cL_{I,J}^{\lambda,\mu} : \Db \HC^{\widehat{\lambda}, \widehat{\mu}} \to \Db \Modc(\osD_{I,J}^{\widehat{\lambda}, \widehat{\mu}}, \bG).
\end{equation*}
%still denoted $L(\overline{f}_{I,J}^{\lambda,\mu})^*$.
\item
\label{it:loc-HC-3}
The functors $\Gamma_{I,J}^{\lambda,\mu}$ and $\cL_{I,J}^{\lambda,\mu}$
%~\eqref{eqn:completed-loc-functor-HC-1} and~\eqref{eqn:completed-loc-functor-HC-2} 
are quasi-inverse equivalences of categories.
%The functor $L(\overline{f}_{I,J}^{\lambda,\mu})^*$ of~\eqref{eqn:completed-loc-functor-HC-2} is left adjoint to the functor $R(\overline{f}_{I,J}^{\lambda,\mu})_*$ of~\eqref{eqn:completed-loc-functor-HC-1}, and moreover the adjunction morphism $\id \to R(\overline{f}_{I,J}^{\lambda,\mu})_* \circ L(\overline{f}_{I,J}^{\lambda,\mu})^*$ is an isomorphism; in particular, $L(\overline{f}_{I,J}^{\lambda,\mu})^*$ is fully faithful.
\end{enumerate}
\end{thm}

%\begin{proof}[Proof of Theorem~\ref{thm:localization-HC}]
\begin{proof}
\eqref{it:loc-HC-1}
The acyclicity statement in Proposi\-tion~\ref{prop:acyclicity}\eqref{it:acyclicity-D} shows that the diagram
\[
\xymatrix@C=1.5cm@R=0.5cm{
D^+ \Modqc(\osD_{I,J}^{\widehat{\lambda}, \widehat{\mu}}, \bG) \ar[d] \ar[r]^-{R(\overline{f}_{I,J}^{\lambda,\mu})_*} & D^+ \widetilde{\HC}^{\widehat{\lambda}, \widehat{\mu}} \ar[d] \\
D^+ \Modqc(\tsD_{I,J}^{\widehat{\lambda},\widehat{\mu}} ) \ar[r]^-{R(f_{I,J}^{\lambda,\mu})_*} & D^+ \Mod((\cU \bg \otimes \cU\bg^\op)^{\widehat{\lambda},\widehat{\mu}})
}
\]
commutes, where the vertical arrows are induced by the functors in~\eqref{eqn:forget-HC-loc-thm}. As explained above, by Theorem~\ref{thm:comp-loc} the lower functor restricts to a functor from $\Db \Modc(\tsD_{I,J}^{\widehat{\lambda},\widehat{\mu}} )$ to $\Db \Modfg((\cU \bg \otimes \cU\bg^\op)^{\widehat{\lambda},\widehat{\mu}})$, which implies that the restriction of $R(\overline{f}_{I,J}^{\lambda,\mu})_*$ to the full subcategory $\Db \Modc(\osD_{I,J}^{\widehat{\lambda}, \widehat{\mu}}, \bG)$
% (see Lemma~\ref{lem:Modc-Modqc}) 
takes values in the full subcategory $\Db \HC^{\widehat{\lambda}, \widehat{\mu}}$.
% (see Lemma~\ref{lem:HC-tHC}). We therefore have a functor
%\[
%R(\overline{f}_{I,J}^{\lambda,\mu})_* : \Db \Modc(\osD_{I,J}^{\widehat{\lambda}, \widehat{\mu}}, \bG) \to \Db \HC^{\widehat{\lambda}, \widehat{\mu}}.
%\]

\eqref{it:loc-HC-2}
The proof is similar to that of~\eqref{it:loc-HC-1}, using the commutative diagram~\eqref{eqn:pullback-diagram}.

%%Lemma~\ref{lem:surjection-diag-ind-comp} and 
%Proposi\-tion~\ref{prop:acyclicity}\eqref{it:acyclicity-HC} implies that 
%%the category $\HC^{\widehat{\lambda}, \widehat{\mu}}$ has enough acyclic objects for the functor $(\overline{f}_{I,J}^{\lambda,\mu})^*$. Hence there is a left derived functor
%%\[
%%L(\overline{f}_{I,J}^{\lambda,\mu})^* : D^- \HC^{\widehat{\lambda}, \widehat{\mu}} \to D^- \Modc(\osD_{I,J}^{\widehat{\lambda}, \widehat{\mu}}, \bG),
%%\]
%%and moreover 
%the diagram
%\[
%\xymatrix@C=1.5cm{
%D^- \widetilde{\HC}^{\widehat{\lambda}, \widehat{\mu}} \ar[r]^-{L(\overline{f}_{I,J}^{\lambda,\mu})^*} \ar[d] & D^- \Modqc(\osD_{I,J}^{\widehat{\lambda}, \widehat{\mu}}, \bG) \ar[d] \\
%D^- \Mod((\cU \bg \otimes \cU\bg^\op)^{\widehat{\lambda},\widehat{\mu}}) \ar[r]^-{L(f_{I,J}^{\lambda,\mu})^*} & D^- \Modqc(\tsD_{I,J}^{\widehat{\lambda},\widehat{\mu}})
%}
%\]
%commutes, where the vertical arrows are induced by the functors in~\eqref{eqn:forget-HC-loc-thm}. By Theorem~\ref{thm:comp-loc} the lower functor restricts to a functor from $\Db \Modfg((\cU \bg \otimes \cU\bg^\op)^{\widehat{\lambda},\widehat{\mu}})$ to $\Db \Modc(\tsD_{I,J}^{\widehat{\lambda},\widehat{\mu}})$, which implies the desired claim.
%between bounded derived categories; we deduce that $L(\overline{f}_{I,J}^{\lambda,\mu})^*$ restricts to a functor
%\[
%L(\overline{f}_{I,J}^{\lambda,\mu})^* : \Db \HC^{\widehat{\lambda}, \widehat{\mu}} \to \Db \Modc(\osD_{I,J}^{\widehat{\lambda}, \widehat{\mu}}, \bG).
%\]

\eqref{it:loc-HC-3}
By general properties of derived functors (see~\cite[\href{https://stacks.math.columbia.edu/tag/0DVC}{Tag 0DVC}]{stacks-project}) the functor $\cL_{I,J}^{\lambda,\mu}$
%$L(\overline{f}_{I,J}^{\lambda,\mu})^*$ of~\eqref{eqn:completed-loc-functor-HC-2} 
is left adjoint to $\Gamma_{I,J}^{\lambda,\mu}$.
%the functor $R(\overline{f}_{I,J}^{\lambda,\mu})_*$ of~\eqref{eqn:completed-loc-functor-HC-1}. 
We therefore have canonical morphisms of functors
%\[
%L(\overline{f}_{I,J}^{\lambda,\mu})^* \circ R(\overline{f}_{I,J}^{\lambda,\mu})_* \to \id, \qquad 
%$\id \to R(\overline{f}_{I,J}^{\lambda,\mu})_* \circ L(\overline{f}_{I,J}^{\lambda,\mu})^*$ and $L(\overline{f}_{I,J}^{\lambda,\mu})^* \circ R(\overline{f}_{I,J}^{\lambda,\mu})_* \to \id$ 
$\id \to \Gamma_{I,J}^{\lambda,\mu} \circ \cL_{I,J}^{\lambda,\mu}$ and $\cL_{I,J}^{\lambda,\mu} \circ \Gamma_{I,J}^{\lambda,\mu} \to \id$ 
which are related to the similar adjunction morphisms for $f_{I,J}^{\lambda,\mu}$
by the functors induced by~\eqref{eqn:forget-HC-loc-thm}.
% to the similar morphisms for the functors $L(f_{I,J}^{\lambda,\mu})^*$ and $R(\overline{f}_{I,J}^{\lambda,\mu})_*$. 
 Using the fact that the latter morphisms are isomorphisms (as explained above) one easily deduces that the former morphisms are also isomorphisms, which completes the proof.
 % is an isomorphism (by Theorem~\ref{thm:comp-loc}), one easily deduces that the former is also an isomorphism. As for Theorem~\ref{thm:comp-loc}, the fully faithfulness of $L(\overline{f}_{I,J}^{\lambda,\mu})^*$ is an easy consequence of this fact.
\end{proof}

%---------------------------------------------------------------
\subsection{Proof of Proposition~\ref{prop:acyclicity}\eqref{it:acyclicity-HC}}
\label{ss:proof-acyclicity-diag-induced}
%---------------------------------------------------------------

In this subsection we explain the proof of Proposition~\ref{prop:acyclicity}\eqref{it:acyclicity-HC}.
We therefore fix $\lambda,\mu$ and $V$ as in this statement. If we set
\[
\tD_{I,J}^{\widehat{\lambda},\widehat{\mu}} = \scO(\FN_{\bt^*/(\bW_I,\bullet) \times \bt^*/(\bW_J,\bullet)}( \{ (\tla_I,\tmu_J) \} )) \otimes_{\scO(\bt^*/(\bW_I,\bullet) \times \bt^*/(\bW_J,\bullet))} \tD_{I,J},
\]
then as in~\eqref{eqn:equivalences-Dmod-xi} we have an equivalence of categories
\[
\Modqc(\tsD_{I,J}^{\widehat{\lambda},\widehat{\mu}}) \cong \Modqc(\tD_{I,J}^{\widehat{\lambda},\widehat{\mu}}),
\]
and under this identification we have
\[
L(f_{I,J}^{\lambda,\mu})^* \sfC^{\widehat{\lambda},\widehat{\mu}}(V \otimes \cU\bg) \cong \tD_{I,J}^{\widehat{\lambda},\widehat{\mu}} \lotimes_{(\cU \bg \otimes \cU\bg^\op)^{\widehat{\lambda},\widehat{\mu}}} \sfC^{\widehat{\lambda},\widehat{\mu}}(V \otimes \cU\bg).
\]
As in~\eqref{eqn:isom-FN}, our assumptions on $\lambda,\mu$ imply that the natural morphism
\[
\FN_{\bt^*/(\bW_I,\bullet) \times \bt^*/(\bW_J,\bullet)}( \{ (\tla_I,\tmu_J) \} ) \to \FN_{\bt^*/(\bWf,\bullet) \times \bt^*/(\bWf,\bullet)}( \{ (\tla,\tmu) \} )
\]
is an isomorphism, and moreover by exactness of completion we have a canonical isomorphism
\begin{multline*}
\sfC^{\widehat{\lambda},\widehat{\mu}}(V \otimes \cU\bg) \cong \\
\scO(\FN_{\bt^*/(\bWf,\bullet) \times \bt^*/(\bWf,\bullet)}( \{ (\tla,\tmu) \} )) \otimes_{\scO(\bt^*/(\bWf,\bullet) \times \bt^*/(\bWf,\bullet))} (V \otimes \cU \bg).
\end{multline*}
We deduce an isomorphism
\begin{multline*}
\tD_{I,J}^{\widehat{\lambda},\widehat{\mu}} \lotimes_{(\cU \bg \otimes \cU\bg^\op)^{\widehat{\lambda},\widehat{\mu}}} \sfC^{\widehat{\lambda},\widehat{\mu}}(V \otimes \cU\bg) \cong \\
\scO(\FN_{\bt^*/(\bW_I,\bullet) \times \bt^*/(\bW_J,\bullet)}( \{ (\tla_I,\tmu_J) \} )) \otimes_{\scO(\bt^*/(\bW_I,\bullet) \times \bt^*/(\bW_J,\bullet))} \\
\bigl( \tD_{I,J} \lotimes_{\cU\bg \otimes \cU\bg^\op} (V \otimes \cU\bg) \bigr),
\end{multline*}
which shows that to conclude it suffices to prove that the complex
\[
\tD_{I,J} \lotimes_{\cU\bg \otimes \cU\bg^\op} (V \otimes \cU\bg)
\]
is concentrated in degree $0$. Now, recall the isomorphism
\[
\bigl( \cU\bg \otimes \cU\bg^\op \bigr) \otimes_{\cU\bg} V \simto V \otimes \cU\bg
\]
from~\cite[\S 3.4]{br-Hecke}, where the morphism $\cU\bg \to \cU\bg \otimes \cU\bg^\op$ is the antidiagonal embedding. For this morphism, $\cU\bg \otimes \cU\bg^\op$ is free over $\cU\bg$; we deduce an isomorphism
\[
\tD_{I,J} \lotimes_{\cU\bg \otimes \cU\bg^\op} (V \otimes \cU\bg) \cong \tD_{I,J} \lotimes_{\cU\bg} V.
\]
To conclude, we therefore only have to show that the right-hand side is concentrated in degree $0$.

Consider the Chevalley--Eilenberg resolution
\[
\cU\bg \otimes \bigwedge \hspace{-3pt} {}^\bullet \bg \to \bk
\]
where $\wedge \hspace{-1pt} {}^a \bg$ is in degree $-a$. Tensoring with $V$ we obtain a resolution
\[
\cU\bg \otimes_\bk \bigwedge \hspace{-3pt} {}^\bullet \bg \otimes V \to V
\]
where $\cU\bg$ acts diagonally on $\cU\bg$ and $V$. Writing $V_{\mathrm{triv}}$ for the vector space $V$ endowed with the trivial action of $\cU\bg$, we have a natural isomorphism of $\cU\bg$-modules $\cU\bg \otimes V \cong \cU\bg \otimes V_{\mathrm{triv}}$, which provides a resolution
\[
\cU\bg \otimes \bigwedge \hspace{-3pt} {}^\bullet \bg \otimes V_{\mathrm{triv}} \to V
\]
where now $\cU\bg$ acts only on the first factor (but the differential involves the action on $V$).
Computing with this resolution, we obtain that $\tD_{I,J} \lotimes_{\cU\bg} V$ is the image of the complex
\begin{equation}
\label{eqn:complex-Chevalley-Eilenberg}
\tD_{I,J} \otimes \bigwedge \hspace{-3pt} {}^\bullet \bg \otimes V,
\end{equation}
for a certain Chevalley--Eilenberg type differential. The sheaf of algebras $\tD_{I,J}$ is by definition the enveloping algebra of a Lie algebroid (see~\cite[\S 1.2.1]{bmr2}); it therefore admits a canonical filtration $(F_{\leq n} (\tD_{I,J}) : n \in \Z_{\geq 0})$, whose associated graded is the pushforward of the structure sheaf of $\tbg_I \times \tbg_J$ to $\bG/\bP_I \times \bG/\bP_J$. We consider the filtration of the complex~\eqref{eqn:complex-Chevalley-Eilenberg} such that
\[
F_{\leq n}^p \left( \tD_{I,J} \otimes \bigwedge \hspace{-3pt} {}^\bullet \bg \otimes V \right) = F_{\leq n+p} (\tD_{I,J}) \otimes \bigwedge \hspace{-3pt} {}^{-p} \bg \otimes V.
\]
Below we will prove that
the associated graded of this filtered complex has cohomology only in degree $0$, which will conclude the proof.

The associated graded of our complex is of the form
\[
\scO_{\tbg_I \times \tbg_J} \otimes \bigwedge \hspace{-3pt} {}^{\bullet} \bg \otimes V,
\]
where we omit the pushforward functor, and where the differential only involves the first two factors; in fact if $e_1, \cdots, e_d$ is a basis of $\bg$, then $\scO_{\tbg_I \times \tbg_J} \otimes \bigwedge \hspace{-3pt} {}^{\bullet} \bg$ with this differential identifies with the Koszul complex on the images of $e_1, \cdots, e_d$ under the morphism
\[
\scO(\bg^*) \to \scO(\bg^* \times \bg^*) \to \scO(\tbg_I \times \tbg_J)
\]
(where the first morphism is induced by the map $\bg^* \times \bg^* \to \bg^*$ given by $(\xi,\eta) \mapsto \xi-\eta$) in the sense of~\cite[\href{https://stacks.math.columbia.edu/tag/062L}{Tag 062L}]{stacks-project}. Since a regular sequence is Koszul-regular (see~\cite[\href{https://stacks.math.columbia.edu/tag/062F}{Tag 062F}]{stacks-project}), to prove that this complex has cohomology only in degree $0$ it suffices to show that these images form a regular sequence in sections over an affine cover of $\tbg_I \times \tbg_J$. By the standard characterization of regular sequences in terms of dimension of the quotient (see~\cite[\href{https://stacks.math.columbia.edu/tag/02JN}{Tag 02JN}]{stacks-project}), this fact follows from the standard observation that $\dim(\tbg_I \times_{\bg^*} \tbg_J)=\dim(\bg) = \dim(\tbg_I \times \tbg_J) - \dim(\bg)$.

\begin{rmk}
\label{rmk:derived-fiber-prod}
The last step of this proof can be equivalently formulated as saying that the fiber product in the definition of $\St_{I,J}$ is also a \emph{derived} fiber product, i.e.~a fiber product in the sense of Derived Algebraic Geometry.
\end{rmk}

%this complex computes
%\[
%\left( \scO_{\tbg_I \times \tbg_J} \lotimes_{\scO(\bg^* \times \bg^*)} \scO(\Delta \bg^*) \right) \otimes_\bk V,
%\]
%where $\Delta \bg^* \subset \bg^* \times \bg^*$ is the diagonal copy.

%As explained in~\S\ref{ss:comp-loc-statement}, the conditions on this weights imply that the natural morphism
%\[
%\bZ_{I,J}^{\widehat{\lambda}, \widehat{\mu}} \to \bZ^{\widehat{\lambda}, \widehat{\mu}}
%\]
%is an isomorphism.

%---------------------------------------------------------------
\subsection{Variant for a fixed central character}
\label{ss:variant-fixed-central-char}
%---------------------------------------------------------------

Recall the constructions of~\S\ref{ss:HC-fixed-character}.
Given $I,J \subset \fRs$, we consider
\[
\St'_{I,J} := \St_{I,J} \times_{\bt^*/\bW_I \times_{\bt^*/\bWf} \bt^*/\bW_J} (\bt^*/\bW_I \times_{\bt^*/\bWf} \{0\}) = \St_{I,J} \times_{\bt^*/\bW_J} \{0\},
\]
where the second fiber product is taken with respect to the composition $\St_{I,J} \to \tbg_J \to \bt^*/\bW_J$.
The following statement is an analogue of Lemma~\ref{lem:steinberg-completion}, and follows from the same arguments. (In that statement, the formal neighborhood that appears is a finite $\bk$-scheme.)

\begin{lem}
\label{lem:steinberg-completion-prime}
Let $\lambda,\mu \in X^*(\bT)$ and $I,J \subset \fRs$,
 and assume that we have
 $\bW_I \subset \mathrm{Stab}_{(\bWf,\bullet)}(\ola)$.
 %, \quad \bW_J \subset \mathrm{Stab}_{(\bWf,\bullet)}(\omu).
Then the projection morphism
\[
\St_{I,J}^{\prime (1)} \times_{\bt^{*(1)} / \bW_I \times_{\bt^{*(1)}/\bWf} \{0\} } (\bt^*/(\bW_I,\bullet) \times_{\bt^{*(1)}/\bWf} \{\tmu_J\}) \to \St_{I,J}^{\prime (1)}
\]
induces an
isomorphism of schemes
\[
\St^{\prime(1)}_{I,J} \times_{\bt^{*(1)}/\bW_I \times_{\bt^{*(1)}/\bWf} \{0\}} \FN_{\bt^*/(\bW_I,\bullet) \times_{\bt^{*(1)}/\bWf} \{\tmu_J\}}( \{ (\tla_I,\tmu_J) \} )
\simto \St^{\prime(1)}_{I,J}.
\]
\end{lem}

Let us consider the pullback
\[
\osD^{\hla,\mu}_{I,J}
\]
of $\osD_{I,J}$ to
\begin{equation}
\label{eqn:scheme-oD-hla-mu}
\St^{\prime (1)}_{I,J} \times_{\bt^{*(1)}/\bW_I \times_{\bt^{*(1)}/\bWf} \{0\}} \FN_{\bt^*/(\bW_I,\bullet) \times_{\bt^{*(1)}/\bWf} \{\tmu_J\}}( \{ (\tla_I,\tmu_J) \} )
\end{equation}
(a scheme of finite type over $\bk$).
We have a canonical action of $\bG$ on this scheme, and $\osD_{I,J}^{\hla,\mu}$ has a canonical $\bG$-equivariant structure. We also have a morphism $\cU\bg \to \Gamma(\osD^{\hla,\mu}_{I,J})$ induced by the diagonal action of $\bG$ on $\bG/\bU_I \times \bG/\bU_J$, hence we can consider the categories
\[
\Modc^{\bG}(\osD^{\hla,\mu}_{I,J}) \subset \Modqc^{\bG}(\osD^{\hla,\mu}_{I,J}), \quad \Modc(\osD^{\hla,\mu}_{I,J}, \bG) \subset \Modqc(\osD^{\hla,\mu}_{I,J}, \bG).
\]

The scheme~\eqref{eqn:scheme-oD-hla-mu} identifies with a closed subscheme of~\eqref{eqn:scheme-oD-hla-hmu}, in such a way that the restriction of $\osD^{\hla,\hmu}_{I,J}$ identifies with $\osD^{\hla,\mu}_{I,J}$. We therefore have a natural pushforward functor
\[
\Modqc^{\bG}(\osD^{\hla,\mu}_{I,J}) \to \Modqc^{\bG}(\osD^{\hla,\hmu}_{I,J})
\]
which restricts to an exact functor
\begin{equation}
\label{eqn:functor-D-hla-mu-hmu}
\Modc(\osD^{\hla,\mu}_{I,J}, \bG) \to \Modc(\osD^{\hla,\hmu}_{I,J}, \bG).
\end{equation}

The proof of the following theorem is similar to that of Theorem~\ref{thm:localization-HC}; details are left to the reader. (In this case one does not need the considerations of Section~\ref{sec:compl-loc}; the results of~\cite{bmr,bmr2} can be used instead.)

\begin{thm}
\label{thm:localization-HC-fixed}
Let $\lambda, \mu \in X^*(\bT)$ and $I,J \subset \fRs$ be such that
\[
\mathrm{Stab}_{(\bWaff,\bullet)}(\lambda) = \bW_I, \quad
\mathrm{Stab}_{(\bWaff,\bullet)}(\mu) = \bW_J.
\]
%Let $\lambda, \mu \in X^*(\bT)$ which belong to the lower closure of the fundamental alcove, and assume that
%\[
%I=\{\alpha \in \fRs \mid \langle \lambda+\varsigma, \alpha^\vee \rangle=0\}, \quad J=\{\alpha \in \fRs \mid \langle \mu+\varsigma, \alpha^\vee \rangle=0\}.
%\]
Global sections induce an equivalence of triangulated categories
\[
\Db \Modc(\osD^{\hla,\mu}_{I,J}, \bG) \simto \Db \HC^{\hla,\mu}
\]
such that the following diagram commutes:
\[
\xymatrix@C=1.5cm@R=0.6cm{
\Db \Modc(\osD^{\hla,\mu}_{I,J}, \bG) \ar[r]^-{\sim} \ar[d]_-{\eqref{eqn:functor-D-hla-mu-hmu}} & \Db \HC^{\hla,\mu} \ar[d]^-{\eqref{eqn:functor-HC-hla-mu-hmu}} \\
\Db \Modc(\osD_{I,J}^{\widehat{\lambda}, \widehat{\mu}}, \bG) \ar[r]_-{\sim}^-{\Gamma_{I,J}^{\lambda,\mu}} & \Db \HC^{\widehat{\lambda}, \widehat{\mu}}.
}
\]
\end{thm}

\section{Splitting}
\label{sec:splitting}
%%%%%%%%%%%%

%----------------------------------------------
\subsection{Azumaya property of \texorpdfstring{$\cU\bg$}{Ug}}
%----------------------------------------------

Recall that an element $\xi \in \bt^*$ is called \emph{unramified} if for any $\alpha \in \fR$ we have $\langle \xi + \overline{\varsigma}, \alpha^\vee \rangle \notin \mathbb{F}_\ell \smallsetminus \{0\}$, where $\mathbb{F}_\ell$ is seen as the prime subfield of $\bk$. We will denote by $\bt^*_{\mathrm{unr}} \subset \bt^*$ the affine open subscheme of unramified elements (which is the complement of a union of finitely many affine hyperplanes). It is clear that this open subscheme is stable under the dot-action of $\bWf$; in fact there exists an affine open subscheme of $\bt^* / (\bWf, \bullet)$ whose preimage in $\bt^*$ is $\bt^*_{\mathrm{unr}}$, namely the complement of the closed subscheme defined by the element
\[
\prod_{\substack{\alpha \in \fR \\ i \in \mathbb{F}_\ell \smallsetminus \{0\}}} ( \langle (-) + \overline{\varsigma}, \alpha^\vee \rangle - i) \quad \in \scO(\bt^*)^{(\bWf, \bullet)}. 
\]
This open subscheme therefore identifies with $\bt^*_{\mathrm{unr}} / (\bWf, \bullet)$. We can then consider the open affine subscheme
\[
\bg^{*(1)} \times_{\bt^{*(1)}/\bWf} \bt^*_{\mathrm{unr}}/(\bWf,\bullet) \subset \bg^{*(1)} \times_{\bt^{*(1)}/\bWf} \bt^*/(\bWf,\bullet).
\]
We set
\[
\cZ_{\mathrm{unr}} := \scO(\bg^{*(1)} \times_{\bt^{*(1)}/\bWf} \bt^*_{\mathrm{unr}}/(\bWf,\bullet)),
\]
a localization of $\cZ = \scO(\bg^{*(1)} \times_{\bt^{*(1)}/\bWf} \bt^*/(\bWf,\bullet))$,
and
\[
(\cU\bg)_{\mathrm{unr}} := \cZ_{\mathrm{unr}} \otimes_{\cZ} \cU\bg.
\]

The following statement is due to Brown--Gordon~\cite{brown-gordon}. (We recall some of the details of its proof below for the reader's convenience. For a quick reminder on Azumaya algebras, see~\cite[\S 4.1]{br-Hecke}.)
% It involves the notion of Azumaya algebra over a commutative ring; for references on this subject we refer to~\cite[\S 4.1]{br-Hecke}.

\begin{prop}
\label{prop:Ug-Azumaya}
The $\cZ_{\mathrm{unr}}$-algebra $(\cU\bg)_{\mathrm{unr}}$
is an Azumaya algebra.
\end{prop}

\begin{proof}
In view of the second characterization of Azumaya algebras recalled in~\cite[\S 4.1]{br-Hecke}, to prove the proposition it suffices to prove that for any $\chi \in \bg^{*(1)}$ and any $\eta \in \bt^*_{\mathrm{unr}}/(\bWf,\bullet)$ whose images in $\bt^{*(1)}/\bWf$ coincide the algebra
\[
\bk_{(\chi,\eta)} \otimes_{\scO(\bg^{*(1)} \times_{\bt^{*(1)}/\bWf} \bt^*/(\bWf,\bullet))} \cU\bg
\]
is a simple algebra. (Here, $\bk_{(\chi,\eta)}$ denotes the $1$-dimensional $\cZ$-module attached to the pair $(\chi,\lambda)$. Note that this simple algebra will automatically be a central simple algebra by the Wedderburn--Artin theorem, since $\bk$ is algebraically closed.) However, by~\cite[Theorems~2.5, 2.6 and~3.10]{brown-gordon} the maximal ideal of $\scO(\bg^{*(1)} \times_{\bt^{*(1)}/\bWf} \bt^*/(\bWf,\bullet))$ defined by~$(\chi,\eta)$ belongs to the Azumaya locus of $\cU\bg$, which by definition means that $\bk_{(\chi,\eta)} \otimes_{\scO(\bg^{*(1)} \times_{\bt^{*(1)}/\bWf} \bt^*/(\bWf,\bullet))} \cU\bg$ is a simple algebra.
%
%It is a standard fact that there exists a Borel subgroup $\bB'$ in $\bG$ with unipotent radical $\bU'$ whose Lie algebra $\bu'$ is such that $\chi_{|\bu'{}^{(1)}}=0$, see~\cite[\S 3.2]{brown-gordon} for references. Then $\bB'/\bU'$ identifies canonically with $\bT$, and denoting by $\bb'$ the Lie algebra of $\bB'$ and replacing (if necessary) $\lambda$ by a $(\bWf,\bullet)$-conjugate we can assume that $\chi_{|\bb'{}^{(1)}}$ is the image of $\lambda$. We can then consider the baby Verma module $M_{(\bb',\chi;\lambda)}$, see e.g.~\cite[\S 3.1.4]{bmr}.
\end{proof}

As explained in~\cite[\S 4.1]{br-Hecke}, Proposition~\ref{prop:Ug-Azumaya} implies that $(\cU\bg)_{\mathrm{unr}}$ is 
%faithfully 
projective over $\cZ_{\mathrm{unr}}$, and that the natural morphism
\begin{equation}
\label{eqn:Azumaya-isom-unr}
(\cU\bg)_{\mathrm{unr}}
\otimes_{\cZ_{\mathrm{unr}}}
(\cU\bg)_{\mathrm{unr}}^{\op} \to \End_{\cZ_{\mathrm{unr}}}((\cU\bg)_{\mathrm{unr}})
\end{equation}
is an isomorphism.

Recall now the algebra $\sfU^{\widehat{-\varsigma},\widehat{-\varsigma}}$ from~\S\ref{ss:central-characters}, and its central subalgebra isomorphic to $\scO(\bg^{*(1)} \times_{\bt^{*(1)}/\bWf} \bZ^{\widehat{-\varsigma},\widehat{-\varsigma}})$.
By construction, this algebra acts naturally (via left and right multiplication) on the algebra
\[
(\cU\bg)^{\widehat{-\varsigma}} =
\scO(\FN_{\bt^*/(\bWf,\bullet)}(\{\widetilde{-\varsigma}\})) \otimes_{\scO(\bt^*/(\bWf,\bullet))} \cU\bg
\]
considered in~\S\ref{ss:comp-loc-statement}.
%We also set
%\[
%\cZ^{\widehat{-\varsigma}} := \scO(\FN_{\bt^*/(\bWf,\bullet)}(\widetilde{-\varsigma})) \otimes_{\scO(\bt^*/(\bWf,\bullet))} \cZ
%\]

\begin{lem}
\label{lem:U-rho-Azumaya}
The $\scO(\bg^{*(1)} \times_{\bt^{*(1)}/\bWf} \bZ^{\widehat{-\varsigma},\widehat{-\varsigma}})$-module $(\cU\bg)^{\widehat{-\varsigma}}$ is projective, and the
action morphism
\[
\sfU^{\widehat{-\varsigma},\widehat{-\varsigma}} \to \End_{\scO(\bg^{*(1)} \times_{\bt^{*(1)}/\bWf} \bZ^{\widehat{-\varsigma},\widehat{-\varsigma}})}((\cU\bg)^{\widehat{-\varsigma}})
\]
is an isomorphism.
\end{lem}

\begin{proof}
As seen in the course of the proof of Lemma~\ref{lem:steinberg-completion} (in the special case $I=\fRs$), the morphism $\bt^*/(\bWf,\bullet) \to \bt^{*(1)}/\bWf$ is \'etale at $\widetilde{-\varsigma}$. It follows that the diagonal embedding
\[
\bt^*/(\bWf,\bullet) \to \bt^*/(\bWf,\bullet) \times_{\bt^{*(1)}/\bWf} \bt^*/(\bWf,\bullet) = \bZ
\]
induces an isomorphism
\[
\FN_{\bt^*/(\bWf,\bullet)}(\{ \widetilde{-\varsigma} \}) \simto \bZ^{\widehat{-\varsigma},\widehat{-\varsigma}},
\]
hence an isomorphism
\[
\sfU^{\widehat{-\varsigma},\widehat{-\varsigma}} = \scO(\bZ^{\widehat{-\varsigma},\widehat{-\varsigma}}) \otimes_{\scO(\bZ)} \sfU \simto \scO(\FN_{\bt^*/(\bWf,\bullet)}(\{ \widetilde{-\varsigma} \})) \otimes_{\scO(\bZ)} \sfU,
\]
where the right-hand side identifies with
\[
\scO(\FN_{\bt^*/(\bWf,\bullet)}( \{ \widetilde{-\varsigma} \})) \otimes_{\scO(\bt^*/(\bWf,\bullet))} ( \cU\bg \otimes_{\cZ} (\cU\bg)^{\op} ).
\]
Now since $\widetilde{-\varsigma}$ belongs to the open subscheme $\bt^*_{\mathrm{unr}}/(\bWf,\bullet)$, we have identifications
\[
\FN_{\bt_{\mathrm{unr}}^*/(\bWf,\bullet)}( \{ \widetilde{-\varsigma} \}) \simto \FN_{\bt^*/(\bWf,\bullet)}( \{ \widetilde{-\varsigma} \})
\]
and
\[
(\cU\bg)^{\widehat{-\varsigma}} \simto
\scO(\FN_{\bt^*_{\mathrm{unr}}/(\bWf,\bullet)}( \{ \widetilde{-\varsigma} \})) \otimes_{\scO(\bt_{\mathrm{unr}}^*/(\bWf,\bullet))} (\cU\bg)_{\mathrm{unr}}.
\]
Hence the first claim follows from the fact that $(\cU\bg)_{\mathrm{unr}}$ is projective over $\cZ_{\mathrm{unr}}$, and the second one
%
%On the other hand, we have a natural isomorphism
%\[
%\sfU^{\widehat{-\varsigma},\widehat{-\varsigma}} = \scO(\bZ^{\widehat{-\varsigma},\widehat{-\varsigma}}) \otimes_{\scO(\bZ)} \sfU \simto \scO(\FN_{\bt^*/(\bWf,\bullet)}(\{ \widetilde{-\varsigma} \})) \otimes_{\scO(\bZ)} \sfU,
%\]
%and the right-hand side identifies with
%\[
%\scO(\FN_{\bt_{\mathrm{unr}}^*/(\bWf,\bullet)}( \{ \widetilde{-\varsigma} \})) \otimes_{\scO(\bt_{\mathrm{unr}}^*/(\bWf,\bullet))} \bigl( (\cU\bg)_{\mathrm{unr}}
%\otimes_{\cZ_{\mathrm{unr}}}
%(\cU\bg)_{\mathrm{unr}}^{\op} \bigr).
%\]
%Hence the second claim follows 
from the fact that~\eqref{eqn:Azumaya-isom-unr} is an isomorphism.
\end{proof}

%---------------------------------------------
\subsection{Relation with \texorpdfstring{$\osD_{I,J}^{\widehat{-\varsigma},\widehat{-\varsigma}}$}{D}}
%---------------------------------------------

Let $I,J \subset \fRs$ 
%If $\nu \in X^*(\bT)$ 
be two subsets, and consider the natural morphism
\[
\omega_{I,J} : 
\St_{I,J}^{(1)} \times_{\bt^{*(1)} / \bW_I \times_{\bt^{*(1)}/\bWf} \bt^{*(1)} / \bW_J} \bZ_{I,J} \to \bg^{*(1)} \times_{\bt^{*(1)}/\bWf} \bZ.
\]
This morphism induces a morphism
\begin{equation}
\label{eqn:omegaIJ-completion}
\St_{I,J}^{(1)} \times_{\bt^{*(1)} / \bW_I \times_{\bt^{*(1)} / \bWf} \bt^{*(1)} / \bW_J} \bZ_{I,J}^{\widehat{-\varsigma},\widehat{-\varsigma}} \to \bg^{*(1)} \times_{\bt^{*(1)}/\bWf} \bZ^{\widehat{-\varsigma},\widehat{-\varsigma}},
\end{equation}
which will also be denoted $\omega_{I,J}$. Since $\sfU^{\widehat{-\varsigma},\widehat{-\varsigma}}$ is a finite algebra over the structure sheaf of the affine scheme $\bg^{*(1)} \times_{\bt^{*(1)}/\bWf} \bZ^{\widehat{-\varsigma},\widehat{-\varsigma}}$, it defines a coherent sheaf of algebras on this scheme, which will also be denoted $\sfU^{\widehat{-\varsigma},\widehat{-\varsigma}}$.

\begin{lem}
\label{lem:splitting-DIJ-varsigma}
For any $I,J$,
there exists a canonical isomorphism of coherent sheaves of $\scO_{\St_{I,J}^{(1)} \times_{\bt^{*(1)} / \bW_I \times_{\bt^{*(1)} / \bWf} \bt^{*(1)} / \bW_J} \bZ_{I,J}^{\widehat{-\varsigma},\widehat{-\varsigma}}}$-algebras
\[
\osD_{I,J}^{\widehat{-\varsigma},\widehat{-\varsigma}} \simto (\omega_{I,J})^* \sfU^{\widehat{-\varsigma},\widehat{-\varsigma}}.
\]
\end{lem}

\begin{proof}
Consider the natural morphism
\[
\omega_I : \tbg_I^{(1)} \times_{\bt^{*(1)} / \bW_I} \bt^*/(\bW_I,\bullet) \to \bg^{*(1)} \times_{\bt^{*(1)} / \bWf} \bt^*/(\bWf,\bullet).
\]
Then by~\cite[Proposition~1.2.3(d)]{bmr2} (see also~\cite[Proposition~5.2.1(b)]{bmr} for the case $I=\varnothing$) there exists a canonical isomorphism
\[
\tsD_I{}_{| \tbg_I^{(1)} \times_{\bt^{*(1)} / \bW_I} \bt_{\mathrm{unr}}^*/(\bW_I,\bullet)} \simto (\omega_I)^* (\cU\bg)_{\mathrm{unr}},
\]
where we still denote by $(\cU\bg)_{\mathrm{unr}}$ the coherent sheaf of $\scO_{\cZ_{\mathrm{unr}}}$-algebras defined by $(\cU\bg)_{\mathrm{unr}}$, and by $\omega_I$ the restriction of this morphism to $\tbg_I^{(1)} \times_{\bt^{*(1)} / \bW_I} \bt_{\mathrm{unr}}^*/(\bW_I,\bullet)$. Of course the same isomorphism holds for $J$, and we deduce an isomorphism
\begin{multline*}
\osD_{I,J}{}_{|\St_{I,J}^{(1)} \times_{\bt^{*(1)} / \bW_I \times_{\bt^{*(1)} / \bWf} \bt^{*(1)} / \bW_J} \bt_{\mathrm{unr}}^*/(\bW_I,\bullet) \times_{\bt^{*(1)} / \bWf} \bt_{\mathrm{unr}}^*/(\bW_J,\bullet)} \simto \\
(\omega_{I,J})^* \sfU_{| \bg^{*(1)} \times_{\bt^{*(1)} / \bWf} \bt_{\mathrm{unr}}^*/(\bWf,\bullet) \times_{\bt^{*(1)} / \bWf} \bt_{\mathrm{unr}}^*/(\bWf,\bullet)},
\end{multline*}
where we use the same notational conventions as above. Now, as in the proof of Lemma~\ref{lem:U-rho-Azumaya} we have an identification
\[
\bZ_{I,J}^{\widehat{-\varsigma},\widehat{-\varsigma}}
 = \FN_{\bt_{\mathrm{unr}}^*/(\bW_I,\bullet) \times_{\bt^{*(1)} / \bWf} \bt_{\mathrm{unr}}^*/(\bW_J,\bullet)}(\{(\widetilde{-\varsigma}_I,\widetilde{-\varsigma}_J)\})
\]
and similarly for $\bZ^{\widehat{-\varsigma},\widehat{-\varsigma}}$. Hence the desired isomorphism follows using pullback to $\St_{I,J}^{(1)} \times_{\bt^{*(1)} / \bW_I \times_{\bt^{*(1)} / \bWf} \bt^{*(1)} / \bW_J} \bZ_{I,J}^{\widehat{-\varsigma},\widehat{-\varsigma}}$.
\end{proof}

%---------------------------------------------
\subsection{\texorpdfstring{$\scD$}{D}-modules and coherent sheaves}
\label{ss:D-mod-Coh}
%---------------------------------------------

%Consider some subsets $I,J \subset \fRs$. Let $\lambda,\mu \in X^*(\bT)$, and assume that
%$\langle \lambda+\varsigma,\alpha^\vee \rangle = 0$ for any $\alpha \in I$ and $\langle \mu+\varsigma,\beta^\vee \rangle = 0$ for any $\beta \in J$. Then $\lambda+\varsigma$ defines in the natural way a line bundle $\scO_{\bG/\bP_I}(\lambda+\varsigma)$ on $\bG/\bP_I$, and $\mu$ defines in the natural way a line bundle $\scO_{\bG/\bP_J}(\mu+\sigma)$ on $\bG/\bP_J$. Their exterior product defines a line bundle $\scO_{\bG/\bP_I \times \bG/\bP_J}(\lambda+\varsigma,\mu+\varsigma)$ on $\bG/\bP_I \times \bG/\bP_J$.

%, resp.~$\mu \in X^*(\bT)$ satisfies $\langle \nu, \alpha^\vee \rangle = 0$ for any $\alpha \in I$, resp.~$\alpha \in J$, then $\nu$ defines in the natural way a line bundle on $\bG/\bP_I$, resp.~$\bG/\bP_J$, which will be denoted $\scO_{\bG/\bP_I}(\nu)$, resp.~$\scO_{\bG/\bP_J}(\nu)$.
%The following statement is proved in~\cite[Lemma in~\S 1.2.1]{bmr2}.

By Lemma~\ref{lem:U-rho-Azumaya}, $(\cU\bg)^{\widehat{-\varsigma}}$ defines a vector bundle on $\bg^{*(1)} \times_{\bt^{*(1)}/\bWf} \bZ^{\widehat{-\varsigma},\widehat{-\varsigma}}$, whose pullback under the morphism~\eqref{eqn:omegaIJ-completion} will be denoted $\scM_{I,J}^{\widehat{-\varsigma},\widehat{-\varsigma}}$. Then Lemma~\ref{lem:U-rho-Azumaya} and Lemma~\ref{lem:splitting-DIJ-varsigma} imply that there exists a canonical isomorphism
\begin{equation}
\label{eqn:splitting-oDIJ-varsigma}
\osD_{I,J}^{\widehat{-\varsigma},\widehat{-\varsigma}} \simto \mathscr{E} \hspace{-1pt} \mathit{nd}_{\scO_{\St_{I,J}^{(1)} \times_{\bt^{*(1)} / \bW_I \times_{\bt^{*(1)} / \bWf} \bt^{*(1)} / \bW_J} \bZ_{I,J}^{\widehat{-\varsigma},\widehat{-\varsigma}}}}(\scM_{I,J}^{\widehat{-\varsigma},\widehat{-\varsigma}}).
\end{equation}

Now, fix $I,J \subset \fRs$, and let $\lambda,\mu \in X^*(\bT)$ be weights which satisfy 
\[
\bW_I \subset \mathrm{Stab}_{(\bWaff,\bullet)}(\lambda), \quad
\bW_J \subset \mathrm{Stab}_{(\bWaff,\bullet)}(\mu).
\]
%$\langle \lambda+\varsigma,\alpha^\vee \rangle = 0$ for any $\alpha \in I$ and $\langle \mu+\varsigma,\beta^\vee \rangle = 0$ for any $\beta \in J$. 
%The comments in Remark~\ref{rmk:translation-DI} show that the assignment $(\xi_1,\xi_2) \mapsto (\xi_1 + \overline{\lambda+\varsigma}, \xi_2 + \overline{\mu+\varsigma})$ induces an automorphism of $\bZ_{I,J}$, which induces an isomorphism of schemes
%\[
%\bZ_{I,J}^{\widehat{-\varsigma}, \widehat{-\varsigma}} \simto \bZ_{I,J}^{\hla,\hmu},
%\]
%and then an isomorphism
%\[
%\St_{I,J}^{(1)} \times_{\bt^{*(1)} / \bW_I \times_{\bt^{*(1)} / \bWf} \bt^{*(1)} / \bW_J} \bZ_{I,J}^{\widehat{-\varsigma}, \widehat{-\varsigma}} \simto
%\St_{I,J}^{(1)} \times_{\bt^{*(1)} / \bW_I \times_{\bt^{*(1)} / \bWf} \bt^{*(1)} / \bW_J} \bZ_{I,J}^{\hla, \hmu}.
%\]
%By~\cite[Lemma~1.2.1]{bmr2} (see also~\cite[Lemma~2.3.1]{bmr}), there exists an equivalence of Azumaya algebras between the pushforward of $\osD_{I,J}^{\widehat{-\varsigma},\widehat{-\varsigma}}$ under this isomorphism and $\osD_{I,J}^{\hla,\hmu}$, the equivalence between the corresponding module categories being given by the functor
%\[
%\scO_{\bG/\bP_I \times \bG/\bP_J}(\lambda+\varsigma, 0) \otimes_{\scO_{\bG/\bP_I \times \bG/\bP_J}} (-) \otimes_{\scO_{\bG/\bP_I \times \bG/\bP_J}} \scO_{\bG/\bP_I \times \bG/\bP_J}(0, \mu + \varsigma).
%\]
%Using~\eqref{eqn:splitting-oDIJ-varsigma} we deduce that, setting
By the comments in Remark~\ref{rmk:translation-DI}, setting
%\begin{multline*}
\[
\scM_{I,J}^{\hla,\hmu} =
\scO_{\bG/\bP_I \times \bG/\bP_J}(\lambda+\varsigma, -\mu-\varsigma) \otimes_{\scO_{\bG/\bP_I \times \bG/\bP_J}} \scM_{I,J}^{\widehat{-\varsigma},\widehat{-\varsigma}} 
%\otimes_{\scO_{\bG/\bP_I \times \bG/\bP_J}} \scO_{\bG/\bP_I \times \bG/\bP_J}(0, \mu + \varsigma)
\]
%\scO_{\bG/\bP_I \times \bG/\bP_J}(\lambda+\varsigma, \mu + \varsigma) \otimes_{\scO_{\bG/\bP_I \times \bG/\bP_J}} \scM_{I,J}^{\widehat{-\varsigma},\widehat{-\varsigma}},
%\end{multline*}
we define a vector bundle on $\St_{I,J}^{(1)} \times_{\bt^{*(1)} / \bW_I \times_{\bt^{*(1)} / \bWf} \bt^{*(1)} / \bW_J} \bZ_{I,J}^{\hla, \hmu}$ with a canonical action of $\osD_{I,J}^{\hla,\hmu}$, such that the action morphism
\[
\osD_{I,J}^{\hla,\hmu} \to \mathscr{E} \hspace{-1pt} \mathit{nd}_{\scO_{\St_{I,J}^{(1)} \times_{\bt^{*(1)} / \bW_I \times_{\bt^{*(1)} / \bWf} \bt^{*(1)} / \bW_J} \bZ_{I,J}^{\hla,\hmu}}}(\scM_{I,J}^{\hla, \hmu})
\]
is an isomorphism. In particular, the functor
\[
\scM_{I,J}^{\hla, \hmu} \otimes_{\scO_{\St_{I,J}^{(1)} \times_{\bt^{*(1)} / \bW_I \times_{\bt^{*(1)} / \bWf} \bt^{*(1)} / \bW_J} \bZ_{I,J}^{\hla,\hmu}}} (-)
\]
defines an equivalence of categories
\[
\Modc^\bG(\osD_{I,J}^{\widehat{\lambda},\widehat{\mu}}) \cong \Coh^{\bG}(\St_{I,J}^{(1)} \times_{\bt^{*(1)} / \bW_I \times_{\bt^{*(1)} / \bWf} \bt^{*(1)} / \bW_J} \bZ_{I,J}^{\hla,\hmu}).
\]
Arguing as in~\cite[Corollary~4.8]{br-Hecke}, one checks that this equivalence restricts to an equivalence
\begin{equation}
\label{eqn:equiv-splitting-oDIJ}
\Modc(\osD_{I,J}^{\widehat{\lambda},\widehat{\mu}},\bG) \cong \Coh^{\bG^{(1)}}(\St_{I,J}^{(1)} \times_{\bt^{*(1)} / \bW_I \times_{\bt^{*(1)} / \bWf} \bt^{*(1)} / \bW_J} \bZ_{I,J}^{\hla,\hmu}).
\end{equation}

Combining these considerations with Lemma~\ref{lem:steinberg-completion}
we deduce the following claim.

\begin{prop}
\label{prop:splitting-HC-equiv}
%Let $I,J \subset \fRs$ and $\lambda,\mu \in X^*(\bT)$, and assume that
%\[
%I=\{\alpha \in \fRs \mid \langle \lambda+\varsigma, \alpha^\vee \rangle=0\} \quad \text{and} \quad
%J=\{\alpha \in \fRs \mid \langle \lambda+\varsigma, \alpha^\vee \rangle=0\}.
%\]
Let $\lambda,\mu \in X^*(\bT)$ and $I,J \subset \fRs$ which satisfy
%belong to the lower closure of the fundamental alcove, and assume that
\[
\bW_I \subset \mathrm{Stab}_{(\bWaff,\bullet)}(\lambda), \quad
\bW_J \subset \mathrm{Stab}_{(\bWaff,\bullet)}(\mu).
\]
%$\langle \lambda+\varsigma, \alpha^\vee \rangle=0$ for any $\alpha \in I$, and $\langle \mu+\varsigma, \beta^\vee \rangle=0$ for any $\beta \in J$.
Then we have canonical equivalences of categories
%\begin{align*}
\[
\Modc^\bG(\osD_{I,J}^{\widehat{\lambda},\widehat{\mu}}) \cong \Coh^{\bG}(\St_{I,J}^{\wedge (1)}), \quad
\Modc(\osD_{I,J}^{\widehat{\lambda},\widehat{\mu}},\bG) \cong \Coh^{\bG^{(1)}}(\St_{I,J}^{\wedge (1)} )
\]
%\end{align*}
under which the embedding
\[
\Modc(\osD_{I,J}^{\widehat{\lambda},\widehat{\mu}},\bG) \to \Modc^\bG(\osD_{I,J}^{\widehat{\lambda},\widehat{\mu}})
\]
corresponds to the natural pullback functor
\[
\Coh^{\bG^{(1)}}(\St_{I,J}^{\wedge (1)} ) \to \Coh^{\bG}(\St_{I,J}^{\wedge (1)} ).
\]
\end{prop}

Finally, combining Proposition~\ref{prop:splitting-HC-equiv} with Theorem~\ref{thm:localization-HC}, for
$\lambda, \mu \in X^*(\bT)$ and $I,J \subset \fRs$ which satisfy
\[
\mathrm{Stab}_{(\bWaff,\bullet)}(\lambda) = \bW_I, \quad
\mathrm{Stab}_{(\bWaff,\bullet)}(\mu) = \bW_J,
\]
%belong to the lower closure of the fundamental alcove and satisfy
%\[
%I=\{\alpha \in \fRs \mid \langle \lambda+\varsigma, \alpha^\vee \rangle=0\}, \quad J=\{\alpha \in \fRs \mid \langle \mu+\varsigma, \alpha^\vee \rangle=0\},
%\]
we obtain equivalences of triangulated categories
\begin{equation}
\label{eqn:equiv-HC-Coh}
\Db \HC^{\hla,\hmu} \xrightarrow[\sim]{\cL_{I,J}^{\lambda,\mu}} \Db\Modc(\osD_{I,J}^{\widehat{\lambda},\widehat{\mu}},\bG) \simto \Db \Coh^{\bG^{(1)}}(\St_{I,J}^{\wedge (1)})
\end{equation}
whose composition will be denoted $\Phi^{\hla,\hmu}$.

\begin{rmk}
Let $\lambda,\mu \in X^*(\bT)$ and $I,J \subset \fRs$ be such that $\mathrm{Stab}_{(\bWaff,\bullet)}(\lambda) = \bW_I$, $\mathrm{Stab}_{(\bWaff,\bullet)}(\mu) = \bW_J$. Let also $\lambda',\mu' \in X^*(\bT)$ which satisfy $\langle \lambda', \alpha^\vee \rangle=0$ for any $\alpha \in I$ and $\langle \mu', \alpha^\vee \rangle=0$ for any $\alpha \in J$. Then the pair $(-\lambda',\mu')$ defines a line bundle $\scO_{(\bG/\bP_I \times \bG/\bP_J)^{(1)}}(-\lambda',\mu')$ on $(\bG/\bP_I \times \bG/\bP_J)^{(1)}$, whose pullback to $\St_{I,J}^{\wedge(1)}$ will be denoted $\scO_{\St_{I,J}^{\wedge(1)}}(-\lambda',\mu')$. 
%We also have the inverse line bundle $\scO_{\St_{I,J}^{\wedge(1)}}(-\lambda',-\mu')$. 
On the other hand, we also have
\[
\mathrm{Stab}_{(\bWaff,\bullet)}(\lambda+\ell\lambda') = \bW_I, \quad \mathrm{Stab}_{(\bWaff,\bullet)}(\mu+\ell \mu') = \bW_J
\]
(see Remark~\ref{rmk:assumption-weights-loc}\eqref{it:assumption-weights-loc-2}), and $\HC^{\hla,\hmu} = \HC^{\widehat{\lambda+\ell\lambda'},\widehat{\mu+\ell\mu'}}$. In this setting we have a functorial isomorphism
\begin{equation}
\label{eqn:Phi-translation-weights}
\Phi^{\widehat{\lambda+\ell\lambda'},\widehat{\mu+\ell\mu'}}(M) \cong \scO_{\St_{I,J}^{\wedge(1)}}(-\lambda',\mu') \otimes_{\scO_{\St_{I,J}^{\wedge(1)}}} \Phi^{\hla,\hmu}(M)
\end{equation}
for any $M$ in $\Db\HC^{\hla,\hmu}$.
\end{rmk}

%--------------------------------------------------------------
\subsection{Proof of Proposition~\ref{prop:acyclicity}\eqref{it:acyclicity-D}}
\label{ss:proof-injectives-acyclic}
%--------------------------------------------------------------

For simplicity
we have stated 
%Proposition~\ref{prop:splitting-HC-equiv} and 
Proposition~\ref{prop:splitting-HC-equiv} for coherent modules only, but it is clear that the same considerations provide an equivalence of categories
\[
\Modqc(\osD_{I,J}^{\widehat{\lambda},\widehat{\mu}},\bG) \cong \QCoh^{\bG^{(1)}}(\St_{I,J}^{\wedge (1)} ).
\]
It is a standard fact that the right-hand side has enough injective objects (see e.g.~\cite[\S A.2]{mr1}), and we deduce that the left-hand side also does.

Next, we have to prove that the image in $\Modqc(\tsD_{I,J}^{\widehat{\lambda},\widehat{\mu}})$ of any injective object in $\Modqc(\osD_{I,J}^{\widehat{\lambda},\widehat{\mu}},\bG)$ is acyclic for the functor $(f^{\lambda,\mu}_{I,J})_*$. For that, we remark that the equivalences above also hold in the non-equivariant setting, so that we have a commutative diagram where all the horizontal arrows are equivalences:
\begin{equation}
\label{eqn:equivalences-inj}
\vcenter{
\xymatrix@R=0.5cm{
\Modqc(\osD_{I,J}^{\widehat{\lambda},\widehat{\mu}},\bG) \ar[d]_-{\For_1} \ar[r]^-{\sim} & \QCoh^{\bG^{(1)}}(\St_{I,J}^{\wedge (1)} ) \ar[d]^-{\For_2} \\
\Modqc^{\bG}(\osD_{I,J}^{\widehat{\lambda},\widehat{\mu}}) \ar[r]^-{\sim} \ar[d]_-{\For_{\bG}} & \QCoh^{\bG}(\St_{I,J}^{\wedge (1)} ) \ar[d]^-{\For_{\bG}} \\
\Modqc(\osD_{I,J}^{\widehat{\lambda},\widehat{\mu}}) \ar[r]^-{\sim} & \QCoh(\St_{I,J}^{\wedge (1)} ).
}
}
\end{equation}
We also set $\For_{\bG^{(1)}} = \For_{\bG} \circ \For_2$.

We claim that for any injective object $\scF \in \QCoh^{\bG^{(1)}}(\St_{I,J}^{\wedge (1)} )$ there exists an injective object $\scG \in \QCoh^{\bG}(\St_{I,J}^{\wedge (1)} )$ such that $\For_{\bG^{(1)}}(\scF)$ is a direct summand in $\For_{\bG}(\scG)$. In fact, recall that the functor $\For_{\bG^{(1)}}$ has a right adjoint
\[
\Av_{\bG^{(1)}} : \QCoh(\St_{I,J}^{\wedge (1)} ) \to \QCoh^{\bG^{(1)}}(\St_{I,J}^{\wedge (1)} )
\]
which can be described as $\Av_{\bG^{(1)}}=a_*p^*$ where $a,p : \bG^{(1)} \times \St_{I,J}^{\wedge (1)} \to \St_{I,J}^{\wedge (1)}$ are the action and projection maps respectively (see once again~\cite[\S A.2]{mr1}). Similarly, the functor $\For_{\bG}$ has a right adjoint
\[
\Av_{\bG} : \QCoh(\St_{I,J}^{\wedge (1)} ) \to \QCoh^{\bG}(\St_{I,J}^{\wedge (1)} )
\]
which can be described as $\Av_{\bG}=a'_*(p')^*$ where $a',p' : \bG \times \St_{I,J}^{\wedge (1)} \to \St_{I,J}^{\wedge (1)}$ are the action and projection maps respectively. If $\Fr_{\bG} : \bG \to \bG^{(1)}$ is the Frobenius morphism, then we have
\[
a' = a \circ (\Fr_\bG \times \id), \quad p' = p \circ (\Fr_\bG \times \id).
\]
Using adjunction it is easy to see
that any injective object in $\QCoh^{\bG^{(1)}}(\St_{I,J}^{\wedge (1)} )$ is a direct summand of an object of the form $\Av_{\bG^{(1)}}(\scF')$ where $\scF' \in \QCoh(\St_{I,J}^{\wedge (1)} )$ is injective. Hence we can assume that $\scF$ is of this form. In this case we will prove that $\For_{\bG^{(1)}}(\scF)$ is a direct summand in $\For_{\bG}(\Av_{\bG}(\scF'))$, which will complete the proof of the claim since $\Av_{\bG}(\scF')$ is injective in $\QCoh^{\bG}(\St_{I,J}^{\wedge (1)} )$ by adjunction. In fact we have
\[
\For_{\bG}(\Av_{\bG}(\scF')) = a'_* (p')^* \scF' = a_* (\Fr_\bG \times \id)_* (\Fr_\bG \times \id)^* p^* \scF'.
\]
Using the projection formula (in the form of~\cite[Proposition~3.9.4]{lipman}) and the fact that $\Fr_\bG$ is affine and flat, we deduce an isomorphism
\[
\For_{\bG}(\Av_{\bG}(\scF')) \cong a_* \bigl( (p^* \scF') \otimes_{\scO_{\bG^{(1)} \times \St_{I,J}^{\wedge (1)}}} (\Fr_\bG \times \id)_*\scO_{\bG \times \St_{I,J}^{\wedge (1)}} \bigr).
\]
Now since $\bG$ is a smooth affine variety, $\scO_{\bG^{(1)}}$ is a direct summand in $(\Fr_{\bG})_* \scO_{\bG}$ by~\cite[Proposition~1.1.6]{bk}, hence $\scO_{\bG^{(1)} \times \St_{I,J}^{\wedge (1)}}$ is a direct summand in $(\Fr_\bG \times \id)_*\scO_{\bG \times \St_{I,J}^{\wedge (1)}}$, which proves our claim since $\For_{\bG^{(1)}}(\Av_{\bG^{(1)}}(\scF')) = a_* p^* \scF'$.

Using the diagram~\eqref{eqn:equivalences-inj}, the claim we have just proven says that for any injective object $\scF$ in $\Modqc(\osD_{I,J}^{\widehat{\lambda},\widehat{\mu}},\bG)$ there exists an injective object $\scG$ in $\Modqc^{\bG}(\osD_{I,J}^{\widehat{\lambda},\widehat{\mu}})$ such that $\For_\bG \circ \For_1(\scF)$ is a direct summand in $\For_\bG(\scF)$. Hence, to conclude the proof it suffices to prove that the image in $\Modqc(\tsD_{I,J}^{\widehat{\lambda},\widehat{\mu}})$ of any injective object in $\Modqc^{\bG}(\osD_{I,J}^{\widehat{\lambda},\widehat{\mu}})$ is acyclic for the functor $(f^{\lambda,\mu}_{I,J})_*$. The latter fact can be checked as in~\cite[Lemma~A.9]{mr1}, since $\osD_{I,J}^{\widehat{\lambda},\widehat{\mu}}$ is a $\bG$-equivariant quasi-coherent sheaf of algebras.

%The acyclicity claim follows from this claim using the diagram~\eqref{eqn:equivalences-inj} and the fact that for any injective object $\scF$ in $\Modqc^{\bG}(\osD_{I,J}^{\widehat{\lambda},\widehat{\mu}})$, the object $\For_\bG(\scF) \in \Modqc(\osD_{I,J}^{\widehat{\lambda},\widehat{\mu}})$ is acyclic for $(f_{I,J}^{\lambda,\mu})_*$ (which can be proved as in~\cite[Lemma~A.9]{mr1} using the fact that $\osD_{I,J}^{\widehat{\lambda},\widehat{\mu}}$ is a $\bG$-equivariant quasi-coherent sheaf of algebras).

%----------------------------------------------------------------
\subsection{Variant for a fixed Harish-Chandra character}
\label{ss:localization-fixed-HC-char}
%----------------------------------------------------------------

The same considerations as above apply in the setting considered in~\S\ref{ss:variant-fixed-central-char}, using in particular Lemma~\ref{lem:steinberg-completion-prime}. We obtain the following result.
%Namely, set
%\[
%\St_{I,J}^{\prime \wedge} := \St^{\prime}_{I,J} \times_{\bt^{*(1)}/\bW_I \times_{\bt^{*(1)}/\bWf} \{0\}} \FN_{\bt^{*(1)}/\bW_I \times_{\bt^{*(1)}/\bWf} \{0\}}( \{ (0,0) \} ).
%\]
%(This scheme is of finite type over $\bk$.)

\begin{prop}
\label{prop:splitting-HC-equiv-fix}
%Let $I,J \subset \fRs$ and $\lambda,\mu \in X^*(\bT)$, and assume that
%\[
%I=\{\alpha \in \fRs \mid \langle \lambda+\varsigma, \alpha^\vee \rangle=0\} \quad \text{and} \quad
%J=\{\alpha \in \fRs \mid \langle \lambda+\varsigma, \alpha^\vee \rangle=0\}.
%\]
Let $\lambda,\mu \in X^*(\bT)$ and $I,J \subset \fRs$ which satisfy
%belong to the lower closure of the fundamental alcove, and assume that
\[
\bW_I \subset \mathrm{Stab}_{(\bWaff,\bullet)}(\lambda), \quad
\bW_J \subset \mathrm{Stab}_{(\bWaff,\bullet)}(\mu).
\]
%$\langle \lambda+\varsigma, \alpha^\vee \rangle=0$ for any $\alpha \in I$, and $\langle \mu+\varsigma, \beta^\vee \rangle=0$ for any $\beta \in J$.
Then we have canonical equivalences of categories
%\begin{align*}
\[
\Modc^\bG(\osD_{I,J}^{\widehat{\lambda},\mu}) \cong \Coh^{\bG}(\St_{I,J}^{\prime (1)}), \quad
\Modc(\osD_{I,J}^{\widehat{\lambda},\mu},\bG) \cong \Coh^{\bG^{(1)}}(\St_{I,J}^{\prime (1)} )
\]
%\end{align*}
under which the embedding
\[
\Modc(\osD_{I,J}^{\widehat{\lambda},\mu},\bG) \to \Modc^\bG(\osD_{I,J}^{\widehat{\lambda},\mu})
\]
corresponds to the natural pullback functor
\[
\Coh^{\bG^{(1)}}(\St_{I,J}^{\prime (1)} ) \to \Coh^{\bG}(\St_{I,J}^{\prime (1)} ).
\]
\end{prop}

Combining Proposition~\ref{prop:splitting-HC-equiv-fix} with Theorem~\ref{thm:localization-HC-fixed}, for
$\lambda, \mu \in X^*(\bT)$ and $I,J \subset \fRs$ which satisfy
\[
\mathrm{Stab}_{(\bWaff,\bullet)}(\lambda) = \bW_I, \quad
\mathrm{Stab}_{(\bWaff,\bullet)}(\mu) = \bW_J
\]
we obtain equivalences of triangulated categories
\begin{equation}
\label{eqn:equiv-HC-Coh-fix}
\Db \HC^{\hla,\mu} \simto \Db\Modc(\osD_{I,J}^{\widehat{\lambda},\mu},\bG) \simto \Db \Coh^{\bG^{(1)}}(\St_{I,J}^{\prime (1)}).
\end{equation}
%whose composition will be denoted $\Phi^{\hla,\hmu}$.
It is easily seen that the following diagram commutes, where the right vertical arrows are the push/pull functors associated with the closed embedding $\St_{I,J}^{\prime (1)} \hookrightarrow \St_{I,J}^{\wedge (1)}$:
\[
\xymatrix@C=2cm{
\Db \HC^{\hla,\hmu} \ar[r]^-{\Phi^{\hla,\hmu}} \ar@<1ex>[d]^-{\mathsf{Sp}_{\lambda,\mu}} & \Db \Coh^{\bG^{(1)}}(\St_{I,J}^{\wedge (1)}) \ar@<1ex>[d] \\
\Db \HC^{\hla,\mu} \ar[r]^-{\eqref{eqn:equiv-HC-Coh-fix}} \ar@<1ex>[u]^-{\eqref{eqn:functor-HC-hla-mu-hmu}} & \Db \Coh^{\bG^{(1)}}(\St_{I,J}^{\prime (1)}). \ar@<1ex>[u]
}
\]

\section{Convolution}
\label{sec:monoidality}
%%%%%%%%%%%%%%%%%%%%%%%%%%%%%%%%%%%%%%%%%%%%%%%%%%%%%

From now on, in addition to the conditions of~\S\ref{ss:HC-notation}, we assume that $\ell \geq h$. We also fix a weight $\lambda \in A_0 \cap X^*(\bT)$. (Here $A_0$ is the fundamental alcove, see~\S\ref{ss:translation-bimod}.) Below we will mainly consider the constructions of the preceding sections in case $I=J=\varnothing$, and for the pair of weights $(\lambda,\lambda)$.
%weight in the fundamental alcove $A_0$. 
%Our goal is to show that the equivalences in~\eqref{eqn:equiv-HC-Coh} are monoidal for appropriate convolution products.

%--------------------------------------------------------------------
\subsection{Convolution bifunctors for coherent sheaves}
\label{ss:convolution}
%--------------------------------------------------------------------

We have a natural action of the group $\bG \times \Gm$ on $\tbg$ where $\bG$ acts as in~\S\ref{ss:parabolic-Groth} and $z \in \Gm$ acts by dilation by $z^2$ on the fibers of the projection $\tbg \to \bG/\bB$. (Our conventions here follow those of~\cite{br-Baff}.) From this action we deduce a diagonal action of $(\bG \times \Gm)^{(1)}$ on any product of copies of $\tbg^{(1)}$.
% and then, for any $\bk$-algebraic group $\bH$ endowed with a morphism to $(\bG \times \Gm)^{(1)}$, a similar action of $\bH$. (In practice we will be interested in the case when $\bH$ is one of $\bG$, $\bG \times \Gm$, $\bG^{(1)}$, $(\bG \times \Gm)^{(1)}$ or the trivial group, with the natural morphism to $(\bG \times \Gm)^{(1)}$.)
Consider the natural morphisms
\[
p_{i,j} : \tbg^{(1)} \times \tbg^{(1)} \times \tbg^{(1)} \to \tbg^{(1)} \times \tbg^{(1)}
\]
of projection to the $i$-th and $j$-th factors, where $i,j \in \{1,2,3\}$. Then it is a classical fact that
%, for any $\bH$ as above, 
we have a monoidal structure on the derived category
\[
\Db \Coh^{(\bG \times \Gm)^{(1)}}_{\St^{(1)}}( \tbg^{(1)} \times \tbg^{(1)})
\]
of equivariant coherent sheaves set-theoretically supported on $\St^{(1)}$, with monoidal product given by
\[
\scF \star \scG = R(p_{1,3})_* \left( L(p_{1,2})^* \scF \lotimes_{\scO_{\tbg^{(1)} \times \tbg^{(1)} \times \tbg^{(1)}}} L(p_{2,3})^*\scG \right)
\]
and unit object the structure sheaf $\scO_{\Delta \tbg^{(1)}}$ of the diagonal copy $\Delta \tbg^{(1)} \subset \tbg^{(1)} \times \tbg^{(1)}$. Considering the Springer resolution
\[
\tcN := \tbg \times_{\bt^*} \{0\},
\]
we also have a canonical action of the category $\Db \Coh^{(\bG \times \Gm)^{(1)}}_{\St^{(1)}}( \tbg^{(1)} \times \tbg^{(1)})$ on the category $\Db \Coh^{(\bG \times \Gm)^{(1)}}_{\St^{\prime (1)}}( \tbg^{(1)} \times \tcN^{(1)})$, defined by a similar formula as above. Similar comments apply to categories of $\bG^{(1)}$-equivariant coherent sheaves, or $\bG$-equivariant coherent sheaves (where $\bG$ acts via the Frobenius morphism $\bG \to \bG^{(1)}$), or nonequivariant coherent sheaves.

Now we consider the schemes $\St = \tbg \times_{\bg^*} \tbg$ and $\St^\wedge$. The latter scheme can also be described as a fiber product of the same form. Namely, set
\begin{gather*}
\tbg^\wedge = \tbg \times_{\bt^*/\bWf} \FN_{\bt^*/\bWf}(\{0\}) \cong \tbg \times_{\bt^*} \FN_{\bt^*}(\{0\}), \\
 \bg^{*\wedge} = \bg^* \times_{\bt^*/\bWf} \FN_{\bt^*/\bWf}(\{0\}).
\end{gather*}
Then Remark~\ref{rmk:St-fiber-product} implies that we have a canonical identification
\begin{equation}
\label{eqn:Stwedge-fiber-prod}
\St^\wedge = \tbg^\wedge \times_{\bg^{*\wedge}} \tbg^\wedge.
\end{equation}

%In~\cite{dr} we show that 
A similar construction as above provides monoidal products on the categories
\[
\Db \Coh^{(\bG \times \Gm)^{(1)}}(\St^{(1)}) \quad \text{and} \quad \Db \Coh^{\bG^{(1)}}(\St^{\wedge(1)})
\]
(and appropriate variants with different equivariance conditions, such that the obvious forgetful functors are monoidal)
such that the natural pushforward functor
\[
\Db \Coh^{(\bG \times \Gm)^{(1)}}(\St^{(1)}) \to \Db \Coh^{\bG^{(1)}}_{\St^{(1)}}( \tbg^{(1)} \times \tbg^{(1)})
\]
and the natural pullback functor
\begin{equation}
\label{eqn:pullback-convolution}
\Db \Coh^{\bG^{(1)}}(\St^{(1)}) \to \Db \Coh^{\bG^{(1)}}(\St^{\wedge(1)})
\end{equation}
admit canonical monoidal structures. The monoidal products on these categories will also be denoted $\star$. For this product, the image of the fully faithful functor of Lemma~\ref{lem:DbCoh-nil} (for $I=J=\varnothing$) is an ideal; in particular, this endows its domain with a canonical monoidal structure. We also have a canonical action of $\Db \Coh^{\bG^{(1)}}(\St^{(1)})$ on $\Db \Coh^{\bG^{(1)}}(\St^{\prime(1)})$, which factors through an action of $\Db \Coh^{\bG^{(1)}}(\St^{\wedge(1)})$.

This construction requires the use of concepts from Derived Algebraic Geometry; here we only sketch how it proceeds (more details will appear in~\cite{adr}). In fact, the fiber products defining $\St$ and $\St^\wedge$ are isomorphic to the corresponding \emph{derived} fiber products (see e.g.~Remark~\ref{rmk:derived-fiber-prod}), which are therefore ``honnest'' schemes. But this is \emph{not} the case for the triple products
\[
\St_{\mathrm{conv}} := \tbg \rtimes_{\bg^*} \tbg \rtimes_{\bg^*} \tbg \quad \text{and} \quad
\St_{\mathrm{conv}}^\wedge := \tbg^\wedge \rtimes_{\bg^{*\wedge}} \tbg^\wedge \rtimes_{\bg^{*\wedge}} \tbg^\wedge
\]
which are involved in the definition of the convolution product, which should be defined using the formula
\[
\scF \star \scG = R(q_{1,3})_* \left( L(q_{1,2})^* \scF \lotimes_{\scO} L(q_{2,3})^*\scG \right)
\]
where $q_{i,j}$ is the projection on the $i$-th and $j$-th factors from $\St_{\mathrm{conv}}^{(1)}$ or $\St_{\mathrm{conv}}^{\wedge(1)}$, and the tensor product is taken over the structure (derived) sheaf of the triple fiber product. This formula is initially defined on the $\infty$-category of (nonequivariant) quasi-coherent sheaves on these derived schemes; it defines a monoidal structure (in the sense of $\infty$-categories) because it can be interpreted as composition of ``convolution kernels;'' see~\cite[Corollary~4.10 and~\S 5.2]{bzfn}. Then one obtains monoidal structures on $\infty$-categories of equivariant quasi-coherent sheaves (equivalently, quasi-coherent sheaves on the associated quotient derived stacks) by taking categorical invariants under the appropriate weak actions (see~\cite[\S\S 2.2.7--2.2.8]{beraldo}).\footnote{The results of~\cite{bzfn} apply to derived stacks, but only when they are \emph{perfect}. Quotient stacks over fields of positive characteristic are not perfect in general, so that these results cannot be applied directly, which justifies our need for a second step.} Passing to homotopy categories we deduce a monoidal structure on the (unbounded) derived categories of equivariant quasi-coherent sheaves.

It remains finally to explain why these convolution products restrict to bounded derived categories of coherent sheaves; here equivariance plays no role, hence it can be forgotten for simplicity. For the case of $\St^{(1)}$, this can be justified using the compatibility with the monoidal product on $\Db \Coh^{\bG^{(1)}}_{\St^{(1)}}( \tbg^{(1)} \times \tbg^{(1)})$. For the case of $\St^{\wedge (1)}$ one can proceed similarly using compatibility with the bifunctor on the derived category $\Db \Coh^{\bG^{(1)}}_{\St^{\wedge (1)}}( (\tbg \times \tbg)^{\wedge(1)})$, where
\[
(\tbg \times \tbg)^{\wedge} = (\tbg \times \tbg) \times_{\bt^* \times \bt^*} \FN_{\bt^* \times \bt^*}(\{(0,0)\}),
\]
defined using formulas similar to those above for the morphisms from
\[
(\tbg \times \tbg \times \tbg)^{\wedge} := (\tbg \times \tbg \times \tbg) \times_{\bt^* \times \bt^* \times \bt^*} \FN_{\bt^* \times \bt^* \times \bt^*}(\{(0,0,0)\})
\]
to $(\tbg \times \tbg)^{\wedge}$ induced by the obvious projections. (We do not claim, and do not expect, that this bifunctor defines a monoidal structure.) This bifunctor respects bounded complexes of coherent sheaves because these maps are flat, proper on the support of the complexes under consideration, and because their domain is smooth over a regular scheme, hence regular, of finite dimension. (Flatness can be justified by factoring this map via $(\tbg \times \tbg) \times_{\bt^* \times \bt^*} \FN_{\bt^* \times \bt^*}(\{(0,0)\}) \times \tbg$, where the first morphism is induced by the natural map $\FN_{\bt^* \times \bt^* \times \bt^*}(\{(0,0,0)\}) \to \FN_{\bt^* \times \bt^*}(\{(0,0)\}) \times \bt^*$, which is flat since $\scO(\FN_{\bt^* \times \bt^* \times \bt^*}(\{(0,0,0)\}))$ is a completion of $\scO(\FN_{\bt^* \times \bt^*}(\{(0,0)\}) \times \bt^*)$.)

%\begin{rmk}
%The reason why this construction requires some care is that, although the fiber product defining $\St$ is a \emph{derived} fiber product, see Remark~\ref{rmk:derived-fiber-prod}, the same assertion is \emph{not} true for the fiber product $\tbg \times_{\bg^*} \tbg \times_{\bg^*} \tbg$. As a consequence, some version of Derived Algebraic Geometry is required in the construction of these monoidal products.
%\end{rmk}

%--------------------------------------------------------------------
\subsection{Monoidality}
\label{ss:monoidality}
%--------------------------------------------------------------------

%We also show in~\cite{dr} that 

The following is the main result of this section

\begin{thm}
\phantomsection
\label{thm:monoidality}
\begin{enumerate}
\item
\label{it:monoidality-1}
The composition
\begin{equation}
\label{eqn:composition-Phi-monoidal}
\Db \Coh^{\bG^{(1)}}(\St^{\wedge(1)}) \xrightarrow[\sim]{(\Phi^{\hla,\hla})^{-1}} \Db \HC^{\hla,\hla} \to D^- \HC^{\hla,\hla}
\end{equation}
(where the second functor is the obvious one) admits a canonical monoidal structure, where the monoidal product on the domain is as discussed in~\S\ref{ss:convolution}, and the one on the codomain is the structure considered in~\S\ref{ss:monoidal-structure-DHC}.
\item
\label{it:monoidality-2}
The composition
\[
\Db \Coh^{\bG^{(1)}}(\St^{\prime(1)}) \xrightarrow[\sim]{\eqref{eqn:equiv-HC-Coh-fix}} \Db \HC^{\hla,\lambda} \to D^- \HC^{\hla,\lambda}
\]
is compatible in the obvious sense with the action of $\Db \Coh^{\bG^{(1)}}(\St^{\wedge(1)})$ on its domain (see~\S\ref{ss:convolution}) and of $D^- \HC^{\hla,\hla}$ on its codomain (see~\S\ref{ss:HC-fixed-character}).
\end{enumerate}
\end{thm}

Let us remark that these results imply that the subcategory
\[
\Db \HC^{\hla,\hla} \subset D^- \HC^{\hla,\hla}
\]
is stable under the monoidal product (hence a monoidal category), and that its action on $D^- \HC^{\hla,\lambda}$ stabilizes $\Db \HC^{\hla,\lambda}$.

Below we give a sketch of proof of~\eqref{it:monoidality-1}, which can easily be adapted for the proof of~\eqref{it:monoidality-2}. A detailed treatment of these questions requires the use of Derived Algebraic Geometry, and will appear in~\cite{adr}.

\begin{proof}[Sketch of proof of Theorem~\ref{thm:monoidality}\eqref{it:monoidality-1}]
The functor $(\Phi^{\hla,\hla})^{-1}$ is a composition
\begin{equation}
\label{eqn:Phi-composition-monoidality}
\Db \Coh^{\bG^{(1)}}(\St^{\wedge(1)}) \simto \Db \Modc(\osD^{\hla, \hla}, \bG) \simto \Db \HC^{\hla,\hla}.
\end{equation}
We first note that the bounded above derived category $D^- \Modc(\osD^{\hla, \hla}, \bG)$ admits a monoidal structure, constructed as follows. Recall the sheaf of algebras $\tsD^{\hla,\hla}$, which we regard here as living on $(\tbg \times \tbg)^{\wedge(1)}$. One can consider the category of strongly $\bG$-equivariant coherent modules for this sheaf of algebras, which identifies with $\Modc(\osD^{\hla, \hla}, \bG)$. (This is similar to the situation for $\cU\bg$-modules considered in~\S\ref{ss:def-HC}.) Denoting by
\[
r_{i,j} : (\tbg \times \tbg \times \tbg)^{\wedge(1)} \to (\tbg \times \tbg)^{\wedge(1)}
\]
the morphism induced by projection on the $i$-th and $j$-th components, we have a canonical bifunctor
\[
D^- \Modc(\tsD^{\hla, \hla}, \bG) \times D^- \Modc(\tsD^{\hla, \hla}, \bG) \to D^- \Modc(r_{1,3}^* \tsD^{\hla, \hla}, \bG)
\]
sending a pair $(\scF,\scG)$ to the derived tensor product of $Lr_{1,2}^* \scF$ and $Lr_{2,3}^* \scG$ over the pullback of $\tsD^{\hla}$ under the morphism $(\tbg \times \tbg \times \tbg)^{\wedge(1)} \to \tbg^{\wedge(1)}$ induced by projection on the second component. Composing with pushforward along $r_{1,3}$ we deduce the desired bifunctor defining the monoidal product.

Next, we claim that the second equivalence in~\eqref{eqn:Phi-composition-monoidality} ``extends'' to an equivalence of triangulated categories
\begin{equation}
\label{eqn:equivalences-D-}
D^- \Modc(\osD^{\hla, \hla}, \bG) \simto D^- \HC^{\hla,\hla}.
\end{equation}
In fact the functors in both directions naturally extend to bounded above derived categories, and these extentions are still adjoint to each other. Let us write $L$ for the left adjoint (given by a pullback functor from $D^- \HC^{\hla,\hla}$ to $D^- \Modc(\osD^{\hla, \hla}, \bG)$) and $R$ for the right adjoint (given by global sections). By the same argument as for bounded derived categories the adjunction morphism $\id \to RL$ is an isomorphism; to conclude the proof it therefore suffices to prove that $R$ does not kill any nonzero object. However if $\scF \in D^- \Modc(\osD^{\hla, \hla}, \bG)$ is a nonzero complex, consider the highest degree $N$ in which $\scF$ is nonzero. For $M>0$ we can consider the truncation triangle
\[
\tau_{<N-M} \scF \to \scF \to \tau_{\geq N-M} \scF \xrightarrow{[1]},
\]
where the third term is bounded. If $M \gg 0$, the complex $R(\tau_{<N-M} \scF)$ vanishes in degrees $\geq N$, so that to conclude it suffices to show that $R(\tau_{\geq N-M} \scF)$ has nonzero cohomology in some degree $\geq N$. But we already know that $R$ and $L$ are inverse equivalences on bounded complexes, so that $\tau_{\geq N-M} \scF \cong L R(\tau_{\geq N-M} \scF)$. Since $L$ is right t-exact (for the tautological t-structures) and $\scF$ is nonzero in degree $N$, this implies the desired claim and finishes the proof.

Next, we claim that~\eqref{eqn:equivalences-D-} has a natural monoidal structure. In fact, to construct such a structure, using the notation above it suffices to construct an isomorphism of functors
\[
R(L(M) \star L(N)) \simto M \star N
\]
for all $M,N \in D^- \HC^{\hla,\hla}$, where we denote both monoidal products by $\star$. This isomorphism can be deduced from the projection formula, realizing the monoidal structures in terms of tensor product with the pullback of $\tsD^{\hla} \boxtimes \tsD^{\hla} \boxtimes \tsD^{\hla}$ under the natural morphism
\[
(\tbg \times \tbg \times \tbg)^{\wedge(1)} \to \tbg^{\wedge(1)} \times \tbg^{\wedge(1)} \times \tbg^{\wedge(1)},
\]
resp.~with
\[
(\cU\bg \otimes_\bk \cU\bg \otimes_\bk \cU\bg) \otimes_{\scO((\bt^*/(\bW, \bullet))^3)} \scO(\FN_{(\bt^*/(\bW, \bullet))^3}(\{(0,0,0)\})),
\]
and using~\eqref{eqn:global-sections-DI-xi}.

Finally, to conclude it remains to construct a monoidal structure on the functor
\[
\Db \Coh^{\bG^{(1)}}(\St^{\wedge(1)}) \to D^- \Modc(\osD^{\hla, \hla}, \bG).
\]
Now we have a canonical equivalence of categories
\[
\Modc(\osD^{\hla, \hla}, \bG) \simto \Modc(\osD^{\widehat{-\varsigma}, \widehat{-\varsigma}}, \bG).
\]
The same considerations as above provide a monoidal structure on the bounded above derived category of the right-hand side, such that the induced equivalence between bounded above derived categories is monoidal. This reduces the proof to the construction of a monoidal structure on the similarly defined functor
\[
\Db \Coh^{\bG^{(1)}}(\St^{\wedge(1)}) \to D^- \Modc(\osD^{\widehat{-\varsigma}, \widehat{-\varsigma}}, \bG),
\]
which is easy because the sheaf of algebras $\tsD^{\widehat{-\varsigma}}$ is the pullback of an Azumaya algebra on $\bg^{*\wedge(1)}$.
\end{proof}

\subsection{Bott--Samelson and Soergel type complexes of coherent sheaves}
\label{ss:BSCoh}
%--------------------------------------------------------------------

%From now on we assume that $\lambda \in A_0$.
Recall the notation introduced in~\S\ref{ss:BSHC}. For any $s \in \bSaff$, resp.~$\omega \in \mathbf{\Omega}$, we set
\[
\scR_s := \Phi^{\hla,\hla}(\sfR_s), \quad \text{resp.} \quad
\scR_\omega := \Phi^{\hla,\hla}(\sfR_\omega).
\]
(It can be deduced from Proposition~\ref{prop:translation-push-pull} below that these objects do not depend on the choice of $\lambda$ up to isomorphism, so that it is legitimate not to indicate this weight in the notation.)
It follows from Lemma~\ref{lem:sfP-properties} and the monoidality of the functor~\eqref{eqn:composition-Phi-monoidal}
%$\Phi^{\hla,\hla}$ 
that for $s \in \bSaff$ and $\omega,\omega' \in \mathbf{\Omega}$ we have
\begin{equation}
\label{eqn:convolution-R}
\scR_\omega \star \scR_s \star \scR_{\omega^{-1}} \cong \scR_{\omega s \omega^{-1}}, \quad \scR_\omega \star \scR_{\omega'} \cong \scR_{\omega \omega'}.
\end{equation}
We then define the category
\[
\BSCoh^{\bG^{(1)}}(\St^{\wedge(1)})
\]
as the strictly full subcategory of $\Db \Coh^{\bG^{(1)}}(\St^{\wedge(1)})$ generated under the monoidal product $\star$ by the unit object and the objects $\scR_s$ ($s \in \bSaff$) and $\scR_\omega$ ($\omega \in \mathbf{\Omega}$). By~\eqref{eqn:convolution-R}, any object in this category is isomorphic to an object
\[
\scR_{s_1} \star \cdots \star \scR_{s_r} \star \scR_\omega
\]
where $s_1, \cdots, s_r \in \bSaff$ and $\omega \in \mathbf{\Omega}$. We will also denote by
\[
\SCoh^{\bG^{(1)}}(\St^{\wedge(1)})
\]
the karoubian envelope of the additive hull of $\BSCoh^{\bG^{(1)}}(\St^{\wedge(1)})$.

By construction, the functor $\Phi^{\hla,\hla}$ restricts to equivalences of monoidal categories
\begin{equation}
\label{eqn:BSCoh-BSHC}
\BSHC^{\hla,\hla} \simto \BSCoh^{\bG^{(1)}}(\St^{\wedge(1)}), \qquad \SHC^{\hla,\hla} \simto \SCoh^{\bG^{(1)}}(\St^{\wedge(1)}).
\end{equation}

\begin{prop}
\label{prop:Ext-vanishing-Coh}
For any $\scF,\scG$ in $\SCoh^{\bG^{(1)}}(\St^{\wedge(1)})$ and $n \in \Z \smallsetminus \{0\}$ we have
\[
\Hom_{\Db \Coh^{\bG^{(1)}}(\St^{\wedge(1)})}(\scF,\scG [n]) = 0.
\]
\end{prop}

\begin{proof}
Using the equivalence $\Phi^{\hla,\hla}$ one sees that the claim follows from Proposition~\ref{prop:Ext-vanishing}.
\end{proof}

%-----------------------------------------------------------
\subsection{Some kernels}
\label{ss:kernels}
%-----------------------------------------------------------

%We consider the natural action of the group $\bG \times \Gm$ on $\tbg$ considered e.g.~in~\cite{br-Baff}; i.e.~$\bG$ acts as in~\S\ref{ss:parabolic-Groth} and $z \in \Gm$ acts by dilations by $z^2$ on the fibers of the projection $\tbg \to \bG/\bB$. We will also consider the diagonal action on $\tbg \times \tbg$, which stabilizes the closed subscheme $\St = \tbg \times_{\bg^*} \tbg$. We can then consider the derived category
%\[
%\Db \Coh^{(\bG \times \Gm)^{(1)}}(\St^{(1)})
%\]
%of $(\bG \times \Gm)^{(1)}$-equivariant coherent sheaves on $\St^{(1)}$. Using the formalism of Section~\ref{sec:monoidality} we obtain a monoidal product $\star$ on this category, with unit the pushforward of the structure sheaf of $\tbg^{(1)}$ under the diagonal closed immersion $\tbg^{(1)} \to \St^{(1)}$, and such that the pullback functor
%\begin{equation}
%\label{eqn:pullback-St-Stwedge}
%\Db \Coh^{(\bG \times \Gm)^{(1)}}(\St^{(1)}) \to \Db \Coh^{\bG^{(1)}}(\St^{\wedge (1)})
%\end{equation}
%is monoidal.

%The formalism of~\cite[\S 5]{br-Baff} shows that this category admits a natural convolution product $\star$ which endows it with a monoidal structure. The unit object for this product is $\scO_{\Delta \tbg_\varnothing}$ seen (via pushforward) as a coherent sheaf on $\St$, where $\Delta \tbg_\varnothing \subset \St_{\varnothing,\varnothing}$ is the diagonal copy of $\tbg_\varnothing$.

For $s \in \bSf$ we consider the $(\bG \times \Gm)^{(1)}$-stable closed subscheme $\tbg^{(1)} \times_{\tbg^{(1)}_{s}} \tbg^{(1)} \subset \St^{(1)}$, and set
\[
\scX_s := \scO_{\tbg^{(1)} \times_{\tbg^{(1)}_{s}} \tbg^{(1)}} \quad
\in \Db\Coh^{(\bG \times \Gm)^{(1)}}(\St^{(1)}).
\]
As explained in~\cite[\S 1.10]{br-Baff} the scheme $\tbg^{(1)} \times_{\tbg_{s}^{(1)}} \tbg^{(1)}$ is reduced, with two irreductible components. One is the diagonal copy $\Delta \tbg^{(1)}$ of $\tbg^{(1)}$, and the pushforward of the structure sheaf of the other one will be denoted $\scZ_s$. Then, there exist exact sequences of equivariant coherent sheaves
\begin{equation}
\label{eqn:exact-seq-braid-gp-action}
\scO_{\Delta \tbg^{(1)}} \langle -2 \rangle \hookrightarrow \scX_s
%\scO_{\tbg_{\varnothing} \times_{\tbg_{s}} \tbg_{\varnothing}} 
\twoheadrightarrow \scZ_s, \qquad
\scZ_s(-\varsigma, \varsigma-\alpha_s) \hookrightarrow \scX_s
%\scO_{\tbg_{\varnothing} \times_{\tbg_{s}} \tbg_{\varnothing}} 
\twoheadrightarrow \scO_{\Delta \tbg_\varnothing}.
\end{equation}
Here, $\alpha_s$ is the simple root attached to $s$, $\scZ_s(-\varsigma, \varsigma-\alpha_s)$ is the twist of $\scZ_s$ by the pullback of the line bundle on $(\bG^{(1)}/\bB^{(1)})^2$ attached to the pair of weights $(-\varsigma, \varsigma-\alpha_s)$, and for $n \in \Z$ we denote by $\langle n \rangle$ the functor of tensoring with the $1$-dimensional $\Gm^{(1)}$-module of weight $n$.
%$\scO_{\bG/\bB}(-\varsigma) \boxtimes \scO_{\bG/\bB}(\varsigma-\alpha_s)$ on $(\bG/\bB)^2$.

%Recall also the braid group $\Br_{\bW}$ associated with the group $\bW$, see~\cite[\S 1.1]{br-Baff}. This group admits a presentation with generators $(T_w : w \in \bW)$ and relations
%\[
%T_y T_w = T_{yw} \quad \text{if $\ell(yw)=\ell(y)+\ell(w)$.}
%\]
%It also contains some ``Bernstein elements'' $(\theta_\lambda : \lambda \in X^*(\bT))$, and it is also generated by the elements
%\[
%\{T_s : s \in \bSf \} \cup \{\theta_\lambda : \lambda \in X^*(\bT)\}.
%\]
%%these elements and the elements $(T_s : s \in \bSf)$. 
%(See~\cite[\S 1.1]{br-Baff} for a description of a presentation of $\Br_{\bW}$ in terms of these generators.)

Recall the braid group $\Br_{\bW}$ introduced in~\S\ref{ss:further-properties}.
For $\lambda \in X^*(\bT)$, we will denote by $\scO_{\Delta \tbg^{(1)}}(\lambda)$ the pushforward to $\St^{(1)}$ of the pullback to $\Delta \tbg^{(1)}$ of the line bundle on $\bG^{(1)}/\bB^{(1)}$ naturally associated with $\lambda$.
It follows from~\cite[Theorem~1.3.1]{br-Baff} that there exists a unique group morphism from $(\Br_{\bW})^{\op}$ to the group of isomorphism classes of invertible elements in the monoidal category 
\begin{equation}
\label{eqn:DbCoh-St}
\bigl( \Db \Coh^{(\bG \times \Gm)^{(1)}}(\St^{(1)}), \star \bigr)
\end{equation}
which sends $T_s$ to $\scZ_s \langle 1 \rangle$ for any $s \in \bSf$, and $\theta_\lambda$ to $\scO_{\Delta \tbg^{(1)}}(\lambda)$ for any $\lambda \in X^*(\bT)$. The image of an element $b \in \Br_{\bW}$ will be denoted $\scI_b$. For any $s \in \bSf$ we have
\[
\scI_{T_s^{-1}} = \scZ_s(-\varsigma, \varsigma-\alpha_s) \langle 1 \rangle.
\]

\begin{rmk}
\label{rmk:convention-action-Br}
We insist that for $b,c \in \Br_{\bW}$ we have $\scI_b \star \scI_c \cong \scI_{cb}$. Our normalization is therefore different from that in~\cite{br-Baff}, and rather follows that considered in~\cite{mr1}; see~\cite[\S 3.3]{mr1} for comments.
\end{rmk}

Consider now some element $s \in \bSaff \smallsetminus \bSf$. As explained in~\cite[Lemma~6.2]{riche} or~\cite[Lemma~2.1.1]{bm-loc}, there exist $b \in \Br_{\bW}$ and $t \in \bSf$ such that $T_s = b T_t b^{-1}$. We fix such elements once and for all (for every $s$ as above) and set
\[
\scX_s := \scI_{b^{-1}} \star \scX_t \star \scI_{b}.
\]
(Note the change of orders of the factors, which is intentional and related to the comments in Remark~\ref{rmk:convention-action-Br}.)
In view of~\eqref{eqn:exact-seq-braid-gp-action}, for any $s \in \bSaff$ we have distinguished triangles
\begin{equation}
\label{eqn:exact-seq-braid-gp-action-2}
\scO_{\Delta \tbg^{(1)}} \langle -1 \rangle \to \scX_s \langle 1 \rangle
\to \scI_{T_s} \xrightarrow{[1]}, \quad
\scI_{T_s^{-1}} \to \scX_s \langle 1 \rangle
\to \scO_{\Delta \tbg^{(1)}} \langle 1 \rangle \xrightarrow{[1]}
\end{equation}
in $\Db \Coh^{(\bG \times \Gm)^{(1)}}(\St^{(1)})$.

Recall that we have identified $X^*(\bT^{(1)})$ with $X^*(\bT)$, see~\S\ref{ss:comp-loc-end-proof}. In particular, for any $\lambda \in X^*(\bT)^+$ we have an indecomposable tilting $\bG^{(1)}$-module $\sfT^{(1)}(\lambda)$ of highest weight $\lambda$.

\begin{lem}
\label{lem:free-tilting-SCoh}
For any $\lambda \in X^*(\bT)^+$, there exist $n \in \Z$, $\omega \in \mathbf{\Omega}$ and $s_1, \cdots, s_r \in \bSaff$ such that the object
\[
\sfT^{(1)}(\lambda) \otimes \scO_{\St^{(1)}}
\]
is a direct summand of the object
\[
\scO_{\St^{(1)}} \star \scI_{T_\omega} \star \scX_{s_1} \star \cdots \star \scX_{s_r} \langle n \rangle
\]
in $\Db \Coh^{(\bG \times \Gm)^{(1)}}(\St^{(1)})$.
\end{lem}

\begin{proof}
Consider the projections
\[
p_1, p_2 : \St^{(1)} \to \tbg^{(1)}
\]
on the first and second factor respectively.
To any $\scF \in \Db \Coh^{(\bG \times \Gm)^{(1)}}(\St^{(1)})$ we can associate a Fourier--Mukai transform
\[
(-) \star \scF : \Db \Coh^{\bG \times \Gm}(\tbg^{(1)}) \to \Db \Coh^{\bG \times \Gm}(\tbg^{(1)})
\]
defined by
\[
\scG \star \scF := R(p_2)_* \bigl( L(p_1)^* \scG \lotimes_{\scO_{\St^{(1)}}} \scF \bigr).
\]
(This functor can also be expressed in terms of a similar formula involving the smooth scheme $\tbg^{(1)} \times \tbg^{(1)}$, which justifies that it indeed takes values in bounded complexes.) By standard considerations involving the flat base change theorem and the projection formula, this assignment defines a right action of the monoidal category~\eqref{eqn:DbCoh-St}
%$( \Db \Coh^{\bG \times \Gm}(\St_{\varnothing,\varnothing}), \star )$ 
on $\Db \Coh^{(\bG \times \Gm)^{(1)}}(\tbg^{(1)})$, which justifies our notation.
For any $\scF,\scG$ as above we have a canonical isomorphism
\begin{equation}
\label{eqn:convolution-Lp2}
L(p_2)^*(\scG \star \scF) \cong \bigl( L(p_2)^* \scG \bigr) \star \scF.
\end{equation}

By~\cite[Lemma~6.7]{mr}, there exist $n$, $\omega$, and $s_1, \cdots, s_r$ as in the statement such that $\sfT^{(1)}(\lambda) \otimes \scO_{\tbg^{(1)}}$ is a direct summand of
\[
\scO_{\tbg^{(1)}} \star \scI_{T_\omega} \star \scX_{s_1} \star \cdots \star \scX_{s_r} \langle n \rangle
\]
in $\Db \Coh^{(\bG \times \Gm)^{(1)}}(\tbg^{(1)})$.
Applying the functor $L(p_2)^*$ and using~\eqref{eqn:convolution-Lp2}, we deduce the desired claim.
\end{proof}

\section{Study of braid objects}
\label{sec:braid-gp}
%%%%%%%%%%%%%%%%%%%%%%%%%%%%%%%%%%%%%%%%%%%%%%%%%%%%%

We continue with the setting of Section~\ref{sec:monoidality} (hence, in particular, with our fixed weight $\lambda \in X^*(\bT) \cap A_0$).

%--------------------------------------------------------------
\subsection{Translation functors and $\scD$-modules}
\label{ss:translation-D}
%--------------------------------------------------------------

%Recall that the set of regular elements in $\mathbb{R} \otimes_{\Z} X^*(\bT)$ (i.e.~those with trivial stabilizer for the $\bullet$-action of $\bWaff$) is the union of the $\bWaff$-translates of $A_0$. There exists a canonical right action $*$ of $\bWaff$ on this subset, such that for any $v \in A_0$, $w \in \bWaff$ and $y \in \bWaff$ we have
%\[
%(w \bullet v) * y = wy \bullet v.
%\]
%Recall also that given an alcove $A$ there exists a unique $w \in \bWaff$ such that $A=w \bullet A_0$. Any wall of $A$ is then the image under $w$ of a unique wall of $A_0$; since the latter walls are in a canonical bijection with $\bSaff$, this allows to attach to any wall of $A$ an element in $\bSaff$.

Recall the ``translation bimodules'' of~\S\ref{ss:BSHC}.
The following statement is an analogue of~\cite[Lemma~6.1.2]{bmr} (see also~\cite[Lemma~2.2.3]{bmr2}). (In~\eqref{it:translation-D-out} we consider the standard order on $X^*(\bT)$ determined by our choice of $\fR_+$.)

\begin{lem}
\label{lem:translation-functors-Dmod}
Let $\mu_1,\mu_2,\mu_3 \in X^*(\bT)$ and $J \subset \fRs$ which satisfy
\[
\mathrm{Stab}_{(\bWaff, \bullet)}(\mu_1) = \varnothing \quad \text{and} \quad \mathrm{Stab}_{(\bWaff, \bullet)}(\mu_3) = \bW_J.
\]

\begin{enumerate}
\item
\label{it:translation-D-in}
Assume that $\mu_2$ belongs to the closure of the alcove of $\mu_1$.
For any $\scF$ in $\Db\Modc(\osD_{\varnothing,J}^{\widehat{\mu_1},\widehat{\mu_3}},\bG)$ there exists a canonical isomorphism
\[
\sfP^{\widehat{\mu_2},\widehat{\mu_1}} \star \Gamma_{\varnothing,J}^{\mu_1,\mu_3}(\scF) \cong 
%\Gamma_{\varnothing,J}^{\mu_2,\mu_3}
R(\overline{f}_{\varnothing,J}^{\mu_2,\mu_3})_*(\scO_{\bG/\bB \times \bG/\bP_J}(\mu_2-\mu_1,0) \otimes_{\scO_{\bG/\bB \times \bG/\bP_J}} \scF).
\]
\item
\label{it:translation-D-out}
Assume that $\mu_2$ lies on a wall of the alcove of $\mu_1$, and denote by $\mu_1'$ the image of $\mu_1$ under the reflection in $\bWaff$ whose fixed-point hyperplane contains that wall. If $\mu_1' < \mu_1$, then for any $\scF$ in $\Db\Modc(\osD_{\varnothing,J}^{\widehat{\mu_2},\widehat{\mu_3}},\bG)$ there exists a distinguished triangle
\begin{multline*}
\Gamma_{\varnothing,J}^{\mu_1', \mu_3}(\scO_{\bG/\bB \times \bG/\bP_J}(\mu_1' - \mu_2,0) \otimes_{\scO_{\bG/\bB \times \bG/\bP_J}} \scF) \to \sfP^{\widehat{\mu_1},\widehat{\mu_2}} \star 
%\Gamma_{\varnothing,J}^{\mu_2,\mu_3}
R(\overline{f}_{\varnothing,J}^{\mu_2,\mu_3})_*(\scF) \\
\to \Gamma_{\varnothing,J}^{\mu_1, \mu_3}(\scO_{\bG/\bB \times \bG/\bP_J}(\mu_1 - \mu_2,0) \otimes_{\scO_{\bG/\bB \times \bG/\bP_J}} \scF) \xrightarrow{[1]}
\end{multline*}
in $\Db \HC^{\widehat{\mu_1},\widehat{\mu_3}}$.
\end{enumerate}
\end{lem}

\begin{proof}
\eqref{it:translation-D-in}
Let $\nu$ be the unique dominant weight in $\bWf(\mu_2-\mu_1)$.
By definition of convolution, for any $M$ in $\Db \HC^{\widehat{\mu_1},\widehat{\mu_3}}$, the complex $\sfP^{\widehat{\mu_2},\widehat{\mu_1}} \star M$ is the component corresponding to $(\mu_2,\mu_3)$ of the complex
\[
\sfC^{\wedge}(\sfL(\nu) \otimes \cU\bg) \star M
\]
in the decomposition induced by~\eqref{eqn:decomp-Zwedge-cat}. Now by Remark~\ref{rmk:completion-diag-induced} we have $\sfC^{\wedge}(\sfL(\nu) \otimes \cU\bg) \cong \sfL(\nu) \otimes (\cU\bg)^\wedge$, so that
\[
\sfC^{\wedge}(\sfL(\nu) \otimes \cU\bg) \star M \cong \sfL(\nu) \otimes M.
\]
The rest of the proof is similar to that of~\cite[Lemma~6.1.2(a)]{bmr} (see also~\cite[Lemma~5.5]{br-Hecke} for similar considerations).

\eqref{it:translation-D-out}
The proof is similar, and parallel to that of~\cite[Lemma~6.1.2(b)]{bmr}.
\end{proof}

%Recall the objects $\sfD_s$, $\sfD_s'$ and the distinguished triangles~\eqref{eqn:triangles-Ds-Ds'}.

%--------------------------------------------------------------
\subsection{Geometric description of translation functors}
\label{ss:geometric-translation}
%--------------------------------------------------------------

Let us consider three weights $\mu_1,\mu_2,\mu_3 \in X^*(\bT)$ and two subsets $J,K \subset \fRs$ which satisfy the following conditions:
\begin{enumerate}
\item
$\mathrm{Stab}_{(\bWaff, \bullet)}(\mu_1) = \varnothing$;
\item
$\mathrm{Stab}_{(\bWaff, \bullet)}(\mu_2) = \bW_J$;
\item
$\mathrm{Stab}_{(\bWaff, \bullet)}(\mu_3) = \bW_K$;
\item
$\mu_2$ belongs to the closure of the alcove of $\mu_1$.
\end{enumerate}
Then we have 
%$I \subset J$, hence 
a natural morphism $\tbg \to \tbg_J$, which induces a morphism
\[
q_{J,K} : \St_{\varnothing,K}^{\wedge (1)} \to \St_{J,K}^{\wedge (1)}.
\]
We also have equivalences of categories
\begin{align*}
\Phi^{\widehat{\mu_1},\widehat{\mu_3}} &: \Db \HC^{\widehat{\mu_1},\widehat{\mu_3}} \simto \Db \Coh^{\bG^{(1)}}(\St^{\wedge (1)}_{\varnothing,K}), \\
\Phi^{\widehat{\mu_2},\widehat{\mu_3}} &: \Db \HC^{\widehat{\mu_2},\widehat{\mu_3}} \simto \Db \Coh^{\bG^{(1)}}(\St^{\wedge (1)}_{J,K}),
\end{align*}
see~\eqref{eqn:equiv-HC-Coh},
and ``translation bimodules''
$\sfP^{\widehat{\mu_1},\widehat{\mu_2}} \in \HC^{\widehat{\mu_1},\widehat{\mu_2}}$, $\sfP^{\widehat{\mu_2},\widehat{\mu_1}} \in \HC^{\widehat{\mu_2},\widehat{\mu_1}}$.
%see~\cite[\S 3.5]{br-Hecke}.

\begin{prop}
\label{prop:translation-push-pull}
The following diagrams commute up to isomorphism:
\[
\xymatrix@C=1.5cm{
\Db \HC^{\widehat{\mu_1},\widehat{\mu_3}} \ar[r]_-{\sim}^-{\Phi^{\widehat{\mu_1},\widehat{\mu_3}}} \ar[d]_-{\sfP^{\widehat{\mu_2},\widehat{\mu_1}} \hatstar (-)} & \Db \Coh^{\bG^{(1)}}(\St^{\wedge (1)}_{\varnothing,K}) \ar[d]^-{R(q_{J,K})_*} \\
\Db \HC^{\widehat{\mu_2},\widehat{\mu_3}} \ar[r]_-{\sim}^-{\Phi^{\widehat{\mu_2},\widehat{\mu_3}}} & \Db \Coh^{\bG^{(1)}}(\St^{\wedge (1)}_{J,K}),
}
\]
\[
\xymatrix@C=1.5cm{
\Db \HC^{\widehat{\mu_1},\widehat{\mu_3}} \ar[r]_-{\sim}^-{\Phi^{\widehat{\mu_1},\widehat{\mu_3}}} & \Db \Coh^{\bG^{(1)}}(\St^{\wedge (1)}_{\varnothing,K}) \\
\Db \HC^{\widehat{\mu_2},\widehat{\mu_3}} \ar[r]_-{\sim}^-{\Phi^{\widehat{\mu_2},\widehat{\mu_3}}} \ar[u]^-{\sfP^{\widehat{\mu_1},\widehat{\mu_2}} \hatstar (-)} & \Db \Coh^{\bG^{(1)}}(\St^{\wedge (1)}_{J,K}). \ar[u]_-{L(q_{J,K})^*}
}
\]
\end{prop}

\begin{proof}
The proof of commutativity of the first diagram is similar to that in~\cite[\S 2.2.5]{bmr2}, based on Lemma~\ref{lem:translation-functors-Dmod}\eqref{it:translation-D-in}. The commutativity of the second one follows, by adjunction (see the comments after Lemma~\ref{lem:translation-bimod-simples}.)
%(In fact, as in~\cite[Lemma~3.6]{br-Hecke} one checks that the functor $\sfP^{\hla,\hmu} \hatstar (-)$ is left and right adjoint to the functor $\sfP^{\hmu,\hla} \hatstar (-)$.)
\end{proof}

Let us note the following consequence. Here, for $s \in \bSaff$ we denote by $\scX_s^\wedge$ the image of $\scX_s$ under the pullback functor~\eqref{eqn:pullback-convolution} (for $\bH=(\bG \times \Gm)^{(1)}$).
%~\eqref{eqn:pullback-St-Stwedge}.
%Taking the Frobenius twists of the constructions of~\S\ref{ss:kernels} we obtain a monoidal category
%\[
%\bigl( \Db \Coh^{\bG^{(1)} \times \Gm^{(1)}}(\St_{\varnothing,\varnothing}^{(1)}), \star \bigr)
%\]
%such that the natural pullback morphism
%\begin{equation}
%\label{eqn:pullback-St-Stwedge}
%\Db \Coh^{\bG^{(1)} \times \Gm^{(1)}}(\St_{\varnothing,\varnothing}^{(1)}) \to
%\Db \Coh^{\bG^{(1)}}(\St_{\varnothing,\varnothing}^{\wedge (1)})
%\end{equation}
%is monoidal. 
%For any $s \in \bSaff$, we will denote by $\scY_s$ the image under~\eqref{eqn:pullback-St-Stwedge} of the Frobenius twist of $\scX_s$.

\begin{cor}
\label{cor:isom-R-X}
For any $s \in \bSf$, we have $\scR_s \cong \scX_s^\wedge$.
\end{cor}

\begin{proof}
We apply Proposition~\ref{prop:translation-push-pull} in the special case $\mu_1=\mu_3=\lambda$, $\mu_2=\mu_s$ (in the notation of~\S\ref{ss:BSHC}), $J=\{s\}$, $K=\varnothing$, and for the object $\scO_{\Delta \tbg^{(1)}}$.
Using the classical fact that the cartesian diagram
\[
\xymatrix{
\tbg \times_{\tbg_{s}} \tbg \ar[r] \ar[d] & \tbg \ar[d] \\
\tbg \ar[r] & \tbg_s
}
\]
is tor-independent in the sense of~\cite[Definition~3.10.2]{lipman} (which follows from the same considerations as in Remark~\ref{rmk:derived-fiber-prod})
%the derived pushforward under the morphism $\tbg \to \tbg_s$ sends the structure sheaf to the structure sheaf 
we see that
\[
\scX_s = L(q'_{\{s\},\varnothing})^* \circ R(q'_{\{s\},\varnothing})_* \scO_{\Delta \tbg^{(1)}},
\]
where $q'_{\{s\},\varnothing} : \St \to \St_{\{s\}, \varnothing}$ is the morphism from which $q_{\{s\},\varnothing}$ is obtained by base change.
The desired claim follows, in view of the definition of $\scR_s$.
\end{proof}

%-----------------------------------------------------------
\subsection{Localization of braid bimodules}
\label{ss:bimodules-W}

Recall the objects $(\sfD_w : w \in \bW)$ in $\Db \HC^{\hla,\hla}$ introduced in~\S\ref{ss:further-properties}.

%Recall the right action of $\bWaff$ on the set of regular weights mentioned in~\S\ref{ss:translation-D}. 
Since the stabilizer of $\lambda$ in $\bW$ is trivial (see Remark~\ref{rmk:assumption-weights-loc}), the orbit $\bW \bullet \lambda$ identifies naturally with $\bW$. Via this identification the right action of $\bW$ on itself given by multiplication on the right defines an action on $\bW \bullet \lambda$, which we will denote by $(\mu,w) \mapsto \mu * w$.
%, and the restriction of the action of $\bWaff$ identifies with right multiplication on $\bW$. This action therefore extends in a natural way to a right action of $\bW$ on $\bW \bullet \lambda$, still denoted $*$.
%Given $\mu \in \bW \bullet \lambda$ and $w \in \bW$, we will say that $w$ \emph{increases $\mu$} if there exists a reduced decomposition $w=\omega s_1 \cdots s_r$ (with $s_1, \cdots, s_r \in \bSaff$ and $\omega \in \Omega$) such that 
%\[
%\mu * \omega < \mu * (\omega s_1) < \cdots < \mu * (\omega s_1 \cdots s_{r-1}) < \mu * (\omega s_1 \cdots s_{r}).
%\]
Given $\mu \in \bW \bullet \lambda$ and $w \in \bW$, we will say that $w$ \emph{decreases $\mu$} if, denoting by $\omega$ the unique element in $\mathbf{\Omega}$ such that $\omega^{-1} w \in \bWaff$, there exists a reduced expression $\omega^{-1} w= s_1 \cdots s_r$ for some $s_1, \cdots, s_r \in \bSaff$ 
%and $\omega \in \Omega$ 
such that 
\[
\mu * \omega > \mu * (\omega s_1) > \cdots > \mu * (\omega s_1 \cdots s_{r-1}) > \mu * (\omega s_1 \cdots s_{r})
\]
(for the standard order on $X^*(\bT)$ determined by our choice of $\fR_+$).

%As in~\cite[\S 2.1.3]{bmr2} we will say that $w$ \emph{increases $\lambda$} if $w=s_1 \cdots s_r\omega$ for some $s_1, \cdots, s_r \in \bSaff$ and $\omega \in \Omega$ such that $\lambda * (s_1 \cdots s_i) < \lambda * (s_1 \cdots s_{i+1})$ for any $i \in \{0, \cdots, r-1\}$.

\begin{lem}
\label{lem:image-Dw}
%For any $w \in \bW$ and $\mu \in \bW \bullet \lambda$ such that $w$ increases $\mu$ we have
%\[
%\Phi^{\widehat{\mu * w}, \hmu} (\sfD_w) \cong \scO_{\Delta \tbg^{\wedge(1)}}.
%\]
For any $w \in \bW$ and $\mu \in \bW \bullet \lambda$ such that $w^{-1}$ decreases $\mu$ we have
\[
\Phi^{\widehat{\mu * w^{-1}}, \hmu} (\sfD_w) \cong \scO_{\Delta \tbg^{\wedge(1)}}.
\]
\end{lem}

\begin{proof}
%\textbf{THE COMBINATORICS DOES NOT QUITE FIT IN THIS PROOF!}
%
We proceed by induction on $\ell(w)$. 
%First, if $w=e$ the claim follows from the monoidality of $\Phi^{\hla,\hla}$. 
If $\ell(w)=0$ we have
\[
\sfD_w = \sfR_w = \sfP^{\hla, \widehat{w \bullet \lambda}} = \sfP^{\widehat{\mu * w^{-1}}, \hmu}
\]
(where we use Remark~\ref{rmk:P-action-W}).
By Lemma~\ref{lem:translation-functors-Dmod}\eqref{it:translation-D-in} and our choice of splitting bundles (see~\S\ref{ss:D-mod-Coh}) we have
\[
\sfP^{\widehat{\mu * w^{-1}}, \hmu} \star (\Phi^{\hmu,\hmu})^{-1}(\scO_{\Delta \tbg^{\wedge(1)}}) \cong (\Phi^{\widehat{\mu * w^{-1}}, \hmu})^{-1}(\scO_{\Delta \tbg^{\wedge(1)}}).
\]
By monoidality, the object $(\Phi^{\hmu,\hmu})^{-1}(\scO_{\Delta \tbg^{\wedge(1)}})$ is the unit object for convolution, from which we deduce the desired isomorphism.
%The desired isomorphism follows.

Now let $w \in \bW$ be an element of positive length, and assume the isomorphism is known for elements of length strictly smaller that $\ell(w)$. 
Since $w^{-1}$ decreases $\mu$, there exists $s \in \bSaff$ such that $\ell(sw)=\ell(w)-1$, $w^{-1}s$ decreases $\mu$ and $\mu * w^{-1} < \mu * (w^{-1} s)$. Then $\sfD_{w} \cong \sfD_s \star \sfD_{sw}$, and
%Choose $s \in \bSaff$ such that $\ell(sw)<\ell(w)$. Then b
by induction we have
\[
\Phi^{\widehat{\mu * (w^{-1} s)},\hmu} (\sfD_{sw}) \cong \scO_{\Delta \tbg^{\wedge(1)}}.
\]
By Lemma~\ref{lem:translation-functors-Dmod} and our choice of splitting bundles there exists a distinguished triangle
\begin{multline*}
(\Phi^{\widehat{\mu * w^{-1}},\hmu})^{-1}(\scO_{\Delta \tbg^{\wedge(1)}}) \to \sfR_s \star (\Phi^{\widehat{\mu * (w^{-1}s)},\hmu})^{-1}(\scO_{\Delta \tbg^{\wedge(1)}}) \\
\to (\Phi^{\widehat{\mu * (w^{-1}s)},\hmu})^{-1}(\scO_{\Delta \tbg^{\wedge(1)}}) \xrightarrow{[1]}
\end{multline*}
in $\Db \HC^{\hla,\hla}$, which can be rewritten as
\[
(\Phi^{\widehat{\mu * w^{-1}},\hmu})^{-1}(\scO_{\Delta \tbg^{\wedge(1)}}) \to \sfR_s \star \sfD_{sw} \to \sfD_{sw} \xrightarrow{[1]}.
\]
Consider the image of this triangle by right multiplication with $\sfN_{w^{-1}s}$ (i.e.~the inverse of $\sfD_{sw}$):
\[
(\Phi^{\widehat{\mu * w^{-1}},\hmu})^{-1}(\scO_{\Delta \tbg^{\wedge(1)}}) \star \sfN_{w^{-1}s} \to \sfR_s \to (\cU \bg)^{\hla} \xrightarrow{[1]}.
\]
We claim that the second morphism in this triangle is a generator of the space of morphisms from $\sfR_s$ to $(\cU\bg)^{\hla}$; this will imply that its cocone identifies with $\sfD_s$, hence that
%From the fact that $\tbS$ is an affine scheme and Remark~\ref{rmk:image-Ds-2} one sees that the image in $\Rep(\bbI^\wedge_{\bS})$ of the first term in this triangle in concentrated in degree $0$, so that the image of the second morphism is surjective, hence a generator of the $\scO(\FN_{\bt^{*(1)}}(\{0\}))$-module of morphisms from the image of $\sfR_s$ to the image of $(\cU \bg)^{\hla}$. Since this functor is fully faithful on $\HC^{\hla,\hla}_{\mathrm{diag}}$ (see Proposition~\ref{prop:restriction-Kostant-HC-ff}), this morphism itself is a generator of the space of morphisms from $\sfR_s$ to $(\cU\bg)^{\hla}$. Hence its cocone identifies with $\sfD_s$. Finally we obtain that
%Since the object $\sfD_{sw}$ is invertible, the second morphism in this triangle is induced by a morphism $\sfR_s \to (\cU\bg)^{\hla}$. The cone of this morphism is isomorphic to $\sfD_s$. (EXPLAIN! THIS SHOULD FOLLOW FROM DESCRIPTION OF ENDOMORPHISMS OF THE 3RD TERM IN THE TRIANGLE.) We deduce that
\[
(\Phi^{\widehat{\mu * w^{-1}},\hmu})^{-1}(\scO_{\Delta \tbg^{\wedge(1)}}) \cong \sfD_s \star \sfD_{sw} \cong \sfD_w,
\]
which will conclude the proof.

The proof of this claim will use the constructions of~\cite{br-Hecke}. Namely, in~\cite[\S 3.9]{br-Hecke} we have constructed an exact functor of ``restriction to a Kostant section'' on $\HC^{\hla,\hla}$, and in~\cite[Proposition~3.7]{br-Hecke} we have proved that this functor is fully faithful on $\HC^{\hla,\hla}_{\mathrm{diag}}$. To conclude, it suffices to prove that the image under this functor of our morphism $\sfR_s \to (\cU \bg)^{\hla}$ is surjective, or in other words that the image of the complex $(\Phi^{\widehat{\mu * w^{-1}},\hmu})^{-1}(\scO_{\Delta \tbg^{\wedge(1)}}) \star \sfN_{w^{-1}s}$ is concentrated in nonpositive degrees. Here $\sfN_{w^{-1}s}$ is itself concentrated in nonpositive degrees by construction; since our functor intertwines convolution in $\Db \HC^{\hla,\hla}$ with a derived version of the convolution considered in~\cite[\S 3.9]{br-Hecke}, it therefore suffices to prove that the image of the complex $(\Phi^{\widehat{\mu * w^{-1}},\hmu})^{-1}(\scO_{\Delta \tbg^{\wedge(1)}})$ is concentrated in nonpositive degrees. In fact this complex is concentrated in degree $0$, as follows from the general form of the base change theorem (see~\cite[Theorem~3.10.3]{lipman}) using the fact that if $\bS^*$ is as in~\cite[\S 3.8]{br-Hecke} the fiber product
$\tbg \times_{\bg^*} \bS^*$ is tor-independent in the sense of~\cite[Definition~3.10.2]{lipman} and an affine scheme. Here the second property follows from the fact that the natural morphism $\tbg \to \bt^*$ restricts to an isomorphism $\tbg \times_{\bg^*} \bS^* \simto \bt^*$, see e.g.~\cite[Proposition~3.5.5]{riche-kostant}, and the first one follows from a standard dimension argument (as e.g.~in Remark~\ref{rmk:derived-fiber-prod}).
\end{proof}

%-----------------------------------------------------------
\subsection{Study of the realization functors}
\label{ss:study-real}
%-----------------------------------------------------------

Consider the composition of natural functors
\begin{equation}
\label{eqn:realization-SHC}
\Kb \SHC^{\hla,\hla} \to \Kb \HC^{\hla,\hla} \to \Db \HC^{\hla,\hla}.
\end{equation}
%The composition of this functor with the obvious embedding $\Db \HC^{\hla,\hla} \to D^- \HC^{\hla,\hla}$ 
This functor is easily seen to be monoidal, and Corollary~\ref{cor:Kb-Db-HC} shows that it is fully faithful.
Consider also the functor
\begin{equation}
\label{eqn:realization-SCoh}
\Kb \SCoh^{\bG^{(1)}}(\St^{\wedge(1)}) \to \Db \Coh^{\bG^{(1)}}(\St^{\wedge(1)})
\end{equation}
%and the 
defined so that the following diagram commutes, where the upper horizontal arrow is induced by the restriction of $\Phi^{\hla,\hla}$, see~\eqref{eqn:BSCoh-BSHC}:
\begin{equation}
\label{eqn:diagram-Phi-real}
\vcenter{
\xymatrix@C=1.5cm{
\Kb \SHC^{\hla,\hla} \ar[r]^-{\sim} \ar[d]_-{\eqref{eqn:realization-SHC}} & \Kb \SCoh^{\bG^{(1)}}(\St^{\wedge(1)}) \ar[d]^-{\eqref{eqn:realization-SCoh}} \\
\Db \HC^{\hla,\hla} \ar[r]^-{\Phi^{\hla,\hla}}_-{\sim} & \Db \Coh^{\bG^{(1)}}(\St^{\wedge(1)}).
}
}
\end{equation}
This functor is also monoidal and fully faithful.

Our next goal is to prove the following result.

\begin{prop}
\label{prop:realization-equiv}
The functors~\eqref{eqn:realization-SHC} and~\eqref{eqn:realization-SCoh} are equivalences of categories.
\end{prop}

%Before explaining the proof we need some preliminaries.

%\begin{rmk}
%In particular, Proposition~\ref{prop:realization-equiv} implies that the subcategory $\Db \HC^{\hla,\hla}$ of $D^- \HC^{\hla,\hla}$ is stable under the monoidal product, hence admits a canonical structure of monoidal category.
%\end{rmk}

The proof of Proposition~\ref{prop:realization-equiv} will use two preliminary lemmas.

\begin{lem}
\label{lem:triang-categories-generators}
Let $\sfD$ be a triangulated category. Assume we are given:
\begin{itemize}
\item
triangulated categories $\sfD_0, \sfD_1, \cdots, \sfD_r$ with $\sfD_0=0$ and $\sfD_r=\sfD$;
\item
triangulated functors $F_i : \sfD_i \to \sfD_{i+1}$ ($i \in \{0, \cdots, r-1\}$);
\item
for each $i \in \{1, \cdots, r\}$, a set $A_i$ and a collection $(X^i_a : a \in A_i)$ of objects of $\sfD_i$ whose images generate the Verdier quotient $\sfD_i/ \langle F_{i-1} \rangle $ as a triangulated category, where $\langle F_{i-1} \rangle$ is the triangulated subcategory generated by the essential image of $F_{i-1}$.
\end{itemize}
% endowed with a filtration
%\[
%0 = \sfD_0 \subset \sfD_1 \subset \cdots \subset \sfD_{r-1} \subset \sfD_r = \sfD
%\]
%by triangulated full subcategories. Assume we are given, for any $i \in \{1, \cdots, r\}$, a set $A_i$ and a collection $(X^i_a : a \in A_i)$ of objects of $\sfD_i$ whose images generate $\sfD_i/\sfD_{i-1}$ as a triangulated category. 
Then the collection
\[
\bigl( F_{r-1} \circ \cdots \circ F_i(X^i_a) : i \in \{1, \cdots, r\}, \, a \in A_i \bigr)
\]
generates $\sfD$ as a triangulated category.
\end{lem}

\begin{proof}
Arguing by induction, we can assume that $r=2$. So we are given a triangulated functor $F_1 : \sfD_1 \to \sfD$, a collection $(X_a : a \in A)$ of objects of $\sfD_1$ which generates $\sfD_1$ as a triangulated category, and a collection $(Y_b : b \in A')$ of objects of $\sfD$ whose images generate $\sfD/\langle F_1 \rangle$. Let $\sfD'$ be the triangulated subcategory of $\sfD$ generated by the objects $(F_1(X_a) : a \in A)$ and $(Y_b : b \in A')$. This subcategory contains $\langle F_1 \rangle$, so that we can consider $\sfD' / \langle F_1 \rangle$ and the natural triangulated functor $\sfD' / \langle F_1 \rangle \to \sfD / \langle F_1 \rangle$. Using the fact that $\sfD'$ is triangulated we see that this functor is fully faithful. Its essential image is therefore a triangulated subcategory of $\sfD / \langle F_1 \rangle$; since it contains a generating family of objects it coincides with the whole of $\sfD / \langle F_1 \rangle$.

If now $X$ is an object of $\sfD$, its image in $\sfD / \langle F_1 \rangle$ is isomorphic to the image of an object $Y$ of $\sfD'$. This means that we have a diagram
\[
X \xleftarrow{f} Z \xrightarrow{g} Y
\]
where the cones of both $f$ and $g$ belong to $\langle F_1 \rangle$. Then $Z$ belongs to $\sfD'$, and $X$ therefore also does.
\end{proof}

Given $w \in \bWf$, we will write ${}^w \hspace{-1pt} \bB$, resp.~${}^w \hspace{-1pt} \bu$, for the conjugate of $\bB$, resp.~$\bu$, by any lift of $w$ in $\mathrm{N}_{\bG}(\bT)$. The following lemma is similar to~\cite[Lemma~4.1]{ahr}; we leave it to the reader to adapt the proof.

\begin{lem}
\label{lem:generators-triang-cat}
For any $w \in \bWf$, the triangulated category
\[
\Db \Coh^{\bB \cap {}^w \hspace{-1pt} \bB} \bigl( (\bg/(\bu + {}^w \hspace{-1pt} \bu))^* \times_{\bt^*} \FN_{\bt^*}(\{0\}) \bigr)
\]
is generated by the objects $\scO_{(\bg/(\bu + {}^w \hspace{-1pt} \bu))^*  \times_{\bt^*} \FN_{\bt^*}(\{ 0 \})} \otimes_\bk \bk_{\bB \cap {}^w \hspace{-1pt} \bB}(\nu)$ for $\nu \in X^*(\bT)$.
\end{lem}

%\begin{proof}
%The proof should be similar to that of~\cite[Lemma~4.1]{ahr}.
%\end{proof}

We can finally give the proof of Proposition~\ref{prop:realization-equiv}.

\begin{proof}[Proof of Proposition~\ref{prop:realization-equiv}]
%By Corollary~\ref{cor:Kb-Db-HC} and its analogue for the triangulated category $\Db \Coh^{\bG^{(1)}}(\St^{\wedge(1)})$ (based on Proposition~\ref{prop:Ext-vanishing-Coh}) 
We have already explained that our functors are fully faithful; it follows that their essential images are triangulated subcategories, so that to conclude it suffices to show that the category $\Db \HC^{\hla,\hla}$, resp.~$\Db \Coh^{\bG^{(1)}}(\St^{\wedge(1)})$, is generated as a triangulated by the subcategory $\SHC^{\hla,\hla}$, resp.~$\SCoh^{\bG^{(1)}}(\St^{\wedge(1)})$. In view of the equivalence $\Phi^{\hla,\hla}$ (see~\S\ref{ss:D-mod-Coh}) the two cases are equivalent, so we concentrate on~\eqref{eqn:realization-SCoh}.

%In view of the diagram~\eqref{eqn:diagram-Phi-real}, the two cases are equivalent, so we concentrate on~\eqref{eqn:realization-SCoh}.
%The fact that the functor is fully faithful follows from Proposition~\ref{prop:Ext-vanishing-Coh} and Be{\u\i}linson's lemma. What remains to be proved is essential surjectivity. 
%%We will denote by $\sfC$ the essential image of our functor. 
%Since this functor is fully faithful and monoidal, its essential image $\sfC$ is a triangulated monoidal subcategory in $\Db \Coh^{\bG^{(1)}}(\St^{\wedge(1)})$.

Let us denote by $\sfC$ the triangulated subcategory of $\Db \Coh^{\bG^{(1)}}(\St^{\wedge(1)})$ generated by $\SCoh^{\bG^{(1)}}(\St^{\wedge(1)})$. Note that $\sfC$ is a \emph{monoidal} subcategory in $\Db \Coh^{\bG^{(1)}}(\St^{\wedge(1)})$.
Recall that $\bG^{(1)}$ has finitely many orbits for the diagonal action on $(\bG/\bB \times \bG/\bB)^{(1)}$. Choose a numbering of these orbits which refines the order given by closure inclusions, and consider the associated filtration
\[
\varnothing = X_0 \subset X_1 \subset X_2 \subset \cdots \subset X_{r-1} \subset X_r = (\bG/\bB \times \bG/\bB)^{(1)}
\]
where each $X_i$ is closed and reduced and each $X_i \smallsetminus X_{i-1}$ is a single $\bG$-orbit. 
%(Here $X_1$ is the diagonal copy of $(\bG/\bB)^{(1)}$.) 
Consider also the pullback
\[
\varnothing = Y_0 \subset Y_1 \subset Y_2 \subset \cdots \subset Y_{r-1} \subset Y_r = \St^{\wedge (1)}
\]
of this filtration to $\St^{\wedge (1)}$. Then each $Y_i \smallsetminus Y_{i-1}$ is of the form
\[
\Bigl( \bG \times^{\bB \cap {}^{w_i} \hspace{-1pt} \bB} \bigl( (\bg/(\bu + {}^{w_i} \hspace{-1pt} \bu))^* \times_{\bt^*} \FN_{\bt^*}(\{0\}) \bigr) \Bigr)^{(1)}
\]
for some $w_i \in \bWf$.

%We will prove by induction on $i$ that $\sfC$ contains the full subcategory of complexes supported on $Y_i$.
%the essential image of the pushforward functor
%\[
%\Db \Coh^{\bG^{(1)}}(Y_i) \to \Db \Coh^{\bG^{(1)}}(\St_{\varnothing,\varnothing}^{\wedge(1)}).
%\]
%First, consider the case $i=1$.
If $\eta \in X^*(\bT)^+$, the element $t_{-\eta}$ decreases the weight $\lambda + \ell\eta=t_\eta \bullet \lambda$.
By Lemma~\ref{lem:image-Dw}, we deduce that
%case $w=t_\eta$ for some $\eta \in X^*(\bT)^+$. Then $w$ increases $\lambda$, so that
\[
\Phi^{\hla,\widehat{\lambda+\ell\eta}} (\sfD_{t_\eta}) \cong \scO_{\Delta \tbg^{\wedge(1)}}.
\]
On the other hand, by~\eqref{eqn:Phi-translation-weights} we have
\[
\Phi^{\hla,\widehat{\lambda+\ell\eta}} (\sfD_{t_\eta}) \cong \scO_{\St^{\wedge(1)}}(0,\eta) \otimes_{\scO_{\St^{\wedge(1)}}} \Phi^{\hla,\hla} (\sfD_{t_\eta}),
\]
which implies that
\begin{equation}
\label{eqn:Phi-D'-dominant}
\Phi^{\hla,\hla} (\sfD_{t_\eta}) \cong \scO_{\Delta \tbg^{\wedge(1)}}(-\eta).
\end{equation}
By construction of $\sfD_{t_\eta}$, 
%and the commutativity of~\eqref{eqn:diagram-Phi-real}, 
$\Phi^{\hla,\hla} (\sfD_{t_\eta})$ belongs to $\sfC$, hence so does $\scO_{\Delta \tbg^{\wedge(1)}}(-\eta)$. By monoidality, $\sfC$ then contains all objects $\scO_{\Delta \tbg^{\wedge(1)}}(\nu)$ with $\nu \in X^*(\bT)$.

Now let $i \in \{1, \cdots, r\}$, consider the element $w_i \in \bWf$, and fix a reduced expression $w_i=s_1 \cdots s_r$. Then using Lemma~\ref{cor:isom-R-X} one sees that the object $\scA_i := \scR_{s_r} \star \cdots \star \scR_{s_1}$ is the pushforward of a complex on $Y_i$, and that its restriction to $Y_i \smallsetminus Y_{i-1}$ is the structure sheaf. Using Lemma~\ref{lem:generators-triang-cat} we deduce that the images in $\Db\Coh^{\bG^{(1)}}(Y_i \smallsetminus Y_{i-1})$ of the objects
\[
\scO_{\Delta \tbg^{\wedge(1)}}(\nu) \star \scA_i
\]
for $\nu \in X^*(\bT)$ generate this category.

Consider now, for any $i$, the pushforward functor
\[
\Db\Coh^{\bG^{(1)}}(Y_{i-1}) \to \Db\Coh^{\bG^{(1)}}(Y_i).
\]
By~\cite[Remark after Lemma~2.12]{arinkin-bezrukavnikov}, the quotient of the target category by the subcategory generated by the essential image of this functor identifies with $\Db\Coh^{\bG^{(1)}}(Y_i \smallsetminus Y_{i-1})$. We are therefore in the setting of Lemma~\ref{lem:triang-categories-generators}, and this result shows that the objects $\scO_{\Delta \tbg^{\wedge(1)}}(\nu) \star \scA_i$ for $\nu \in X^*(\bT)$ and $i \in \{1, \cdots, r\}$ generate $\Db \Coh^{\bG^{(1)}}(\St^{\wedge(1)})$ as a triangulated category. Since these objects all belong to $\sfC$, this shows that $\sfC$ is the whole of $\Db \Coh^{\bG^{(1)}}(\St^{\wedge(1)})$, and therefore finishes the proof.
\end{proof}

%--------------------------------------------------------------
\subsection{Image of braid objects}
\label{ss:image-braid}
%--------------------------------------------------------------

%Recall the braid group $\Br_{\bW}$ and the objects $(\scI_b : b \in \Br_{\bW})$ introduced in~\S\ref{ss:kernels}. 
For $b \in \Br_{\bW}$ we will denote by $\scI_b^\wedge$ the image of $\scI_b$ under the monoidal functor~\eqref{eqn:pullback-convolution}. Then we have
\[
\scI^\wedge_b \star \scI_c^\wedge \cong \scI^\wedge_{cb}
\]
for any $b,c \in \Br_{\bW}$. 
%In~\S\ref{ss:kernels} we have also defined some objects $(\scX_s : s \in \bSaff)$. Their images under~\eqref{eqn:pullback-St-Stwedge} will be denoted $(\scX^\wedge_s : s \in \bSaff)$. 
Recall also the objects $(\scX^\wedge_s : s \in \bSaff)$ introduced in~\S\ref{ss:geometric-translation}, and the objects $(\sfN_b : b \in \Br_{\bW})$ introduced in~\S\ref{ss:further-properties}.

Consider the group anti-automorphism $\imath$ of $\Br_{\bW}$ which satisfies $\imath(T_w) = T_{w^{-1}}$ for any $w \in \bW$. (This map is clearly an involution.)

\begin{lem}
\label{lem:image-braid}
For any $b \in \Br_{\bW}$ we have $\Phi^{\hla,\hla}(\sfN_b) \cong \scI^\wedge_{\imath(b)}$.
\end{lem}

\begin{proof}
It suffices to prove the isomorphism for $b$ in a generating subset of $\Br_{\bW}$; in practice we will use the subset $\{T_s : s \in \bSf\} \cup \{T_{t_{-\nu}} : \nu \in X^*(\bT)^+\}$.

First, consider the case $b=T_s$ for some $s \in \bSf$. By Corollary~\ref{cor:isom-R-X} we have $\Phi^{\hla,\hla}(\sfR_s) \cong \scX_s^\wedge$, hence a distinguished triangle
\[
\scO_{\tbg^{\wedge(1)}} \to \scX_s^\wedge \to \Phi^{\hla,\hla}(\sfN_s) \xrightarrow{[1]}
\]
in which the first morphism is a generator of the space of morphisms from $\scO_{\tbg^{\wedge(1)}}$ to $\scX_s^\wedge$. Comparing this triangle with the image of 
the left triangle in~\eqref{eqn:exact-seq-braid-gp-action} we deduce an isomorphism
\[
\Phi^{\hla,\hla}(\sfN_s) \cong \scI^\wedge_{T_s},
\]
as desired.
If $\nu \in X^*(\bT)^+$, by~\eqref{eqn:Phi-D'-dominant} we have $\Phi^{\hla,\hla} (\sfD_{t_\nu}) \cong \scO_{\Delta \tbg^{\wedge(1)}}(-\nu)$, hence (passing to inverses)
\[
\Phi^{\hla,\hla} (\sfN_{t_{-\nu}}) \cong \scO_{\Delta \tbg^{\wedge(1)}}(\nu) \cong \scI^\wedge_{T_{t_\nu}},
\]
which concludes the proof.
%Since $\{T_s : s \in \bSf\} \cup \{T_{t_{-\nu}} : \nu \in X^*(\bT)^+\}$ generates $\Br_{\bW}$, this concludes the proof.
\end{proof}

In particular, for $\omega \in \mathbf{\Omega}$, applying Lemma~\ref{lem:image-braid} with $b=T_\omega$ we obtain that $\Phi^{\hla,\hla}(\sfN_\omega) \cong \scI^\wedge_{T_{\omega^{-1}}}$, hence that
\begin{equation}
\label{eqn:Romega-I}
\scR_\omega \cong \scI^\wedge_{T_{\omega^{-1}}}.
\end{equation}
Similarly, for $s \in \bSaff$, using Lemma~\ref{lem:image-braid} in case $b \in \{T_s, T_s^{-1}\}$ and applying the functor $\Phi^{\hla,\hla}$ to the triangles in~\eqref{eqn:triangles-Ds-Ds'} we obtain distinguished triangles
\begin{equation}
\label{eqn:triangle-Rs-I}
\scO_{\tbg^{\wedge(1)}} \to \scR_s \to \scI^\wedge_{T_s} \xrightarrow{[1]}, \quad
\scI^\wedge_{T_s^{-1}} \to \scR_s \to \scO_{\tbg^{\wedge(1)}} \xrightarrow{[1]}.
\end{equation}
In fact one can also deduce a generalization of Corollary~\ref{cor:isom-R-X}.

\begin{cor}
\label{cor:isom-R-X-2}
For any $s \in \bSaff$, we have $\scR_s \cong \scX_s^\wedge$.
\end{cor}

\begin{proof}
The case $s \in \bSf$ has been treated in Corollary~\ref{cor:isom-R-X}, hence we can assume that $s \in \bSaff \smallsetminus \bSf$. In this case, to define $\scX_s$ we have fixed $b \in \Br_\bW$ and $t \in \bSf$ such that $T_s = b T_t b^{-1}$, and set $\scX_s = \scI_{b^{-1}} \star \scX_t \star \scI_b$ (see~\S\ref{ss:kernels}); we then have
\begin{equation}
\label{eqn:Xs-conjugation}
\scX_s^\wedge = \scI^\wedge_{b^{-1}} \star \scX^\wedge_t \star \scI^\wedge_b.
\end{equation}
Using Lemma~\ref{lem:image-braid} and the case of $t$ which we have already treated we deduce that
\[
(\Phi^{\hla,\hla})^{-1}(\scX_s^\wedge) \cong \sfN_{\imath(b)^{-1}} \star \sfR_t \star \sfN_{\imath(b)}.
\]
Here $T_s = \imath(b)^{-1} T_t \imath(b)$, hence by
Lemma~\ref{lem:conjugation-sfR} the right-hand side is isomorphic to $\sfR_s$, which finishes the proof.
%Now we have $T_s = \imath(b)^{-1} T_t \imath(b)$, hence
%\[
%\nabla^{\wedge}_{\imath(b)^{-1}} \hatstar \nabla_t^\wedge \hatstar \nabla^\wedge_{\imath(b)} \cong \nabla^\wedge_s.
%\]
%Passing to inverses we also have
%\[
%\nabla^{\wedge}_{\imath(b)^{-1}} \hatstar \Delta_t^\wedge \hatstar \nabla^\wedge_{\imath(b)} \cong \Delta^\wedge_s.
%\]
%Hence from the exact sequences
%\[
%\Delta^\wedge_t \hookrightarrow \Xi^\wedge_{t,!} \twoheadrightarrow \Delta^\wedge_e, \quad \nabla^\wedge_e \hookrightarrow \Xi^\wedge_{t,!} \twoheadrightarrow \nabla^\wedge_t
%\]
%we obtain distinguished triangles showing that the complex $\nabla^{\wedge}_{\imath(b)^{-1}} \hatstar \Xi_{t,!}^\wedge \hatstar \nabla^\wedge_{\imath(b)}$ is a tilting perverse sheaf, and in fact that it must be isomorphic to $\Xi^\wedge_{s,!}$. This concludes the proof since $\Theta^{\hla,\hla}(\Xi_{s,!}^\wedge) \cong \sfR_s$ by construction.
\end{proof}

Let us note the following consequence for later use.

\begin{lem}
\label{lem:free-tilting-Soergel}
For any $V \in \Tilt(\bG^{(1)})$, the object $V \otimes \scO_{\St^{\wedge(1)}}$ belongs to the subcategory $\SCoh^{\bG^{(1)}}(\St^{\wedge(1)})$.
\end{lem}

\begin{proof}
First we consider the case when $V=\bk$ is the trivial module.
By Proposition~\ref{prop:translation-push-pull} (applied with $\mu_1=\mu_3=\lambda$, $\mu_2=-\varsigma$, $J=\fRs$ and $K=\varnothing$) and Remark~\ref{rmk:derived-fiber-prod} we have
\[
\scO_{\St^{\wedge(1)}} = \Phi^{\hla,\hla}(\sfP^{\hla,\widehat{-\varsigma}} \hatstar \sfP^{\widehat{-\varsigma},\hla}).
\]
By Lemma~\ref{lem:SHC-w0} the object $\sfP^{\hla,\widehat{-\varsigma}} \hatstar \sfP^{\widehat{-\varsigma},\hla}$ belongs to $\SHC^{\hla,\hla}$. Hence $\scO_{\St^{\wedge(1)}}$ indeed belongs to $\SCoh^{\bG^{(1)}}(\St^{\wedge(1)})$.

To conclude the proof it then suffices to observe that for any indecomposable $V \in \Tilt(\bG^{(1)})$ there exist $s_1, \cdots, s_r \in \bSaff$ and $\omega \in \mathbf{\Omega}$ such that $V \otimes \scO_{\St^{\wedge(1)}}$ is a direct summand in the complex
\[
\scO_{\St^{\wedge(1)}} \star \scR_\omega \star \scR_{s_1} \star \cdots \star \scR_{s_r}.
\]
In fact, this claim follows from Lemma~\ref{lem:free-tilting-SCoh}, using~\eqref{eqn:Romega-I} and Corollary~\ref{cor:isom-R-X-2}.
\end{proof}

%\begin{lem}
%\label{lem:image-braid}
%For any $b \in \Br_\bW$, the image of $\nabla^\wedge_b$ in $\Db \Coh^{\bG^{(1)}}(\St_{\varnothing,\varnothing}^{\wedge(1)})$ under the equivalence of Theorem~\ref{thm:equivalences-comp} is $\scI^\wedge_b$.
%\end{lem}
%
%THERE SEEMS TO BE A PROBLEM HERE!
%
%\begin{proof}
%By monoidality it suffices to prove the claim for $b$ in a generating subset of $\Br_\bW$. If $b=T_s$ for some $s \in \bSaff$, the claim is obtained by comparing the left triangle in~\eqref{eqn:exact-seq-braid-gp-action-2} with the exact sequence of perverse sheaves
%\[
%\nabla^\wedge_e \hookrightarrow \Xi^\wedge_s \twoheadrightarrow \nabla^\wedge_s.
%\]
%And when $b=T_\nu$ for some $\nu \in X^*(\bT)^+$ the claim has been obtained in the course of the proof of Lemma~\ref{lem:realization-equiv}, see~\eqref{eqn:Phi-D'-dominant}. Since these elements generate $\Br_\bW$, this is sufficient to imply the desired claim.
%\end{proof}

%%%%%%%%%%%%%%%%%%%%%%%%%%%%%%%%%%%%%%%%%%%%%%%%%%%%%
\section{Restriction to a Steinberg section}
\label{sec:Steinberg-section}
%%%%%%%%%%%%%%%%%%%%%%%%%%%%%%%%%%%%%%%%%%%%%%%%%%%%%

We continue with the setting of Sections~\ref{sec:monoidality}--\ref{sec:braid-gp}.
In this section we introduce a functor of ``restriction to a Steinberg section'' and show that it allows to identify the category $\SHC^{\hla,\hla}$ with a suitable category of ``Soergel bimodules.'' This construction is similar to one considered in~\cite{br-Hecke}, but slightly different: in~\cite{br-Hecke} we worked with a \emph{Kostant section} (section to the adjoint quotient for $\bg$) while here we work with a \emph{Steinberg section} (section to the adjoint quotient for $\bG$).
% in order to make a connection with the constructions of~\cite{br-pt2}.
All the results in this section have analogues involving sheaves on varieties attached to the Lie algebra rather than the group and a Kostant section rather than a Steinberg section, which are in a sense more natural. But here we need to work with the group versions in order to make a connection with the constructions of~\cite{br-pt2}.

\subsection{Pseudo-logarithm}
\label{ss:pseudo-log}
%------------------------------------------------

We first explain how one can compare the geometry of some schemes attached to $\bG$ with that of schemes similarly attached to $\bg$ or $\bg^*$. For that
we will assume that there exists a ``pseudo-logarithm,'' i.e.~a $\bG$-equi\-variant morphism
\[
\varphi' : \bG \to \bg
\]
%(where $\bg'$ is the Lie algebra of $\bG'$, and $\bG'$ acts on both sides via the adjoint actions)
which sends $e \in \bG$ to $0 \in \bg$ and is \'etale at $e$, and we fix such a morphism. 

\begin{rmk}
\label{rmk:existence-pseudolog}
The condition above is satisfied at least if one of the following conditions are satisfied:
\begin{itemize}
\item
$\bG=\mathrm{GL}_n(\bk)$;
\item
$\bG$ is semisimple and simply connected and $\ell$ is very good for $\bG$.
\end{itemize}
In fact, in the former case one can take for $\varphi'$ the morphism given by $X \mapsto X-\mathrm{I}_n$. For the latter case the claim clearly reduces to the case $\bG$ is quasi-simple. In this case, if $\bG=\mathrm{SL}_n$ one can take $X \mapsto X-\frac{\mathrm{tr}(X)}{n} \mathrm{I}_n$, and if $\bG$ is not of type $\mathbf{A}$ this is a standard consequence of results of Springer--Steinberg (see e.g.~\cite[\S 5.3]{modrap1} for references).
\end{rmk}

%(For sufficient conditions that guarantee the existence of such a map, see~\cite[Proposition~9.3.3]{bardsley-richardson}.) 

Composing $\varphi'$ with our fixed $\bG$-equivariant isomorphism $\bg \simto \bg^*$ we obtain a $\bG$-equivariant morphism $\varphi : \bG \to \bg^*$. For $I \subset \fRs$ we consider the smooth scheme $\tbG_I$ defined by
\[
\tbG_I = \{(g,h\bP_I) \in \bG \times \bG/\bP_I \mid g \in h\bP_I h^{-1}\}.
\]
(Here the natural morphism $\tbG_I \to \bG/\bP_I$ is Zariski locally trivial, with fibers isomorphic to $\bP_I$, which justifies smoothness.) This scheme is equipped with a natural morphism to $\bG$, and we have an identification
\begin{equation}
\label{eqn:tbG-induced-variety}
\tbG_I = \bG \times^{\bP_I} \bP_I
\end{equation}
where $\bP_I$ acts on itself by conjugation. In particular, from the composition
\[
\bP_I \to \bL_I \to \bL_I/\bL_I \cong \bT/\bW_I,
\]
(where the first map is the projection with kernel $\bU_I$, the second one is the adjoint quotient, and the isomorphism is the standard one, see e.g.~\cite[\S 2.2]{br-pt2})
which is $\bP_I$-equivariant for the trivial action on $\bT/\bW_I$, we deduce a morphism
\begin{equation}
\label{eqn:morph-tbG-adj-quotient}
\tbG_I \to \bT/\bW_I.
\end{equation}
As usual, when $I=\varnothing$ the subscript will often be omitted from notation.

We claim that the morphism $\varphi \times \id_{\bG/\bP_I}$ restricts to a morphism
\[
\widetilde{\varphi}_I : \tbG_I \to \tbg_I.
\]
In fact, since we work with reduced schemes it suffices to check this at the level of $\bk$-points. Now, for any parabolic subgroup $\bP \subset \bG$ there exists a $1$-parameter subgroup $\chi : \Gm \to \bG$ such that $\bP$, resp.~$\mathrm{Lie}(\bP)$, is the attractor for the action of $\Gm$ on $\bG$, resp.~$\bg$, by conjugation via $\chi$, see e.g.~\cite[Proposition~2.2.9]{cgp}. Since $\varphi$ is $\bG$-equivariant it commutes with these actions of $\Gm$, hence sends $\tbG_I$ into $\tbg_I$.

Note also that the morphism $\bG/\bG \to \bg^*/\bG$ (where in both cases we consider the (co)adjoint quotients) is \'etale at the image of $e \in \bG$ by~\cite[Theorem~4.1]{bardsley-richardson}; it therefore induces an isomorphism of schemes
\begin{equation}
\label{eqn:isom-varphi-FN}
\FN_{\bG/\bG}(\{e\}) \xrightarrow{\sim} \FN_{\bg^*/\bG}(\{0\}).
\end{equation}
Here, as in~\S\ref{ss:center-Ug}, the Chevalley isomorphism identifies $\bg^*/\bG$ with $\bt^*/\bWf$, and the adjoint quotient $\bG/\bG$ identifies naturally with $\bT/\bWf$, see the comments above. Hence we can interpret the isomorphism above as an isomorphism of schemes
\[
\FN_{\bT/\bWf}(\{e\}) \simto \FN_{\bt^*/\bWf}(\{0\}).
\]

\begin{lem}
\label{lem:isom-varphi-tbG}
For any $I \subset \fRs$, the morphism $\widetilde{\varphi}_I$ induces an isomorphism of schemes
\[
\FN_{\bG/\bG}(\{e\}) \times_{\bG/\bG} \tbG_I \simto \FN_{\bg^*/\bG}(\{0\}) \times_{\bg^*/\bG} \tbg_I.
\]
\end{lem}

\begin{proof}
The claim follows from (the positive-characteristic version of) Luna's \'etale slice theorem. More specifically one can proceed as follows.
By~\cite[Theorem~6.2]{bardsley-richardson}, there exists $f \in \scO(\bg^*/\bG) = \scO(\bg^*)^\bG$ with value $1$ at $0$ and such that $\varphi$ induces a (surjective) \'etale morphism $(\bG/\bG)_{f \circ \varphi} \to (\bg^*/\bG)_f$ and an isomorphism of affine varieties
\[
\bG_{f \circ \varphi} \simto (\bG/\bG)_{f \circ \varphi} \times_{(\bg^*/\bG)_f} (\bg^*)_f,
\]
where the subscripts mean the open subschemes defined by the given function. Consider the induced isomorphism
\[
\bG_{f \circ \varphi} \times \bG/\bP_I \simto (\bG/\bG)_{f \circ \varphi} \times_{(\bg^*/\bG)_f} \bigl( (\bg^*)_f \times \bG/\bP_I \bigr).
\]
We claim that the preimage of the closed subscheme
\begin{multline*}
(\bG/\bG)_{f \circ \varphi} \times_{\bg^*/\bG} \tbg_I = (\bG/\bG)_{f \circ \varphi} \times_{(\bg^*/\bG)_f} ((\bg^*/\bG)_f \times_{\bg^*/\bG} \tbg_I) \\
 \subset (\bG/\bG)_{f \circ \varphi} \times_{(\bg^*/\bG)_f} \bigl( (\bg^*)_f \times \bG/\bP_I \bigr)
\end{multline*}
is the closed subscheme
\[
(\bG/\bG)_{f \circ \varphi} \times_{\bG/\bG} \tbG_I \subset \bG_{f \circ \varphi} \times \bG/\bP_I.
\]
In fact, all the schemes under consideration are smooth (note that the scheme $(\bG/\bG)_{f \circ \varphi} \times_{(\bg^*/\bG)_f} \tbg_I$ admits an \'etale morphism to the smooth scheme $\tbg_I$, hence is smooth), in particular reduced, so that this claim can be checked at the level of $\bk$-points. Then it follows from the same considerations as above for $\widetilde{\varphi}_I$, involving attractors.

From this claim we deduce that $\widetilde{\varphi}_I$ induces an isomorphism
\begin{equation}
\label{eqn:isom-varphi-opens}
(\bG/\bG)_{f \circ \varphi} \times_{\bG/\bG} \tbG_I \simto (\bG/\bG)_{f \circ \varphi} \times_{\bg^*/\bG} \tbg_I.
\end{equation}
Now we observe that the natural projection
\[
\FN_{\bG/\bG}(\{e\}) \times_{\bG/\bG} (\bG/\bG)_{f \circ \varphi} \to \FN_{\bG/\bG}(\{e\})
\]
is an isomorphism; in fact the ring of functions on $\FN_{\bG/\bG}(\{e\}) \times_{\bG/\bG} (\bG/\bG)_{f \circ \varphi}$ identifies with the localization of $\scO(\FN_{\bG/\bG}(\{e\}))$ with respect to the multiplicative subset generated by the image of $f \circ \varphi$, but this element is already invertible. 
%A similar argument applies for $\bg^*/\bG$.
Using this claim, from the isomorphism~\eqref{eqn:isom-varphi-opens} we deduce an isomorphism
\[
\FN_{\bG/\bG}(\{e\}) \times_{\bG/\bG} \tbG_I \simto \FN_{\bG/\bG}(\{e\}) \times_{\bg^*/\bG} \tbg_I.
\]
Composing with the isomorphism induced by the isomorphism in~\eqref{eqn:isom-varphi-FN} we deduce the desired isomorphism.
\end{proof}

\begin{rmk}
\label{rmk:identification-G-g}
In case $I=\fRs$, the isomorphism of Lemma~\ref{lem:isom-varphi-tbG} reads
\[
\FN_{\bG/\bG}(\{e\}) \times_{\bG/\bG} \bG \simto \FN_{\bg^*/\bG}(\{0\}) \times_{\bg^*/\bG} \bg^*.
\]
\end{rmk}

%------------------------------------------------
\subsection{The multiplicative Steinberg variety}
%------------------------------------------------

We set
\[
\Stm := \tbG \times_{\bG} \tbG, \quad \Stm^\wedge := \Stm \times_{\bG/\bG} \FN_{\bG/\bG}(\{e\}).
\]
As in remark~\ref{rmk:St-fiber-product} we have a canonical isomorphism
\[
 \Stm^\wedge := \Stm \times_{\bT^{(1)} \times_{\bT^{(1)}/\bWf} \bT^{(1)}} \FN_{\bT^{(1)} \times_{\bT^{(1)}/\bWf} \bT^{(1)}}(\{(e,e)\})
\]
We consider these schemes as ``multiplicative versions'' of the schemes $\St$ and $\St^\wedge$; they are endowed with natural (diagonal) actions of $\bG$. We will more specifically consider the bounded derived categories
\[
\Db \Coh^{\bG^{(1)}}(\Stm^{(1)}), \quad \Db \Coh^{\bG^{(1)}}(\Stm^{\wedge(1)})
\]
of $\bG^{(1)}$-equivariant coherent sheaves on $\Stm^{(1)}$ and $\Stm^{\wedge(1)}$ respectively. We will denote by $\cU \subset \bG^{(1)}$ the unipotent cone, i.e.~the preimage of the image of $e$ under the quotient morphism $\bG^{(1)} \to \bG^{(1)}/\bG^{(1)}$, and by $\Coh_{\cU}^{\bG^{(1)}}(\Stm^{(1)})$ the full subcategory of $\Coh^{\bG^{(1)}}(\Stm^{(1)})$ whose objects are the coherent sheaves supported set-theoretically on $\cU$.

The following statement is an analogue of Lemma~\ref{lem:DbCoh-nil}, and admits the same proof.

\begin{lem}
\label{lem:DbCoh-unip}
The obvious functor
\[
\Db \Coh^{\bG^{(1)}}_{\cU}(\Stm^{(1)}) \to \Db \Coh^{\bG^{(1)}}(\Stm^{\wedge(1)})
\]
is fully faithful; its essential image is the full subcategory whose objects are the complexes $\scF$ such that the morphism
\[
\scO(\FN_{\bT^{(1)} \times_{\bT^{(1)}/\bWf} \bT^{(1)}}(\{(e,e)\})) \to \End(\scF)
\]
vanishes on a power of the unique maximal ideal.
\end{lem}

%The morphism $\varphi$ induces a morphism
%\[
%\Stm \to \St.
%\]
%Moreover, since 
As in~\eqref{eqn:Stwedge-fiber-prod} we have a natural identification
\[
\Stm^\wedge = (\tbG \times_{\bG/\bG} \FN_{\bG/\bG}(\{e\})) \times_{\bG \times_{\bG/\bG} \FN_{\bG/\bG}(\{e\})} (\tbG \times_{\bG/\bG} \FN_{\bG/\bG}(\{e\}));
\]
Lemma~\ref{lem:isom-varphi-tbG} therefore implies that we have a $\bG$-equivariant isomorphism
\begin{equation}
\label{eqn:identification-Stm-St}
\Stm^\wedge \simto \St^\wedge,
\end{equation}
%This isomorphism is clearly $\bG$-equivariant, hence it 
which induces an equivalence of triangulated categories
\begin{equation}
\label{eqn:Coh-St-Stm}
\Db \Coh^{\bG^{(1)}}(\Stm^{\wedge(1)}) \simto \Db\Coh^{\bG^{(1)}}(\St^{\wedge(1)}).
\end{equation}
The same constructions as in Section~\ref{sec:monoidality} allow to define a natural convolution product on the category $\Db \Coh^{\bG^{(1)}}(\Stm^{\wedge(1)})$ (which will once again be denoted $\star$), and this equivalence has a canonical monoidal structure. We will denote by
\[
\Psi^{\hla,\hla} : \Db \HC^{\hla,\hla} \simto \Db \Coh^{\bG^{(1)}}(\Stm^{\wedge(1)})
\]
the composition of $\Phi^{\hla,\hla}$ (see~\S\ref{ss:D-mod-Coh}) with the inverse of this equivalence. We will denote by $\Delta \tbG^{\wedge} \subset \Stm^\wedge$ the diagonal copy of $\tbG \times_{\bG/\bG} \FN_{\bG/\bG}(\{e\})$ in $\Stm^\wedge$, so that the monoidal unit in $\Db \Coh^{\bG^{(1)}}(\Stm^{\wedge(1)})$ is $\scO_{\Delta \tbG^{\wedge(1)}}$

For any $s \in \bSaff$, resp.~$\omega \in \mathbf{\Omega}$, we will denote by $\scS_s$, resp.~$\scS_\omega$, the inverse image of $\scR_s$, resp.~$\scR_\omega$, under~\eqref{eqn:Coh-St-Stm}.
Then for $s \in \bSaff$ and $\omega,\omega' \in \mathbf{\Omega}$ we have
\begin{equation*}
\scS_\omega \star \scS_s \star \scS_{\omega^{-1}} \cong \scS_{\omega s \omega^{-1}}, \quad \scS_\omega \star \scS_{\omega'} \cong \scS_{\omega \omega'}.
\end{equation*}
We define the category
\[
\BSCoh^{\bG^{(1)}}(\Stm^{\wedge(1)})
\]
as the strictly full subcategory of $\Db \Coh^{\bG^{(1)}}(\Stm^{\wedge(1)})$ generated under the monoidal product $\star$ by the unit object and the objects $\scS_s$ ($s \in \bSaff$) and $\scS_\omega$ ($\omega \in \mathbf{\Omega}$). Any object in this category is isomorphic to an object
$\scS_{s_1} \star \cdots \star \scS_{s_r} \star \scS_\omega$
where $s_1, \cdots, s_r \in \bSaff$ and $\omega \in \mathbf{\Omega}$. We will also denote by
\[
\SCoh^{\bG^{(1)}}(\Stm^{\wedge(1)})
\]
the karoubian envelope of the additive hull of $\BSCoh^{\bG^{(1)}}(\Stm^{\wedge(1)})$.
Then~\eqref{eqn:Coh-St-Stm} induces equivalences of categories
\begin{align}
\BSCoh^{\bG^{(1)}}(\Stm^{\wedge(1)}) &\simto \BSCoh^{\bG^{(1)}}(\St^{\wedge(1)}), \\
\label{eqn:equiv-SCoh-St-Stm}
 \SCoh^{\bG^{(1)}}(\Stm^{\wedge(1)}) &\simto \SCoh^{\bG^{(1)}}(\St^{\wedge(1)}).
\end{align}
%It also follows from Proposition~\ref{prop:Ext-vanishing} that
%for any $\scF,\scG$ in $\SCoh^{\bG^{(1)}}(\Stm^{\wedge(1)})$ and $n \in \Z \smallsetminus \{0\}$ we have
%\[
%\Hom_{\Db \Coh^{\bG^{(1)}}(\Stm^{\wedge(1)})}(\scF,\scF' [n]) = 0.
%\]
Proposition~\ref{prop:realization-equiv} also implies that we have an equivalence of monoidal categories
\begin{equation}
\label{eqn:realization-SCoh-Stm}
\Kb \SCoh^{\bG^{(1)}}(\Stm^{\wedge(1)}) \simto \Db \Coh^{\bG^{(1)}}(\Stm^{\wedge(1)}).
\end{equation}

%-------------------------------------------------------
\subsection{Restriction to a Steinberg section}
\label{ss:restriction-Steinberg}
%-------------------------------------------------------

Let $\mathbf{\Sigma} \subset \bG$ be a Steinberg section as in~\cite[\S 2.2]{br-pt2}, and set $\Sigma = \mathbf{\Sigma}^{(1)} \subset \bG^{(1)}$. Then we have the universal centralizer $\bbJ_{\mathbf{\Sigma}}$, a smooth affine group scheme over $\mathbf{\Sigma}$ equipped with a canonical closed immersion of group schemes $\bbJ_{\mathbf{\Sigma}} \hookrightarrow \bG \times \mathbf{\Sigma}$, see~\cite[\S 2.8]{br-pt2}. The composition
\[
\mathbf{\Sigma} \hookrightarrow \bG \to \bG/\bG \cong \bT/\bWf
\]
is an isomorphism, where the second morphism is the adjoint quotient. In particular there exists a canonical morphism
\[
\FN_{\bT^{(1)} \times_{\bT^{(1)}/\bWf} \bT^{(1)}}(\{(e,e)\}) \to \Sigma,
\]
and we will denote by $\bbI_\Sigma^\wedge$ the pullback of $\bbJ_{\mathbf{\Sigma}}^{(1)}$ to $\FN_{\bT^{(1)} \times_{\bT^{(1)}/\bWf} \bT^{(1)}}(\{(e,e)\})$. (The present $\bbI_\Sigma^\wedge$ is therefore the Frobenius twist of the group scheme $\bbI_{\mathbf{\Sigma}}^\wedge$ of~\cite[\S 3.3]{br-pt2}.) Then we can consider the category $\Rep(\bbI_\Sigma^\wedge)$ of coherent representations of this group scheme.

By~\cite[Lemma~3.3]{br-pt2} we have a canonical identification
\[
\FN_{\bT^{(1)} \times_{\bT^{(1)}/\bWf} \bT^{(1)}}(\{(e,e)\}) \cong \FN_{\bT^{(1)}}(\{e\}) \times_{\FN_{\bT^{(1)}/\bWf}(\{e\})} \FN_{\bT^{(1)}}(\{e\}),
\]
so that a representation of $\bbI_{\Sigma}^\wedge$ is a bimodule over the regular ring $\scO(\FN_{\bT^{(1)}}(\{e\}))$ with some extra structure. As a consequence, the derived tensor product of bimodules induces a monoidal structure on the derived category
\[
\Db \Rep(\bbI_{\Sigma}^\wedge).
\]
In~\cite[\S 3.3]{br-pt2} we have considered a full subcategory 
\[
\SRep(\bbI_{\Sigma}^{\wedge})
\]
of $\Rep(\bbI_\Sigma^\wedge)$, which is stable under the convolution product on $\Db \Rep(\bbI_{\Sigma}^\wedge)$.
(More precisely, we apply Frobenius twists everywhere in these constructions from~\cite{br-pt2}.)
%(This category is denoted $\SRep(\bbJ_{\mathsf{D}}^\wedge)$ in~\cite{br-pt2}.) 
This category contains some distinguished objects $(\scB^\wedge_s : s \in \bSaff)$ and $(\scM^\wedge_w : w \in \bW)$, which satisfy in particular
\begin{equation}
\label{eqn:convolution-Ms}
\scM^\wedge_y \star \scM^\wedge_w \cong \scM^\wedge_{yw} \quad \text{for any $y,w \in \bW$.}
\end{equation}
%It contains some distinguished objects associated with elements in $\bSaff$ and in $\Omega$.

It follows from~\cite[Proposition~2.13]{br-pt2} that the natural morphism $\tbG \to \bT$ restricts to an isomorphism
$\tbG \times_{\bG} \mathbf{\Sigma} \simto \bT$. In particular we deduce a natural closed immersion
\[
\bT \times_{\bT/\bWf} \bT \hookrightarrow \Stm,
\]
and then a closed immersion
\begin{equation}
\label{eqn:section-Stm}
\FN_{\bT^{(1)} \times_{\bT^{(1)}/\bWf} \bT^{(1)}}(\{(e,e)\}) \hookrightarrow \Stm^{\wedge(1)}.
\end{equation}
It follows also from~\cite[Proposition~2.13]{br-pt2} that
the restriction to the closed subscheme $\FN_{\bT^{(1)} \times_{\bT^{(1)}/\bWf} \bT^{(1)}}(\{(e,e)\})$
of the universal stabilizer for the action of $\bG^{(1)}$ on $\Stm^{\wedge (1)}$ identifies canonically with $\bbI_{\Sigma}^{\wedge}$.
Using the general construction spelled out e.g.~in~\cite[\S 2.2]{mr},
we deduce that there exists a natural functor
\begin{equation}
\label{eqn:restriction-Stm-centralizer-Ab}
\Coh^{\bG^{(1)}}(\Stm^{\wedge (1)}) \to \Rep(\bbI_{\Sigma}^{\wedge})
\end{equation}
%which, as for~\eqref{eqn:functor-restriction-S} and~\eqref{eqn:restriction-S-oD}, can be shown to be 
which is exact. (In fact, in the noncompleted setting, this functor identifies, via the third equivalence in~\cite[Proposition~2.20]{br-pt2}, with restriction to an open subscheme. The completed case can be treated similarly.) It therefore induces a functor
\begin{equation}
\label{eqn:restriction-Stm-centralizer}
\Db \Coh^{\bG^{(1)}}(\Stm^{\wedge (1)}) \to \Db \Rep(\bbI_{\Sigma}^{\wedge}),
\end{equation}
which is easily seen to be monoidal.

Our goal in the rest of this section is to prove the following claim.

\begin{prop}
\label{prop:restriction-SCoh-Sigma}
The functor~\eqref{eqn:restriction-Stm-centralizer} restricts to an equivalence of monoidal categories
\[
\SCoh^{\bG^{(1)}}(\Stm^{\wedge (1)}) \simto \SRep(\bbI_{\Sigma}^{\wedge}).
\]
\end{prop}

%----------------------------------------------------
\subsection{Restriction of Harish-Chandra bimodules to the Steinberg section}
%----------------------------------------------------

For any $\mu,\nu \in X^*(\bT)$ we have the algebra $\sfU^{\hmu,\hnu}$, which can be considered as a coherent sheaf of algebras on the scheme
\[
\bg^{*(1)} \times_{\bt^{*(1)}/\bWf} \FN_{\bt^*/(\bWf,\bullet) \times_{\bt^{*(1)}/\bWf} \bt^*/(\bWf,\bullet)}(\{(\tmu,\tnu)\}).
\]
Now this scheme identifies with
\begin{multline*}
(\bg^{*(1)} \times_{\bt^{*(1)}/\bWf} \FN_{\bt^{*(1)}/\bWf}(\{0\})) \times_{\FN_{\bt^{*(1)}/\bWf}(\{0\})} \\
\FN_{\bt^*/(\bWf,\bullet) \times_{\bt^{*(1)}/\bWf} \bt^*/(\bWf,\bullet)}(\{(\tmu,\tnu)\}),
\end{multline*}
which by Remark~\ref{rmk:identification-G-g} identifies with
\begin{multline*}
(\bG^{(1)} \times_{\bT^{(1)}/\bWf} \FN_{\bT^{(1)}/\bWf}(\{e\})) \times_{\FN_{\bt^{*(1)}/\bWf}(\{0\})} \\
\FN_{\bt^*/(\bWf,\bullet) \times_{\bt^{*(1)}/\bWf} \bt^*/(\bWf,\bullet)}(\{(\tmu,\tnu)\}).
\end{multline*}
We will denote by $\sfU_\Sigma^{\hmu,\hnu}$ the restriction of $\sfU^{\hmu,\hnu}$ to the (affine) closed subscheme
\begin{multline*}
(\Sigma \times_{\bT^{(1)}/\bWf} \FN_{\bT^{(1)}/\bWf}(\{e\})) \times_{\FN_{\bt^{*(1)}/\bWf}(\{0\})} \\
\FN_{\bt^*/(\bWf,\bullet) \times_{\bt^{*(1)}/\bWf} \bt^*/(\bWf,\bullet)}(\{(\tmu,\tnu)\})
\end{multline*}
(which identifies with $\FN_{\bt^*/(\bWf,\bullet) \times_{\bt^{*(1)}/\bWf} \bt^*/(\bWf,\bullet)}(\{(\tmu,\tnu)\})$).
If we set
\[
\bbM_\Sigma := \bbJ_{\mathbf{\Sigma}}^{(1)} \times_{\bG^{(1)} \times \Sigma} (\bG \times \Sigma)
\]
(where the morphism $\bG \times \Sigma \to \bG^{(1)} \times \Sigma$ is $\Fr_{\bG} \times \id$),
then this algebra is endowed with a natural action of the group scheme
\begin{multline*}
\bbM_\Sigma^{\hmu,\hnu} := (\bbM_\Sigma \times_{\bT^{(1)}/\bWf} \FN_{\bT^{(1)}/\bWf}(\{e\})) \times_{\FN_{\bt^{*(1)}/\bWf}(\{0\})} \\
\FN_{\bt^*/(\bWf,\bullet) \times_{\bt^{*(1)}/\bWf} \bt^*/(\bWf,\bullet)}(\{(\tmu,\tnu)\}).
\end{multline*}

We will denote by $\Modfg^{\bbM}(\sfU_\Sigma^{\hmu,\hnu})$ the abelian category of equivariant finitely generated modules over $\sfU_\Sigma^{\hmu,\hnu}$, and by $\HC_\Sigma^{\hmu,\hnu}$ its subcategory of Harish-Chandra bimodules, i.e.~objects such that the restriction of the action of $\bbM_\Sigma^{\hmu,\hnu}$ to the subgroup
\[
\bG_1 \times \FN_{\bt^*/(\bWf,\bullet) \times_{\bt^{*(1)}/\bWf} \bt^*/(\bWf,\bullet)}(\{(\tmu,\tnu)\})
\]
coincides with the restriction of the action of $\sfU_\Sigma^{\hmu,\hnu}$ via the natural morphism $\cU\bg \to \sfU_\Sigma^{\hmu,\hnu}$. Then we have a natural (exact) restriction functor
\[
\Modfg^{\bG}(\sfU^{\hmu,\hnu}) \to \Modfg^{\bbM}(\sfU_\Sigma^{\hmu,\hnu}),
\]
which restricts to an exact functor
\begin{equation}
\label{eqn:rest-Sigma-HC}
\HC^{\hmu,\hnu} \to \HC_\Sigma^{\hmu,\hnu}.
\end{equation}

The following statement is an analogue of~\cite[Proposition~3.7]{br-Hecke}.

\begin{prop}
\label{prop:restriction-HC-Sigma}
For any $\mu,\nu \in X^*(\bT)$, the functor~\eqref{eqn:rest-Sigma-HC} is fully faithful on the subcategory $\HC^{\hmu,\hnu}_{\mathrm{diag}}$.
\end{prop}

\begin{proof}
The proof is similar to that of~\cite[Proposition~3.7]{br-Hecke}. Namely, we can consider $\sfU^\wedge$ as a sheaf of algebras on the scheme
\begin{multline*}
(\bg^{*(1)} \times_{\bt^{*(1)}/\bWf} \FN_{\bt^{*(1)}/\bWf}(\{0\})) \times_{\FN_{\bt^{*(1)}/\bWf}(\{0\})} \\
\Bigl( \FN_{\bt^{*(1)}/\bWf}(\{0\}) \times_{\bt^{*(1)}/\bWf} (\bt^*/(\bWf,\bullet) \times_{\bt^{*(1)}/\bWf} \bt^*/(\bWf,\bullet)) \Bigr),
\end{multline*}
which by Remark~\ref{rmk:identification-G-g} identifies with
\begin{multline*}
(\bG^{(1)} \times_{\bT^{(1)}/\bWf} \FN_{\bT^{(1)}/\bWf}(\{e\})) \times_{\FN_{\bt^{*(1)}/\bWf}(\{0\})} \\
\Bigl( \FN_{\bt^{*(1)}/\bWf}(\{0\}) \times_{\bt^{*(1)}/\bWf} (\bt^*/(\bWf,\bullet) \times_{\bt^{*(1)}/\bWf} \bt^*/(\bWf,\bullet)) \Bigr).
\end{multline*}
One can then follow exactly the same arguments as in~\cite{br-Hecke}, replacing everywhere the algebra $\cU\bg$ by its pullback to the scheme
$\bG^{(1)} \times_{\bt^{*(1)}/\bWf} \bt^*/(\bWf,\bullet)$.
\end{proof}

%----------------------------------------------------
\subsection{Localization over the Steinberg section}
%----------------------------------------------------

Recall from~\S\ref{ss:HC-Dmod} the sheaf of algebras $\osD^{\hla,\hla}_{\varnothing,\varnothing}$ over the scheme
\[
\St^{(1)} \times_{\bt^{*(1)} \times_{\bt^{*(1)}/\bWf} \bt^{*(1)}} \FN_{\bt^*/(\bWf,\bullet) \times_{\bt^{*(1)}/\bWf} \bt^*/(\bWf,\bullet)}(\{(\tla,\tla)\}),
\]
which 
%identifies with $\St^{\wedge(1)}$, hence
%Using 
by Lemma~\ref{lem:isom-varphi-tbG} identifies with
\[
\Stm^{\wedge(1)} \times_{\FN_{\bt^{*(1)} \times_{\bt^{*(1)}/\bWf} \bt^{*(1)}}(\{(0,0)\})} \FN_{\bt^*/(\bWf,\bullet) \times_{\bt^{*(1)}/\bWf} \bt^*/(\bWf,\bullet)}(\{(\tla,\tla)\}).
\]
%one sees that this scheme identifies canonically 
%with $\Stm^{\wedge(1)}$. 
We will denote by $\osD_{\Sigma}^{\hla,\hla}$ the restriction of $\osD^{\hla,\hla}_{\varnothing,\varnothing}$ along the closed immersion of $\FN_{\bt^*/(\bWf,\bullet) \times_{\bt^{*(1)}/\bWf} \bt^*/(\bWf,\bullet)}(\{(\tla,\tla)\})$ into this scheme
induced by~\eqref{eqn:section-Stm}. This algebra admits a canonical action of the group scheme $\bbM_\Sigma^{\hla,\hla}$; we can therefore consider the category $\Modc^{\bbM}(\osD_{\Sigma}^{\hla,\hla})$ of (weakly) equivariant modules for this sheaf of algebras, and its subcategory $\Modc(\osD_{\Sigma}^{\hla,\hla}, \bbM)$ of modules on which the action of the subgroup scheme $\bG_1 \times \FN_{\bt^*/(\bWf,\bullet) \times_{\bt^{*(1)}/\bWf} \bt^*/(\bWf,\bullet)}(\{(\tla,\tla)\})$ coincides with the action coming from that of $\osD_{\Sigma}^{\hla,\hla}$. We then have restriction functors
\begin{equation}
\label{eqn:restriction-osD-Sigma}
\Modc^{\bG}(\osD^{\hla,\hla}_{\varnothing,\varnothing}) \to \Modc^{\bbM}(\osD_{\Sigma}^{\hla,\hla}), \quad \Modc(\osD^{\hla,\hla}_{\varnothing,\varnothing}, \bG) \to \Modc(\osD_{\Sigma}^{\hla,\hla}, \bbM).
\end{equation}

In~\S\ref{ss:D-mod-Coh} we have constructed a splitting bundle for $\osD^{\hla,\hla}_{\varnothing,\varnothing}$, which by restriction provides a splitting bundle for $\osD_{\Sigma}^{\hla,\hla}$. Using this splitting bundle we obtain equivalences of categories
\[
\Modfg^{\bbM}(\osD_\Sigma^{\hla,\hla}) \simto \Rep(\bbM_\Sigma^{\hla,\hla}), \quad \Modc(\osD_{\Sigma}^{\hla,\hla}, \bbM) \simto \Rep(\bbI^{\hla,\hla}_\Sigma),
\]
where $\bbI^{\hla,\hla}_\Sigma$ is the pullback of $\bbJ_{\mathbf{\Sigma}}^{(1)}$ along the natural morphism
\[
\FN_{\bt^*/(\bWf,\bullet) \times_{\bt^{*(1)}/\bWf} \bt^*/(\bWf,\bullet)}(\{(\tla,\tla)\}) \to \Sigma,
\]
i.e.~the quotient of $\bbM_\Sigma^{\hla,\hla}$ by the normal subgroup scheme
\[
\bG_1 \times \FN_{\bt^*/(\bWf,\bullet) \times_{\bt^{*(1)}/\bWf} \bt^*/(\bWf,\bullet)}(\{(\tla,\tla)\}).
\]
Now we have an identification
\[
\FN_{\bT^{(1)} \times_{\bT^{(1)}/\bWf} \bT^{(1)}}(\{(e,e)\}) \cong \FN_{\bt^*/(\bWf,\bullet) \times_{\bt^{*(1)}/\bWf} \bt^*/(\bWf,\bullet)}(\{(\tla,\tla)\}),
\]
and the pullback of $\bbI^{\hla,\hla}_\Sigma$ along this isomorphism is $\bbI^\wedge_\Sigma$. We therefore finally obtain an equivalence of categories
\[
\Modc(\osD_{\Sigma}^{\hla,\hla}, \bbM) \simto \Rep(\bbI^\wedge_\Sigma)
\]
such that the diagram
\[
\xymatrix{
\Modc(\osD^{\hla,\hla}_{\varnothing,\varnothing}, \bG) \ar[d]_-{\eqref{eqn:restriction-osD-Sigma}} \ar[r]^-{\sim} & \Coh^{\bG^{(1)}}(\Stm^{\wedge(1)}) \ar[d]^-{\eqref{eqn:restriction-Stm-centralizer-Ab}} \\
\Modc(\osD_{\Sigma}^{\hla,\hla}, \bbM) \ar[r]^-{\sim} & \Rep(\bbI_\Sigma^\wedge)
}
\]
commutes, where the upper horizontal arrow is the composition of the equivalence of Proposition~\ref{prop:splitting-HC-equiv} with (the restriction of)~\eqref{eqn:Coh-St-Stm}.

Now the scheme $\FN_{\bt^*/(\bWf,\bullet) \times_{\bt^{*(1)}/\bWf} \bt^*/(\bWf,\bullet)}(\{(\tla,\tla)\})$ is affine, and using the computation of the derived global sections of $\tsD_\varnothing$ (see~\cite[Proposition~3.4.1]{bmr}) and the general form of the base change theorem (see~\cite[Theorem~3.10.3]{lipman}) one checks that the global sections of $\osD_{\Sigma}^{\hla,\hla}$ identify canonically with $\sfU^{\hla,\hla}_\Sigma$. We deduce equivalences of categories
\[
\Modfg^{\bbM}(\osD_\Sigma^{\hla,\hla}) \simto \Modfg^{\bbM}(\sfU_\Sigma^{\hla,\hla}), \quad
\Modc(\osD_{\Sigma}^{\hla,\hla}, \bbM) \simto \HC_\Sigma^{\hla,\hla}
\]
such that the following diagram commutes:
\[
\xymatrix{
\Db \Modc(\osD^{\hla,\hla}_{\varnothing,\varnothing}, \bG) \ar[d]_-{\eqref{eqn:restriction-osD-Sigma}} \ar[r]^-{\Gamma^{\lambda,\lambda}_{\varnothing,\varnothing}}_-{\sim} & \Db \HC^{\hla,\hla} \ar[d]^-{\eqref{eqn:rest-Sigma-HC}} \\
\Db\Modc(\osD_{\Sigma}^{\hla,\hla}, \bbM) \ar[r]^-{\sim} & \Db\HC_\Sigma^{\hla,\hla}.
}
\]

Combining the above informations, we therefore obtain an equivalence of categories
\[
\HC_\Sigma^{\hla,\hla} \simto \Rep(\bbI_{\Sigma}^{\wedge})
\]
such that the following diagram commutes:
\[
\xymatrix@C=1.5cm{
\Db \HC^{\hla, \hla} \ar[d]_-{\eqref{eqn:rest-Sigma-HC}}  \ar[r]^-{\Psi^{\hla,\hla}}_-{\sim} & \Db \Coh^{\bG^{(1)}}(\Stm^{\wedge(1)}) \ar[d]^-{\eqref{eqn:restriction-Stm-centralizer}} \\
\Db \HC_\Sigma^{\hla,\hla} \ar[r]^-{\sim} & \Db \Rep(\bbI_{\Sigma}^{\wedge}).
}
\]

From this commutative diagram and Proposition~\ref{prop:restriction-HC-Sigma} we deduce the following.

\begin{cor}
\label{cor:restriction-SCoh-Sigma-ff}
The functor~\eqref{eqn:restriction-Stm-centralizer} is fully faithful on
$\SCoh^{\bG^{(1)}}(\Stm^{\wedge (1)})$.
\end{cor}

%---------------------------------------------
\subsection{Image of Soergel-type coherent sheaves}
%---------------------------------------------

In view of Corollary~\ref{cor:restriction-SCoh-Sigma-ff},
to conclude the proof of Proposition~\ref{prop:restriction-SCoh-Sigma} it remains to describe the essential image of $\SCoh^{\bG^{(1)}}(\Stm^{\wedge (1)})$. In fact, by construction of this subcategory, it will be enough to prove that for any $s \in \bSaff$, resp.~$\omega \in \mathbf{\Omega}$, this functor sends the object $\scS_s$, resp.~$\scS_\omega$, to $\scB^\wedge_s$, resp.~$\scB^\wedge_\omega$.

Recall the objects $(\scX_s^\wedge : s \in \bSaff)$ in $\Db\Coh^{\bG^{(1)}}(\St^{\wedge(1)})$ considered in~\S\ref{ss:geometric-translation}. By Corollary~\ref{cor:isom-R-X-2}, for any $s \in \bSaff$, $\scS_s$ is isomorphic to the image of $\scX_s^\wedge$ in $\Db \Coh^{\bG^{(1)}}(\Stm^{\wedge(1)})$.

In case $s \in \bSf$, this image is the structure sheaf of the closed subscheme
\[
(\tbG^{(1)} \times_{\tbG_s^{(1)}} \tbG^{(1)} ) \times_{\bG^{(1)}/\bG^{(1)}} \FN_{\bG^{(1)}/\bG^{(1)}}(\{e\}) \subset \Stm^{\wedge(1)},
\]
where we write $\tbG_s$ for $\tbG_{\{s\}}$. Now the preimage of $\mathbf{\Sigma}$ in $\tbG_s$ identifies (via~\eqref{eqn:morph-tbG-adj-quotient}) with $\bT/\{1,s\}$. Hence the image of $\scS_s$ under~\eqref{eqn:restriction-Stm-centralizer} is the structure sheaf of the closed subscheme
\[
\FN_{\bT^{(1)} \times_{\bT^{(1)}/\{1,s\}} \bT^{(1)}}(\{(e,e)\}) \subset \FN_{\bT^{(1)} \times_{\bT^{(1)}/\bWf} \bT^{(1)}}(\{(e,e)\}).
\]
Looking at the definition of $\scB^\wedge_s$ in this case, we deduce the following.

\begin{lem}
\label{lem:scS-scB}
For any $s \in \bSf$, the functor~\eqref{eqn:restriction-Stm-centralizer} sends $\scS_s$ to $\scB^\wedge_s$.
\end{lem}

%\begin{proof}
%Let $s \in \bSf$. Using Corollary~\ref{cor:isom-R-X} we obtain that $\scS_s$ is the structure sheaf of the closed subscheme
%\[
%(\tbG^{(1)} \times_{\tbG_s^{(1)}} \tbG^{(1)} ) \times_{\bG^{(1)}/\bG^{(1)}} \FN_{\bG^{(1)}/\bG^{(1)}}(\{e\}) \subset \St^{\wedge(1)},
%\]
%where we write $\tbG_s$ for $\tbG_{\{s\}}$. Now it is easily seen that the preimage of $\Sigma$ identifies (via~\eqref{eqn:morph-tbG-adj-quotient}) with $\bT/\{1,s\}$. Hence the image of $\scS_s$ is the structure sheaf of the closed subscheme
%\[
%\FN_{\bT^{(1)} \times_{\bT^{(1)}/\{1,s\}} \bT^{(1)}}(\{(1,1)\}) \subset \FN_{\bT^{(1)} \times_{\bT^{(1)}/\bWf} \bT^{(1)}}(\{(1,1)\}),
%\]
%endowed with the trivial action of $\bbI^\wedge_\Sigma$. By definition, this is $\scB^\wedge_s$.
%\end{proof}

Recall now the objects $(\scI_b^\wedge : b \in \Br_{\bW})$ considered in~\S\ref{ss:image-braid}. For any $b \in \Br_{\bW}$ we will denote by $\scJ^\wedge_b$ the image of $\scI^\wedge_b$ in $\Db \Coh^{\bG^{(1)}}(\Stm^{\wedge(1)})$, and by $\scJ_b^\Sigma$ the image of $\scJ^\wedge_b$ in $\Db \Rep(\bbI_\Sigma^\wedge)$. We will also denote by $\jmath : \Br_{\bW} \to \bW$ the canonical group morphism, such that $\jmath(T_w) = w$ for any $w \in \bW$.

%Let us denote by $\jmath : \Br_\bW \to \bW$ the unique ring anti-morphism which sends $T_w$ to $w^{-1}$ for any $w \in \bW$.

\begin{lem}
\label{lem:scJ-scM}
For any $b \in \Br_\bW$, we have $\scJ^\Sigma_b \cong \scM_{\jmath \circ \imath(b)}^\wedge$.
\end{lem}

\begin{proof}
In view of~\eqref{eqn:convolution-Ms}, it suffices to prove the isomorphism for $b$ in a generating subset of $\Br_{\bW}$. Here we will use the subset~\eqref{eqn:generators-BrW}.

If $s \in \bSf$, using the identification of the image of $\scX_s^\wedge$ discussed above, from the exact sequences~\eqref{eqn:exact-seq-braid-gp-action} we obtain exact sequences
\[
\scJ_e^\Sigma \hookrightarrow \scB^\wedge_s \twoheadrightarrow \scJ_{T_s}^\Sigma, \qquad \scJ_{T_s^{-1}}^\Sigma \hookrightarrow \scB^\wedge_s \twoheadrightarrow \scJ_{e}^\Sigma.
\]
Comparing with the exact sequences appearing in~\cite[\S 3.2]{br-pt2}, we deduce isomorphisms
\[
\scJ_{T_s}^\Sigma \cong \scM_s^\wedge \cong \scJ_{T_s^{-1}}^\Sigma.
\]
On the other hand, for $\mu \in X^*(\bT)$, the object $\scJ^\wedge_{\theta_\mu}$ is the pushforward under the diagonal embedding
\[
\Bigl( \tbG \times_{\bG/\bG} \FN_{\bG/\bG}(\{e\}) \Bigr)^{(1)} \to \Stm^{\wedge(1)}
\]
of the pullback of the line bundle on $(\bG/\bB)^{(1)}$ associated with $\mu$. By~\cite[Lemma~2.21]{br-pt2}, the image of this object in $\Rep(\bbI_\Sigma^\wedge)$ is $\scM^\wedge_{t_\mu}$, which concludes the proof.
\end{proof}

We can now conclude the proof of Proposition~\ref{prop:restriction-SCoh-Sigma}. In fact, for $\omega \in \mathbf{\Omega}$ we have $\scS_\omega \cong \scJ^\wedge_{T_{\omega^{-1}}}$ by~\eqref{eqn:Romega-I}, so that its image in $\SRep(\bbI_{\Sigma}^{\wedge})$ in $\scJ^\Sigma_{T_{\omega^{-1}}}$, which is isomorphic to $\scM_{\omega}^\wedge$ by Lemma~\ref{lem:scJ-scM}. For $s \in \bSf$, we have already proved that the image of $\scS_s$ is $\scB_s^\wedge$ in Lemma~\ref{lem:scS-scB}. Finally, in case $s \in \bSaff \smallsetminus \bSf$, recall the elements $b \in \Br_\bW$ and $t \in \bSf$ such that $T_s = b T_t b^{-1}$, see~\S\ref{ss:kernels}. In view of~\eqref{eqn:Xs-conjugation}, Lemma~\ref{lem:scS-scB} and Lemma~\ref{lem:scJ-scM}, the image of $\scS_s$ is then
\[
\scM^\wedge_{\jmath \circ \imath(b)^{-1}} \star \scB^\wedge_t \star \scM^\wedge_{\jmath \circ \imath(b)}.
\]
Here we have $s = \jmath \circ \imath(b)^{-1} \cdot t \cdot \jmath \circ \imath(b)$, hence by~\cite[Lemma~3.5]{br-pt2} the right-hand side is isomorphic to $\scB^\wedge_s$, which concludes the proof.

% FIN RELECTURE

%%%%%%%%%%%%%%%%%%%%%%%%%%%%%%%%%%%%%%%%%%%%%%%%%%%%%
\section{Equivalences}
\label{sec:equiv}
%%%%%%%%%%%%%%%%%%%%%%%%%%%%%%%%%%%%%%%%%%%%%%%%%%%%%

We continue with the setting of Section~\ref{sec:Steinberg-section}, hence in particular with the assumption introduced in~\S\ref{ss:pseudo-log}. In order to apply the results of~\cite{br-pt2}, we also assume from now on that $\ell \neq 19$, resp.~$\ell \neq 31$, if $\fR$ contains a component of type $\mathbf{E}_7$, resp.~$\mathbf{E}_8$. We also consider our fixed element
%In particular we assume $\ell \geq h$, and $\lambda$ is a fixed element in 
$\lambda \in A_0 \cap X^*(\bT)$.

%-----------------------------------------------------------
\subsection{Relation with constructible sheaves}
\label{ss:relation-const-sheaves}
%-----------------------------------------------------------

%Using Proposition~\ref{prop:Ext-vanishing-Coh}, standard arguments (see~\cite[\S 2.5]{amrw}) show that there exists a triangulated functor
%\begin{equation}
%\label{eqn:realization-SCoh}
%\Kb \SCoh^{\bG^{(1)}}(\St^{\wedge(1)}) \to \Db \Coh^{\bG^{(1)}}(\St^{\wedge(1)})
%\end{equation}
%whose restriction to $\SCoh^{\bG^{(1)}}(\St^{\wedge(1)})$ is the obvious embedding. Inspecting the construction of this functor one sees that it admits a natural monoidal structure. 

Consider now a reductive group $G$ over an algebraically closed field $\F$ of positive characteristic $p \neq \ell$ with a maximal torus $T$ such that $G^\vee_\bk = \bG^{(1)}$ (where $G^\vee_\bk$ is the Langlands dual reductive group provided by the geometric Satake equivalence of~\cite{mv} applied to the group $G$). In particular we have identifications $X_*(T)=X^*(\bT^{(1)}) =X^*(\bT)$. We will denote by $B \subset G$ the Borel subgroup containing $T$ whose roots are the coroots of $\bB$. Note that the extended affine Weyl group $\bW$ used above then identifies with the group $W$ of~\cite[\S 4.1]{br-pt2}, and all the structures related to this group identify in the obvious way. From now on we will thus use the conventions of~\cite{br-pt2}, and often write $W$ for $\bW$, $\Wf$ for $\bWf$, $\Omega$ for $\mathbf{\Omega}$, etc.~when the considerations are unrelated to the geometry of the group $\bG$.

We will use the constructions and notation introduced in~\cite{brr-pt1,br-pt2} for this group $G$. In particular we consider the category $\sfD^\wedge_{\Iwu, \Iwu}$ of sheaves on the natural $T$-torsor $\widetilde{\Fl}_G$ over the affine flag variety $\Fl_G$ of $G$, and its subcategory $\sfT^\wedge_{\Iwu, \Iwu}$ of tilting perverse sheaves; the former category contains some distinguished objects $(\Xi^\wedge_{s,!} : s \in \Saff)$ labelled by $\Saff$ (see~\cite[\S 11.1]{br-pt2}), and two families $(\nabla^\wedge_w : w \in W)$ and $(\Delta^\wedge_w : w \in W)$ of objects labelled by $W$ (see~\cite[\S 6.2]{br-pt2}). Here, $\Xi^\wedge_{s,!}$ belongs to $\sfT^\wedge_{\Iwu, \Iwu}$ for any $s \in \Saff$, and for $\omega \in \Omega$ we have $\Delta^\wedge_\omega = \nabla^\wedge_\omega$ and this object also belongs to $\sfT^\wedge_{\Iwu, \Iwu}$. 
%Below we will apply the main result of~\cite{br-pt2} to $G$; for this statement to apply, in addition to our running assumptions, we will assume that $\ell \neq 19$, resp.~$\ell \neq 31$, if $\fR$ contains a component of type $\mathbf{E}_7$, resp.~$\mathbf{E}_8$.

\begin{rmk}
In~\cite{br-pt2}, the geometric constructions of the first part of the paper were applied to the group $G^\vee_\bk$, while here they are applied to the group $\bG$ whose Frobenius twist is $G^\vee_\bk$. This is due to the fact that in the present paper we also consider the category of Harish-Chandra bimodules, which should be considered for $\bG$ and not $G^\vee_\bk$. There is no contradiction since in practice all varieties that already appeared in~\cite{br-pt2} will appear here with a Frobenius twist, and the Frobenius twist of the Grothendieck, resp.~Steinberg, etc.~variety of $\bG$ is the Grothendieck, resp.~Steinberg, etc.~variety of $\bG^{(1)} = G^\vee_\bk$.
\end{rmk}

The category $\sfD^\wedge_{\Iwu, \Iwu}$ admits a monoidal product $\hatstar$ (see~\cite[\S 6.1]{br-pt2}), and $\sfT^\wedge_{\Iwu, \Iwu}$ is a monoidal subcategory. We also have a monoidal equivalence of categories
\begin{equation}
\label{eqn-D-KbTilt}
\Kb \sfT^\wedge_{\Iwu, \Iwu} \simto
\sfD^\wedge_{\Iwu, \Iwu},
\end{equation}
see~\cite[Equation~(6.11)]{br-pt2}. 

%Let $\Sigma \subset \bG$ be a Steinberg section as in~\cite[\S 2.2]{br-pt2}. Then we have the universal centralizer $\bbJ_\Sigma$, a smooth affine group scheme over $\Sigma$ equipped with a canonical closed immersion of group schemes $\bbJ_\Sigma \hookrightarrow \bG \times \Sigma$, and applying a natural variant of the construction of the category $\SRep(\bbI_{\bS}^{\wedge})$ we obtain a category
%\[
%\SRep(\bbI_{\Sigma}^{\wedge}),
%\]
%see~\cite[\S 3.3]{br-pt2}. (This category is denoted $\SRep(\bbJ_{\mathsf{D}}^\wedge)$ in~\cite{br-pt2}.) It contains some distinguished objects associated with elements in $\bSaff$ and in $\Omega$.
Using the notation of~\S\ref{ss:restriction-Steinberg}, in~\cite[Theorem~11.2]{br-pt2} we have constructed an equivalence of monoidal categories
\begin{equation}
\label{eqn:equiv-pt2}
\sfT^\wedge_{\Iwu, \Iwu} \simto \SRep(\bbI_{\Sigma}^{\wedge})
\end{equation}
which sends $\Xi^\wedge_s$ to $\scB_s^\wedge$ for any $s \in \Saff$ and $\Delta^\wedge_\omega$ to $\scM^\wedge_\omega$ for any $\omega \in \Omega$.
%matching the distinguished objects associated with elements in $\bSaff$ and in $\Omega$.

%\begin{prop}
%\label{prop:equiv-unip-nilp}
%There exists an equivalence of categories
%\[
%\SRep(\bbI_{\Sigma}^{\wedge}) \simto \SRep(\bbI_{\bS}^{\wedge})
%\]
%which matches the distinguished objects associated with elements in $\bSaff$ and in $\Omega$.
%\end{prop}
%
%\begin{proof}
%Such an equivalence is constructed in~\cite[Equation~(3.19)]{br-pt2}.\footnote{\textbf{This construction this requires $\ell$ very good, i.e.~$\ell>h$, cf.~\cite[Remark~3.18]{br-pt2}.}} \textbf{MAYBE WE NEED A MORE CLEVER CONSTRUCTION!! (Cf.~Lemma~\ref{lem:free-tilting-SRep}.)}
%\end{proof}

%given a description of the monoidal category $(\sfT^\wedge_{\Iwu, \Iwu}, \hatstar)$ in terms of a certain ``Hecke category.'' By~\cite[Equation~(3.19)]{br-pt2} this category identifies with the category $\SRep(\bbI_{\bS}^{\wedge})$ of~\S\ref{ss:relation-SCoh-SRep}; we therefore have an equivalence of monoidal categories
%\[
%\sfT^\wedge_{\Iwu, \Iwu} \simto \SRep(\bbI_{\bS}^{\wedge}).
%\]
Combining the equivalences~\eqref{eqn:equiv-pt2}, \eqref{eqn:equiv-SCoh-St-Stm} and~\eqref{eqn:BSCoh-BSHC}
%~\eqref{eqn:restriction-Kostant-Coh-S} and~\eqref{eqn:equiv-pt2} 
with that of Proposition~\ref{prop:restriction-SCoh-Sigma} we deduce equivalences of monoidal categories
\begin{equation}
\label{eqn:equivalences-Soergel-cats}
\sfT^\wedge_{\Iwu, \Iwu} \simto \SCoh^{\bG^{(1)}}(\Stm^{\wedge(1)}) \simto \SCoh^{\bG^{(1)}}(\St^{\wedge(1)}) \simto \SHC^{\hla,\hla}
\end{equation}
which, for $s \in \Saff$ and $\omega \in \Omega$, match our distinguished objects in the following way:
\[
\Xi^\wedge_{s,!} \leftrightarrow \scS_s \leftrightarrow \scR_s \leftrightarrow \sfR_s, \quad \Delta^\wedge_{\omega} \leftrightarrow \scS_\omega \leftrightarrow \scR_\omega \leftrightarrow \sfR_\omega.
\]

We finally obtain the following theorem, which is the main result of this paper.
%using Proposition~\ref{prop:realization-SCoh-equiv}, Remark~\ref{rmk:equiv-SHC-HC}, and~\eqref{eqn-D-KbTilt} we deduce the following statement.

\begin{thm}
\label{thm:equivalences-comp}
There exist equivalences of monoidal triangulated categories
\[
\sfD^\wedge_{\Iwu, \Iwu} 
\xrightarrow[\sim]{\Theta^{\hla,\hla}} 
%\simto 
\Db \Coh^{\bG^{(1)}}(\Stm^{\wedge(1)})
\simto
\Db \Coh^{\bG^{(1)}}(\St^{\wedge(1)})
\xleftarrow[\sim]{\Phi^{\hla,\hla}}
\Db \HC^{\hla,\hla}.
\]
\end{thm}

\begin{proof}
The first equivalence is obtained from the first equivalence in~\eqref{eqn:equivalences-Soergel-cats} by passing to homotopy categories and conjugating by the equivalences~\eqref{eqn-D-KbTilt} and~\eqref{eqn:realization-SCoh-Stm}. The second equivalence is~\eqref{eqn:Coh-St-Stm}. Finally, the equivalence in the right-hand side was constructed in~\S\ref{ss:D-mod-Coh}.
\end{proof}

\begin{rmk}
\label{rmk:matching-standards-costandards}
For $s \in \Saff$,
comparing the triangles~\eqref{eqn:triangles-Ds-Ds'} with similar triangles in $\sfD^\wedge_{\Iwu,\Iwu}$ one sees that under the equivalence $\sfD^\wedge_{\Iwu, \Iwu} \simto \Db \HC^{\hla,\hla}$ of Theorem~\ref{thm:equivalences-comp} the object $\Delta^\wedge_s$, resp.~$\nabla^\wedge_s$, corresponds to $\sfD_s$, resp.~$\sfN_s$. By monoidality it then follows that, for any $w \in W$, the object $\Delta^\wedge_w$, resp.~$\nabla^\wedge_w$, corresponds to $\sfD_w$, resp.~$\sfN_w$. By Lemma~\ref{lem:image-braid}, the object in $\Db \Coh^{\bG^{(1)}}(\St^{\wedge(1)})$ corresponding to $\Delta^\wedge_w$ and $\sfD_w$, resp.~to $\nabla^\wedge_w$ and $\sfN_w$, is $\scI^\wedge_{T_w^{-1}}$, resp.~$\scI^\wedge_{T_{w^{-1}}}$.
\end{rmk}

%\begin{proof}
%The equivalence is obtained from the equivalences in~\eqref{eqn:equivalences-Soergel-cats} by passing to bounded homotopy categories and using the equivalences~\eqref{eqn:realization-SCoh}, \eqref{eqn:realization-SHC} and~\eqref{eqn-D-KbTilt}.
%\end{proof}

%In view of Remark~\ref{rmk:equiv-SHC-HC}, we also have an equivalence of monoidal triangulated categories
%\begin{equation}
%\label{eqn:equiv-D-HC}
%\sfD^\wedge_{\Iwu, \Iwu} \simto \Db \HC^{\hla,\hla}.
%\end{equation}

%The first equivalence in Theorem~\ref{thm:equivalences-comp} is the functor denoted $\Theta^{\hla,\hla}$ above.

Once the equivalences in Theorem~\ref{thm:equivalences-comp} are established we can ``forget about completions'' in the following sense. Recall the category $\sfD_{\Iwu, \Iwu}$ and its monoidal product $\star_{\Iwu}$ considered in~\cite[\S 5.1]{br-pt2}.

\begin{thm}
\label{thm:equivalences}
The equivalences of Theorem~\ref{thm:equivalences-comp} restrict to equivalences of monoidal categories
\[
\sfD_{\Iwu, \Iwu} 
%\simto
\xrightarrow[\sim]{\Theta_{\mathrm{nil}}^{\hla,\hla}} 
\Db \Coh_{\cU}^{\bG^{(1)}}(\Stm^{(1)}) \simto \Db \Coh_{\cN}^{\bG^{(1)}}(\St^{(1)})
\xleftarrow[\sim]{\Phi_{\mathrm{nil}}^{\hla,\hla}} \Db \HC_{\mathrm{nil}}^{\hla,\hla}.
\]
\end{thm}

\begin{proof}
The equivalences in Theorem~\ref{thm:equivalences-comp} are linear for the actions of the local ring
\[
\scO(\bZ^{\hla,\hla}) \cong \scO(\FN_{\bt^{*(1)} \times_{\bt^{*(1)}/\bWf} \bt^{*(1)}}(\{(0,0)\})) \cong \scO(\FN_{\bT^{(1)} \times_{\bT^{(1)}/\bWf} \bT^{(1)}}(\{(e,e)\}));
\]
they therefore restrict to equivalences between the full subcategories of objects on which the action of the maximal ideal of this ring is nilpotent. In the case of $\sfD^\wedge_{\Iwu, \Iwu}$, this subcategory has been identified with $\sfD_{\Iwu, \Iwu}$ in~\cite[Lemma~6.2]{br-pt2}. For $\Db \HC^{\hla,\hla}$, this subcategory has been identified with $\Db \HC_{\mathrm{nil}}^{\hla,\hla}$ in Lemma~\ref{lem:HC-central-char}. Finally, for $\Db \Coh^{\bG^{(1)}}(\St^{\wedge(1)})$, resp.~$\Db \Coh^{\bG^{(1)}}(\Stm^{\wedge(1)})$, this subcategory has been identified with $\Db \Coh_{\cN}^{\bG^{(1)}}(\St^{(1)})$, resp.~$\Db \Coh_{\cU}^{\bG^{(1)}}(\Stm^{(1)})$, in Lemma~\ref{lem:DbCoh-nil}, resp.~in Lemma~\ref{lem:DbCoh-unip}.
\end{proof}

%--------------------------------------------------------------
\subsection{Some objects associated with tilting \texorpdfstring{$\bG^{(1)}$}{G(1)}-modules}
%--------------------------------------------------------------

Recall from~\cite[Proposition~7.9]{br-pt2} that for any $V \in \Tilt(\bG^{(1)})$ we have a tilting perverse sheaf
\[
\scZ^\wedge(V) \hatstar \Xi_!^\wedge \quad \in \sfT^\wedge_{\Iwu,\Iwu}.
\]
(Here $\Xi_!^\wedge$ is the ``big tilting pro-object'' on the basic affine space of $G$, see~\cite[\S 6.6]{br-pt2}, and $\scZ^\wedge$ is the pro-monodromic variant of Gaitsgory's central functor constructed in~\cite{gaitsgory}, see~\cite[\S 7.4]{br-pt2}.)
On the other hand, we can consider the equivariant coherent sheaf
\[
V \otimes \scO_{\Stm^{\wedge(1)}} \quad \in \Db \Coh^{\bG^{(1)}}(\Stm^{\wedge(1)}),
\]
which by Lemma~\ref{lem:free-tilting-Soergel} belongs to $\SCoh^{\bG^{(1)}}(\Stm^{\wedge(1)})$. Recall also that we have a canonical isomorphism
\begin{equation*}
\scO(\FN_{\bT^{(1)} \times_{\bT^{(1)}/\bWf} \bT^{(1)}}(\{(e,e)\})) \simto \End_{\sfT^\wedge_{\Iwu,\Iwu}}(\Xi_!^\wedge)
\end{equation*}
induced by left and right monodromy operations (this is the main result of~\cite{br-soergel}; see also~\cite[Equation~(8.7)]{br-pt2}). On the other hand, the projection morphism
\[
\Stm^{\wedge(1)} \to \FN_{\bT^{(1)} \times_{\bT^{(1)}/\bWf} \bT^{(1)}}(\{(e,e)\})
\]
induces a canonical algebra morphism
\[
\scO(\FN_{\bT^{(1)} \times_{\bT^{(1)}/\bWf} \bT^{(1)}}(\{(e,e)\})) \to \End_{\Coh^{\bG^{(1)}}(\Stm^{\wedge(1)})}(\scO_{\Stm^{\wedge(1)}}).
\]

\begin{prop}
\label{prop:image-central}
For any $V \in \Tilt(\bG^{(1)})$ there exists a functorial (in $V$) isomorphism
\[
%\Phi^{\hla,\hla} \circ 
\Theta^{\hla,\hla}
(\scZ^\wedge(V) \hatstar \Xi_!^\wedge) \cong V \otimes \scO_{\Stm^{\wedge(1)}},
\]
which moreover has the property that for any $r \in \scO(\FN_{\bT^{(1)} \times_{\bT^{(1)}/\bWf} \bT^{(1)}}(\{(e,e)\}))$ with images $f \in \End(\Xi_!^\wedge)$ and $g \in \End(\scO_{\Stm^{\wedge(1)}})$ the following diagram commutes:
\[
\xymatrix{
\Theta^{\hla,\hla}(\scZ^\wedge(V) \hatstar \Xi_!^\wedge) \ar[r]^-{\sim} \ar[d]_-{\Theta^{\hla,\hla}(\id \hatstar f)} & V \otimes \scO_{\Stm^{\wedge(1)}} \ar[d]^-{\id \otimes g} \\
\Theta^{\hla,\hla}(\scZ^\wedge(V) \hatstar \Xi_!^\wedge) \ar[r]^-{\sim} & V \otimes \scO_{\Stm^{\wedge(1)}}.
}
\]
\end{prop}

%Before doing so we prove the following lemma
%
%Restricting to the Kostant section, Lemma~\ref{lem:free-tilting-Soergel} implies in particular for any $V \in \Tilt(\bG^{(1)})$, the object
%\[
%V \otimes \scO_{\FN_{\bt^{*(1)} \times_{\bt^{*(1)}/\bWf} \bt^{*(1)}}(\{(0,0)\})}
%\]
%belongs to $\SRep(\bbI_\bS^\wedge)$. On the other hand, we have observed in~\cite[\S 11.3]{br-pt2} that the object
%\[
%V \otimes \scO_{\FN_{\bT^{(1)} \times_{\bT^{(1)}/\bWf} \bT^{(1)}}(\{(1,1)\})}
%\]
%belongs to $\SRep(\bbI_\Sigma^\wedge)$
%
%\begin{lem}
%\label{lem:free-tilting-SRep}
%The equivalence of Proposition~\ref{prop:equiv-unip-nilp} can be chosen so that it matches the object $V \otimes \scO_{\FN_{\bT^{(1)} \times_{\bT^{(1)}/\bWf} \bT^{(1)}}(\{(1,1)\})}$ with $V \otimes \scO_{\FN_{\bt^{*(1)} \times_{\bt^{*(1)}/\bWf} \bt^{*(1)}}(\{(0,0)\})}$, for any $V \in \Rep(\bG^{(1)})$.
%\end{lem}
%
%\begin{proof}
%\textbf{THIS IS NOT CLEAR TO ME!!}
%\end{proof}

\begin{proof}
By construction of $\Theta^{\hla,\hla}$, to construct an isomorphism as in the proposition it suffices to construct a canonical isomorphism between the images of $\scZ^\wedge(V) \hatstar \Xi_!^\wedge$ and $V \otimes \scO_{\Stm^{\wedge(1)}}$ in $\Rep(\bbI^\wedge_\Sigma)$. In~\cite[Theorem~11.2]{br-pt2} we have identified the image of $\scZ^\wedge(V) \hatstar \Xi_!^\wedge$ with
\[
V \otimes \scO_{\FN_{\bT^{(1)} \times_{\bT^{(1)}/\bWf} \bT^{(1)}}(\{(e,e)\})},
\]
with the action coming from the embedding
\[
\bbI^\wedge_\Sigma \subset \bG^{(1)} \times \FN_{\bT^{(1)} \times_{\bT^{(1)}/\bWf} \bT^{(1)}}(\{(e,e)\}).
\]
It is clear that this object is also the image of $V \otimes \scO_{\Stm^{\wedge(1)}}$, which provides the desired isomorphism. The commutativity of the lemma is clear from construction.
%[Proof of Proposition~\ref{prop:image-central}]
%The claim follows from Lemma~\ref{lem:free-tilting-SRep} and the fact that the equivalence~\eqref{eqn:equiv-pt2} sends $\sfZ^\wedge(V) \hatstar \Xi_!^\wedge$ to the object $V \otimes \scO_{\FN_{\bT^{(1)} \times_{\bT^{(1)}/\bWf} \bT^{(1)}}(\{(1,1)\})}$ (see~\cite[Theorem~11.2]{br-pt2}), while the equivalence~\eqref{eqn:restriction-Kostant-Coh-S} sends $V \otimes \scO_{\St^{\wedge(1)}}$ to $V \otimes \scO_{\FN_{\bt^{*(1)} \times_{\bt^{*(1)}/\bWf} \bt^{*(1)}}(\{(0,0)\})}$.
\end{proof}

%--------------------------------------------------------------
\subsection{Variant for a fixed central character}
%--------------------------------------------------------------

Recall the categories considered in~\S\ref{ss:HC-fixed-character} and~\S\ref{ss:localization-fixed-HC-char}, and the category $\sfD_{\Iwu,\Iw}$ of~\cite[\S 4.2]{br-Hecke}. Set also
\[
\Stm' := \Stm \times_{\bT^{(1)} \times_{\bT^{(1)}/\bWf} \bT^{(1)}} \bigl( \bT^{(1)} \times_{\bT^{(1)}/\bWf} \{e\} \bigr).
\]
Then the isomorphism~\eqref{eqn:identification-Stm-St} restricts to an isomorphism
\begin{equation}
\label{eqn:isom-St'-Stm'}
\Stm' \simto \St'.
\end{equation}

The following statement is the analogue of Theorems~\ref{thm:equivalences-comp} and~\ref{thm:equivalences} in the present context.

\begin{thm}
\label{thm:equivalences-fixed}
There exist equivalences of triangulated categories
\[
\sfD_{\Iwu, \Iw} 
\simto
\Db \Coh^{\bG^{(1)}}(\Stm^{\prime (1)})
\simto
\Db \Coh^{\bG^{(1)}}(\St^{\prime (1)})
\simto
\Db \HC^{\hla,\lambda}.
\]
These equivalences are compatible with those of Theorem~\ref{thm:equivalences-comp} in the sense that they are equivalences of module categories over the monoidal categories appearing in the latter statement, and also that the diagrams involving the various push/pull functors relating these categories are commutative.
\end{thm}

\begin{proof}
The second equivalence is induced by the ($\bG$-equivariant) isomorphism~\eqref{eqn:isom-St'-Stm'}, while the third one is proved in~\S\ref{ss:localization-fixed-HC-char}. The compatibilities of these equivalences with the operations considered in the statement are either obvious or discussed in~\S\ref{ss:localization-fixed-HC-char} and Section~\ref{sec:monoidality}. It remains to explain the construction of the equivalence 
\begin{equation}
\label{eqn:equiv-D-HC-fixed-char}
\sfD_{\Iwu, \Iw} \simto \Db \HC^{\hla,\lambda}
\end{equation}
with appropriate compatibility properties.

We have the additive subcategory $\sfT_{\Iwu,\Iw} \subset \sfD_{\Iwu, \Iw}$ of tilting perverse sheaves, and an equivalence of categories
\begin{equation}
\label{eqn:real-T-Iwu-Iw}
\Kb \sfT_{\Iwu, \Iw} \simto \sfD_{\Iwu, \Iw}.
\end{equation}
We also have a natural functor
\[
\pi_\dag : \sfT^\wedge_{\Iwu, \Iwu} \to \sfT_{\Iwu, \Iw},
\]
see~\cite[\S 6.3]{br-pt2}.
Let us denote by $\mathsf{BST}^\wedge_{\Iwu,\Iwu}$, resp.~$\mathsf{BST}_{\Iwu,\Iw}$ the full subcategory of $\sfT^\wedge_{\Iwu,\Iwu}$, resp.~$\sfT_{\Iwu, \Iw}$, whose objects are the tilting perverse sheaves isomorphic to an object of the form
\[
\Delta^\wedge_\omega \hatstar \Xi_{s_1,!}^\wedge \hatstar \cdots \hatstar \Xi_{s_r,!}^\wedge, \quad \text{resp.} \quad
\pi_\dag(\Delta^\wedge_\omega \hatstar \Xi_{s_1,!}^\wedge \hatstar \cdots \hatstar \Xi_{s_r,!}^\wedge),
\]
with $\omega \in \Omega$ and $s_1, \cdots, s_r \in \Saff$. Then $\pi_\dag$ restricts to an essentially surjective functor $\mathsf{BST}^\wedge_{\Iwu,\Iwu} \to \mathsf{BST}_{\Iwu,\Iw}$, which has the property that for any $M,N \in \mathsf{BST}^\wedge_{\Iwu,\Iwu}$ it induces an isomorphism
\[
\bk \otimes_{\scO(\FN_{\bT^{(1)}}(\{e\}))} \Hom_{\sfT^\wedge_{\Iwu, \Iwu}}(M,N) \simto \Hom_{\sfT_{\Iwu, \Iw}}(\pi_\dag M, \pi_\dag N),
\]
see~\cite[Lemma~5.9]{br-soergel}.
Moreover the category $\sfT^\wedge_{\Iwu,\Iwu}$, resp.~$\sfT_{\Iwu, \Iw}$, identifies with the karoubian closure of the additive hull of $\mathsf{BST}^\wedge_{\Iwu,\Iwu}$, resp.~$\mathsf{BST}_{\Iwu,\Iw}$.

On the other hand, recall the functor $\mathsf{Sp}_{\lambda,\lambda}$ introduced in~\S\ref{ss:HC-fixed-character}.
 Consider the full subcategory $\BSHC^{\hla,\lambda} \subset \HC^{\hla,\lambda}$ whose objects are the images under this functor
 %~\eqref{eqn:specialization-ZHC-char} 
 of the bimodules
\[
\sfR_\omega \star \sfR_{s_1} \star \cdots \star \sfR_{s_r}
\]
with $\omega \in \mathbf{\Omega}$ and $s_1, \cdots, s_r \in \bSaff$, and denote by $\SHC^{\hla,\lambda}$ the karoubian envelope of the additive hull of $\BSHC^{\hla,\lambda}$. Then $\mathsf{Sp}_{\lambda,\lambda}$ restricts to an essentially surjective functor $\BSHC^{\hla,\hla} \to \BSHC^{\hla,\lambda}$, which has the property that for any $M,N \in \BSHC^{\hla,\lambda}$ it induces an isomorphism
\[
\bk_\mu \otimes_{\scO(\FN_{\bt^*/(\bWf,\bullet)}(\{\tla\}))} \Hom_{\HC^{\hla,\hla}}(M,N) \simto \Hom_{\HC^{\hla,\lambda}}(\mathsf{Sp}_{\lambda,\lambda}(M), \mathsf{Sp}_{\lambda,\lambda}(N)),
\]
see Proposition~\ref{prop:Ext-vanishing-fixed-char}.
Arguing as in~\S\ref{eqn:realization-SHC} (and using the second part in Proposition~\ref{prop:Ext-vanishing-fixed-char}), one also checks that the natural functor
\begin{equation}
\label{eqn:real-HC-fixed}
\Kb \SHC^{\hla,\lambda} \to \Db \HC^{\hla,\lambda}
\end{equation}
is an equivalence of monoidal categories.

It is clear that the equivalence $\sfD^\wedge_{\Iwu,\Iwu} \simto \Db \HC^{\hla,\hla}$ restricts to an equivalence of categories
\[
\mathsf{BST}^\wedge_{\Iwu,\Iwu} \simto
\BSHC^{\hla,\hla}.
\]
The comments above show that this equivalence in turn induces an equivalence of categories
\[
\mathsf{BST}_{\Iwu,\Iw} \simto
\BSHC^{\hla,\lambda}.
\]
Passing to bounded homotopy categories of the karoubian closure of the additive hull, and finally conjugating by the equivalences~\eqref{eqn:real-T-Iwu-Iw} and~\eqref{eqn:real-HC-fixed}, we deduce the wished-for equivalence~\eqref{eqn:equiv-D-HC-fixed-char}.

It is easily seen from the construction that this equivalence is compatible in the required sense with the convolution actions of $\sfD_{\Iwu,\Iwu}^\wedge$ and $\Db \HC^{\hla,\hla}$, and that it matches the functor $\pi_\dag$ with $\mathsf{Sp}_{\lambda,\lambda}$. The other compatibility follows by adjunction.
\end{proof}

%------------------------------------------------------
\subsection{Antispherical quotient}
\label{ss:antispherical}
%------------------------------------------------------

%Prove a modular version of the main result of~\cite{ab} as follows. 
In this subsection we indicate how one can deduce from the results of this paper an analogue for positive-characteristic coefficients of the main result of~\cite{ab}.

Recall the Springer resolution $\tcN$ considered in~\S\ref{ss:convolution}, and the ``Iwhori--Whittaker'' category $\sfD_{\IW,\Iw}$ of sheaves on $\Fl_G$, see~\cite[\S 7.1]{brr-pt1}.
We consider the functor
\begin{equation}
\label{eqn:functor-modAB}
\Db \Coh^{\bG^{(1)}}(\tcN^{(1)}) \to \sfD_{\IW, \Iw}
\end{equation}
defined as the composition
\[
\Db \Coh^{\bG^{(1)}}(\tcN^{(1)}) \to \Db \Coh^{\bG^{(1)}}(\St^{\prime (1)}) \cong \sfD_{\Iwu,\Iw} \xrightarrow{\Av_{\IW}} \sfD_{\IW, \Iw}
\]
where the left functor is pushforward under the diagonal embedding $\tcN^{(1)} \to \St^{\prime (1)}$, the middle equivalence is given by Theorem~\ref{thm:equivalences-fixed}, and the functor on the right is given by Iwahori--Whittaker averaging (see~\cite[\S 7.5]{br-pt2}).

\begin{thm}
\label{thm:modAB}
The functor~\eqref{eqn:functor-modAB} is an equivalence of categories.
\end{thm}

\begin{proof}
From the construction we know that the functor sends the object $\scO_{\tcN^{(1)}}(\lambda)$ to $\Av_{\IW}(\Wak_\lambda)$ for any $\lambda$, and $V \otimes_\bk \scO_{\tcN^{(1)}}$ to $\Av_{\IW}(\scZ(V))$ for any $V \in \Rep(G^\vee_\bk)$. (Here, $\scZ$ is the composition of Gaitsgory's central functor, see~\cite[\S 4.2]{brr-pt1}, with the geometric Satake equivalence.) To prove that this functor is an equivalence we follow the same steps as in~\cite{ab}; see~\cite[\S 6.6]{ar-book}. The only step that has to be modified is~\cite[Lemma~6.6.4]{ar-book}. For this lemma we need to prove the composition
\begin{multline*}
\Hom_{\Db \Coh^{\bG^{(1)}}(\tcN^{(1)})}(\scO_{\tcN^{(1)}}, V \otimes \scO_{\tcN^{(1)}}) \to \\
\Hom_{\Db \Coh^{\bG^{(1)}}(\St^{\prime (1)})}(\scO_{\Delta \tcN^{(1)}}, V \otimes \scO_{\Delta \tcN^{(1)}}) \to \\
\Hom_{\sfD_{\Iwu,\Iw}}(\delta_{\Fl}, \scZ(V)) \to \Hom_{\sfD_{\IW, \Iw}}(\Av_{\IW}(\delta_{\Fl}), \Av_{\IW}(\scZ(V)))
\end{multline*}
is injective. 
Here the first arrow is an isomorphism by fully faithfulness of (non derived!) pushforward under a closed embedding, the second arrow is an isomorphism because it is induced by an equivalence of categories. 
Finally, the third arrow is injective because any nonzero morphism in $\Hom(\delta_{\Fl}, \scZ(V))$ is injective (by simplicity of $\delta_{\Fl}$) and $\Av_{\IW}$ is t-exact.
%For the last arrow we observe that the functor $\Av_{\IW} : \sfP_{\Iw,\Iw} \to \sfP_{\IW, \Iw}$ factors as a composition
%\[
%\sfP_{\Iw,\Iw} \to \sfP_{\Iw,\Iw}^{\asp} \to \sfP_{\IW, \Iw}
%\]
%where the second functor is
(The latter fact can be proved as in~\cite[Corollary~7.5]{brr-pt1}.)
\end{proof}

\begin{rmk}
The equivalence of Theorem~\ref{thm:modAB} does not obviously satisfy some of the properties of its counterpart in~\cite{ab}, in particular compatibility with natural incarnations of ``affine Hecke category actions" on both sides. For this reason it might be interesting to construct this equivalence in a different way, closer to the strategy used in~\cite{ab}. Some technical difficulties have prevented us from completing such an approach so far, but we understand that several groups of authors have made progress in this direction in different sheaf-theoretic contexts, so that such a construction might be within reach now.
\end{rmk}

%------------------------------------------------------
\subsection{Discussion of our assumptions}
\label{ss:assumptions}
%------------------------------------------------------

Let us summarize the assumptions we have made for the proof of
the main results of the paper (Theorems~\ref{thm:equivalences-comp},~\ref{thm:equivalences} and~\ref{thm:equivalences-fixed}). In this process we start with the connected reductive group $G$ over an algebraically closed field $\F$ of characteristic $p>0$, its Borel subgroup $B$, and its maximal torus $T$. Then we choose a field $\bk$ which is an algebraic closure of a finite field of characteristic $\ell \neq p$, and the geometric Satake equivalence provides us with the dual group $G^\vee_\bk$, its Borel subgroup $B^\vee_\bk$ and maximal torus $T^\vee_\bk$. Since this dual group is canonically the base-change to $\bk$ of a reductive group scheme over $\Z_\ell$, there is a canonical group $\bG$ with $\bG^{(1)} = G^\vee_\bk$. We have to assume that this group satisfies the assumptions of~\S\ref{ss:HC-notation}, Section~\ref{sec:monoidality}, \S\ref{ss:pseudo-log} and~\S\ref{ss:relation-const-sheaves}; in concrete terms this means that:
\begin{enumerate}
\item 
\label{it:assumptions-discussion-1}
the quotient of $X^*(T)$ by the root lattice of $(G,T)$ is free;\footnote{This assumption means that the derived subgroup of $G^\vee_\bk$ is simply connected, or equivalently that the center of $G$ is a torus.}
\item
the quotient of $X_*(T)$ by the coroot lattice of $(G,T)$ has no $\ell$-torsion;
\item
there exists a $\bG$-equivariant morphism $\bG \to \bg$ sending $e$ to $0$ and \'etale at $e$;
\item
$\bg$ admits a nondegenerate $\bG$-invariant bilinear form;
\item 
\label{it:assumptions-discussion-5}
$\ell$ is odd, $\ell \geq h$, and if $\fR$ contains a component of type $\mathbf{E}_7$, resp.~$\mathbf{E}_8$, then $\ell \neq 19$, resp.~$\ell \neq 31$ (see~\S\ref{ss:relation-const-sheaves}).
\end{enumerate}

The assumptions on the group $G$ can be relaxed a posteriori by using the following observation. Assume we are given a connected reductive group $G'$ over $\F$ and a central isogeny $G' \to G$ such that the dual central isogeny $G^\vee_\bk \to (G')^\vee_\bk$ is \'etale. Then equivalences as in Theorems~\ref{thm:equivalences-comp},~\ref{thm:equivalences} and~\ref{thm:equivalences-fixed} for the group $G'$ can be deduced from similar equivalences for $G$, as we explain below. Note that, given a connected reductive group $G'$ and a field $\bk$ as above, if $\ell$ is strictly larger than the bound of Figure~\ref{fig:bounds} for any component of the root of system of $G'$, then a central isogeny $G' \to G$ with $G$ satisfying the conditions~\eqref{it:assumptions-discussion-1}--\eqref{it:assumptions-discussion-5} above automatically exists. (This follows from the considerations in~\cite[\S II.1.18]{jantzen}, using Remarks~\ref{rmk:assumptions-group} and~\ref{rmk:existence-pseudolog} and the fact that $\ell$ is automatically very good for $G$; one can e.g.~take for $G$ a product of the adjoint quotient of $G'$ with a torus.) This remark justifies the presentation of our results in Section~\ref{sec:intro}.

Let us now sketch a justification of the assertion above about Theorem~\ref{thm:equivalences-comp}; the other cases can be treated similarly. So, we consider groups $G$, $G'$ as above, and the associated \'etale isogeny $\bG \to \bG'$ obtained from the dual groups. What happens is that each of the categories considered in Theorem~\ref{thm:equivalences-comp} for the group $G'$ identifies with a direct summand of the corresponding category for $G$, and the equivalences for $G$ restrict to equivalences between these direct summands. On the constructible side, it is a classical fact that (possibly up to a universal homeomorphism) the affine flag variety of $G'$ identifies with a union of connected components of that of $G$, which implies the desired claim. (See~\cite[\S 9.3]{brr-pt1} for similar considerations.) For the categories associated with $\bG$ and $\bG'$, we consider the kernel $\bK$ of the isogeny $\bG \to \bG'$, a smooth finite diagonalizable group scheme. The isogeny identifies the Lie algebras $\bg$ and $\bg'$ of $\bG$ and $\bG'$, together with most of the associated structures (in particular, the finite Weyl groups and Grothendieck resolutions). Any complex in $\Db \Coh^{\bG^{(1)}}(\St^{\wedge(1)})$ or $\Db \HC^{\hla,\hla}$ admits a canonical decomposition according to the action of $\bK^{(1)}$, parametrized by the character lattice of this group scheme, and the corresponding categories attached to $\bG'$ identify with the subcategories of objects whose summands associated with nontrivial characters vanish. Finally, the morphism relating the \emph{multiplicative} Grothendieck resolutions $\tbG$ and $\tbG'$ is not an isomorphism, but arguing as in Lemma~\ref{lem:isom-varphi-tbG} one sees that it induces an isomorphism of schemes
$\FN_{\bG/\bG}(\{e\}) \times_{\bG/\bG} \tbG \simto \FN_{\bG'/\bG'}(\{e\}) \times_{\bG'/\bG'} \tbG'$. One therefore obtains an identification between the associated ``completed versions'' of the multiplicative Steinberg varieties, and can conclude as before.

\section{t-structures}
\label{sec:t-structures}
%%%%%%%%%%%%%%%%%%%%%%%%%%%

In this section we continue with our running assumptions that conditions~\eqref{it:assumptions-discussion-1}--\eqref{it:assumptions-discussion-5} in~\S\ref{ss:assumptions} are satisfied. (Here again, all the statements in this sections can a posteriori be generalized to more general reductive groups using the procedure in~\S\ref{ss:assumptions}.)

%consider a group $G$ and a field $\bk$ such that the statements of Theorem~\ref{thm:equivalences-comp} and Theorem~\ref{thm:equivalences} are known. (See~\S\ref{ss:assumptions} for a discussions of some conditions that guarantee that this is the case.)

%-------------------------------------------------------------
\subsection{Unicity of t-structures}
\label{ss:unicity-t-str}
%-------------------------------------------------------------

Consider the nilpotent cone $\cN \subset \bg^{*(1)}$, i.e.~the preimage of the image of $0$ under the quotient morphism $\bg^{*(1)} \to \bg^{*(1)}/\bG^{(1)} \cong \bt^{*(1)}/\bWf$. We then consider the derived category
\[
\Db \Coh^{\bG^{(1)}}_{\cN}(\bg^{*(1)})
\]
of $\bG^{(1)}$-equivariant coherent sheaves on the scheme $\bg^{*(1)}$ supported set-theoretically on $\cN$, see~\S\ref{ss:coh-subscheme-statement}. If we denote by $q : \St^{(1)} \to \bg^{*(1)}$ the natural (proper) morphism, then we can consider the natural functor
\begin{equation}
\label{eqn:Rpi*}
Rq_* : \Db \Coh^{\bG^{(1)}}_{\cN}(\St^{(1)}) \to \Db \Coh^{\bG^{(1)}}_{\cN}(\bg^{*(1)}).
\end{equation}

Recall from Lemma~\ref{lem:DbCoh-nil} the fully faithful functor
\[
\Db \Coh^{\bG^{(1)}}_{\cN}(\St^{(1)}) \to \Db \Coh^{\bG^{(1)}}(\St^{\wedge(1)}).
\]
As explained in~\S\ref{ss:convolution}, the essential image of this functor is a two-sided ideal in the monoidal category $\Db \Coh^{\bG^{(1)}}(\St^{\wedge(1)})$.

%Recall from Remark~\ref{rmk:action-DbCoh-nilpotent-objects} the action of the monoidal category $\Db \Coh^{\bG^{(1)}}(\St^{\wedge(1)}_{\varnothing,\varnothing})$ on $\Db \Coh^{\bG^{(1)}}_{\cN}(\St^{(1)}_{\varnothing,\varnothing})$. 
% and the objects $\sfD_w$ introduced in~\S\ref{ss:study-real}.
%For any $w \in \bW$ we set
%\[
%\scD_w := \Phi^{\hla,\hla}(\sfD_w).
%\]
%The property~\eqref{eqn:convolution-Dw} implies that the assignment $T_w \mapsto \scD_w$ extends to a group morphism form $\Br_{\bW}$ to the group of isomorphism classes of invertible objects in the monoidal category $\Db \Coh^{\bG^{(1)}}(\St^{\wedge(1)}_{\varnothing,\varnothing})$. We will denote by $\scD_b$ the image of an element $b \in \Br_{\bW}$. 
%
%The following result will be used below.

Consider the braid group $\Br_{\bW}$ from~\S\ref{ss:further-properties}.
We will denote by $\Br_{\bW}^+$ the sub-semigroup of $\Br_{\bW}$ generated by the elements $T_w$ with $w \in \bW$. Any element $b \in \Br_{\bW}^+$ can be written as $b=T_{s_1} \cdots T_{s_r} \cdot T_\omega$ for some $s_1, \cdots, s_r \in \bSaff$ and $\omega \in \mathbf{\Omega}$. We will denote by $\ell(b)$ the minimal $r$ such that such an expression exists.

\begin{lem}
\label{lem:nonzero-cohomology-braid-translate}
For any nonzero $\scF$ in $\Db \Coh^{\bG^{(1)}}_{\cN}(\St^{(1)})$, there exist $b,c \in \Br_{\bW}^+$ such that $Rq_*(\scI^\wedge_b \star \scF \star \scI^\wedge_{c}) \neq 0$.
\end{lem}

\begin{proof}
Fix $\scF$ as in the statement.
As explained already in the course of the proof of Lemma~\ref{lem:comp-loc-end}, for any strictly dominant weights $\nu, \nu' \in X^*(\bT)$ the line bundle $\scO_{(\bG/\bB)^{(1)}}(\nu) \boxtimes \scO_{(\bG/\bB)^{(1)}}(\nu')$ is ample (see~\cite[Proposition~II.4.4]{jantzen}).
Since the morphism $\St^{(1)} \to (\bG/\bB)^{(1)} \times (\bG/\bB)^{(1)}$ is affine, the pullback $\scO_{\St^{(1)}}(\nu,\nu')$ of this line bundle to $\St^{(1)}$ is also ample,
see~\cite[\href{https://stacks.math.columbia.edu/tag/0892}{Tag 0892}]{stacks-project}. Hence, by standard considerations (based on~\cite[\href{https://stacks.math.columbia.edu/tag/0B5U}{Tag 0B5U}]{stacks-project} and~\cite[\href{https://stacks.math.columbia.edu/tag/01PR}{Tag 01PR}]{stacks-project}) there exist $\nu,\nu' \in X^*(\bT)^+$ such that
\[
R\Gamma(\St^{(1)}, \scF \otimes_{\scO_{\St^{(1)}}} \scO_{\St^{(1)}}(\nu,\nu')) \neq 0.
\]
Now we have
\[
\scF \otimes_{\scO_{\St^{(1)}}} \scO_{\St^{(1)}}(\nu,\nu') = \scI^\wedge_{T_{t_\nu}} \star \scF \star \scI^\wedge_{T_{t_{\nu'}}},
\]
so that this implies the desired claim.
\end{proof}

%It follows from the main result of~\cite{paris} that $\Br_{\bW}^+$ is the monoid generated by the elements $(T_w : w \in \bW)$ with relations $T_y \cdot T_w = T_{yw}$ if $\ell(yw)=\ell(y)+\ell(w)$.

We will say that a t-structure on the category $\Db \Coh^{\bG^{(1)}}_{\cN}(\St^{(1)})$ is \emph{braid positive} if for any $b,c \in \Br_{\bW}^+$ the functor
\[
\scI_b^\wedge \star (-) \star \scI^\wedge_{c} : \Db \Coh^{\bG^{(1)}}_{\cN}(\St^{(1)}) \to \Db \Coh^{\bG^{(1)}}_{\cN}(\St^{(1)})
\]
is right t-exact.

\begin{lem}
\label{lem:properties-braid-pos-t-str}
Assume we are given a braid positive t-structure on the triangulated category $\Db \Coh^{\bG^{(1)}}_{\cN}(\St^{(1)})$.

\begin{enumerate}
\item
\label{it:properties-braid-pos-t-str-1}
For any $b,c \in \Br_{\bW}^+$ the functor
\[
\scI^\wedge_{b^{-1}} \star (-) \star \scI^\wedge_{c^{-1}} : \Db \Coh^{\bG^{(1)}}_{\cN}(\St^{(1)}) \to \Db \Coh^{\bG^{(1)}}_{\cN}(\St^{(1)})
\]
is left t-exact.
\item
\label{it:properties-braid-pos-t-str-2}
For any $s \in \bSaff$ the functors
\[
\scR_s \star (-), \, (-) \star \scR_{s} : \Db \Coh^{\bG^{(1)}}_{\cN}(\St^{(1)}) \to \Db \Coh^{\bG^{(1)}}_{\cN}(\St^{(1)})
\]
are t-exact.
\end{enumerate}
\end{lem}

\begin{proof}
Statement~\eqref{it:properties-braid-pos-t-str-1} follows from the fact that a right adjoint of a right t-exact functor is left t-exact. Then~\eqref{it:properties-braid-pos-t-str-2} follows using the triangles~\eqref{eqn:triangle-Rs-I}.
%Corollary~\ref{cor:isom-R-Y}
\end{proof}

The following statement closely resembles a ``unicity of t-structure'' result encountered in~\cite{bm-loc}.

\begin{prop}
\label{prop:unicity-t-str}
Given any 
%bounded 
t-structure $(\sfD_{\bg^*}^{\leq 0}, \sfD_{\bg^*}^{\geq 0})$ on $\Db \Coh^{\bG^{(1)}}_{\cN}(\bg^{*(1)})$, there exists at most one bounded t-structure $(\sfD_{\St}^{\leq 0}, \sfD_{\St}^{\geq 0})$ on $\Db \Coh^{\bG^{(1)}}_{\cN}(\St^{(1)})$ which is braid positive and such that the functor~\eqref{eqn:Rpi*} is t-exact for the t-structures
\[
(\sfD_{\St}^{\leq 0}, \sfD_{\St}^{\geq 0}) \quad \text{and} \quad (\sfD_{\bg^*}^{\leq 0}, \sfD_{\bg^*}^{\geq 0}).
\]
If this t-structure exists, then its nonpositive part is given by
%\begin{align*}
\begin{equation}
\label{eqn:description-Dleq0}
\sfD_{\St}^{\leq 0} = \{ \scF \in \Db \Coh^{\bG^{(1)}}_{\cN}(\St^{(1)}) \mid \forall b,c \in \Br_{\bW}^+, \, Rq_*(\scI^\wedge_b \star \scF \star \scI^\wedge_{c}) \in \sfD_{\bg^*}^{\leq 0} \}.
\end{equation}
%\sfD_{\St}^{\geq 0} &= \{ \scF \in \Db \Coh^{\bG^{(1)}}_{\cN}(\St_{\varnothing,\varnothing}) \mid \forall b,b' \in \Br_{\bW}^+, \, R\pi_*(b^{-1} \cdot \scF \cdot (b')^{-1}) \in \sfD_{\bg^*}^{\geq 0} \}.
%\end{align*}
\end{prop}

\begin{proof}
It is enough to show that any t-structure $(\sfD_{\St}^{\leq 0}, \sfD_{\St}^{\geq 0})$ as in the lemma satisfies~\eqref{eqn:description-Dleq0}, since a t-structure is determined by its nonpositive part. Let us therefore assume we are given a bounded t-structure $(\sfD_{\St}^{\leq 0}, \sfD_{\St}^{\geq 0})$ on $\Db \Coh^{\bG^{(1)}}_{\cN}(\St^{(1)})$ which is braid positive and such that the functor $Rq_*$ is t-exact for the t-structures $(\sfD_{\St}^{\leq 0}, \sfD_{\St}^{\geq 0})$ and $(\sfD_{\bg^*}^{\leq 0}, \sfD_{\bg^*}^{\geq 0})$. It is clear that any object $\scF$ in $\sfD_{\St}^{\leq 0}$ will then satisfy
\[
Rq_*(\scI^\wedge_b \star \scF \star \scI^\wedge_{c}) \in \sfD_{\bg^*}^{\leq 0}
\]
for any $b,c \in \Br_{\bW}^+$.

Reciprocally, let us fix an object $\scF$ such that $Rq_*(\scI^\wedge_b \star \scF \star \scI^\wedge_{c}) \in \sfD_{\bg^*}^{\leq 0}$ for any $b,c \in \Br_{\bW}^+$. 
%We need to show that $\scF$ is in $\sfD_{\St}^{\leq 0}$. 
Since the t-structure is bounded, there exists $i \in \Z$ such that $\scF \in \sfD_{\St}^{\leq i}$ but $\scF \notin \sfD_{\St}^{\leq i-1}$; to conclude we need to show that $i \leq 0$. Assume for a contradiction that $i>0$, and set $\scF' := \tau_{>0}(\scF)$ where $\tau_{>0}$ is the truncation functor (in positive degrees) associated with our t-structure. From the truncation triangle
\[
\tau_{\leq 0} \scF \to \scF \to \scF' \xrightarrow{[1]}
\]
and the inclusion proven above we see that $\scF'$ also satisfies
\[
Rq_*(\scI^\wedge_b \star \scF' \star \scI^\wedge_{c}) \in \sfD_{\bg^*}^{\leq 0}
\]
for any $b,c \in \Br_{\bW}^+$. By Lemma~\ref{lem:nonzero-cohomology-braid-translate}, there exist $b,c \in \Br_{\bW}^+$ such that
\[
Rq_*(\scI^\wedge_b \star \scF' \star \scI^\wedge_{c}) \neq 0.
\]
Let us choose $b$ and $c$ which satisfy this property and such that $\ell(b) + \ell(c)$ is minimal among elements satisfying it.
Write $b=T_{s_1} \cdots T_{s_{\ell(b)}} T_\omega$ and $c = T_{s_1'} \cdots T_{s'_{\ell(c)}} T_{\omega'}$ for some $s_1, \cdots, s_{\ell(b)},s'_1, \cdots, s'_{\ell(c)} \in \bSaff$ and $\omega, \omega' \in \mathbf{\Omega}$.
Using the triangles~\eqref{eqn:exact-seq-braid-gp-action-2} and minimality we obtain that
\[
Rq_*(\scI^\wedge_b \star \scF' \star \scI^\wedge_{c}) \cong Rq_*( \scX^\wedge_{s_1} \star \cdots \star \scX^\wedge_{s_{\ell(b)}} \star \scI^\wedge_{T_\omega} \star \scF' \star \scX^\wedge_{s'_1} \star \cdots \star \scX^\wedge_{s'_{\ell(c)}} \star \scI^\wedge_{T_{\omega'}}).
\]
Now the right-hand side belongs to $\sfD_{\bg^*}^{\geq 1}$ by Lemma~\ref{lem:properties-braid-pos-t-str}\eqref{it:properties-braid-pos-t-str-2}. (Note that the functors $\scI^\wedge_{T_\omega} \star (-)$ and $(-) \star \scI^\wedge_{T_{\omega'}}$ are left t-exact by Lemma~\ref{lem:properties-braid-pos-t-str}\eqref{it:properties-braid-pos-t-str-1}.)
%since their inverses, namely the functors $\scI^\wedge_{T_{\omega^{-1}}} \star (-)$ and $(-) \star \scI^\wedge_{T_{(\omega')^{-1}}}$, are right t-exact. 
We have therefore produced a nonzero object in $\sfD_{\bg^*}^{\leq 0} \cap \sfD_{\bg^*}^{\geq 1}$, a contradiction.
\end{proof}

%-------------------------------------------------------------
\subsection{The perverse coherent t-structure on coherent sheaves}
%-------------------------------------------------------------

We now consider the action of $\bG$ on $\bg^{*(1)}$ obtained by pullback of the coadjoint action of $\bG^{(1)}$ by the Frobenius morphism $\Fr_{\bG}$, and
%Consider the nilpotent cone $\cN \subset \bg^{*(1)}$, i.e.~the preimage of the image of $0$ under the quotient morphism $\bg^{*(1)} \to \bg^{*(1)}/\bG^{(1)} \cong \bt^{*(1)}/\bWf$. We then consider 
the derived category
\[
\Db \Coh^{\bG}_{\cN}(\bg^{*(1)})
\]
of $\bG$-equivariant coherent sheaves on $\bg^{*(1)}$ supported set-theoretically on $\cN$.
% see~\S\ref{ss:coh-subscheme-statement}. (Here, $\bG$ acts on $\bg^{*(1)}$ via the Frobenius morphism $\Fr_{\bG}$ and the action of $\bG^{(1)}$. This action preserves $\cN$ by construction.) 
For any $n \geq 1$ we consider the $n$-th infinitesimal neighborhood $\cN^{[n]}$ of $\cN$ in $\bg^{*(1)}$; then the action of $\bG$ on $\bg^{*(1)}$ induces an action on $\cN^{[n]}$, and by Lemma~\ref{lem:DbCoh-colim} we have an equivalence of triangulated categories
\begin{equation}
\label{eqn:equiv-coh-N-g*}
\mathrm{colim}_{n \geq 1} \Db \Coh^{\bG}(\cN^{[n]}) \simto \Db \Coh^{\bG}_{\cN}(\bg^{*(1)}).
\end{equation}

For any $n \geq 1$ we can consider the perverse coherent t-structure on the category $\Db \Coh^{\bG}(\cN^{[n]})$ as in~\cite{arinkin-bezrukavnikov} (see also~\cite{bez-pcoh}). More precisely, as in~\cite[Example~4.15]{arinkin-bezrukavnikov}, points of the stack $\cN^{[n]}/\bG$ correspond to $\bG^{(1)}$-orbits on $\cN$. The ``middle perversity'' is defined by $p(x)=-\frac{\dim(O_x)}{2}$ where $O_x$ is the orbit corresponding to $x$, and the procedure in~\cite[Definition~3.7]{arinkin-bezrukavnikov} using this function produces a bounded t-structure whose heart is noetherian and artinian, see~\cite[Corollary~4.13]{arinkin-bezrukavnikov}. For our purposes we will apply a global shift to this t-structure, and say that a complex $\scF$ belongs to the heart of the \emph{perverse coherent t-structure} iff
$\scF[\frac{\dim(\cN)}{2}]$
belongs to the heart of the t-structure of~\cite{arinkin-bezrukavnikov} associated with the middle perversity. (Note that $\dim(\cN) \in 2\Z$, so that this definition makes sense.) With this convention, the restriction to the regular orbit of any object in the heart of the perverse coherent t-structure is concentrated in degree $0$ (for the tautological t-structure).

For any $n \geq 1$ the pushforward functor
\[
\Db \Coh^{\bG}(\cN^{[n]}) \to \Db \Coh^{\bG}(\cN^{[n+1]})
\]
is t-exact for the perverse coherent t-structures, see~\cite[Lemma~3.3(c)]{arinkin-bezrukavnikov}. In fact since this functor does not kill any nonzero object, it ``detects the t-structure'' in the sense that an object $\scF$ in $\Db \Coh^{\bG}(\cN^{[n]})$ belongs to the non-negative, resp.~non-positive, part of the perverse coherent t-structure on $\Db \Coh^{\bG}(\cN^{[n]})$ iff its image belongs to the non-negative, resp.~non-positive, part of the perverse coherent t-structure on $\Db \Coh^{\bG}(\cN^{[n+1]})$. We therefore obtain a t-structure on the categories in~\eqref{eqn:equiv-coh-N-g*} such that each pushforward functor
\[
\Db \Coh^{\bG}(\cN^{[n]}) \to \Db \Coh^{\bG}_{\cN}(\bg^{*(1)})
\]
is t-exact (and, in fact, ``detects the t-structure'' in the same sense as above).

The cohomology functors for the perverse coherent t-structure on the triangulated category $\Db \Coh^{\bG}_{\cN}(\bg^{*(1)})$ will be denoted $(\pcH^n : n \in \Z)$.

\begin{rmk}
\label{rmk:pcoh-tautological-tstr}
Consider the full subcategory
\begin{equation}
\label{eqn:DbCoh-0}
\Db \Coh^{\bG}_{0}(\bg^{*(1)}) \subset
\Db \Coh^{\bG}_{\cN}(\bg^{*(1)})
\end{equation}
whose objects are the complexes which are set-theoretically supported on the zero-orbit $\{0\}$. By definition of the perverse coherent t-structure the embedding~\eqref{eqn:DbCoh-0} is t-exact if the left-hand side is endowed with the shift by $\frac{\dim(\cN)}{2}$ of the tautological t-structure and the right-hand side with the perverse coherent t-structure; in other words, for any $\scF$ in $\Db \Coh^{\bG}_{0}(\bg^{*(1)})$ we have
\[
 \pcH^n(\scF) = \scH^{n+\frac{\dim(\cN)}{2}}(\scF)
\]
for any $n \in \Z$. In particular, if $\delta$ denotes the skyscraper sheaf at $0$, then $\delta[-\frac{\dim(\cN)}{2}]$ belongs to the heart of the perverse coherent t-structure; in fact it is easily seen to be a simple object therein.
\end{rmk}

The same procedure also produces a t-structure on the category $\Db \Coh^{\bG^{(1)}}_{\cN}(\bg^{*(1)})$ considered in~\S\ref{ss:unicity-t-str}, such that the obvious pullback functor
\[
\Db \Coh^{\bG^{(1)}}_{\cN}(\bg^{*(1)}) \to \Db \Coh^{\bG}_{\cN}(\bg^{*(1)})
\]
is t-exact.
This t-structure will also be called the perverse coherent t-structure.

%-------------------------------------------------------------
\subsection{The perverse coherent t-structure on Harish-Chandra bimodules}
%-------------------------------------------------------------

For any $\mu,\nu \in X^*(\bT)$, restriction of the action to $\ZFr$ defines a natural functor
\begin{equation}
\label{eqn:restriction-HC-center}
\Db \HC^{\hmu,\hnu}_{\mathrm{nil}} \to \Db \Coh^{\bG}_{\cN}(\bg^{*(1)}).
\end{equation}

\begin{lem}
\label{lem:existence-pcoh-t-str}
There exists a unique t-structure on $\Db \HC^{\hmu,\hnu}_{\mathrm{nil}}$ such that the functor~\eqref{eqn:restriction-HC-center} is t-exact with respect to the perverse coherent t-structure on the category $\Db \Coh^{\bG}_{\cN}(\bg^{*(1)})$; moreover this t-structure is bounded.
\end{lem}

\begin{proof}
Unicity is clear from the fact that the functor~\eqref{eqn:restriction-HC-center} does not kill any nonzero object: if the t-structure exists, its nonpositive, resp.~nonnegative, part must consist of complexes whose image is in the nonpositive, resp.~nonnegative, part of the perverse coherent t-structure on $\Db \Coh^{\bG}_{\cN}(\bg^{*(1)})$. 
%It is clear also from this description that this t-structure is bounded.

To show existence, we will in fact construct a bounded t-structure on the category $\Db \HC^\wedge_{\mathrm{nil}}$ (see~\S\ref{ss:nilpotent-HC}) such that the forgetful functor
\[
\Db \HC^\wedge_{\mathrm{nil}} \to \Db \Coh^{\bG}_{\cN}(\bg^{*(1)})
\]
is t-exact with respect to the perverse coherent t-structure. By~\eqref{eqn:decomp-HC-nil} we have a decomposition as a sum of triangulated categories
\[
\Db \HC^\wedge_{\mathrm{nil}} = \bigoplus_{\mu,\nu \in X^*(\bT)/(\bW,\bullet)} \Db \HC^{\hmu,\hnu},
\]
hence the t-structure will restrict to a t-structure with the appropriate property on each factor $\Db \HC^{\hmu,\hnu}$.
As in the proof of Lemma~\ref{lem:HC-wedge-nil}, restriction to the left action defines an equivalence of categories between $\HC^\wedge_{\mathrm{nil}}$ and the category of $\bG$-equivariant finitely generated $\cU\bg$-modules such that the maximal ideal $\fn \subset \scO(\bt^{*(1)}/\bWf) = \ZHC \cap \ZFr$ acts nilpotently. The $\ZFr$-algebra $\cU\bg$ defines a $\bG$-equivariant coherent sheaf of $\scO_{\bg^{*(1)}}$-algebras on $\bg^{*(1)}$, and restricting this algebra to $\cN^{[n]}$ produces a sheaf of algebras $(\cU\bg)^{[n]}$. Then by Lemma~\ref{lem:DbCoh-colim} we therefore have a canonical equivalence of triangulated categories
\[
\Db \HC^\wedge_{\mathrm{nil}} \cong \mathrm{colim}_{n \geq 1} \Db \Coh^{\bG}(\cN^{[n]}, (\cU\bg)^{[n]}).
\]

The construction in~\cite{arinkin-bezrukavnikov, bez-pcoh} applies more generally to categories of coherent sheaves of modules for a coherent sheaf of algebras; in particular,
for any $n \geq 1$ it provides a bounded t-structure on $\Db \Coh^{\bG}(\cN^{[n]}, (\cU\bg)^{[n]})$ such that the forgetful functor
\[
\Db \Coh^{\bG}(\cN^{[n]}, (\cU\bg)^{[n]}) \to \Db \Coh^{\bG}(\cN^{[n]})
\]
is t-exact. Then each embedding
\[
\Db \Coh^{\bG}(\cN^{[n]}, (\cU\bg)^{[n]}) \to \Db \Coh^{\bG}(\cN^{[n+1]}, (\cU\bg)^{[n+1]})
\]
is t-exact, and we obtain a t-structure on the colimit which solves our problem.
%EXISTENCE: COMPLETE!
\end{proof}

The t-structure of Lemma~\ref{lem:existence-pcoh-t-str} on the triangulated category $\Db \HC^{\hmu,\hnu}_{\mathrm{nil}}$ will also be called the perverse coherent t-structure, and its cohomology functors will also be denoted $(\pcH^n : n \in \Z)$.

\begin{lem}
\label{lem:translation-exact-pcoh}
For any $\mu,\nu,\eta \in X^*(\bT)$, the functors
\[
\sfP^{\hmu,\hnu} \star (-) : \Db \HC^{\hnu,\widehat{\eta}}_{\mathrm{nil}} \to \Db \HC^{\hmu,\widehat{\eta}}_{\mathrm{nil}}
\quad \text{and} \quad
(-) \star \sfP^{\hmu,\hnu} : \Db \HC^{\widehat{\eta},\hmu}_{\mathrm{nil}} \to \Db \HC^{\widehat{\eta},\hnu}_{\mathrm{nil}}
\]
are t-exact for the perverse coherent t-structures.
\end{lem}

\begin{proof}
Using~\eqref{eqn:completion-diag-induced} one sees that the functors under consideration are obtained from a functor of the form
\[
V \otimes (-) : \Db \HC^{\wedge}_{\mathrm{nil}} \to \Db \HC^{\wedge}_{\mathrm{nil}}
\]
with $V \in \Rep(\bG)$
by restriction to a direct summand and composition with projection to a (possibly different) direct summand. Since the functor of tensoring by $V$ is clearly t-exact for the perverse coherent t-structure on $\Db \Coh^{\bG}_{\cN}(\bg^{*(1)})$, we deduce the desired claim.
\end{proof}

%---------------------------------------------------------------
\subsection{Compatibility}
%---------------------------------------------------------------

The main result of this section is following.

\begin{thm}
\label{thm:t-structure}
The equivalence
%\Theta_{\mathrm{nil}}^{\hla,\hla} : 
$\sfD_{\Iwu, \Iwu} \simto \Db \HC_{\mathrm{nil}}^{\hla,\hla}$
of Theorem~\ref{thm:equivalences} 
is t-exact with respect to the perverse t-structure on $\sfD_{\Iwu, \Iwu}$ and the perverse coherent t-structure on $\Db \HC_{\mathrm{nil}}^{\hla,\hla}$.
\end{thm}

Our strategy will be the following. By Proposition~\ref{prop:unicity-t-str}, there exists at most one bounded t-structure on $\Db \Coh^{\bG^{(1)}}_{\cN}(\St^{(1)})$ which is braid positive and such that the functor~\eqref{eqn:Rpi*} is t-exact, where the target category is equipped with the perverse coherent t-structure. We will prove in~\S\ref{ss:image-perv} that the image under 
%$\Phi_{\mathrm{nil}}^{\hla,\hla} \circ \Theta_{\mathrm{nil}}^{\hla,\hla}$
the equivalence $\sfD_{\Iwu, \Iwu} \simto \Db \Coh^{\bG^{(1)}}_{\cN}(\St^{(1)})$ 
%of Theorem~\ref{thm:equivalences} 
of the perverse t-structure on $\sfD_{\Iwu, \Iwu}$ satisfies these properties, and in~\S\ref{ss:image-pcoh} that the image under the equivalence 
%$\Phi_{\mathrm{nil}}^{\hla,\hla}$
$\Db \HC_{\mathrm{nil}}^{\hla,\hla} \simto \Db \Coh^{\bG^{(1)}}_{\cN}(\St^{(1)})$ 
%(also from Theorem~\ref{thm:equivalences}) 
of the perverse coherent t-structure on $\Db \HC_{\mathrm{nil}}^{\hla,\hla}$ also satisfies these properties. This will imply Theorem~\ref{thm:t-structure}.

%---------------------------------------------------------------
\subsection{Image of the perverse t-structure}
\label{ss:image-perv}
%---------------------------------------------------------------

\begin{prop}
\label{prop:image-perv}
The image of the perverse t-structure on $\sfD_{\Iwu,\Iwu}$ under the equivalence
%\Phi_{\mathrm{nil}}^{\hla,\hla} \circ \Theta_{\mathrm{nil}}^{\hla,\hla} : 
$\sfD_{\Iwu,\Iwu} \simto \Db \Coh_{\cN}^{\bG^{(1)}}(\St^{(1)})$
of Theorem~\ref{thm:equivalences} 
is braid positive and such that the functor~\eqref{eqn:Rpi*} is t-exact with respect to the perverse coherent t-structure on $\Db \Coh^{\bG^{(1)}}_{\cN}(\bg^{*(1)})$.
\end{prop}

\begin{proof}
In this proof we will use the categories $\sfD_{\Iwu,\Iw}$ and $\sfD_{\Iw,\Iw}$, the functors $\pi^\dag : \sfD_{\Iwu,\Iw} \to \sfD_{\Iwu,\Iwu}$ and $\For^{\Iw}_{\Iwu} : \sfD_{\Iw,\Iw} \to \sfD_{\Iwu,\Iw}$, and the objects $(\Delta^{\Iw}_w : w \in W)$ and $(\nabla^{\Iw}_w : w \in W)$ from~\cite{brr-pt1,br-pt2}. Then the nonpositive, resp.~nonnegative, part of the perverse t-structure on $\sfD_{\Iwu,\Iwu}$ is generated under extensions by the objects
\[
(\pi^\dag \For^{\Iw}_{\Iwu}(\Delta_w^{\Iw})[n] : w \in W, n \in \Z_{\geq 0}), \quad \text{resp.} \quad
(\pi^\dag \For^{\Iw}_{\Iwu}(\nabla_w^{\Iw})[n] : w \in W, n \in \Z_{\leq 0}).
\]

It is a standard fact that the functors of left and right convolution with objects $\nabla^\wedge_w$ on $\sfD_{\Iwu,\Iwu}$ are right t-exact for the perverse t-structure; in fact, for left convolution this follows from the observation that if $y,w \in W$ the complex
\[
\nabla^\wedge_w \hatstar \pi^\dag \For^{\Iw}_{\Iwu}(\Delta_y^{\Iw}) \cong \pi^\dag \For^{\Iw}_{\Iwu}(\nabla_w^{\Iw} \star^{\Iw} \Delta_y^{\Iw})
\]
is perverse (see e.g.~\cite[Lemma~4.1.7]{ar-book}). By Remark~\ref{rmk:matching-standards-costandards}, this implies that the t-structure under consideration is braid positive.

Now, consider the composition
\begin{equation}
\label{eqn:functor-projection-nilp}
\sfD_{\Iwu,\Iwu} \simto \Db \Coh_{\cN}^{\bG^{(1)}}(\St^{(1)})
\xrightarrow{Rq_*} \Db \Coh^{\bG^{(1)}}_{\cN}(\bg^{*(1)}).
\end{equation}
To conclude we have to prove that this functor is t-exact, which will follow if we prove that it sends any object $\pi^\dag \For^{\Iw}_{\Iwu}(\Delta^{\Iw}_w)$ or $\pi^\dag \For^{\Iw}_{\Iwu}( \Delta^{\Iw}_w)$ ($w \in W$) to an object in the heart of the perverse coherent t-structure.

Recall the simple objects $(\IC_w : w \in W)$ in the heart of the perverse t-structure on $\sfD_{\Iw,\Iw}$.
We claim that the image under~\eqref{eqn:functor-projection-nilp} of any object of the form $\pi^\dag \For^{\Iw}_{\Iwu}(\IC_w)$ where $w$ in not minimal in $\Wf w \Wf$ vanishes. In fact, let us denote this image by $\scF_w$.
Since $\bg^{*(1)}$ is an affine scheme, to prove that $\scF_w=0$ it suffices to prove that $R\Gamma(\bg^{*(1)}, \scF_w) = 0$, which will follow if we prove that
\[
\Ext^n_{\bG^{(1)}}(V, R\Gamma(\bg^{*(1)}, \scF_w))=0
\]
for any $V \in \Tilt(\bG^{(1)})$ and $n \in \Z$, i.e.~that
\[
\Hom_{\Db \Coh^{\bG^{(1)}}(\St^{\wedge(1)})}(V \otimes \scO_{\St^{\wedge(1)}}, \scF_w[n]) =0
\]
for any $V \in \Tilt(\bG^{(1)})$ and $n \in \Z$. Transferring this condition in $\sfD^\wedge_{\Iwu,\Iwu}$ and using Proposition~\ref{prop:image-central}, this amounts to proving that
\[
\Hom_{\sfD^\wedge_{\Iwu,\Iwu}}(\scZ^\wedge(V) \hatstar \Xi^\wedge_!, \pi^\dag \For^{\Iw}_{\Iwu} (\IC_w)[n])=0
\]
for any $V \in \Tilt(\bG^{(1)})$ and $n \in \Z$. Recall that if $w$ is not minimal in $\Wf w \Wf$, then it is not minimal either in $\Wf w$ or in $w\Wf$. If $w$ is not minimal in $w\Wf$, then $\IC_w$ is obtained by pullback from a partial affine flag variety associated with a subset of $\fRs$, and the claim follows from adjunction using the fact that the pushforward of $\pi_\dag(\Xi_!^\wedge)$ to this partial flag variety vanishes. The case when $w$ is not minimal in $\Wf w$ is similar, using the centrality of $\scZ^\wedge(V)$.

By standard arguments (see e.g.~\cite[Lemma~3(a)]{bez-nilp}), the claim we have proved above implies that the image under~\eqref{eqn:functor-projection-nilp} of the object
\[
\pi^\dag \For^{\Iw}_{\Iwu}(\Delta^{\Iw}_w), \quad \text{resp.} \quad \pi^\dag \For^{\Iw}_{\Iwu}(\Delta^{\Iw}_w),
\]
only depends on the coset $\Wf w \Wf$. Now, any such coset contains the translation associated with a dominant weight, and the translation associated with an antidominant weight. If $\mu \in X^*(\bT)^+$, then by Remark~\ref{rmk:matching-standards-costandards}
%we have
%\Phi^{\hla,\hla} \circ \Theta^{\hla,\hla}
the image of $\nabla^{\wedge}_{t_{-\mu}}$ under the equivalence
\[
\sfD^\wedge_{\Iwu, \Iwu} 
\simto
\Db \Coh^{\bG^{(1)}}(\St^{\wedge(1)})
\]
of Theorem~\ref{thm:equivalences-comp}
is $\scO_{\Delta \tbg^{\wedge(1)}}(\mu)$,
which by linearity for the right action of the ring $\scO(\FN_{\bt^{*(1)}}(\{0\}))$ implies that the image of
$\pi^\dag \For^{\Iw}_{\Iwu}(\nabla^{\Iw}_{t_{-\mu}})$ under the equivalence we consider here is $\scO_{\Delta \widetilde{\cN}^{(1)}}(\mu)$. (Here $\widetilde{\cN} = \tbg \times_{\bt^*} \{0\}$ is the Springer resolution.)
It is a standard fact that the object $Rq_* \scO_{\Delta \widetilde{\cN}^{(1)}}(\mu)$ belongs to the heart of the perverse coherent t-structure (see~\cite[\S 2.2]{bez-nilp} or~\cite[\S 2.3]{achar}), which implies the desired claim for costandard objects. %If $\lambda$ is antidominant, then the image of $\pi^\dag(\Delta^{\Iw}_{t_\lambda})$ is $\scO_{\Delta \widetilde{\cN}}(\lambda)$. 
Similarly, for any $\mu \in X^*(\bT)^+$ the image of
%\Phi^{\hla,\hla}_{\mathrm{nil}} \circ \Theta^{\hla,\hla}_{\mathrm{nil}}
$\pi^\dag \For^{\Iw}_{\Iwu}(\nabla^{\Iw}_{t_{\mu}})$ is $\scO_{\Delta \widetilde{\cN}^{(1)}}(-\mu)$.
Since $Rq_* \scO_{\Delta \widetilde{\cN}^{(1)}}(-\mu)$ also belongs to the heart of the perverse coherent t-structure also in this case (see again~\cite[\S 2.2]{bez-nilp} or~\cite[\S 2.3]{achar}), this proves the desired property for standard objects and finishes the proof.
\end{proof}

%Maybe use [Bezrukavnikov, Perverse sheaves on affine flags and nilpotent cone of the Langlands dual group, \S 2.2].

%---------------------------------------------------------------
\subsection{Image of the perverse coherent t-structure on Harish-Chandra bimodules}
\label{ss:image-pcoh}
%---------------------------------------------------------------

\begin{prop}
The image of the perverse coherent t-structure on $\Db \HC_{\mathrm{nil}}^{\hla,\hla}$ under the equivalence
\[
\Phi^{\hla,\hla}_{\mathrm{nil}} : \Db \HC_{\mathrm{nil}}^{\hla,\hla} \simto
\Db \Coh_{\cN}^{\bG^{(1)}}(\St^{(1)})
\]
%of Theorem~\ref{thm:equivalences} 
is braid positive and such that the functor~\eqref{eqn:Rpi*} is t-exact with respect to the perverse coherent t-structure on $\Db \Coh^{\bG^{(1)}}_{\cN}(\bg^{*(1)})$.
\end{prop}

\begin{proof}
To prove braid positivity, we have to prove that for $b \in \{T_\omega : \omega \in \mathbf{\Omega}\} \cup \{T_s : s \in \bSaff\}$ the functors
\begin{align*}
\scI_b^\wedge \star (-) &: \Db \Coh_{\cN}^{\bG^{(1)}}(\St^{(1)}) \to \Db \Coh_{\cN}^{\bG^{(1)}}(\St^{\wedge(1)}), \\
(-) \star \scI_b^\wedge &: \Db \Coh_{\cN}^{\bG^{(1)}}(\St^{(1)}) \to \Db \Coh_{\cN}^{\bG^{(1)}}(\St^{\wedge(1)})
\end{align*}
are right t-exact with respect to our given t-structure. By~\eqref{eqn:Romega-I}
%Lemma~\ref{lem:image-braid} 
and monoidality,
for any $\omega \in \mathbf{\Omega}$ we have a commutative diagram
\[
\xymatrix@C=1.5cm{
\Db \HC_{\mathrm{nil}}^{\hla,\hla} \ar[d]_-{\Phi^{\hla,\hla}_{\mathrm{nil}}}^-{\wr} \ar[r]^-{\sfR_{\omega^{-1}} \star (-)} & \Db \HC_{\mathrm{nil}}^{\hla,\hla} \ar[d]^-{\Phi^{\hla,\hla}_{\mathrm{nil}}}_-{\wr} \\
\Db \Coh_{\cN}^{\bG^{(1)}}(\St^{(1)}) \ar[r]^-{\scI_{T_\omega}^\wedge \star (-)} & \Db \Coh_{\cN}^{\bG^{(1)}}(\St^{(1)}).
}
\]
%where the vertical arrows are the equivalences of Theorem~\ref{thm:equivalences}. 
By Lemma~\ref{lem:translation-exact-pcoh} the upper arrow is t-exact for the perverse coherent t-structure, which implies that the lower horizontal arrow is t-exact for the t-structure we consider. Similar arguments apply to the functor $(-) \star \scI_{T_\omega}^\wedge$.

By similar arguments one sees that, for any $s \in \bSaff$, the functors
\[
\scR_s \star (-), (-) \star \scR_s : \Db \Coh_{\cN}^{\bG^{(1)}}(\St^{(1)}) \to \Db \Coh_{\cN}^{\bG^{(1)}}(\St^{(1)})
\]
are t-exact for our t-structure.
%we have an exact sequence of perverse sheaves
%\[
%\nabla^\wedge_e \hookrightarrow \Xi^\wedge_s \twoheadrightarrow \nabla^\wedge_s,
%\]
%which, 
%applying $\Phi^{\hla,\hla}$ to the left triangle in~\eqref{eqn:triangles-Ds-Ds'} and using
%Lemma~\ref{lem:image-braid} we obtain a distinguished triangle
%\[
%\scI_e^\wedge \to \scR_s \to \scI_{T_s}^\wedge \xrightarrow{[1]}.
%\]
%Using this triangle 
Using the left triangle in~\eqref{eqn:triangle-Rs-I} we deduce that the functors $\scI_{T_s}^\wedge \star (-)$ and $(-) \star \scI_{T_s}^\wedge$ are right t-exact for our t-structure, which finishes the proof of braid positivity.

Now we will prove that the functor $Rq_*$ is t-exact.
Recall the weight $\varsigma$ fixed in~\S\ref{ss:HC-notation}.
% we consider a weight $\varsigma \in X^*(\bT)$ such that $\langle \varsigma, \alpha^\vee \rangle = 1$ for any simple root $\alpha$. 
 Then we have an equivalence
\[
\Phi^{\widehat{-\varsigma}, \widehat{-\varsigma}} : \Db \HC^{\widehat{-\varsigma}, \widehat{-\varsigma}} \simto \Db \Coh^{\bG^{(1)}}(\St^{\wedge(1)}_{\fRs,\fRs}),
\]
see~\S\ref{ss:D-mod-Coh}, which as in Theorem~\ref{thm:equivalences} restricts to an equivalence
\[
\Phi^{\widehat{-\varsigma}, \widehat{-\varsigma}}_{\mathrm{nil}} : \Db \HC^{\widehat{-\varsigma}, \widehat{-\varsigma}}_{\mathrm{nil}} \simto \Db \Coh^{\bG^{(1)}}_{\cN}(\St^{(1)}_{\fRs,\fRs}),
\]
and moreover $\St_{\fRs,\fRs}=\bg^*$. Since $\Phi^{\widehat{-\varsigma}, \widehat{-\varsigma}}_{\mathrm{nil}}$ is given by tensoring with a vector bundle, it is t-exact for the perverse coherent t-structures on both sides. By Proposition~\ref{prop:translation-push-pull} (and an analogue for convolution on the right) we have a commutative diagram
\[
\xymatrix@C=2cm{
\Db \HC_{\mathrm{nil}}^{\hla,\hla} \ar[d]_-{\Phi^{\hla,\hla}_{\mathrm{nil}}}^-{\wr} \ar[r]^-{\sfP^{\widehat{-\varsigma},\hla} \star (-) \star \sfP^{\hla,\widehat{-\varsigma}}} & \Db \HC_{\mathrm{nil}}^{\widehat{-\varsigma}, \widehat{-\varsigma}} \ar[d]^-{\Phi^{\widehat{-\varsigma},\widehat{-\varsigma}}_{\mathrm{nil}}}_-{\wr} \\
\Db \Coh_{\cN}^{\bG^{(1)}}(\St^{(1)}) \ar[r]^-{Rq_*} & \Db \Coh_{\cN}^{\bG^{(1)}}(\bg^{*(1)}).
}
\]
By Lemma~\ref{lem:translation-exact-pcoh} the upper line is t-exact for the perverse coherent t-structures; we deduce the desired exactness property for the lower line.
\end{proof}

%%%%%%%%%%%%%%%%%%%%%%%%%%%%%%%%%%%%%%%%%%%%%%%%%%%%%
\section{Application to the Finkelberg--Mirkovi{\'c} conjecture}
\label{sec:FM}
%%%%%%%%%%%%%%%%%%%%%%%%%%%%%%%%%%%%%%%%%%%%%%%%%%%%%

We continue with the assumptions of Section~\ref{sec:t-structures}, replacing the condition that $\ell \geq h$ by the condition $\ell > h$.

%-------------------------------------------------------
\subsection{Image of the trivial bimodule}
%-------------------------------------------------------

%From now on we assume (for simplicity) that $\lambda=0$. 
%Recall the functor~\eqref{eqn:Rep-HC-lambda}, which . The image under this functor of the trivial $\bG$-module $\bk$ will be called the
%We endow the trivial $\bG$-module $\bk$ with the structure of an object in $\HC^{0,0}$ where the left and right $\cU$-actions are both trivial. This object will be called the 
Recall the ``trivial Harish-Chandra bimodule'' considered in~\S\ref{ss:bimod-trivial-action}. 
Consider also the simple perverse sheaf $\IC_{w_\circ} \in \sfD_{\Iwu,\Iw}$ associated with the longest element $w_\circ$ in $\bWf$, considered in particular in the proof of Proposition~\ref{prop:image-perv}.
The following statement is analogous to~\cite[Lemma~6.7]{blo2}.

\begin{prop}
\label{prop:image-k}
%Assume that $\ell>h$.
The shifted trivial Harish-Chandra bimodule
\[
\bk \left[ -\frac{\dim(\cN)}{2} \right] \quad \in \Db \HC^{\widehat{0},\widehat{0}}
\]
is the image of the simple perverse sheaf
\[
\pi^\dag(\IC_{w_\circ}) \quad \in \sfD_{\Iwu,\Iwu}
\]
under the equivalence of Theorem~\ref{thm:equivalences}.
\end{prop}

\begin{proof}
%The proof is similar to (the first step of) that of~\cite[Lemma~6.7]{blo1}. Namely, 
%We know that the two objects under consideration are simple objects in the heart of the t-structures on $\Db \HC_{\mathrm{nil}}^{\widehat{0},\widehat{0}}$ and $\sfD_{\Iwu, \Iwu}$ that are matched under the equivalence of Theorem~\ref{thm:equivalences}, see Theorem~\ref{thm:t-structure}. (For the shifted trivial bimodule, see Remark~\ref{rmk:pcoh-tautological-tstr}.) To conclude the proof, it therefore suffices to show that their classes in the Grothendieck groups of these categories coincide under the identification provided by the equivalence of Theorem~\ref{thm:equivalences}.
As explained in Remark~\ref{rmk:pcoh-tautological-tstr}, the object $\bk [ -\frac{\dim(\cN)}{2} ]$ is a simple object in the heart of the perverse coherent t-structure. By Theorem~\ref{thm:t-structure} its image is therefore a simple perverse sheaf in $\sfD_{\Iwu,\Iwu}$, hence is of the form $\pi^\dag(\scF)$ for some simple $\Iw$-equivariant perverse sheaf $\scF$ on $\Fl_G$. For any $s \in \Sf$, by using the comparison with translation functors for $\bG$-modules (see~\cite[Lemma~6.1]{br-Hecke}) we see that
\[
\sfR_s \star \bk = 0,
\]
hence that
\[
\Xi^\wedge_s \hatstar \pi^\dag(\scF)=0,
\]
which implies that $\scF$ is $\Loop^+ G$-equivariant. By symmetry we also have $\bk \star \sfR_s=0$ for $s \in \Sf$, which implies that $\scF$ is the pullback of a simple perverse sheaf $\scG$ in the Satake category $\sfP_{\Loop^+ G,\Loop^+ G}$ (see~\cite[\S 4.3]{br-pt2} or~\S\ref{ss:proof-FM} below).

Now, we claim that each cohomology object of the complex
$\bk \star \bk$
(for the tautological t-structure) is a direct sum of copies of the trivial module $\bk$. In fact, clearly the left and right actions of $\cU\bg$ on these cohomology objects are trivial, so it only remains to show that the action of $\bG$ is trivial too. The convolution $\bk \star \bk$ can be computed using the Chevalley--Eilenberg resolution of the trivial $\bg$-module (which is a resolution as a $\bG$-equivariant $\cU\bg$-module, hence can be seen as a resolution as Harish-Chandra bimodule, see~\eqref{eqn:HC-modules}); we see that it coincides with the cohomology of the complex $\bigwedge^\bullet \bg$, endowed with the differential given by
\[
d(x_1 \wedge \cdots \wedge x_r) = \sum_{i<j} (-1)^{i+j} [x_i, x_j] \wedge \cdots \wedge \widehat{x_i} \wedge \cdots \wedge \widehat{x_j} \wedge \cdots \wedge x_r.
\]
As explained above the action of $\bG_1$ on this cohomology is trivial; in particular the action of $\bT_1$ is trivial. Hence the cohomology under consideration is also the cohomology of the complex $(\bigwedge^\bullet \bg)^{\bT_1}$. By~\cite[Lemma~II.12.10]{jantzen}, under our assumptions on $\ell$ the action of $\bT$ on $(\bigwedge^\bullet \bg)^{\bT_1}$ is trivial; hence the action of $\bT$ on our cohomology is trivial too. This implies that the action of $\bG$ is trivial.

From this claim and Remark~\ref{rmk:pcoh-tautological-tstr}, we see that each cohomology object of the complex
\[
\bk \left[ -\frac{\dim(\cN)}{2} \right] \star \bk \left[ -\frac{\dim(\cN)}{2} \right] 
\]
for the perverse coherent t-structure is a direct sum of copies of $\bk \left[ -\frac{\dim(\cN)}{2} \right]$. Hence each perverse cohomology object of $\pi^\dag(\scF) \hatstar \pi^\dag(\scF)$ is isomorphic to a sum of copies of $\pi^\dag(\scF)$. This is only possible if $\scG$ is the sky-scraper perverse sheaf, i.e.~if $\scF=\IC_{w_\circ}$.
%First, consider the object $\pi^\dag(\IC_{w_\circ})$. Recall the category $\sfD_{\Iwu,\Iw}$ of~\cite{br-pt2}. Then the functor $\pi^\dag : \sfD_{\Iwu,\Iw} \to \sfD_{\Iwu,\Iwu}$ induces an isomorphism on Grothendieck groups, and the Grothendieck group of $\sfD_{\Iwu,\Iw}$ has a natural basis given by classes of standard $\Iwu$-equivariant perverse sheaves on $\Fl$. This basis is parametrized by $\bW$, and we use it to identify the Grothendieck group of $\sfD_{\Iwu,\Iwu}$ with $\Z[\bW]$. Under this identification, it is a standard fact that the class of $\pi^\dag(\IC_{w_\circ})$ is $\sum_{w \in \bWf} (-1)^{\ell(ww_\circ)} w$.
%
%\textbf{DESCRIPTION OF THE CLASS OF THE TRIVIAL BIMODULE: use translation functors to say that the corresponding perverse sheaf is $L^+G$-equivariant. Then use that $\bk \star \bk$ is essentially $\bk$ to show that it is $\pi^\dag(\IC_{w_\circ})$.}
\end{proof}

%---------------------------------------------------------
\subsection{Statement of the Finkelberg--Mirkovi{\'c} conjecture}
\label{ss:FM-statement}
%---------------------------------------------------------

%Now we assume that $\lambda=0$, and consider 
Recall the extended principal block $\Rep_{\langle 0 \rangle}(\bG) \subset \Rep(\bG)$ considered in~\S\ref{ss:HC-notation}. Recall also the extended affine Weyl group $\bW$ introduced in~\S\ref{ss:central-characters}, and denote by ${}^{\mathrm{f}} \hspace{-1pt} \bW \subset \bW$ the subset consisting of elements $w$ which are of minimal length in $\bWf w$. Then it is a standard fact that the map $w \mapsto w \bullet 0$ induces a bijection
\[
{}^{\mathrm{f}} \hspace{-1pt} \bW \simto (\bW \bullet 0) \cap X^*(\bT)^+.
\]
The simple objects in the category $\Rep_{\langle 0 \rangle}(\bG)$ are therefore in a canonical bijection with ${}^{\mathrm{f}} \hspace{-1pt} \bW$. It is well known also that this category admits a highest weight structure with weight poset ${}^{\mathrm{f}} \hspace{-1pt} \bW$ endowed with the restriction of the Bruhat order,\footnote{Here the Bruhat order on $\bW$ is the order such that for $y,y' \in \bWaff$ and $\omega,\omega' \in \mathbf{\Omega}$ we have $y \omega \leq y' \omega'$ iff $\omega=\omega'$ and $y \leq y'$ for the Bruhat order on $\bWaff$ constructed from the Coxeter system $(\bWaff,\bSaff)$.} and standard, resp.~costandard, objects given by the Weyl, resp.~induced, modules. Steinberg's tensor product formula implies that the action of $\Rep(\bG^{(1)})$ on $\Rep(\bG)$ given by
\[
(V, V') \mapsto \Fr_\bG^*(V) \otimes V'
\]
for $V \in \Rep(\bG^{(1)})$ and $V' \in \Rep(\bG)$
(where $\Fr_\bG : \bG \to \bG^{(1)}$ is the Frobenius morphism) stabilizes the subcategory $\Rep_{\langle 0 \rangle}(\bG)$.

Consider now the affine Grassmannian
\[
\Gr_G = \Loop G / \Loop^+ G
\]
for the group $G$, and the action of $\Loop^+ G$ by multiplication on the left. We will denote by $\sfD_{\Loop^+ G,\Loop^+ G}$ the $\Loop^+ G$-equivariant derived category of constructible $\bk$-sheaves on $\Gr_G$, and by $\sfP_{\Loop^+ G,\Loop^+ G}$ its full subcategory of perverse sheaves. The category $\sfD_{\Loop^+ G,\Loop^+ G}$ admits a canonical monoidal product $\star^{\Loop^+ G}$ which is t-exact on each side for the perverse t-structure. We therefore obtain a monoidal category $(\sfP_{\Loop^+ G,\Loop^+ G}, \star^{\Loop^+ G})$. The \emph{geometric Satake equivalence} provides an equivalence of monoidal categories
\[
\Sat : (\sfP_{\Loop^+ G,\Loop^+ G}, \star^{\Loop^+ G}) \simto (\Rep(\bG^{(1)}), \otimes),
\]
see~\cite{mv}.

On the other hand, consider the ``twisted'' version
\[
\Gr_G' = \Loop^+ G \backslash \Loop G,
\]
and the action of $\Iwu$ by multiplication on the right. Let us denote by $\sfD_{\Loop^+ G,\Iwu}$ the $\Iwu$-equivariant derived category of constructible $\bk$-sheaves on $\Gr_G'$, and by $\sfP_{\Loop^+ G,\Iwu}$ its full subcategory of perverse sheaves. Since the $\Iwu$-orbits on $\Gr_G'$ are isomorphic to affine spaces, the category $\sfP_{\Loop^+ G,\Iwu}$ admits a canonical structure of highest weight category, see~\cite[\S\S 3.2--3.3]{bgs}. Namely, the $\Iwu$-orbits on $\Gr_G'$ are naturally labelled by ${}^{\mathrm{f}} \hspace{-1pt} W$, and the order given by inclusions of closures of strata is the restriction of the Bruhat order on $W$. (Here ${}^{\mathrm{f}} \hspace{-1pt} W \subset W$ is the subset corresponding to ${}^{\mathrm{f}} \hspace{-1pt} \bW \subset \bW$.) The weight poset for this highest weight structure is therefore ${}^{\mathrm{f}} \hspace{-1pt} W$ endowed with the Bruhat order. The standard, resp.~costandard, object associated with an element $w \in {}^{\mathrm{f}} \hspace{-1pt} W$ is the $!$-extension, resp.~$*$-extension, of the perversely shifted constant sheaf on the orbit corresponding to $w$. The simple object in $\sfP_{\Loop^+ G,\Iwu}$ associated with $w \in {}^{\mathrm{f}} \hspace{-1pt} W$ will be denoted $\scL_w$; it is the intersection cohomology complex of the orbit associated with $w$.

We also have a convolution bifunctor
\begin{equation}
\label{eqn:convolution-sph}
\star^{\Loop^+ G} : \sfD_{\Loop^+ G,\Loop^+ G} \times \sfD_{\Loop^+ G,\Iwu} \to \sfD_{\Loop^+ G,\Iwu}
\end{equation}
which is also t-exact on both sides for the perverse t-structures, hence induces an action of the monoidal category $(\sfP_{\Loop^+ G,\Loop^+ G}, \star^{\Loop^+ G})$ on $\sfP_{\Loop^+ G,\Iwu}$.

The main result of this section is the following theorem, which was conjectured by Finkelberg--Mirkovi{\'c}, see~\cite{fm}.

\begin{thm}
\label{thm:FM}
There exists an equivalence of highest weight categories
\[
\Phi : \sfP_{\Loop^+ G,\Iwu} \cong \Rep_{\langle 0 \rangle}(\bG)
\]
which satisfies
\[
\Phi(\scL_w) \cong \sfL(w \bullet 0)
\]
for any $w \in {}^{\mathrm{f}} \hspace{-1pt} \bW$, and which is an equivalence of module categories for $\sfP_{\Loop^+ G,\Loop^+ G}$ and $\Rep(\bG^{(1)})$, where the latter two categories are identified via the geometric Satake equivalence $\Sat$.
%for $\scG$ in $\sfP_{L^+G,L^+G}$ and $\scF \in \sfP_{L^+G,\Iwu}$ there exists a bifunctorial isomorphism
%\[
%\Phi(\scG \star^{L^+G} \scF) \cong \Fr^*(\Sat(\scG)) \otimes \Phi(\scF).
%\]
%compatible with the geometric Satake equivalence
%\[
%\sfP_{L^+H, L^+ H} \cong \Rep(\bG^{(1)})
%\]
%and the action of these categories on $\sfP_{\Iwu,L^+ H}$ and $\Rep_{\langle 0 \rangle}(\bG)$ respectively.
\end{thm}

%The proof of the theorem will use the following lemma.

%---------------------------------------------------------
\subsection{Proof of the Finkelberg--Mirkovi{\'c} conjecture}
\label{ss:proof-FM}
%---------------------------------------------------------

Before giving the proof of Theorem~\ref{thm:FM}, we need to study some objects associated with representations of $\bG^{(1)}$. Namely, for any $V \in \Rep(\bG^{(1)})$, the action of $\cU\bg$ on $\Fr_\bG^*(V)$ is trivial. We therefore have an exact functor
\[
\HC^{\hla,\hla} \to \HC^{\hla,\hla}
\]
which sends an object $M$ to $V \otimes M$, where the action of $\bG$ is diagonal and the left and right actions of $\cU\bg$ are induced by the actions on $M$. This functor identifies with left and right convolution with the object $\sfC^{\hla,\hla}(V \otimes \cU\bg)$. Similarly we have functors
\begin{align*}
\Db\Coh^{\bG^{(1)}}(\Stm^{\wedge (1)}) &\to \Db\Coh^{\bG^{(1)}}(\Stm^{\wedge (1)}), \\
 \Db\Coh^{\bG^{(1)}}(\St^{\wedge (1)}) &\to \Db\Coh^{\bG^{(1)}}(\St^{\wedge (1)})
\end{align*}
of tensoring with $V$, which identify with convolution with the objects $V \otimes \scO_{\Delta \tbG^{\wedge (1)}}$ and $V \otimes \scO_{\Delta \tbg^{\wedge(1)}}$ respectively. It is clear that these functors and objects correspond under the relevant equivalences of Theorem~\ref{thm:equivalences-comp}.

The identification of the corresponding functors on $\sfD^\wedge_{\Iwu,\Iwu}$ (in case $V$ is tilting) is the subject of the following lemma.

\begin{lem}
\label{lem:image-central-2}
There exists an isomorphism
\[
\Theta^{\hla,\hla} \circ \scZ^\wedge(-) \cong (-) \otimes \scO_{\Delta \tbG^{\wedge(1)}}
\]
of monoidal functors from $\Tilt(\bG^{(1)})$ to $\Db\Coh^{\bG^{(1)}}(\Stm^{\wedge (1)})$.
%between the composition of the functor $\mathscr{Z}^\wedge$ of~\cite[\S 7.4]{br-pt2} with the equivalence
%\[
%\sfD^{\wedge}_{\Iwu,\Iwu} \simto \Db\Coh^{\bG^{(1)}}(\St^{\wedge (1)})
%\]
%of Theorem~\ref{thm:equivalences-comp}
%and the functor $V \mapsto V \otimes_\bk \scO_{\Delta \tbg^\wedge}$.
\end{lem}

\begin{proof}
It is known that the $\scO(\FN_{\bT^{(1)}}(\{e\}))$-modules
\[
\Hom_{\sfP^\wedge_{\Iwu,\Iwu}}(\Xi^\wedge_!, \delta^\wedge) \quad \text{and} \quad
\Hom_{\sfP^\wedge_{\Iwu,\Iwu}}(\delta^\wedge, \Xi^\wedge_!)
\]
are free of rank one, that any generator of the first, resp.~second, module is surjective, resp.~injective, and that these morphisms remain so after application of the functor $\pi_\dag : \sfD^\wedge_{\Iwu,\Iwu} \to \sfD_{\Iwu,\Iw}$. Let us fix generators $f_1 : \Xi^\wedge_! \to \delta^\wedge$ and $f_2 : \delta^\wedge \to \Xi^\wedge_!$ of these spaces, and set $f=f_2 \circ f_1$. Consider also $g_1 = \Theta^{\hla,\hla}(f_1)$, $g_2 = \Theta^{\hla,\hla}(f_2)$ and $g = \Theta^{\hla,\hla}(f) = g_2 \circ g_1$. Then by monoidality and Proposition~\ref{prop:image-central}  there exist canonical identifications
\[
\Theta^{\hla,\hla}(\delta^\wedge) \cong \scO_{\tbG^{\wedge(1)}}, \quad \Theta^{\hla,\hla}(\Xi^\wedge_!) \cong \scO_{\Stm^{\wedge(1)}}
\]
(so that $g_1$ can be considered as a morphism from $\scO_{\Stm^{\wedge(1)}}$ to $\scO_{\tbG^{\wedge(1)}}$, and similarly for $g_2$ and $g$) and, for any $V \in \Tilt(\bG^{(1)})$, an identification
\[
\Theta^{\hla,\hla}(\scZ^\wedge(V) \hatstar \Xi^\wedge_!) \cong V \otimes \scO_{\Stm^{\wedge(1)}}
\]
such that the following diagram commutes:
\[
\xymatrix@C=1.7cm{
\Theta^{\hla,\hla}(\scZ^\wedge(V) \hatstar \Xi_!^\wedge) \ar[d]_-{\wr} \ar[r]^-{\Theta^{\hla,\hla}(\id \hatstar f)} & \Theta^{\hla,\hla}(\scZ^\wedge(V) \hatstar \Xi_!^\wedge)  \ar[d]^-{\wr} \\
V \otimes \scO_{\Stm^{\wedge(1)}} \ar[r]^-{\id \otimes g} & V \otimes \scO_{\Stm^{\wedge(1)}}.
}
\]

We claim that there exists a unique morphism
\begin{equation}
\label{eqn:isom-Delta}
\Theta^{\hla,\hla} \circ \scZ^\wedge (V) \to V \otimes \scO_{\Delta \tbG^{\wedge(1)}}
\end{equation}
such that the following diagram commutes:
\[
\xymatrix@C=1.7cm{
\Theta^{\hla,\hla}(\scZ^\wedge(V) \hatstar \Xi_!^\wedge) \ar[d]_-{\wr} \ar[r]^-{\Theta^{\hla,\hla}(\id \hatstar f_1)} & \Theta^{\hla,\hla}(\scZ^\wedge(V)) \ar[d]^-{\eqref{eqn:isom-Delta}} \ar[r]^-{\Theta^{\hla,\hla}(\id \hatstar f_2)} & \Theta^{\hla,\hla}(\scZ^\wedge(V) \hatstar \Xi_!^\wedge)  \ar[d]^-{\wr} \\
V \otimes \scO_{\Stm^{\wedge(1)}} \ar[r]^-{\id \otimes g_1} & V \otimes \scO_{\Delta \tbG^{\wedge(1)}} \ar[r]^-{\id \otimes g_2} & V \otimes \scO_{\Stm^{\wedge(1)}},
}
\]
and that this morphism is an isomorphism.
In fact, the functor of convolution with $\scZ^\wedge(V)$ is t-exact for the perverse t-structure by~\cite[Theorem~7.8(5)]{br-pt2}, so that $\Theta^{\hla,\hla}(\scZ^\wedge(V))$ is the image of the morphism $\Theta^{\hla,\hla}(\id \hatstar f)$, for the t-structure on $\Db\Coh^{\bG^{(1)}}(\Stm^{\wedge (1)})$ obtained by transporting the perverse t-structure along $\Theta^{\hla,\hla}$. On the other hand, by Theorem~\ref{thm:equivalences-fixed} the functor $\pi^\dag \pi_\dag$ corresponds, under the equivalence $\Theta^{\hla,\hla}$, to the composition $Li^* \circ i_*$ where $i : \St^{\prime(1)} \hookrightarrow \St^{\wedge(1)}$ is the natural embedding. Using the analogous properties for $f_1$ and $f_2$, we deduce that the latter functor sends the cocone of $g_1$ and the cone of $g_2$ to objects in the heart of the t-structure. By Theorem~\ref{thm:t-structure} this t-structure restricts to a t-structure on $\Db\Coh^{\bG^{(1)}}_{\cU}(\Stm^{(1)})$, which moreover is the transport of the perverse coherent t-structure along the equivalence
\[
\Db\Coh^{\bG^{(1)}}_{\cU}(\Stm^{(1)}) \simto \Db \HC^{\hla,\hla}_{\mathrm{nil}}
\]
of Theorem~\ref{thm:equivalences}. The latter equivalence commutes with the functors of tensoring with $V$ on both sides, and this functor is t-exact for the perverse coherent t-structure. From these remarks we deduce that the image under $Li^* \circ i_*$ of the cocone of $\id \otimes g_1$ and of the cone of $\id \otimes g_2$ are in the heart of our t-structure. Using standard properties of the perverse t-structure in completed categories (see~\cite[Lemma~5.3]{br-soergel}) it follows that these objects themselves belong to the heart of the t-structure. Hence $V \otimes \scO_{\Delta \tbG^{\wedge(1)}}$ identifies with the image in this t-structure of the morphism $\id \otimes g$, which implies our claim.

Note that the isomorphism~\eqref{eqn:isom-Delta} does not depend on the choice of $f_1$ and $f_2$, by compatibility of our constructions with right multiplication by elements in $\scO(\FN_{\bT^{(1)}}(\{e\}))$. The functoriality and monoidality of these isomorphisms can be checked using the characterization in terms of the commutative diagram above.
\end{proof}

\begin{proof}[Proof of Theorem~\ref{thm:FM}]
Recall the category $\Db \Coh^{\bG}_{0}(\bg^{*(1)})$ of Remark~\ref{rmk:pcoh-tautological-tstr}.
If we denote by
\[
\Db \HC^{\widehat{0},\widehat{0}}_0
\]
the preimage of $\Db \Coh^{\bG}_{0}(\bg^{*(1)})$ under the functor~\eqref{eqn:restriction-HC-center},
%in $\Db \HC^{\widehat{0},\widehat{0}}_{\mathrm{nil}}$ of 
we deduce from that remark that for any $M$ in $\Db \HC^{\widehat{0},\widehat{0}}_0$ and $n \in \Z$ we have
\[
H^{n+\frac{\dim(\cN)}{2}}(M) = \pcH^n(M),
\]
where in the left-hand side we consider cohomology with respect to the tautological t-structure.
Any object of the form $\bk \star M$ with $M \in \Db \HC^{\widehat{0},\widehat{0}}$ clearly belongs to $\Db \HC^{\widehat{0},\widehat{0}}_0$;
these comments and Lemma~\ref{lem:Rep-cohomology-convolution} therefore
%IDEA: describe 
imply that $\Rep_{\langle 0 \rangle}(\bG)$ identifies with the full subcategory of $\Db \HC^{\widehat{0},\widehat{0}}_{\mathrm{nil}}$ whose objects are the complexes of the form $\pcH^n(\bk \star M)$
with $M \in \Db \HC^{\widehat{0},\widehat{0}}$ and $n \in \Z$.

On the other hand, it is clear that $\sfP_{\Loop^+ G, \Iwu}$ identifies with the full subcategory of $\sfD_{\Iwu,\Iwu}$ whose objects are the perverse sheaves of the form
\[
\pH^n(\pi^\dag(\IC_{w_\circ}) \hatstar \scF)
\]
for $n \in \Z$ and $\scF \in \sfD^\wedge_{\Iwu,\Iwu}$. In view of Theorem~\ref{thm:t-structure}, Proposition~\ref{prop:image-k} and monoidality, the equivalence $\Phi$ can therefore be obtained by restriction of the equivalence $\sfD_{\Iwu,\Iwu} \simto \Db \HC_{\mathrm{nil}}^{\widehat{0},\widehat{0}}$ of Theorem~\ref{thm:equivalences}.

The description of images of simple objects follows from a standard support argument using the fact that convolution with the object $\Xi^{\wedge}_{s,!}$ on $\sfP_{\Loop^+ G,\Iwu}$ corresponds to the wall-crossing functor attached to $s$ on $\Rep_{\langle 0 \rangle}(\bG)$, for any $s \in \bSaff$.

Finally, we consider compatibility with the actions of $\Rep(\bG^{(1)})$.
Using centrality of the functor $\mathscr{Z}^\wedge$ (see~\cite[Theorem~7.8(4)]{br-pt2}), standard properties of convolution, the compatibility of $\mathscr{Z}^\wedge$ with Gaitsgory's central functor (see~\cite[Theorem~7.8(1)]{br-pt2}), and the fact that the composition of the latter functor with pushforward to $\Gr$ coincides with the equivalence $\Sat$, we see that, for any $V \in \Rep(\bG^{(1)})$, the functor of right convolution with $\mathscr{Z}^\wedge(V)$ stabilizes the image of $\sfP_{\Loop^+ G, \Iwu}$, and identifies with convolution on the left with $\Sat(V)$ in the sense of the bifunctor~\eqref{eqn:convolution-sph}. In view of Lemma~\ref{lem:image-central-2}, this shows that the diagram
\[
\xymatrix{
\sfT_{\Loop^+ G, \Loop^+ G} \times \sfP_{\Loop^+ G, \Iwu}  \ar[d] \ar[r] & \sfP_{\Loop^+ G, \Iwu} \ar[d] \\
\Tilt(\bG^{(1)}) \times \Rep_{\langle 0 \rangle}(\bG) \ar[r] & \Rep_{\langle 0 \rangle}(\bG)
}
\]
commutes, where the horizontal arrows are the action bifunctors, and the vertical arrows are induced by $\Sat$ and $\Phi$. (Here, $\sfT_{\Loop^+ G, \Loop^+ G}$ is the category of tilting objects in the highest weight category $\sfP_{\Loop^+ G, \Loop^+ G}$.) To ``extend'' this isomorphism to the whole of $\Rep(\bG^{(1)})$ one can proceed as follows. For a fixed $\scF \in \sfP_{\Iwu,\Loop^+ G}$, the commutativity above provides an isomorphism between the functors
\[
\Kb \sfT_{\Loop^+ G, \Loop^+ G} \to \Db \Rep_{\langle 0 \rangle}(\bG)
\]
obtained on the one hand by applying $\Sat$ and tensoring with $\Phi(\scF)$, and on the other hand by convolving with $\scF$ and then applying $\Db(\Phi)$. (In this discussion we omit the natural functors from the homotopy category to the derived category.) Now by general properties of highest weight categories there exists a canonical equivalence of monoidal categories
\[
\Kb \sfT_{\Loop^+ G, \Loop^+ G} \simto \Db \sfP_{\Loop^+ G, \Loop^+ G};
\]
composing the functors above with the inverse equivalence, and then restricting to $\sfP_{\Loop^+ G, \Loop^+ G}$, we obtain the desired extension.
\end{proof}

\appendix

%%%%%%%%%%%%%%%%%%%%%%%%%%%%%%%%%%%%%%%%%%%%%%%
\section{Coherent sheaves supported on a subscheme}
\label{sec:app-completions}
%%%%%%%%%%%%%%%%%%%%%%%%%%%%%%%%%%%%%%%%%%%%%%%

In this appendix we prove a general statement on coherent sheaves that generalizes~\cite[Lemma in~\S 3.1.7]{bmr} and is used in various settings in several places of the paper.

%------------------------------------------------------
\subsection{Statement}
\label{ss:coh-subscheme-statement}
%------------------------------------------------------

Let $k$ be a commutative ring. Let $X$ be a noetherian $k$-scheme, let $H$ be a flat affine group scheme over $k$ acting on $X$, and let $\scA$ be an $H$-equivariant coherent sheaf of $\scO_X$-algebras. In other words, $\scA$ is a (not necessarily commutative) sheaf of rings on $X$ endowed with a morphism of sheaves of rings $\scO_X \to \scA$ which takes values in the center of $\scA$ and makes $\scA$ a coherent $\scO_X$-module, and with a structure of $H$-equivariant quasi-coherent sheaf (see~\cite[Appendix]{mr1} for the definition and references) such that the multiplication morphism $\scA \otimes_{\scO_X} \scA \to \scA$ is a morphism of $H$-equivariant quasi-coherent sheaves.
We can then consider the abelian category $\QCoh^H(X,\scA)$ of $H$-equivariant quasi-coherent sheaves on $X$ endowed with a compatible left action of $\scA$, and its full abelian subcategory $\Coh^H(X,\scA)$ whose objects are coherent as $\scO_X$-modules.

Let now $Y \subset X$ be an $H$-stable closed subscheme, corresponding to a quasi-coherent ideal $\scI \subset \scO_X$. Recall that a quasi-coherent sheaf $\scF$ on $X$ is said to be \emph{set-theoretically supported on $Y$} if its restriction to $X \smallsetminus Y$ vanishes, or equivalently if each local section of $\scF$ is annihilated by a power of $\scI$. In particular, a coherent sheaf $\scF$ on $X$ is set-theoretically supported on $Y$ iff the multiplication morphism $\scI^n \otimes_{\scO_X} \scF \to \scF$ vanishes for $n \gg 0$. We will denote by
\[
\QCoh^H_Y(X,\scA) \subset \QCoh^H(X,\scA), \quad \Coh^H_Y(X,\scA) \subset \Coh^H(X,\scA)
\]
the Serre subcategories of sheaves set-theoretically supported on $Y$.

Our goal in this section is to prove the following.

\begin{prop}
\label{prop:sheaves-support}
The obvious functors
\[
\Db \Coh_Y^H(X,\scA) \to \Db \Coh^H(X,\scA) \quad \text{and} \quad
\Db \QCoh_Y^H(X,\scA) \to \Db \QCoh^H(X,\scA)
\]
are fully faithful; their essential images consist of complexes all of whose cohomology objects are set-theoretically supported on $Y$. For the first functor, the essential image can also be described as the full subcategory whose objects are the complexes $\scF$ such that there exists $n \geq 0$ such that for any open subscheme $U \subset X$ the natural morphism
\[
\Gamma(U,\scI^n) \to \End_{\Db \Coh(U)}(\scF_{|U})
\]
vanishes.
\end{prop}

%------------------------------------------------------
\subsection{Preliminaries}
\label{ss:prelim-support}
%------------------------------------------------------

%We start by considering the special case when $X=\Spec(R)$ is affine. In this case $R$ admits an action of $H$ by algebra automorphisms, and $A:=\Gamma(X,\scA)$ also admits a compatible action of $H$ by algebra automorphisms. The category $\QCoh^H(X,\scA)$ identifies with the category $\Mod^H(A)$ of $H$-equivariant $A$-modules, and an object is set-theoretically supported on $Y$ iff it is $I$-power torsion as an $R$-module in the sense of~\cite[\href{https://stacks.math.columbia.edu/tag/05E6}{Tag 05E6}]{stacks-project}, where $I:=\Gamma(X,\scI)$.

%Before giving the proof of this lemma we explain some general facts, which involve the notion of $J$-power torsion module from~\cite[\href{https://stacks.math.columbia.edu/tag/05E6}{Tag 05E6}]{stacks-project}.
%Let $R$ be a noetherian commutative ring, $J \subset R$ be an ideal, $H$ be a flat affine $R$-group scheme, and let $A$ be a (non necessarily commutative) $R$-algebra endowed with a structure of $H$-module such that the unit is $H$-invariant and the multiplication morphism $A \otimes_R A \to A$ is a morphism of $H$-modules. We will assume that $A$ is finitely generated as an $R$-module.
%; then an $A$-module is finitely generated iff it is finitely generated as an $R$-module; in particular, $A$ is left and right noetherian. 
%Denote by $\Mod^H(A)$ the category of $H$-equivariant $A$-modules, i.e.~$A$-modules $M$ endowed with a structure of $H$-module such that the action morphism $A \otimes_R M \to M$ is a morphism of $H$-modules.

First we consider the categories $\QCoh_Y(X,\scA)$, $\QCoh(X,\scA)$ as above for the trivial group scheme (which we omit from notation).

\begin{lem}
\label{lem:support-modules-inj}
For any $\scF$ in $\QCoh_Y(X,\scA)$, there exists $\scG$ in $\QCoh_Y(X,\scA)$ which is injective in $\QCoh(X,\scA)$ and an embedding $\scF \hookrightarrow \scG$ in $\QCoh_Y(X,\scA)$.
%If $M$ is an $H$-equivariant $A$-module which is $I$-power torsion as an $R$-module, there exists an injective object $N$ in $\Mod^H(A)$ which is $I$-power torsion as an $R$-module and an embedding $M \hookrightarrow N$ in $\Mod^H(A)$.
\end{lem}

\begin{proof}
This lemma is the main ingredient of the proof in~\cite[Lemma in~\S 3.1.7]{bmr}. We repeat its proof for the reader's convenience.

First, assume that $X=\Spec(R)$ is affine (with $R$ a noetherian $k$-algebra), and set $I:=\Gamma(X,\scI) \subset R$. Then $A:=\Gamma(X,\scA)$ is an $R$-algebra, $\QCoh(X,\scA)$ identifies with the category $\Mod(A)$ of $A$-modules, and an object is set-theoretically supported on $Y$ iff it is $I$-power torsion as an $R$-module in the sense of~\cite[\href{https://stacks.math.columbia.edu/tag/05E6}{Tag 05E6}]{stacks-project}. We will use the notation of~\cite[\href{https://stacks.math.columbia.edu/tag/0ALX}{Tag 0ALX}]{stacks-project} and, for an $R$-module $M$, denote by $M[I^\infty]$ its maximal $I$-power torsion submodule, i.e.~the submodule consisting of elements annihilated by a power of $I$.

 Consider the natural forgetful functor
\[
 \Mod(A) \to \Mod(R).
\]
This functor is exact, and admits as right adjoint the ``coinduction'' functor $\mathrm{Coind}_A$ defined by
\[
 \mathrm{Coind}_A(M) = \Hom_R(A,M),
\]
where $A$ acts on $\Hom_R(A,M)$ via right multiplication on $A$. 
%(Here, since $A$ is finite as an $R$-module and $H$ is flat over $R$, the $R$-module $\Hom_R(A,M)$ admits a canonical structure of $H$-module, see~\cite[Lemma~3.8(ii)]{br-Hecke}.) 
This coinduction functor therefore sends injective $R$-modules to injective $A$-modules. By construction, it sends $R$-modules which are $I$-power torsion to objects of $\Mod(A)$ which are $I$-power torsion as $R$-modules. (This claim uses our assumption that $\scA$ is coherent, i.e.~that $A$ is a finite $R$-module.) Moreover, if $M \in \Mod(A)$ and $M' \in \Mod(R)$, the adjunction isomorphism
\[
 \Hom_R(M,M') \simto \Hom_{A}(M, \mathrm{Coind}_A(M'))
\]
sends injective morphisms to injective morphisms.

Let $M$ be an $A$-module which is $I$-power torsion as an $R$-module. Consider an injective $R$-module $Q$ and an embedding $M \hookrightarrow Q$. Then if $Q':=Q[I^\infty]$, 
%where we use the notation of~\cite[\href{https://stacks.math.columbia.edu/tag/0ALX}{Tag 0ALX}]{stacks-project}, 
the morphism $M \hookrightarrow Q$ factors through an embedding $M \hookrightarrow Q'$, and $Q'$ is an injective $R$-module which is $I$-power torsion
by~\cite[\href{https://stacks.math.columbia.edu/tag/08XW}{Tag 08XW}]{stacks-project}. By adjunction, from the morphism $M \to Q'$ we deduce a morphism of $A$-modules $M \to \mathrm{Coind}_A(Q')$ which is injective, and the right-hand side is injective in $\Mod(A)$ and $I$-power torsion as an $R$-module; this proves the lemma in this case.

For the general case, we consider a finite affine covering $X = \bigcup_a U_a$. For any $a$ we denote by $j_a : U_a \to X$ the embedding; then the (exact) restriction functor
\[
\QCoh(X,\scA) \to \QCoh(U_a, \scA_{| U_a})
\]
admits as right adjoint the pushforward functor $(j_a)_*$; the latter functor therefore sends injective objects to injective objects. (Here $U_a$ is a noetherian scheme, see~\cite[\href{https://stacks.math.columbia.edu/tag/02IK}{Tag 02IK}]{stacks-project}, hence $j_a$ is quasi-compact, see~\cite[\href{https://stacks.math.columbia.edu/tag/01P0}{Tag 01P0}]{stacks-project}. This morphism is also separated since $U_a$ is affine, see~\cite[\href{https://stacks.math.columbia.edu/tag/01KN}{Tag 01KN}]{stacks-project}. Hence the functor $(j_a)_*$ sends quasi-coherent sheaves to quasi-coherent sheaves by~\cite[\href{https://stacks.math.columbia.edu/tag/01LC}{Tag 01LC}]{stacks-project}.) By the case treated above, there exists $\scG_a \in \QCoh(U_a,\scA_{| U_a})$ which is set-theoretically supported on $Y \cap U_a$ and an embedding $\scF_{| U_a} \hookrightarrow \scG_a$, which by adjunction provides a morphism $\scF \to (j_a)_* \scG_a$. Here both sheaves are supported set-theoretically on $Y$, the direct sum 
\[
\scF \to \bigoplus_a (j_a)_* \scG_a
\]
is injective, and the right-hand side is an injective object in $\QCoh(X,\scA)$.
\end{proof}

Next we generalize Lemma~\ref{lem:support-modules-inj} to the equivariant setting.

\begin{lem}
\label{lem:support-modules-inj-equiv}
For any $\scF$ in $\QCoh^H_Y(X,\scA)$, there exists $\scG$ in $\QCoh^H_Y(X,\scA)$ which is injective in $\QCoh^H(X,\scA)$ and an embedding $\scF \hookrightarrow \scG$ in $\QCoh_Y^H(X,\scA)$.
\end{lem}

\begin{proof}
Consider the action and projection morphisms
\[
a,p : H \times_k X \to X.
\]
Then the pullback functor $p^*$ sends quasi-coherent sheaves to quasi-coherent sheaves by~\cite[\href{https://stacks.math.columbia.edu/tag/01BG}{Tag 01BG}]{stacks-project}, and so does the functor $a_*$ by~\cite[\href{https://stacks.math.columbia.edu/tag/01LC}{Tag 01LC}]{stacks-project}. (Here $a$ is affine, being the composition of an isomorphism with the affine morphism $p$; hence it is quasi-compact and separated by~\cite[\href{https://stacks.math.columbia.edu/tag/01S7}{Tag 01S7}]{stacks-project}.)

The structure we have on $\scA$ provides an isomorphism of $\scO_{H \times_k X}$-algebras $a^* \scA \cong p^* \scA$. The functor $p^*$ provides an equivalence of categories
\[
\QCoh(X, \scA) \simto \QCoh^H(H \times_k X, p^* \scA),
\]
and we have a natural pushforward functor
\[
a_* : \QCoh^H(H \times_k X, a^* \scA) \to \QCoh^H(X, \scA).
\]
Standard arguments show that the composition
\[
\mathsf{Av}_H := a_* \circ p^* : \QCoh(X, \scA) \to \QCoh^H(X,\scA)
\]
is right adjoint to the (exact) forgetful functor
\[
\For_H : \QCoh^H(X, \scA) \to \QCoh(X, \scA);
\]
this functor therefore sends injective objects in $\QCoh(X, \scA)$ to injective objects in $\QCoh^H(X, \scA)$. It is also easily seen that $\mathsf{Av}_H$ sends sheaves set-theoretically supported on $Y$ to sheaves set-theoretically supported on $Y$, and that for $\scF$ in $\QCoh^H(X)$ and $\scG$ in $\QCoh(X)$ the adjunction isomorphism
\[
\Hom_{\QCoh(X)}(\For_H(\scF), \scG) \cong \Hom_{\QCoh^H(X)}(\scF, \Av_H(\scG))
\]
sends injective morphisms to injective morphisms. Now by Lemma~\ref{lem:support-modules-inj} there exists $\scG'$ in $\QCoh_Y(X,\scA)$ which is injective in $\QCoh(X,\scA)$ and an embedding $\For_H(\scF) \hookrightarrow \scG'$ in $\QCoh_Y(X,\scA)$. By adjunction we deduce an injective morphism $\scF \hookrightarrow \mathsf{Av}_H(\scG')$, where $\mathsf{Av}_H(\scG')$ is set-theoretically supported on $Y$ and injective in $\QCoh^H(X, \scA)$.
\end{proof}

%---------------------------------------------------------------
\subsection{Proof of Proposition~\ref{prop:sheaves-support}}
%---------------------------------------------------------------

%Our assumption that $A$ is finitely generated as an $R$-module implies that an $A$-module is finitely generated iff it is finitely generated as an $R$-module; in particular, $A$ is left and right noetherian. We can therefore consider the abelian category
%\[
% \Modfg^H(A)
%\]
%of $H$-equivariant $A$-modules which are finitely generated as $R$-modules. We will denote by
%\[
% \Mod^H_{\mathrm{fg},J\mhyphen\mathrm{tor}}(A)
%\]
%the full subcategory whose objects are the modules which are $J$-power torsion as $R$-modules.
%
%\begin{prop}
%\label{prop:nilp-modules-ff}
% The obvious functor
% \[
%  \Db \Mod^H_{\mathrm{fg},J\mhyphen\mathrm{tor}}(A) \to \Db \Modfg^H(A)
% \]
%is fully faithful; its essential image consists of complexes all of whose cohomology objects belong to $\Mod^H_{\mathrm{fg},J\mhyphen\mathrm{tor}}(A)$.
%\end{prop}
%
%\begin{proof}

Lemma~\ref{lem:support-modules-inj-equiv} and standard arguments (see~\cite[Chap.~I, Proposition~4.8]{hartshorne-RD}) imply that the canonical functor
\[
D^+ \QCoh_Y^H(X,\scA) \to D^+ \QCoh^H(X,\scA)
\]
is fully faithful.
%, with essential image the subcategory whose objects are the complexes all of whose cohomology modules belong to $\Mod^H_{J\mhyphen\mathrm{tor}}(A)$. 
Restricting this functor to complexes with bounded cohomology, we deduce that the natural functor
\[
\Db \QCoh_Y^H(X,\scA) \to \Db \QCoh^H(X,\scA)
\]
is fully faithful.
%, with essential image the subcategory whose objects are the complexes all of whose cohomology modules belong to $\Mod^H_{J\mhyphen\mathrm{tor}}(A)$. 
Using the fact that $\scA$ is a coherent sheaf on the noetherian scheme $X$, standard arguments (see e.g.~\cite[\S 2.2]{arinkin-bezrukavnikov}) show that the natural functors
\[
\Db \Coh^H(X,\scA) \to \Db \QCoh^H(X,\scA) \quad \text{and} \quad \Db \Coh_Y^H(X,\scA) \to \Db \QCoh_Y^H(X,\scA)
\]
are fully faithful.
%, with essential image consisting of complexes all of whose cohomology objects are finitely generated. 
We therefore have a diagram
\[
\xymatrix{
\Db \Coh_Y^H(X,\scA) \ar[r] \ar[d] & \Db \QCoh_Y^H(X,\scA) \ar[d] \\
\Db \Coh^H(X,\scA) \ar[r] & \Db \QCoh^H(X,\scA)
}
\]
in which the horizontal and right vertical arrows are fully faithful, which implies that the left vertical arrow is fully faithful hence finishes the proof of the first claim.

Now that we know that our functors are fully faithful, we deduce that their essential images are triangulated subcategories. As a consequence, the essential image of the first, resp.~second, functor is the full triangulated subcategory generated by $\Coh_Y^H(X,\scA)$, resp.~$\QCoh_Y^H(X,\scA)$, which clearly coincides with the full subcategory whose objects are the complexes all of whose cohomology objects belong to $\Coh_Y^H(X,\scA)$, resp.~$\QCoh_Y^H(X,\scA)$.

Finally we consider the case of coherent sheaves. Any complex in the essential image of the functor $\Db \Coh_Y^H(X,\scA) \to \Db \Coh^H(X,\scA)$ can be represented by a bounded complex of coherent sheaves set-theoretically supported on $Y$, so that for $n$ large enough the morphism $\Gamma(U,\scI^n) \to \End_{\Db \Coh(U)}(\scF_{|U})$ vanishes for any open subscheme $U$. On the other hand, if a complex satisfies this condition, then all of its cohomology sheaves are set-theoretically supported on $Y$, so that the complex belongs to the essential image of $\Db \Coh_Y^H(X,\scA)$ by the first description of this essential image.
%\end{proof}

%------------------------------------------------------
\subsection{A special case}
\label{ss:app-special-case}
%------------------------------------------------------

The setting we encounter in the body of the paper is the following. Let $k$ be a noetherian commutative ring and $H$ a flat affine group scheme over $\Spec(k)$. Let 
%$A$ be a commutative noetherian $k$-algebra equipped with an action of $H$ by algebra automorphisms, let 
$X$ be a $k$-scheme of finite type endowed with an action of $H$; 
%and let $f : X \to \Spec(A)$ be an $H$-equivariant morphism of finite type; in particular 
then $X$ is noetherian by~\cite[\href{https://stacks.math.columbia.edu/tag/01T6}{Tag 01T6}]{stacks-project}. Let also $\scA$ be an $H$-equivariant coherent sheaf of $\scO_X$-algebras.

We fix an ideal $J \subset k$, and denote by $k^\wedge$ the completion of $k$ with respect to $J$ (a noetherian ring, see~\cite[\href{https://stacks.math.columbia.edu/tag/05GH}{Tag 05GH}]{stacks-project}). We set
\[
X^\wedge := X \times_{\Spec(k)} \Spec(k^\wedge), \quad H^\wedge := H \times_{\Spec(k)} \Spec(k^\wedge);
\]
then $H^\wedge$ is a flat group scheme over $\Spec(k^\wedge)$, it acts naturally on $X^\wedge$, and $f$ induces a morphism of finite type $X^\wedge \to \Spec(k^\wedge)$ (see~\cite[\href{https://stacks.math.columbia.edu/tag/01T4}{Tag 01T4}]{stacks-project}), so that again $X^\wedge$ is a noetherian scheme. We also denote by $\scA^\wedge$ the pullback of $\scA$ to $X^\wedge$ (an $H^\wedge$-equivariant coherent sheaf of $\scO_{X^\wedge}$-algebras). Let $J^\wedge$ be the completion of the $R$-module $J$ with respect to $J$; then $J^\wedge$ identifies with the ideal in $k^\wedge$ generated by $J$, see~\cite[\href{https://stacks.math.columbia.edu/tag/031C}{Tag 031C}]{stacks-project}. Let also $Y \subset X$, resp.~$Y^\wedge \subset X^\wedge$, be the pullback of the closed subscheme in $\Spec(k)$, resp.~$\Spec(k^\wedge)$, defined by $J$, resp.~$J^\wedge$. We can then consider the categories
\[
\Coh^H_Y(X, \scA) \quad \text{and} \quad \Coh^{H^\wedge}_{Y^\wedge}(X^\wedge, \scA^\wedge).
\]
Here $\Coh^H_Y(X, \scA)$, resp.~$\Coh^{H^\wedge}_{Y^\wedge}(X^\wedge, \scA^\wedge)$, identifies with the full subcategory of $\Coh^H(X, \scA)$, resp.~$\Coh^{H^\wedge}(X^\wedge, \scA^\wedge)$, whose objects are the sheaves which are annihilated by a power of $J$, resp.~$J^\wedge$. Since, for any $n \geq 0$, the natural morphism $k/J^n \to k^\wedge / (J^\wedge)^n$ is an isomorphism (see~\cite[\href{https://stacks.math.columbia.edu/tag/031C}{Tag 031C}]{stacks-project}), we deduce that the projection morphism $X^\wedge \to X$ induces an equivalence of categories
\[
\Coh^{H^\wedge}_{Y^\wedge}(X^\wedge, \scA^\wedge) \simto \Coh^H_Y(X, \scA).
\]
In this setting, Proposition~\ref{prop:sheaves-support} therefore provides a fully faithful functor
\[
\Db \Coh^H_Y(X, \scA) \to \Db \Coh^{H^\wedge}(X^\wedge, \scA^\wedge)
\]
whose essential image is the full subcategory whose objects are the complexes $\scF$ all of whose cohomology objects belong to $\Coh^{H^\wedge}_{Y^\wedge}(X^\wedge, \scA^\wedge)$, or equivalently such that there exists $n \geq 0$ such that the morphism
\[
J^n \to \End_{\Db \Coh^{H^\wedge}(X^\wedge, \scA^\wedge)}(\scF)
\]
vanishes.

%------------------------------------------------------
\subsection{Description in terms of a colimit of categories}
%------------------------------------------------------

We come back to the general setting of~\S\ref{ss:coh-subscheme-statement}.
In the body of the paper we will also require a different description of the category $\Db \Coh_Y^H(X,\scA)$, as follows. For any $n \geq 1$ we can consider the closed subscheme $Y^{[n]} \subset X$ determined by the ideal $\scI^n$; in other words, $Y^{[n]}$ is the $n$-th infinitesimal neighborhood of $Y$ in $X$. The action of $H$ on $X$ restricts to an action on $Y^{[n]}$, so that we can consider the category $\Coh^H(Y^{[n]}, \scA_{| Y^{[n]}})$ of $H$-equivariant modules over the restriction $\scA_{|Y^{[n]}}$ of $\scA$ to $Y^{[n]}$ which are coherent as $\scO_{Y^{[n]}}$-modules. For any $n \geq 1$ we have a closed immersion $Y^{[n]} \to Y^{[n+1]}$, and the associated pushforward functor
\[
\Db \Coh^H(Y^{[n]}, \scA_{| Y^{[n]}}) \to \Db \Coh^H(Y^{[n+1]}, \scA_{| Y^{[n+1]}}),
\]
and we will consider the colimit category
\[
\mathrm{colim}_{n \geq 1} \Db \Coh^H(Y^{[n]}, \scA_{| Y^{[n]}}).
\]
(Note that the transition functors in this colimit are \emph{not} fully faithful.)
Recall that an object in this category is a pair $(n, \scF)$ where $n \geq 1$ and $\scF$ is an object in $\Db \Coh^H(Y^{[n]}, \scA_{| Y^{[n]}})$, and that the space of morphisms from $(n,\scF)$ to $(n',\scF')$ is
\[
\mathrm{colim}_{m \geq \max(n,n')} \Hom_{\Db \Coh^H(Y^{[m]}, \scA_{| Y^{[m]}})}(\scF,\scF')
\]
where the omit of the pushforward functors under the closed immersions $Y^{[n]} \to Y^{[m]}$ and $Y^{[n']} \to Y^{[m]}$.

\begin{lem}
\label{lem:DbCoh-colim}
There exists a canonical equivalence of categories
\[
\mathrm{colim}_{n \geq 1} \Db \Coh^H(Y^{[n]}, \scA_{| Y^{[n]}}) \simto \Db \Coh_Y^H(X,\scA).
\]
\end{lem}

\begin{proof}
For any $n \geq 1$ the pushforward functor under the closed immersion $Y^{[n]} \to X$ induces a functor $\Db \Coh^H(Y^{[n]}, \scA_{| Y^{[n]}}) \to \Db \Coh_Y^H(X,\scA)$, and taking all of these functors together provides a functor
\[
F : \mathrm{colim}_{n \geq 1} \Db \Coh^H(Y^{[n]}, \scA_{| Y^{[n]}}) \to \Db \Coh_Y^H(X,\scA).
\]
It is clear that this functor is essentially surjective. To prove that it is fully faithful we consider some objects $(n,\scF)$ and $(n',\scF')$ in the colimit category. Replacing $\scF$ or $\scF'$ by a pushforward we can (and will) assume that $n=n'$. A morphism in $\Db \Coh_Y^H(X,\scA)$ from $F(\scF)$ to $F(\scF')$ is a diagram
\[
F(\scF) \xleftarrow{f} \scG \xrightarrow{g} F(\scF')
\]
where $\scG$ is a bounded complex of objects in $\Coh_Y^H(X,\scA)$ and $f,g$ are morphisms of complexes in this category, with $f$ a quasi-isomorphism. For $m \gg 0$, $\scG$ is a complex of objects in $\Coh^H(Y^{[m]}, \scA_{| Y^{[m]}})$, proving that $F$ is full. On the other hand, a morphism from $(n,\scF)$ to $(n,\scF')$ is a diagram
\[
\scF \xleftarrow{f} \scG \xrightarrow{g} \scF'
\]
where $\scG$ is a bounded complex of objects in $\Coh^H(Y^{[m]}, \scA_{| Y^{[m]}})$ for some $m \geq n$, $f,g$ are morphisms of complexes in this category with $f$ a quasi-isomorphism, and we omit the pushforward functor under the immersion $Y^{[n]} \to Y^{[m]}$. If the image of this morphism under $F$ vanishes, by standard considerations (see e.g.~\cite[\href{https://stacks.math.columbia.edu/tag/04VJ}{Tag 04VJ}]{stacks-project}) there exists a bounded complex $\scH$ of objects in $\Coh_Y^H(X,\scA)$ and a quasi-isomorphism $h : \scH \to \scG$ such that $g \circ h=0$. For $m \gg 0$, $\scH$ is a complex of objects in $\Coh^H(Y^{[m]}, \scA_{| Y^{[m]}})$, and then the image of our morphism in $\Db \Coh^H(Y^{[m]}, \scA_{| Y^{[m]}})$ vanishes, which proves faithfulness.
\end{proof}

\section{Equivariant modules}
\label{sec:extension-scalars}
%%%%%%%%%%%%%%%%%%%%%

In this appendix we prove some generalities on categories of equivariant modules over equivariant algebras. These results (and their proofs) are very similar to those discussed in~\cite[Appendix~A]{mr1}.

%------------------------------------------------
\subsection{Morphisms and invariants}
%------------------------------------------------

Let $k$ be a noetherian commutative ring, let $A$ be a left noetherian $k$-algebra, and let $H$ be a flat affine group scheme over $k$ acting on $A$ by algebra automorphisms. Consider the categories $\Mod^H(A)$, $\Modfg^H(A)$ of $H$-equivariant $A$-modules and finitely generated $H$-equivariant $A$-modules. Since $A$ is noetherian, standard arguments (as e.g.~in~\cite[\S 2.2]{arinkin-bezrukavnikov}) show that the natural functor
\[
\Db \Modfg^H(A) \to \Db \Mod^H(A)
\]
is fully faithful.

It is a standard fact that the forgetful functor
\[
\For^H : \Mod^H(A) \to \Mod(A)
\]
admits a right adjoint
\[
\Av_H : \Mod(A) \to \Mod^H(A);
\]
explicitly this functor sends an $A$-module $M$ to the $H$-module $\mathrm{Ind}_{\{e\}}^H(M) = M \otimes_k \scO(H)$, seen as the $H$-module of algebraic functions $H \to M$ (with the action induced by multiplication on the right on $H$), with the action of $A$ given by $(a \cdot \varphi)(h) = (h \cdot a) \cdot \varphi(h)$. In other words, as an $A$-module, $\mathrm{Ind}_{\{e\}}^H(M)$ is obtained from the $A \otimes_k \scO(H)$-module $M \otimes_k \scO(H)$ (with the obvious action) by restriction of scalars along the algebra morphism $A \to A \otimes_k \scO(H)$ sending $a$ to the function $h \mapsto h \cdot a$ (i.e.~the coaction morphism).
In particular, this fact implies that the category $\Mod^H(A)$ has enough injectives; in fact any injective object is a direct summand in an object of the form $\Av_H(M)$ with $M$ an injective $A$-module. This fact implies that we can consider the derived bifunctor
\begin{equation}
\label{eqn:RHom-equiv}
R\Hom_{\Mod^H(A)} : D^- \Mod^H(A) \times D^+ \Mod^H(A) \to D^+ \Mod(k).
\end{equation}

\begin{lem}
\label{lem:injectives-equiv}
If $M \in \Modfg^{H}(A)$ and if $N \in \Mod^H(A)$ is injective, then we have
\[
\Ext^n_{\Mod(A)}(M,N)=0
\]
for any $n>0$.
\end{lem}

\begin{proof}
By the comments above we can assume that $N=\Av_H(N')$ for some injective $A$-module $N'$. Since $A \otimes_k \scO(H)$ is flat over $A$ (for the action induced by the coaction morphism as above) we have
\begin{multline*}
\Ext^n_{A}(M,N) = \Ext^n_A(M, \Av_H(N')) \\
\cong \Ext^n_{A \otimes_k \scO(H)}((A \otimes_k \scO(H)) \otimes_A M, N' \otimes_k \scO(H)).
\end{multline*}
Now we have an isomorphism of $A \otimes_k \scO(H)$-modules
\[
(A \otimes_k \scO(H)) \otimes_A M \simto M \otimes_k \scO(H)
\]
(where on the right-hand side we consider the obvious module structure), given in Sweedler's notation by
\[
(a \otimes \varphi) \otimes m \mapsto (a \cdot m_{(1)}) \otimes \varphi m_{(2)}.
\]
We deduce isomorphisms
\begin{multline*}
\Ext^n_{A \otimes_k \scO(H)}((A \otimes_k \scO(H)) \otimes_A M, N' \otimes_k \scO(H)) \\
\cong \Ext^n_{A \otimes_k \scO(H)}(M \otimes_k \scO(H), N' \otimes_k \scO(H)) \\
\cong \Ext^n_{A}(M, N' \otimes_k \scO(H))
\end{multline*}
where now $A$ acts on $N' \otimes_k \scO(H)$ via the restriction of the action of $A \otimes_k \scO(H)$ along the obvious algebra morphism $A \to A \otimes_k \scO(H)$; in other words the action is induced by the action on $N'$.

For this action, $N' \otimes_k \scO(H)$ is injective as an $A$-module, which will conclude the proof. In fact, by Baer's criterion,
%(see~\cite[\href{https://stacks.math.columbia.edu/tag/05NU}{Tag 05NU}]{stacks-project}), 
to prove this it suffices to prove that for any left ideal $I \subset A$ the map
\[
\Hom_A(A,N' \otimes_k \scO(H)) \to \Hom_A(I,N' \otimes_k \scO(H))
\]
is surjective. However $A$ and $I$ are finitely generated $A$-modules and $\scO(H)$ is flat over $k$, so that by~\cite[Lemma~3.8(i)]{br-Hecke} these spaces identify with $\Hom_A(A,N') \otimes_k \scO(H)$ and $\Hom_A(I,N') \otimes_k \scO(H)$ respectively, in such a way that the map above is induced by the similar map $\Hom_A(A,N') \to \Hom_A(I,N')$. We conclude by injectivity of $N'$ as an $A$-module and flatness of $\scO(H)$.
\end{proof}

As explained e.g.~in~\cite[Lemma~3.8(ii)]{br-Hecke}, for $M$ in $\Modfg^H(A)$ and $N$ in $\Mod^H(A)$ the $k$-module $\Hom_A(M,N)$ admits a canonical $H$-module structure. If $M$ is a bounded complex of objects in $\Modfg^H(A)$, we can therefore consider the functor
\[
\Hom^\bullet_A(M,-) : C^+ \Mod^H(A) \to C^+ \Rep^\infty(H).
\]
Since $\Mod^H(A)$ has enough injectives this functor admits a derived functor, which will be denoted
\[
R\Hom_A(M,-) : D^+ \Mod^H(A) \to D^+ \Rep^\infty(H);
\]
this notation is justified by the fact that, by Lemma~\ref{lem:injectives-equiv}, the underlying complex of $k$-modules of $R\Hom_A(M,N)$ is the complex $R\Hom_A(\For^H(M),\For^H(N))$. Note that for any $M'$ in $C^{\mathrm{b}} \Modfg^H(A)$ and any quasi-isomorphism $M \to M'$, the induced morphism $R\Hom_A(M',N) \to R\Hom_A(M,N)$ is an isomorphism, so that this construction provides a bifunctor
\[
\Db \Modfg^H(A) \times D^+ \Mod^H(A) \to D^+ \Rep^\infty(H).
\]

Now we consider the functor of $H$-invariants $(-)^H$, and denote its derived functor by
\[
\Inv^H : D^+ \Rep^\infty(H) \to D^+ \Mod(k).
\]
(This functor exists because $\Rep^\infty(H)$ has enough injectives.)

\begin{lem}
\label{lem:RHom-equiv}
For any $M \in \Db \Modfg^H(A)$ and $N \in D^+ \Mod^H(A)$, there exists a canonical (in particular, bifunctorial), isomorphism
\[
R\Hom_{\Mod^H(A)}(M,N) \simto \Inv^H \circ R\Hom_A(M,N).
\]
\end{lem}

\begin{proof}
As explained in~\cite[Lemma~3.8(ii)]{br-Hecke}, for $M$ in $\Modfg^H(A)$ and $N$ in $\Mod^H(A)$ there exists a canonical isomorphism
\[
\Hom_{\Mod^H(A)}(M,N) \simto (\Hom_A(M,N))^H.
\]
By standard properties of derived functors, we deduce a bifunctorial morphism
\[
R\Hom_{\Mod^H(A)}(M,N) \to \Inv^H \circ R\Hom_A(M,N)
\]
for $M,N$ as in the lemma. To prove that this morphism is an isomorphism, it suffices to prove that if $M \in \Modfg^H(A)$ and $N$ is an injective object of $\Mod^H(A)$ we have
\[
\mathsf{H}^n (\Inv^H(\Hom_A(M,N))=0
\]
for all $n>0$. Here we can assume that $N=\Av_H(N')$ for some injective $A$-module $N'$. Then as in the proof of Lemma~\ref{lem:injectives-equiv} we have
\[
\Hom_A(M,N) \cong \Hom_A(M, N' \otimes_k \scO(H)) \cong \Hom_A(M,N') \otimes_k \scO(H),
\]
where the second isomorphism uses~\cite[Lemma~3.8(i)]{br-Hecke}. The desired claim therefore follows from~\cite[Lemma~I.4.7]{jantzen}.
\end{proof}

%------------------------------------------------
\subsection{Extension of scalars}
%------------------------------------------------

Let now $k'$ be another noetherian commutative ring,
%of finite global dimension, 
and consider the data $A' := k' \otimes_k A$, $H' := \Spec(k') \times_{\Spec(k)} H$ obtained by base change. We will make the following assumptions:
\begin{enumerate}
\item
\label{it:assumption-ext-scalars-00}
$A$ is flat over $k$;
\item
\label{it:assumption-ext-scalars-0}
$A'$ is noetherian;
\item
\label{it:assumption-ext-scalars-1}
%$k'$ has finite flat dimension as a $k$-module. 
$k$ has finite global dimension;
\item
\label{it:assumption-ext-scalars-2}
for any $M \in \Rep(H)$, there exists an object $M' \in \Rep(H)$ which is flat over $k$ and a surjection $M' \twoheadrightarrow M$.
\end{enumerate}
%Here,~\eqref{it:assumption-ext-scalars-1} is e.g.~automatic in case $k$ has finite global dimension. 
The assumption~\eqref{it:assumption-ext-scalars-2} is subtle, as discussed (in a more general setting) in~\cite{thomason}. Note that by~\cite[Corollary~2.9]{thomason}, it is satisfied at least if $H$ is a split reductive group scheme over $k$.
%more subtle, already in case $A=k$, as discussed (in a more general setting) in~\cite{thomason}. Here we explain why it is satisfied in the setting we consider in the body of the paper.

Let us note the following consequence of~\eqref{it:assumption-ext-scalars-2}.

\begin{lem}
\label{lem:equiv-modules-resolution}
%Assume that $H$ is a split reductive group scheme over $k$. Then assumption~\eqref{it:assumption-ext-scalars-2} holds.
For any $M \in \Modfg^H(A)$, resp.~$M \in \Mod^H(A)$, there exists an object $M' \in \Modfg^{H}(A)$, resp.~$M' \in \Mod^{H}(A)$, which is flat over $A$ and a surjection $M' \twoheadrightarrow M$.
\end{lem}

\begin{proof}
It suffices to prove the claim for finitely generated modules; the general case then follows by the same considerations as in~\cite[\S 2.2]{thomason}.
Let $M \in \Modfg^H(A)$, and let $M' \subset M$ be an $H$-stable finitely generated $k$-submodule which generates $M$ as an $A$-module. (Such a submodule exists by~\cite[I.2.13(3)]{jantzen}.) By assumption there exists $N' \in \Rep(H)$ which is flat over $k$ and a surjection $N \twoheadrightarrow M'$. Then $A \otimes_k N$ is naturally an object of $\Modfg^H(A)$ which is flat as an $A$-module, and the natural morphism $A \otimes_k N \to M$ is surjective.
\end{proof}

Consider the categories $\Mod^H(A)$, $\Modfg^H(A)$ and $\Mod^{H'}(A')$, $\Modfg^{H'}(A')$, and the natural functor
\begin{equation}
\label{eqn:ext-scalars-functor}
\Mod^H(A) \to \Mod^{H'}(A')
\end{equation}
given by $M \mapsto A' \otimes_A M = k' \otimes_k M$. By Lemma~\ref{lem:equiv-modules-resolution}, any $M \in \Mod^H(A)$ is a quotient of an object of $\Mod^H(A)$ which is flat as an $A$-module. (Moreover, in case $M$ is finitely generated, this module can be taken to be finitely generated too.) It follows that the functor~\eqref{eqn:ext-scalars-functor} has a derived functor
\begin{equation*}
%\label{eqn:ext-scalars-functor}
D^- \Mod^H(A) \to D^- \Mod^{H'}(A')
\end{equation*}
such that the diagram
\[
\xymatrix@C=2cm{
D^- \Mod^H(A) \ar[r] \ar[d]_-{\For^H} & D^- \Mod^{H'}(A') \ar[d]^-{\For^H} \\
D^- \Mod(A) \ar[r]^{A' \lotimes_A (-)} & D^- \Mod(A')
}
\]
commutes. This functor will therefore also be denoted $A' \lotimes_A (-)$. It restricts to a functor
\[
D^- \Modfg^H(A) \to D^- \Modfg^{H'}(A').
\]
%given by $M \mapsto A' \lotimes_A M$. 
Since $A$ is flat over $k$ (see Assumption~\eqref{it:assumption-ext-scalars-00}), the underlying complex of $k'$-modules of $A' \lotimes_A M$ is $k' \lotimes_k M$; by our assumption~\eqref{it:assumption-ext-scalars-1} this implies that the functor above restricts to a functor
\[
\Db \Mod^H(A) \to \Db \Mod^{H'}(A'),
\]
hence also to a functor
\[
\Db \Modfg^H(A) \to \Db \Modfg^{H'}(A').
\]

Consider now the bifunctor~\eqref{eqn:RHom-equiv}, and its analogue for $A'$.

\begin{lem}
\label{lem:RHom-equiv-ext-scalars}
For any $M \in \Db \Modfg^H(A)$ and $N \in D^+ \Mod^H(A)$, 
%the functor~\eqref{eqn:ext-scalars-functor} induces 
we have a canonical isomorphism
\[
k' \lotimes_k R\Hom_{\Mod^H(A)}(M,N) \simto R\Hom_{\Mod^{H'}(A')}(A' \lotimes_A M, A' \lotimes_A N).
\]
\end{lem}

\begin{proof}
By Lemma~\ref{lem:RHom-equiv}, for $M,N$ as in the lemma there exists a canonical isomorphism
\[
k' \lotimes_k R\Hom_{\Mod^H(A)}(M,N) \cong k' \lotimes_k \left( \Inv^H \circ R\Hom_{A}(M,N) \right).
\]
Using~\cite[Proposition~A.8]{mr1} we deduce a canonical isomorphism
\[
k' \lotimes_k R\Hom_{\Mod^H(A)}(M,N) \cong \Inv^{H'}  \left( k' \lotimes_k  R\Hom_{A}(M,N) \right),
\]
where in the right-hand side we consider the derived extension-of-scalars functor for $H$-modules (which is well defined thanks to our assumption~\eqref{it:assumption-ext-scalars-2}). Now, as in~\cite[\href{https://stacks.math.columbia.edu/tag/0A6A}{Tag 0A6A}]{stacks-project} we have a canonical isomorphism
\[
k' \lotimes_k  R\Hom_{A}(M,N) \cong  R\Hom_{A'}(A' \lotimes_A M, A' \lotimes_A N)
\]
in $D^+ \Rep^\infty(H')$. 
%(Here we use the fact that $k$ has finite global dimension, so that $A' \lotimes_A M \cong k' \lotimes_k M$ can be computed by taking a \emph{finite} resolution by finitely generated $A$-modules which are flat over $k$.) 
The claim then follows from another application of Lemma~\ref{lem:RHom-equiv}.
\end{proof}


\begin{thebibliography}{ALWY}

\bibitem[Ab1]{abe-Hecke}
N.~Abe, \emph{A bimodule description of the Hecke category},
Compos. Math. \textbf{157} (2021), no. 10, 
2133--2159.

\bibitem[Ab2]{abe}
N.~Abe, \emph{On one-sided singular Soergel bimodules}, J. Algebra~\textbf{633} (2023), 722--753.
%preprint~\href{https://arxiv.org/abs/2004.09014}{arXiv:2004.09014}.

% \bibitem[ACR]{acr}
% P.~Achar, N.~Cooney, and S.~Riche, \emph{The parabolic exotic t-structure}, pre\-print~\texttt{arXiv:1805.05624}.

\bibitem[Ac]{achar}
P.~Achar, \emph{On exotic and perverse-coherent sheaves}, in \emph{Representations of reductive groups}, 11--49, Progr. Math. 312,
Birkh\"auser/Springer, Cham, 2015.

\bibitem[ADR]{adr}
P.~Achar, G.~Dhillon and S.~Riche, \emph{Categorical fixed points and monoidality for localization theorems}, in preparation.

\bibitem[AHR]{ahr}
P.~Achar, W.~Hardesty, and S.~Riche, \emph{Integral exotic sheaves and the modular Lusztig-Vogan bijection}, J. Lond. Math. Soc. (2) \textbf{106} (2022), no.~3, 2403--2458.

%\bibitem[AMRW]{amrw}
%P.~Achar, S.~Makisumi, S.~Riche, and G.~Williamson, \emph{Koszul duality for Kac--Moody groups and characters of tilting modules},  J. Amer. Math. Soc. \textbf{32} (2019), no. 1, 
%261--310.
%% %preprint \texttt{arXiv:1706.00183}, to appear in J. Amer. Math. Soc.

\bibitem[AR1]{modrap1}
P.~Achar and S.~Riche, {\em Modular perverse sheaves on flag varieties I:
  tilting and parity sheaves} (with a joint appendix with G.~Williamson), Ann. Sci. \'{E}c. Norm. Sup\'{e}r.~\textbf{49} (2016), 325--370.

%\bibitem[AR2]{modrap2}
%P.~Achar and S.~Riche, \emph{Modular perverse sheaves on flag varieties II: Koszul duality and formality}, Duke Math. J.~\textbf{165} (2016), 161--215.

%\bibitem[AR2]{prinblock}
%P.~Achar and S.~Riche, \emph{Reductive groups, the loop Grassmannian, and the Springer resolution}, 
%%preprint \texttt{arXiv:1602.04412}, to appear in 
%Invent. Math.~\textbf{214} (2018), 289--436.

\bibitem[AR2]{ar-book}
P.~Achar and S.~Riche, \emph{Central sheaves on affine flag varieties}, book in preparation, available at~\url{https://lmbp.uca.fr/~riche/central.pdf}.

\bibitem[ARV]{arv}
P.~Achar, S.~Riche, and C.~Vay, \emph{Mixed perverse sheaves on flag varieties for Coxeter groups}, Canad. J. Math. \textbf{72} (2020), no.~1, 1--55.

%\bibitem[ARd]{achar-rider}
%P.~Achar and L.~Rider, \emph{Parity sheaves on the affine Grassmannian and the Mirkovi\'c-Vilonen conjecture},
%Acta Math. \textbf{215} (2015), 
%%no. 2, 
%183--216. 

\bibitem[ALWY]{alwy}
J.~Ansch{\" u}tz, J.~Louren{\c c}o, Z.~Wu, and J.~Yu, \emph{Gaitsgory's central functor and the Arkhipov--Bezrukavnikov equivalence in mixed characteristic}, preprint~\href{https://arxiv.org/abs/2311.04043}{arXiv:2311.04043}.

%\bibitem[AG]{ag}
%D.~Arinkin and D.~Gaitsgory, \emph{Asymptotics of geometric Whittaker coefficients}, available at \texttt{http://www.math.harvard.edu/$\sim$gaitsgde/GL/WhitAsympt.pdf}.

\bibitem[AriB]{arinkin-bezrukavnikov}
 D.~Arinkin and R.~Bezrukavnikov, \emph{Perverse coherent sheaves},
 Mosc. Math. J. \textbf{10} (2010), no. 1, 
 3--29, 271.

\bibitem[AB]{ab}
S.~Arkhipov and R.~Bezrukavnikov, \emph{Perverse sheaves on affine flags and Langlands dual group} (with an appendix by R.~Bezrukavnikov and I.~Mirkovi{\'c}), Israel J. Math. \textbf{170} (2009), 135--183.

% \bibitem[ABBGM]{abbgm}
% S.~Arkhipov, R.~Bezrukavnikov, A.~Braverman, D.~Gaitsgory, and I.~Mirkovi{\'c},
% \emph{Modules over the small quantum group and semi-infinite flag manifold},
% Transform. Groups \textbf{10} (2005), 
% %no. 3-4, 
% 279--362. 

%\bibitem[ABG]{abg}
%S.~Arkhipov, R.~Bezrukavnikov, and V.~Ginzburg, {\em Quantum groups, the loop Grassmannian, and the Springer resolution}, J. Amer. Math. Soc. {\bf 17} (2004), 595--678.

\bibitem[BaR]{bardsley-richardson}
P.~Bardsley and R.~Richardson, \emph{\'Etale slices for algebraic transformation groups in characteristic $p$},
Proc. London Math. Soc. \textbf{51} (1985), no. 2, 
295--317. 

%\bibitem[Ba]{barry}
%M.~Barry, \emph{Bases for fixed points of unipotent elements acting on the tensor square and the spaces of alternating and symmetric 2-tensors},
%J. Algebra \textbf{251} (2002), 
%%no. 1, 
%395--412. 

%\bibitem[BaR]{br}
%P.~Baumann and S.~Riche, \emph{Notes on the geometric Satake equivalence}, 
%in \emph{Relative Aspects in Representation Theory, Langlands Functoriality and Automorphic Forms, CIRM Jean-Morlet Chair, Spring 2016} (V. Heiermann, D. Prasad, Eds.), 1--134, Lecture Notes in Math. 2221, Springer, 2018.
%preprint \texttt{arXiv:1703.07288}, to appear in Lecture Notes in Mathematics 2221.

%\bibitem[BBD]{bbd}
%A.~Be{\u\i}linson, J.~Bernstein, and P.~Deligne, \textit{Faisceaux pervers}, in \emph{Analyse et topologie sur les espaces singuliers, I (Luminy, 1981)}, 5--172,
%Ast\'erisque \textbf{100} (1982).
%%, Soc. Math. France, 1982.
%
%\bibitem[BBM]{bbm}
%A.~Be{\u\i}linson, R.~Bezrukavnikov, and I.~Mirkovi{\'c}, \emph{Tilting exercises}, Mosc.~Math.~J.~\textbf{4} (2004), 547--557, 782. 
%
%\bibitem[BD]{bd}
%A.~Be{\u\i}linson and V. Drinfeld, \emph{Quantization of Hitchin's integrable system and Hecke
%eigensheaves}, unpublished preprint available at \url{http://www.math.uchicago.edu/~mitya/langlands.html}.

\bibitem[BGS]{bgs}
A.~Be{\u\i}linson, V.~Ginzburg, and W.~Soergel, {\em Koszul duality patterns in
  representation theory}, J.~Amer.~Math.~Soc.~{\bf 9} (1996), 473--527.
  
  \bibitem[BCHN]{bchn}
  D.~Ben-Zvi, H.~Chen, D.~Helm, and D.~Nadler, \emph{Coherent Springer theory and the categorical Deligne--Langlands correspondence}, 
  Invent. Math. \textbf{235} (2024), no. 2, 255--344.
  %preprint~\href{https://arxiv.org/abs/2010.02321}{arXiv:2010.02321}, to appear in Invent. Math.
  
  \bibitem[BFN]{bzfn}
  D.~Ben-Zvi, J.~Francis, and D.~Nadler, \emph{Integral transforms and Drinfeld centers in derived algebraic geometry},
J. Amer. Math. Soc. \textbf{23} (2010), no. 4, 909--966.

\bibitem[Ber]{beraldo}
D.~Beraldo, \emph{Loop group actions on categories and Whittaker invariants}, Adv. Math. \textbf{322} (2017), 565--636.
  
%\bibitem[BL1]{bl1}
%J.~Bernstein and V.~Lunts, \emph{Equivariant sheaves and functors},
%Lecture Notes in Mathematics 1578, Springer-Verlag, Berlin, 
%1994.

\bibitem[BL]{bl2}
J.~Bernstein and V.~Lunts, \emph{Localization for derived categories of $(\mathfrak{g},K)$-modules},
J. Amer. Math. Soc. \textbf{8} (1995), no. 4, 819--856. 

\bibitem[Be1]{bez-pcoh}
R.~Bezrukavnikov, \emph{Perverse coherent sheaves (after Deligne)}, preprint \href{https://arxiv.org/abs/math/0005152}{arXiv:math/0005152}.

%\bibitem[Be2]{bez}
%R.~Bezrukavnikov, \emph{On tensor categories attached to cells in affine Weyl groups}, with an appendix by D. Gaitsgory, in \emph{Representation theory of algebraic groups and quantum groups}, 69--100,
%Adv. Stud. Pure Math. 40, Math. Soc. Japan, Tokyo, 
%2004. 

\bibitem[Be2]{bez-tilting}
R.~Bezrukavnikov, \emph{Cohomology of tilting modules over quantum groups and t-structures on derived categories of coherent sheaves},
Invent. Math. \textbf{166} (2006), no. 2, 327--357. 

\bibitem[Be3]{bez-nilp}
R.~Bezrukavnikov, \emph{Perverse sheaves on affine flags and nilpotent cone of the Langlands dual group},
Israel J. Math.~\textbf{170} (2009), 185--206.

\bibitem[Be4]{be}
R.~Bezrukavnikov, \emph{On two geometric realizations of an affine Hecke algebra}, Publ. Math. IHES~\textbf{123} (2016), 1--67.

% \bibitem[BBM]{bbrm}
%R.~Bezrukavnikov, A.~Braverman, and I.~Mirkovi{\'c}, \emph{Some results about geometric Whittaker model}, Adv.~Math.~\textbf{186} (2004), 143--152. 

%\bibitem[BGMRR]{bgmrr}
%R.~Bezrukavnikov, D.~Gaitsgory, I.~Mirkovi{\'c}, S.~Riche, and L.~Rider, \emph{An Iwahori-Whittaker model for the Satake category},
%J. \'Ec. polytech. Math. \textbf{6} (2019), 707--735.

\bibitem[BLo1]{blo1}
R.~Bezrukavnikov and I.~Losev, \emph{On dimension growth of modular irreducible representations of semisimple Lie algebras}, in \emph{Lie groups, geometry, and representation theory}, 59--89,
Progr. Math. 326, Birkh\"auser/Springer, Cham, 2018. 

\bibitem[BLo2]{blo2}
R.~Bezrukavnikov and I.~Losev, \emph{Dimensions of modular irreducible representations of semisimple Lie algebras}, 
 J. Amer. Math. Soc. \textbf{36} (2023), no. 4, 1235--1304. 
%preprint~\href{https://arxiv.org/abs/2005.10030}{arXiv:2005.10030}.

%\bibitem[BM1]{bm}
%R.~Bezrukavnikov and I.~Mirkovi{\'c}, appendix to~\cite{ab}.

\bibitem[BM]{bm-loc}
R.~Bezrukavnikov and I.~Mirkovi{\'c}, \emph{Representations of semisimple Lie algebras in prime characteristic and the noncommutative Springer resolution},
with an appendix by E.~Sommers,
Ann. of Math. (2) \textbf{178} (2013), no. 3, 
835--919.

\bibitem[BMR1]{bmr}
R.~Bezrukavnikov, I.~Mirkovi{\'c}, and D.~Rumynin, \emph{Localization of modules for a semisimple Lie algebra in prime characteristic}, with an appendix by R.~Bezrukavnikov and S.~Riche, Ann. of Math. \textbf{167} (2008), 945--991.

\bibitem[BMR2]{bmr2}
R.~Bezrukavnikov, I.~Mirkovi{\'c}, and D.~Rumynin, \emph{Singular localization and intertwining functors for reductive Lie algebras in prime characteristic}, Nagoya Math. J. \textbf{184} (2006), 1--55.

\bibitem[BR1]{br-Baff}
R.~Bezrukavnikov and S.~Riche, \emph{Affine braid group actions on derived categories of Springer resolutions},
Ann. Sci. \'Ec. Norm. Sup\'er. (4) \textbf{45} (2012), 
no. 4, 
535--599 (2013). 

\bibitem[BR2]{br-soergel}
R.~Bezrukavnikov and S.~Riche, \emph{A topological approach to Soergel theory}, in
\emph{Representation theory and algebraic geometry---a conference celebrating the birthdays of Sasha Be{\u\i}linson and Victor Ginzburg}, 267--343, Trends Math., Birkh\"auser/Springer, Cham, 2022. 
%pre\-print~\href{https://arxiv.org/abs/1807.07614}{arXiv:1807.07614}.

\bibitem[BR3]{br-Hecke}
R.~Bezrukavnikov and S.~Riche, \emph{Hecke action on the principal block},
Compos. Math. \textbf{158} (2022), 953--1019.
%preprint~\href{https://arxiv.org/abs/2009.10587}{arXiv:2009.10587}.

\bibitem[BR4]{br-pt2}
R.~Bezrukavnikov and S.~Riche, \emph{Modular affine Hecke category and regular centralizer}, preprint~\href{https://arxiv.org/abs/2206.03738}{arXiv:2206.03738}.

\bibitem[BRR]{brr-pt1}
R.~Bezrukavnikov, S.~Riche, and L. Rider, \emph{Modular affine Hecke category and regular unipotent centralizer}, preprint~\href{https://arxiv.org/abs/2005.05583}{arXiv:2005.05583}.

\bibitem[BY]{by}
R.~Bezrukavnikov and Z.~Yun, \emph{On Koszul duality for Kac--Moody groups}, Represent.~Theory~\textbf{17} (2013), 1--98.

\bibitem[Bj]{bjork}
 J.-E.~Bj{\"o}rk, \emph{Rings of differential operators},
North-Holland Mathematical Library 21, North-Holland Publishing Co., Amsterdam-New York, 1979.
%J.~E.~Bjork, \emph{Rings of Differential Operators}, North Holland, 1979.

\bibitem[Bo]{borel-Haff}
A.~Borel,
\emph{Admissible representations of a semi-simple group over a local field with vectors fixed under an Iwahori subgroup}, Invent. Math. \textbf{35} (1976), 233--259.

%\bibitem[Bo2]{borel}
%A.~Borel, \emph{Linear algebraic groups, second edition}, Springer, 1991.

\bibitem[BC]{bc}
A. Bouthier and K.~Cesnavicius, \emph{Torsors on loop groups and the Hitchin fibration}, 
Ann. Sci. \'Ec. Norm. Sup\'er. (4) \textbf{55} (2022), no. 3, 791--864.
%preprint arXiv:1908.07480, to appear in Ann. Sci. \'Ec. Norm. Sup\'er..

%\bibitem[BG]{bg}
%A.~Braverman and D.~Gaitsgory, \emph{Crystals via the affine Grassmannian},
%Duke Math. J. \textbf{107} (2001), 
%%no. 3, 
%561--575. 

 \bibitem[BK]{bk}
 M.~Brion and S.~Kumar, \emph{Frobenius splitting methods in geometry and representation theory}, Progr. Math.~231, Birkh{\"a}user, 2004.

\bibitem[BG]{brown-gordon}
K.~A.~Brown and I.~Gordon, \emph{The ramification of centres: Lie algebras in positive characteristic and quantised enveloping algebras},
Math. Z. \textbf{238} (2001), no. 4, 
733--779. 

% \bibitem[CPS]{cps}
% E.~Cline, B.~Parshall, and L.~Scott, \emph{Finite-dimensional algebras and highest weight categories},
% J. Reine Angew. Math. \textbf{391} (1988), 85--99.

\bibitem[CD]{chen-dhillon}
H.~Chen and G.~Dhillon, \emph{A Langlands dual realization of coherent sheaves on the nilpotent cone}, preprint~\href{https://arxiv.org/abs/2310.10539}{arXiv:2310.10539}.

\bibitem[CF]{chen-fu}
L.~Chen and Y.~Fu, \emph{An extension of the Kazhdan--Lusztig equivalence}, preprint~\href{https://arxiv.org/abs/2111.14606}{arXiv:2111.14606}.

\bibitem[CG]{cg}
N.~Chriss and V.~Ginzburg, \emph{Representation theory and complex geometry}, Birkh\"auser Boston, Inc., Boston, MA, 
1997.

\bibitem[CGP]{cgp}
B.~Conrad, O.~Gabber, and G.~Prasad, \emph{Pseudo-reductive groups,
Second edition}, New Mathematical Monographs 26, Cambridge University Press, Cambridge, 2015.

%\bibitem[DCM]{dcm}
%C.~De Concini and A.~Maffei, \emph{A generalized Steinberg section and branching rules for quantum groups at roots of $1$}
%Mosc. Math. J. \textbf{12} (2012), no. 3, 
%469--495, 668. 
%
%\bibitem[DM]{dm}
%P.~Deligne and J.~S.~Milne, \emph{Tannakian categories},
%in \emph{Hodge cycles, motives, and Shimura varieties}, 101--228,
%Lecture Notes in Mathematics~900, Springer-Verlag,Berlin-New York,
%1982.

\bibitem[De]{demazure}
M.~Demazure, \emph{Invariants sym\'etriques entiers des groupes de Weyl et torsion},
Invent. Math. \textbf{21} (1973), 287--301.

%\bibitem[DG]{dg}
%M.~Demazure and P.~Gabriel, \emph{Groupes alg\'ebriques. Tome I: G\'eom\'etrie alg\'ebrique, g\'en\'eralit\'es, groupes commutatifs}, with an appendix
%\emph{Corps de classes local} by Michiel Hazewinkel, Masson \& Cie, 
%%Éditeur, Paris; 
%North-Holland Publishing Co., Amsterdam, 
%1970.
% xxvi+700 pp. 

%\bibitem[Dh]{dhillon}
%G.~Dhillon, \emph{An informal introduction to categorical representation theory and the local geometric Langlands program}, preprint~\href{https://arxiv.org/abs/2205.14578}{arXiv:2205.14578}.

%\bibitem[Dod]{dodd} 
%C.~Dodd, \emph{Equivariant coherent sheaves, Soergel bimodules, and categorification
%of affine Hecke algebras}, preprint \href{https://arxiv.org/abs/1108.4028}{arXiv:1108.4028}.

\bibitem[Do]{donkin}
S.~Donkin, \emph{Rational representations of algebraic groups. Tensor products and filtration}, Lecture Notes in Mathematics 1140, Springer-Verlag, Berlin, 1985.

%\bibitem[Dou]{douglass}
%J.~M.~Douglass, \emph{An inversion formula for relative Kazhdan--Lusztig polynomials},
%Comm. Algebra \textbf{18} (1990), no. 2, 
%371--387. 
%
%%\bibitem[Fa]{faltings}
%%G.~Faltings, \emph{Algebraic loop groups and moduli spaces of bundles},
%%J. Eur. Math. Soc. (JEMS) \textbf{5} (2003), 
%%%no. 1, 
%%41--68. 
%
%% \bibitem[EGA1]{ega1}
%% A.~Grothendieck, \emph{\'El\'ements de g\'eom\'etrie alg\'ebrique (r\'edig\'es avec la collaboration de J.~Dieudonn\'e) : I. Le langage des sch\'emas},
%% Inst. Hautes \'Etudes Sci. Publ. Math. \textbf{4} (1960).
%
%\bibitem[EGNO]{egno}
%P.~Etingof, S.~Gelaki, D.~Nikshych, and V.~Ostrik, \emph{Tensor categories},
%Mathematical Surveys and Monographs 205, American Mathematical Society, Providence, RI, 
%2015. 

\bibitem[FM]{fm}
M.~Finkelberg and I.~Mirkovi\'{c}, {\em Semi-infinite flags I. Case of global
  curve $\mathbb{P}^1$}, in \emph{Differential topology, infinite-dimensional {L}ie
  algebras, and applications}, 81--112, Amer. Math. Soc. Transl. Ser. 2, vol. 194, Amer.
  Math. Soc., 1999.
  
%   \bibitem[FK]{fk}
%   E.~Freitag, R.~Kiehl, \emph{\'Etale cohomology and the Weil conjecture},
% Ergebnisse der Mathematik und ihrer Grenzgebiete 13, Springer-Verlag, 
% %Berlin, 
% 1988.
  
% \bibitem[FG]{fg}
% E.~Frenkel, D.~Gaitsgory, \emph{Local geometric Langlands correspondence and affine Kac--Moody algebras}, in \emph{Algebraic geometry and number theory}, 69--260,
% Progr. Math. 253, Birkh\"auser Boston, 
% %Boston, MA, 
% 2006.
%   
%   \bibitem[FGKV]{fgkv}
%   E.~Frenkel, D.~Gaitsgory, D.~Kazhdan, and K.~Vilonen, \emph{Geometric realization of Whittaker functions and the Langlands conjecture}, J. Amer. Math. Soc. \textbf{11} (1998), 
%   %no. 2, 
%   451--484. 
% 
% \bibitem[FGV]{fgv}
% E.~Frenkel, D.~Gaitsgory, and K.~Vilonen, \emph{Whittaker patterns in the geometry of moduli spaces of bundles on curves}
% Ann. of Math. \textbf{153} (2001), 
% %no. 3, 
% 699--748.

%\bibitem[Gab]{gabriel}
%P.~Gabriel, \emph{Des cat\'egories ab\'eliennes},
%Bull. Soc. Math. France \textbf{90} (1962), 323--448. 
  
\bibitem[G1]{gaitsgory}
D.~Gaitsgory,
\emph{Construction of central elements in the affine Hecke algebra via nearby cycles},
Invent. Math.~\textbf{144} (2001), no. 2, 
253--280. 

\bibitem[G2]{gaitsgory-2}
D.~Gaitsgory, \emph{A conjectural extension of the Kazhdan--Lusztig equivalence},
Publ. Res. Inst. Math. Sci.~\textbf{57} (2021), no.~3--4, 1227--1376.

%\bibitem[G2]{gaitsgory-app}
%D.~Gaitsgory, \emph{Braiding compatibilities}, appendix to~\cite{bez}.

%\bibitem[Gi]{ginzburg}
%V.~Ginzburg, \emph{Perverse sheaves on a loop group and Langlands' duality}, preprint \href{https://arxiv.org/abs/alg-geom/9511007}{arXiv:alg-geom/9511007} (1995).

%\bibitem[GG]{gg}
%R.~Gordon and E.~Green, \emph{Graded Artin algebras},
%J. Algebra \textbf{76} (1982), no. 1, 
%111--137.

% \bibitem[G\"o]{goertz}
% U.~G\"ortz, \emph{Affine Springer fibers and affine Deligne--Lusztig varieties}, in \emph{Affine flag manifolds and principal bundles} (A.~Schmitt, Ed.), 1--50,
% Trends Math., Birkh\"auser,
% %/Springer Basel AG, Basel, 
% 2010.

%\bibitem[EGA1]{ega1}
%A.~Grothendieck, \emph{\'El\'ements de g\'eom\'etrie alg\'ebrique (r\'edig\'es avec la collaboration de J.~Dieudonn\'e) : I. Le langage des sch\'emas},
%Inst. Hautes \'Etudes Sci. Publ. Math. \textbf{4} (1960).

\bibitem[Ha]{hartshorne-RD}
R.~Hartshorne, \emph{Residues and duality},
Lecture notes of a seminar on the work of A.~Grothendieck, given at Harvard 1963/64, with an appendix by P. Deligne, Lecture Notes in Mathematics~20, Springer-Verlag, Berlin-New York, 1966.

%\bibitem[Ha2]{hartshorne} 
%R.~Hartshorne, \emph{Algebraic geometry}, Springer, 1977.

%\bibitem[Hei]{heinloth}
%J.~Heinloth, \emph{Uniformization of $\mathscr{G}$-bundles},
%Math. Ann. \textbf{347} (2010), no. 3, 
%499--528. 

\bibitem[HZ]{hz}
T.~Hemo and X.~Zhu, \emph{Unipotent categorical local Langlands correspondence}, in preparation.

%\bibitem[Her]{herpel}
%S.~Herpel, \emph{On the smoothness of centralizers in reductive groups},
%Trans. Amer. Math. Soc. \textbf{365} (2013), no. 7, 
%3753--3774. 

%\bibitem[Hu]{humphreys}
%J.~E.~Humphreys, \emph{Conjugacy classes in semisimple algebraic groups}, American Mathematical Society, 1995.

\bibitem[IM]{im}
N.~Iwahori and H.~Matsumoto, \emph{On some Bruhat decomposition and the structure of the Hecke rings of $\mathfrak{p}$-adic Chevalley groups},
Inst. Hautes \'Etudes Sci. Publ. Math. \textbf{25} (1965), 5--48. 

\bibitem[J1]{jantzen-prime}
J.~C.~Jantzen, \emph{Representations of Lie algebras in prime characteristic}, notes by I.~Gordon, in \emph{Representation theories and algebraic geometry (Montreal, PQ, 1997)}, 185--235, NATO Adv. Sci. Inst. Ser. C Math. Phys. Sci. 514, Kluwer Acad. Publ., Dordrecht, 
1998.

\bibitem[J2]{jantzen}
J.~C.~Jantzen, \emph{Representations of Algebraic Groups, second edition}, Mathematical Surveys and Monographs 107, Amer.~Math.~Soc., 2003.

%\bibitem[J3]{jantzen-nilp}
%J.~C.~Jantzen, \emph{Nilpotent orbits in representation theory}, in \emph{Lie theory}, 1--211,
%Progr. Math. 228, Birkh\"auser Boston, 2004.

%\bibitem[Ju]{juteau}
%D.~Juteau, \emph{Modular representations of reductive groups and geometry of affine Grassmannians}, preprint~\texttt{arXiv:0804.2041}.

% \bibitem[JMW1]{jmw}
% D.~Juteau, C.~Mautner, and G.~Williamson, {\em Parity sheaves}, J. Amer. Math. Soc. {\bf 27} (2014), 1169--1212.

%\bibitem[JMW]{jmw2}
%D.~Juteau, C.~Mautner, and G.~Williamson, \emph{Parity sheaves and tilting modules}, Ann. Sci. {\'E}c. Norm. Sup{\'e}r. \textbf{49} (2016), 257--275.

\bibitem[Ka]{kashiwara}
M.~Kashiwara, \emph{Equivariant derived category and representation of real semisimple Lie groups}, in \emph{Representation theory and complex analysis}, 137--234,
Lecture Notes in Math. 1931, Springer, Berlin, 2008.

\bibitem[KL]{kl}
D.~Kazhdan and G.~Lusztig, \emph{Proof of the Deligne--Langlands conjecture for Hecke algebras},
Invent. Math. \textbf{87} (1987), no. 1, 
153--215. 

%\bibitem[Ko]{korhonen}
%M.~Korhonen, \emph{Unipotent elements forcing irreducibility in linear algebraic groups},
%J. Group Theory \textbf{21} (2018), 365--396.

%\bibitem[KLT]{klt}
%S.~Kumar, N.~Lauritzen, and J.~F.~Thomsen, \emph{Frobenius splitting of cotangent bundles of flag varieties},
%Invent.~Math.~\textbf{136} (1999), 603--621. 

%\bibitem[La]{lawther}
%R.~Lawther, \emph{Jordan block sizes of unipotent elements in exceptional algebraic groups},
%Comm. Algebra \textbf{23} (1995), no. 11, 
%4125--4156.
%
%\bibitem[Le]{lee}
%T.~Y.~Lee, \emph{Adjoint quotients of reductive groups}, in \emph{Autour des sch\'emas en groupes. Vol. III}, 131--145,
%Panor. Synthèses 47, Soc. Math. France, Paris, 
%2015.

\bibitem[Le]{letellier}
E.~Letellier, {\em Fourier transforms of invariant functions on finite
 reductive Lie algebras}, Lecture Notes in Math., vol.~1859,
 Springer-Verlag, Berlin, 2005.

%\bibitem[Lin]{lindsey}
%J.~Lindsey, II,
%\emph{Groups with a t. i. cyclic Sylow subgroup},
%J. Algebra \textbf{30} (1974), 181--235.

\bibitem[Li]{lipman}
J.~Lipman, \emph{Notes on derived functors and Grothendieck duality}, in \emph{Foundations of Grothendieck duality for diagrams of schemes}, Lecture Notes in Math. 1960, Springer, 2009, 1--259.

\bibitem[Lo1]{losev}
I.~Losev, \emph{On modular Soergel bimodules, Harish-Chandra bimodules, and category O}, preprint~\href{https://arxiv.org/abs/2302.05782}{arXiv:2302.05782}.

\bibitem[Lo2]{losev2}
I.~Losev, \emph{Quantum category O vs affine Hecke category}, preprint~\href{https://arxiv.org/abs/2310.03153}{arXiv:2310.03153}.

%\bibitem[L\"ub]{lubeck}
%F.~L\"ubeck, \emph{Small degree representations of finite Chevalley groups in defining characteristic}, LMS J. Comput. Math. \textbf{4} (2001), 135--169.

\bibitem[Lu1]{lusztig-pbs}
G.~Lusztig, \emph{Some problems in the representation theory of finite Chevalley groups}, in \emph{The Santa Cruz Conference on Finite Groups (Univ. California, Santa Cruz, Calif., 1979)}, pp. 313--317,
Proc. Sympos. Pure Math. 37, Amer. Math. Soc., Providence, RI, 1980.

%\bibitem[Lu2]{lusztig}
%G.~Lusztig, \emph{Singularities, character formulas, and a $q$-analog of weight multiplicities}, in \emph{Analysis and topology on singular spaces, II, III (Luminy, 1981)}, 208--229,
%Ast{\'e}risque \textbf{101--102}, Soc. Math. France, 1983.

\bibitem[Lu2]{lusztig-bases}
G.~Lusztig, \emph{Bases in equivariant K-theory}, Represent. Theory 2 (1998), 298--369.

% \bibitem[Ma]{mathieu}
% O.~Mathieu, \emph{Filtrations of $G$-modules},
% Ann. Sci. \'Ecole Norm. Sup. (4) \textbf{23} (1990), 
% %no. 4, 
% 625--644. 

\bibitem[MR1]{mr1}
C.~Mautner and S.~Riche, \emph{On the exotic t-structure in positive characteristic},
Int. Math. Res. Not. IMRN \textbf{2016}, no. 18, 5727--5774. 

\bibitem[MR2]{mr}
C.~Mautner and S.~Riche, {\em Exotic tilting sheaves, parity sheaves on affine
  Grassmannians, and the Mirkovi{\'c}--Vilonen conjecture},  J. Eur. Math. Soc. (JEMS) \textbf{20} (2018), no. 9, 
  2259--2332.
  %preprint \texttt{arXiv:1501.07369}, to appear in J. Eur. Math. Soc.

%\bibitem[McN]{mcninch}
%G. J. McNinch, \emph{Optimal $\mathrm{SL}(2)$-homomorphisms},
%Comment. Math. Helv. \textbf{80} (2005), 391--426.

%\bibitem[MiR1]{mir}
%I.~Mirkovi{\'c} and S.~Riche, \emph{Linear Koszul duality}
%Compos. Math. \textbf{146} (2010), 
%%no. 1, 
%233--258. 

%\bibitem[MiR2]{mir2}
%I.~Mirkovi{\'c} and S.~Riche, \emph{Iwahori--Matsumoto involution and linear Koszul duality},
%Int. Math. Res. Not. IMRN \textbf{2015}, no. 1, 150--196.

\bibitem[MV]{mv}
I.~Mirkovi{\'c} and K.~Vilonen, {\em Geometric {L}anglands duality and representations of algebraic groups over commutative rings}, Ann. of Math. {\bf 166} (2007), 95--143.

% \bibitem[Na]{nadler}
% D.~Nadler, \emph{Perverse sheaves on real loop Grassmannians},
% Invent. Math. \textbf{159} (2005), 
% %no. 1, 
% 1--73. 

%\bibitem[NP]{ngo-polo}
%B.~C.~Ng{\^o} and P.~Polo, \emph{R\'esolutions de Demazure affines et formule de Casselman--Shalika g\'eom\'etrique},
%J. Algebraic Geom. \textbf{10} (2001), 
%%no. 3, 
%515--547.

%\bibitem[PR]{pr}
%G.~Pappas and M.~Rapoport, \emph{Twisted loop groups and their affine flag varieties},
%with an appendix by T.~Haines and M.~Rapoport,
%Adv. Math. \textbf{219} (2008), no. 1, 
%118--198. 

%\bibitem[Pa]{paris}
%L.~Paris, \emph{Artin monoids inject in their groups},
%Comment. Math. Helv. \textbf{77} (2002), 609--637.

%\bibitem[PS]{ps}
%A.~Premet and I.~Suprunenko, \emph{The Weyl modules and the irreducible representations of the symplectic group with the fundamental highest weights},
%Comm. Algebra \textbf{11} (1983), no. 12, 
%1309--1342.

\bibitem[R1]{riche}
S.~Riche, \emph{Koszul duality and modular representations of semi-simple Lie algebras}, Duke Math. J.~\textbf{154} (2010), 31--134.

\bibitem[R2]{riche-kostant}
S.~Riche, \emph{Kostant section, universal centralizer, and a modular derived Satake equivalence}, Math. Z.~\textbf{286} (2017), 223--261.

%\bibitem[R3]{riche-hab}
%S.~Riche, \emph{Geometric Representation Theory in positive characteristic}, habilitation thesis, available on \url{https://tel.archives-ouvertes.fr/tel-01431526}.

%\bibitem[RSW]{rsw}
%S.~Riche, W.~Soergel, and G.~Williamson, \emph{Modular Koszul duality},
%Compos. Math. \textbf{150} (2014), 
%%no. 2, 
%273--332. 

%\bibitem[RW]{tilting}
%S.~Riche and G.~Williamson, \emph{Tilting modules and the $p$-canonical basis}, 
%%preprint \texttt{arXiv:1512.08296}, to appear in 
%Ast\'erisque \textbf{397} (2018).

\bibitem[Se]{seshadri}
C.~S.~Seshadri, \emph{Geometric reductivity over arbitrary base},
Advances in Math. \textbf{26} (1977), no. 3, 225--274.

\bibitem[Si]{situ}
Q.~Situ, \emph{Category $\mathcal{O}$ for hybrid quantum groups and non-commutative Springer resolutions}, preprint~\href{https://arxiv.org/abs/2308.07028}{arXiv:2308.07028}.

%\bibitem[So]{soergel-Kat}
%W.~Soergel, \emph{Kategorie $\mathscr{O}$, perverse Garben und Moduln \"uber den Koinvarianten zur Weylgruppe},
%J. Amer. Math. Soc. \textbf{3} (1990), no. 2, 
%421--445. 

%\bibitem[SR]{sr}
%N.~Saavedra Rivano, \emph{Cat\'egories Tannakiennes}, Lecture Notes in Mathematics 265, Springer-Verlag, 
%Berlin-New York, 
%1972.

\bibitem[SGA1]{sga1}
\emph{Rev\^etements \'etales et groupe fondamental (SGA 1)}.
S\'eminaire de g\'eom\'etrie alg\'ebrique du Bois Marie 1960--61. Directed by A. Grothendieck. With two papers by M. Raynaud. Updated and annotated reprint of the 1971 original.
%[Lecture Notes in Math., 224, Springer, Berlin; MR0354651]. 
Documents Math\'ematiques 
%(Paris) [Mathematical Documents (Paris)], 
3, Soci\'et\'e Math\'ematique de France, Paris, 
2003.

\bibitem[So]{soergel}
W.~Soergel, \emph{Kategorie $\mathscr{O}$, perverse Garben und Moduln \"uber den Koinvarianten zur Weylgruppe}, J. Amer. Math. Soc. \textbf{3} (1990), no. 2, 421--445. 

%\bibitem[Sp]{springer}
%T.~A.~Springer, \emph{Quelques applications de la cohomologie d'intersection}, in
%\emph{S\'eminaire N.~Bourbaki, Vol. 1981/82},
%Ast\'erisque \textbf{92--93} (1982), Exp. 589, 249--273.

%\bibitem[Sp]{springer-some}
%T.~A.~Springer, \emph{Some arithmetical results on semi-simple Lie algebras},
%Inst. Hautes \'Etudes Sci. Publ. Math. \textbf{30} (1966), 115--141.

% \bibitem[Sp2]{springer}
% T.~A. Springer, \emph{Linear algebraic groups, second ed.}, Progr. Math.~9, Birkh{\"a}user, 1998.

\bibitem[SP]{stacks-project}
\emph{Stacks project}, available at \url{http://stacks.math.columbia.edu}.

%\bibitem[S1]{steinberg}
%R.~Steinberg, \emph{Regular elements of semisimple algebraic groups},
%Inst. Hautes \'Etudes Sci. Publ. Math. \textbf{25} (1965), 49--80.
%
%\bibitem[S2]{steinberg-pittie}
%R.~Steinberg, \emph{On a theorem of Pittie}, Topology \textbf{14} (1975), 173--177.

\bibitem[Th]{thomason}
R.~W.~Thomason, \emph{Equivariant resolution, linearization, and Hilbert's fourteenth problem over arbitrary base schemes}, Adv. in Math. \textbf{65} (1987), no. 1, 16--34. 

%\bibitem[Ve]{verdier}
%J.-L. Verdier, {\em Sp\'ecialisation de faisceaux et monodromie mod\'er\'ee}, in
% \emph{Analyse et topologie sur les espaces singuliers, II, III (Luminy, 1981)}, 332--364,
% Ast\'erisque \textbf{101--102} (1983).
%%  1983,

% \bibitem[Wa]{wang}
% J.~Wang, \emph{A new Fourier transform}, Math. Res. Lett. \textbf{22} (2015), 
% %no. 5, 
% 1541--1562.

%\bibitem[Wi]{williamson-IC}
%G.~Williamson, \emph{Modular intersection cohomology complexes on flag varieties}, with an appendix by T.~Braden,
%Math. Z. \textbf{272} (2012), 
%%no. 3-4, 
%697--727. 

%\bibitem[Zh]{zhu-conj}
%X.~Zhu, \emph{On the coherence conjecture of Pappas and Rapoport},
%Ann. of Math. \textbf{180} (2014), no. 1, 
%1--85. 

\end{thebibliography}
\end{document}